\documentclass[a4paper, 11pt]{article}
\usepackage[utf8]{inputenc}
\usepackage[T1]{fontenc}
\usepackage{graphicx}
\usepackage{amsmath,amsfonts,amssymb,amsthm}
\usepackage{mathtools}
\usepackage{mathrsfs}
\usepackage{dsfont}
\usepackage{bm}
\usepackage{xcolor}%
\usepackage{fullpage}
\usepackage{mathtools}
\usepackage{enumitem}
\usepackage{dirtytalk}
\usepackage[normalem]{ulem}

\usepackage{hyperref}
\hypersetup{
  colorlinks, %
  linkcolor=blue, %
  citecolor=blue, %
  linktoc=all %
}

\renewcommand{\leq}{\leqslant}
\renewcommand{\geq}{\geqslant}
\newcommand{\mleq}{\preccurlyeq}
\newcommand{\mgeq}{\succcurlyeq}

\newcommand{\di}{\mathrm{d}}

\newcommand{\eps}{\varepsilon}
\renewcommand{\epsilon}{\varepsilon}

\newcommand{\argmin}{\mathop{\mathrm{arg}\,\mathrm{min}}}
\newcommand{\argmax}{\mathop{\mathrm{arg}\,\mathrm{max}}}

\newcommand{\wt}{\widetilde}
\newcommand{\wh}{\widehat}
\renewcommand{\hat}{\widehat}
\renewcommand{\tilde}{\widetilde}
\newcommand{\ol}{\overline}
\newcommand{\pp}{\, : \;}

\newcommand{\N}{\mathbb N}
\newcommand{\Zn}{\mathbb Z}

\newcommand{\R}{\mathbb R}

\renewcommand{\L}{\mathcal{L}}

\usepackage{xspace}
\newcommand{\ie}{\textit{i.e.}\@\xspace} 
\newcommand{\eg}{e.g.\@\xspace}
\newcommand{\iid}{i.i.d.\@\xspace}

\DeclarePairedDelimiter{\parens}{(}{)}
\DeclarePairedDelimiter{\bracks}{[}{]}
\DeclarePairedDelimiter{\braces}{\{}{\}}

\newcommand{\set}[1]{\{ #1 \}}

\newcommand{\Sym}{\mathfrak S}

\newcommand{\supp}{\mathrm{supp}}

\newcommand{\vspan}{\mathrm{span}}

\DeclarePairedDelimiterX{\innerp}[2]{\langle}{\rangle}{#1, #2}

\newcommand{\tr}{\mathrm{Tr}}

\DeclarePairedDelimiter{\abs}{|}{|}
\DeclarePairedDelimiter{\norm}{\lVert}{\rVert}

\newcommand{\opnorm}[1]{\|{#1}\|_{\mathrm{op}}}

\newcommand{\lamin}{\lambda_{\mathrm{min}}}

\newcommand{\sign}{\mathrm{sign}}

\DeclareMathOperator{\vol}{Vol}
\newcommand{\dist}{\mathrm{dist}}
\newcommand{\conv}{\mathrm{conv}}

\newcommand{\E}{\mathbb{E}}
\renewcommand{\P}{\mathbb{P}}
\DeclarePairedDelimiterXPP{\Probab}[1]{\P}{(}{)}{}{#1}
\DeclarePairedDelimiterXPP{\Expect}[1]{\E}{[}{]}{}{#1}
\newcommand{\Var}{\mathrm{Var}}
\newcommand{\var}{\Var}

\newcommand{\cond}{|}%

\newcommand{\probas}{\mathcal{P}}
\DeclarePairedDelimiterXPP{\indic}[1]{\bm 1}{(}{)}{}{#1}
\newcommand{\ind}[1]{\bm1\! \left\{#1\right\}}
\newcommand{\1}{\bm {1}}
\newcommand{\kl}{D}%
\newcommand{\kll}[2]{\kl(#1\Vert #2)}%

\renewcommand{\supp}[1]{\mathrm{Supp}\left(#1\right)}

\newcommand{\gaussdist}{\mathsf{N}}%
\newcommand{\normal}{\gaussdist}

\newcommand{\trh}{\widetilde{\rho}}

\newcommand{\Xs}{\mathcal{X}}

\newcommand{\A}{\mathcal{A}}

\newcommand{\mc}{\mathsf{C}}

\newcommand{\ainnerp}[2]{\abs{\innerp{#1}{#2}}}
\renewcommand{\enspace}{\,}
\newcommand{\pd}{\, \text{p.d.}}
\newcommand{\model}{\probas_{\mathrm{logit}}}
\newcommand{\cvd}{\xrightarrow{(\mathrm{d})}}
\newcommand{\X}{\bm X}
\newcommand{\V}{\bm V}
\newcommand{\Z}{\bm Z}

\newcommand{\emp}[1]{\widehat{#1}_n}
\renewcommand{\b}{B}%

\newcommand{\parag}[1]{\paragraph{#1.}}

\newtheorem{proposition}{Proposition}%
\newtheorem{theorem}{Theorem}
\newtheorem*{theorem*}{Theorem}
\newtheorem{lemma}{Lemma}

\newtheorem{fact}{Fact}
\theoremstyle{definition}
\newtheorem{definition}{Definition}

\newtheorem{assumption}{Assumption}%
\newtheorem{question}{Problem}

\theoremstyle{remark}
\newtheorem{remark}{Remark}

\title{Finite-sample performance of the maximum likelihood estimator in logistic regression}
\author{Hugo Chardon\thanks{Department of Statistics, University of California, Berkeley, USA}
  \and Matthieu Lerasle\thanks{Department of Statistics, CREST/ENSAE Paris, Palaiseau, France}
  \and Jaouad Mourtada\footnotemark[2]}
\date{\today}

\begin{document}
\maketitle

\begin{abstract}
  Logistic regression is a classical model for describing the probabilistic dependence of binary responses to multivariate covariates.
  We consider the predictive performance of the maximum likelihood estimator (MLE) for logistic regression, assessed in terms of logistic risk.
  We consider two questions: first, that of the existence of the MLE (which occurs when the dataset is not linearly separated), and second, that of its accuracy when it exists.
  These properties depend on both the dimension of covariates and the signal strength.
  In the case of Gaussian covariates and a well-specified logistic model, we obtain sharp non-asymptotic guarantees for the existence and excess logistic risk of the MLE.
  We then generalize these results in two ways: first, to non-Gaussian covariates satisfying a certain two-dimensional margin condition, and second to the general case of statistical learning with a possibly misspecified logistic model.
  Finally, we consider the case of a Bernoulli design, where the behavior of the MLE is highly sensitive to the parameter direction.
\end{abstract}

\setcounter{tocdepth}{2}
\tableofcontents

\newpage

\section{Introduction}
\label{sec:introduction}

Logistic regression~\cite{berkson1944logistic,mccullagh1989glm} is a classical model describing the dependence of binary outcomes on multivariate features.
In this work, we investigate the predictive performance of the most standard method for fitting this model, namely the maximum likelihood estimator (MLE).
Our emphasis is placed on the dependence of the estimation error on the various parameters of the problem, as well as on the conditions under which the MLE performs well.

\subsection{Problem setting and main questions}
\label{sec:problem-setting}

To set the stage for the discussion, we start by recalling the definition of the logistic (logit) model.
Given a dimension $d \geq 2$, the \emph{logistic model} is the family of conditional distributions on the outcome $y \in \{ -1, 1 \}$ given the covariates $x \in \R^d$ defined by:
\begin{equation}
  \label{eq:def-logistic-model}
  \model
  = \big\{ p_\theta : \theta \in \R^d \big\} \, ,
  \quad \text{where} \quad
  p_\theta (y | x) = \sigma (y \innerp{\theta}{x})
  \quad \text{for } (x,y) \in \R^d \times \{-1, 1\}
  \, ,
\end{equation}
where we let $\sigma (s) = e^s / (e^s+1)$ for $s \in \R$ be the sigmoid function, and where $\innerp{\cdot}{\cdot}$ denotes the usual dot product on $\R^d$.
We say that a random pair $(X,Y)$ on $\R^d \times \{ -1, 1\}$ follows the logistic model if the conditional distribution of $Y$ given $X$ belongs to $\model$.

In short, the logistic model is appealing because it constitutes a natural ``linear'' model for
binary outcomes: indeed, the conditional probability $\P (Y=1|X=x)$ is obtained by applying
the link function $\sigma : \R \to [0,1]$ to a linear function of $x$.
Note that the specific choice of the link function $\sigma$ in the logistic
model is not arbitrary: it corresponds to the ``canonical link function'' for the Bernoulli
parameter, in the sense of %
exponential families~\cite{brown1986fundamentals}.

In the statistical setting, the true distribution $P$ of the random pair $(X,Y)$ is unknown,
but one has access to a sample $(X_1, Y_1), \dots, (X_n, Y_n)$ of \iid random variables
with distribution $P$.
Using this sample, one can compute the MLE, defined by
\begin{equation}
  \label{eq:def-mle}
  \wh \theta_n
  = \argmax_{\theta \in \R^d} \prod_{i=1}^n p_{\theta} (Y_i|X_i)
  = \argmin_{\theta \in \R^d} \sum_{i=1}^n \log \big( 1 + e^{-Y_i \innerp{\theta}{X_i}} \big)
  \, .
\end{equation}
In this work, we will be concerned with the following two questions:
\begin{enumerate}
\item \emph{Existence}: When does the MLE exist?
\item \emph{Performance}: When the MLE exists, how accurate is it?
\end{enumerate}
To make these two questions precise, some discussion is in order.

First, we must clarify the geometric meaning of existence (and uniqueness) of the MLE; 
we refer to~\cite{albert1984existence} and to the introduction of~\cite{candes2020phase} for an interesting discussion of this point, with thorough references.
Uniqueness of the MLE is in fact a straightforward question: whenever the points $X_1, \dots, X_n$ span $\R^d$ (a property that holds with high probability for $n \gtrsim d$, under suitable assumptions on $X$), the second function in~\eqref{eq:def-mle} that the MLE minimizes is strictly convex on $\R^d$, and thus admits at most one minimizer.
The property of existence of the MLE has a richer geometric content.
Assume again for simplicity that $X_1, \dots, X_n$ span $\R^d$, so that for every $\theta \neq 0$, there exists $i \in \set{1, \dots, n}$ such that $\innerp{\theta}{X_i} \neq 0$.
Then, \emph{the MLE exists if and only if the dataset is not linearly separated}, by which we mean that
there is no $\theta \neq 0$ such that $Y_i \innerp{\theta}{X_i} \geq 0$ for every $i=1, \dots, n$.
Indeed, if such a $\theta$ exists, then the second function in~\eqref{eq:def-mle} evaluated at $t \theta$ 
remains upper bounded as $t \to + \infty$; since a strictly convex function admitting a global minimizer diverges at infinity, the objective function admits no global minimizer.
Conversely, if no such $\theta$ exists, then simple compactness arguments show that the function in~\eqref{eq:def-mle} diverges at infinity and is continuous, hence admits a global minimizer.

Second, in order to assess the performance of the MLE, one must specify a notion of accuracy.
In this work, we will mainly focus on the predictive performance of the MLE, as measured by its risk for prediction under logistic loss.
Specifically, we consider the problem of assigning probabilities to the possible values $\pm 1$ of $Y$, given the knowledge of the associated covariate vector $X$.
Each parameter $\theta \in \R^d$ gives rise to the conditional distribution $p_\theta$ defined in~\eqref{eq:def-logistic-model}.
We can then define the logistic loss $\ell$ (at a point $(x,y) \in \R^d \times \{ -1, 1 \}$) and risk $L$ of $\theta$ by, respectively:
\begin{equation}
  \label{eq:loss-risk}
  \ell (\theta, (x,y))
  = - \log p_\theta (y|x)
  = \log (1 + e^{-y \innerp{\theta}{x}})
  \quad \text{and} \quad
  L (\theta)
  = \E [ \ell (\theta, (X,Y)) ]
  \, .
\end{equation}
Hence, the logistic loss corresponds to the negative log-likelihood (or logarithmic loss) for the logistic model.
The logarithmic loss is a classical way to assess the quality of probabilistic forecasts: it enforces calibrated predictions by penalizing both overconfident and under-confident probabilities.
In particular, assigning a probability of $0$ to a label $y$ that does appear leads to an infinite loss.
In addition, this criterion is closely related to the MLE, which corresponds to the minimizer of the \emph{empirical risk} $\wh L_n : \R^d \to \R$ under logistic loss, defined by
\begin{equation}
  \label{eq:def-empirical-risk}
  \wh L_n (\theta)
  = \frac{1}{n} \sum_{i=1}^n \ell (\theta, (X_i, Y_i))
  = \frac{1}{n} \sum_{i=1}^n \log (1+ e^{-Y_i \innerp{\theta}{X_i}})
  \, .
\end{equation}
Finally, the logistic loss also naturally arises in statistical learning theory~\cite{boucheron2005survey,koltchinskii2011oracle,bach2024learning}, as a convex surrogate of the classification error~\cite{zhang2004statistical,bartlett2006convexity}.
With these definitions at hand, one can measure the prediction accuracy of the MLE by its excess risk under logistic loss, namely
$L (\wh \theta_n) - L (\theta^*)$, where $\theta^* \in \argmin_{\theta \in \R^d} L (\theta)$ (assuming this set is nonempty).

Thirdly, the existence of the MLE depends on the dataset and is thus a random event, and likewise the excess risk $L (\wh \theta_n) - L (\theta^*)$ is a random quantity.
As such, both the existence and the accuracy of the MLE depend on the joint distribution $P$ of $(X,Y)$.
To give a precise meaning to the questions above, we must therefore specify which distributions $P$ we consider.
Note that the joint distribution $P$ is characterized by (a) the marginal distribution $P_X$ of $X$, and (b) the conditional distribution $P_{Y|X}$ of $Y$ given $X.$

We will actually consider three different settings of increasing generality, depending on the respective assumptions on $P_X$ and $P_{Y|X}$, but for concreteness and in order to compare with previous results, we will in this introduction start with the simplest one:
\begin{enumerate}
\item[(a)] The design follows a Gaussian distribution: $X \sim \gaussdist (0, \Sigma)$ for some positive matrix $\Sigma$.
  By invariance of the problem under invertible linear transformations of $X$, we may assume that $\Sigma = I_d$ is the identity matrix, which we will do in what follows.
\item[(b)] The model is \emph{well-specified}, in that the conditional distribution $P_{Y|X}$ belongs to the logistic model $\model$.
  In other words, there exists $\theta^* \in \R^d$ such that $\P (Y=1|X) = \sigma (\innerp{\theta^*}{X})$.
\end{enumerate}
Besides its natural character, the appeal of this setting is that the problem only depends on a small number of parameters.
These are: 
the sample size $n$, the data dimension $d$, the probability $1-\delta$ with which the guarantees hold,
and importantly the \emph{signal strength} (or signal-to-noise ratio, or inverse temperature) $\b = \max \{ e, \norm{\theta^*}\}$, where $\norm{\cdot}$ stands for the Euclidean norm.

It is worth commenting on the role of the dimension $d$ and of the signal strength $\b$.
Intuitively, there are two distinct effects that may lead the dataset to be linearly separated.
First, the larger the dimension $d$, the more degrees of freedom there are to linearly separate the dataset.
But another effect comes from the signal strength: the stronger the signal $\b$, the more the labels $Y_i$ tend to be of the same sign as $\innerp{\theta^*}{X_i}$---and thus, the more likely it is for the dataset to be separated by $\theta^*$, or by a ``close'' direction.
As we will see, the ``dimensionality'' and the ``signal strength'' effects interact with each other.
We also note that, intuitively, a stronger signal should make the \emph{classification} problem (of predicting the value of the label $Y$, and minimizing the fraction of errors) easier.
This amounts to saying that the larger $\b$ is, the smaller the estimation error for the \emph{direction} $u^* = \theta^* / \norm{\theta^*}$ of the parameter $\theta^*$ should be.
On the other hand, under a stronger signal, the MLE is known (see, \eg,~\cite{sur2019modern} and references therein) to tend to underestimate the uncertainty in the labels, that is, to return overconfident conditional probabilities for $Y$ given $X$.
This holds, for instance, if the dataset is nearly linearly separated, in which case the MLE predicts conditional probabilities close to $0$ or $1$.
Hence, for the \emph{conditional density estimation} problem we consider, a stronger signal may degrade the performance of the MLE.
This should manifest itself by the fact that the \emph{norm} of the MLE (as opposed to its direction) may be far from that of $\theta^*$, so the overall estimation error of $\theta^*$ may be larger.

To summarize, we are interested in explicit and \emph{non-asymptotic} guarantees for the existence and accuracy of the MLE, in terms of the relevant parameters $\b, d, n, \delta$---ideally, in the general situation where these parameters may take arbitrary values.
Our aim is twofold: first, to obtain the optimal dependence on all parameters in the case of a Gaussian design and a well-specified model; second, to investigate to which extent these results extend to more general distributions.

\subsection{Existing results}
\label{sec:existing-results}

Before describing our contributions, we first provide an overview of known results on the questions we consider.
As a basic statistical method, logistic regression has been studied extensively in the literature, hence we focus on those results that are most directly relevant to our setting.
Again, for the sake of comparison, we will
mainly focus on
the case of a Gaussian design and a well-specified model, although extensions will also be discussed.

\parag{Classical asymptotics}

The behavior of the MLE is well-understood in the context of
classical parametric asymptotics~\cite{lecam2000asymptotics,vandervaart1998asymptotic}.
In this setting, the distribution $P$ is fixed (and thus, so are the dimension $d$ and signal strength $\b$) while the sample size $n$ goes to infinity.
In this case, as $n \to \infty$, the MLE $\wh \theta_n$ exists with probability converging to $1$, converges to $\theta^*$ at a $1/\sqrt{n}$ rate, and is asymptotically normal, with asymptotic covariance given by the inverse of the Fisher information matrix~\cite[\S5.2--5.6]{vandervaart1998asymptotic}.
This implies that the excess risk converges to $0$ at a rate $1/n$, and more precisely that
\begin{equation}
  \label{eq:mle-wilks}
  2n \big\{ L (\wh \theta_n) - L (\theta^*) \big\}
  \cvd \chi^2 (d)
  \, ,
\end{equation}
where $\cvd$ denotes convergence in distribution and $\chi^2(d)$ denotes the $\chi^2$ distribution with $d$ degrees of freedom.
Together with a tail bound for the $\chi^2$ distribution, this implies the following:
for fixed $d \geq 1$, $\theta^* \in \R^d$, and $\delta \in (0, 1)$,
we have
\begin{equation}
  \label{eq:lowdim-asymptotic-mle}
  \liminf_{n \to \infty} \P \Big( L (\wh \theta_n) - L (\theta^*) \leq \frac{d + 2 %
    \log (1/\delta)
  }{n} \Big)
  \geq %
  1-\delta
  \, ,
\end{equation}
with the convention that $L(\wh \theta_n) - L (\theta^*) = + \infty$ if the dataset is linearly separated.
Note that the convergence~\eqref{eq:mle-wilks} holds only in the well-specified case, and that in the misspecified case the normalized excess risk $2 n \{ L (\wh \theta_n) - L (\theta^*) \}$ converges to a different limiting distribution that depends on the distribution $P$ of $(X, Y)$; see~\cite[Example~5.25 p.~55]{vandervaart1998asymptotic} and (for instance) the introductions of~\cite{ostrovskii2021finite,mourtada2022logistic} for additional discussions on this point.

On the positive side, the high-probability guarantee~\eqref{eq:lowdim-asymptotic-mle}
is sharp, in light of the convergence in distribution~\eqref{eq:mle-wilks} of the excess risk.
On the other hand, it should be noted that this guarantee is purely asymptotic:
it holds as $n \to \infty$ while all other parameters of the problem are fixed.
This does not allow one to handle the modern high-dimensional regime, where
the dimension $d$ may be large and possibly comparable to $n$.
In addition, it does not state how large the sample size $n$ should be
(in terms of $\b,d, \delta$) for the asymptotic behavior~\eqref{eq:lowdim-asymptotic-mle}
to occur---in particular, it provides no information on the sample size required for the
existence of the MLE.

\parag{High-dimensional asymptotics}

Several 
of the shortcomings of the classical asymptotic theory can be addressed by considering a different asymptotic framework, namely the ``high-dimensional asymptotic regime'', where $d, n \to \infty$ while $d/n$ converges to a fixed constant.
This framework has attracted significant interest in statistics over the last decade (see, \eg, \cite{elkaroui2018random,montanari2018mean} and references therein for a partial overview of this line of work).
The interest of this framework is that it allows one to capture high-dimensional effects, since the dimension is no longer negligible compared to the sample size.

The question of existence of the MLE under high-dimensional asymptotics was addressed
in~\cite{candes2020phase}, extending a previous result of Cover~\cite{cover1965geometrical} in the ``null'' case where $\theta^* = 0$.
Specifically, the main result of Candès and Sur~\cite[Theorems~1--2]{candes2020phase} can be stated as follows:
there exists a function $h : \R^+ \to (0,1)$ such that the following holds.
Fix $b \in \R^+$ and $\gamma \in (0, 1)$, and let $d = d_n \to \infty$ as $n \to \infty$, with $d/n \to \gamma$.
If $X \sim \gaussdist (0, I_d)$ and $\P (Y=1 | X) = \sigma (\innerp{\theta^*}{X})$, with $\theta^* = \theta^*_d \in \R^d$ such that $\norm{\theta^*} = b$, and the dataset consists of $n$ \iid copies of $(X,Y)$, then 
\begin{align}
  \label{eq:phase-transition-mle}
  \lim_{n \to \infty} \P (\text{MLE exists})
  =
  \begin{cases}
    0 & \text{if } \gamma > h (b) \, ; \\
    1 & \text{if } \gamma < h (b)
        \, .
  \end{cases}
\end{align}
In addition, the quantity $h (b)$ is defined as the infimum of the expectation of an explicit family of random variables that depend on $b$
(see eq.~\eqref{eq:boundary-function} below)
and the curve of the function $h$ is plotted numerically in this paper.

The conditions~\eqref{eq:phase-transition-mle} provide a precise characterization
of the existence of the MLE under high-dimensional asymptotics, and in particular
establish a sharp phase transition for this property, depending on the value of the
aspect ratio $\gamma = \lim d/n$.

While this result conclusively answers the question of existence of the MLE in this asymptotic setting, it does not cover the general regime where the problem parameters may be of arbitrary order of magnitude relative to each other.
Indeed, although this regime captures high-dimensional effects by allowing the dimension to grow with the sample size, it assumes the signal strength $\b$ to be fixed while $n \to \infty$.
This excludes ``strong signal'' regimes, where the sample size may not be large enough relative to $\b$ for the asymptotic characterization~\eqref{eq:phase-transition-mle} to provide an accurate approximation.
As an example, a finite-sample
condition of the form $n \gg \exp(\b)$ would always be satisfied under high-dimensional asymptotics, and thus would not be visible from results framed in this setting.
In addition, the characterization~\eqref{eq:phase-transition-mle} is a qualitative zero-one law, stating that the considered probability converges to $0$ or $1$.
However, one may wish for more precise information, namely
sharp
quantitative estimates on the probabilities.

Finally, the characterization~\eqref{eq:phase-transition-mle} is specific to the case of a Gaussian design, and indeed one should expect the precise threshold for existence of the MLE to be sensitive to the distribution of the design (see~\cite{elkaroui2018predictor} for %
results in this spirit
in the case of robust regression).
One may therefore want to identify general conditions on the design distribution under which the MLE behaves in a similar way as for Gaussian design.
Likewise, the characterization~\eqref{eq:phase-transition-mle} holds in the well-specified case, which
raises the question of existence of the MLE 
in the misspecified case.

These considerations motivate a finite-sample analysis that would allow one to handle general values of the problem parameters, and extend to more general situations.
We would like however to clarify that
the finite-sample results do not imply the asymptotic ones:
indeed, the non-asymptotic characterizations we will obtain feature universal constant factors (and even in some cases logarithmic factors), while the asymptotic characterization~\eqref{eq:phase-transition-mle} is precise down to the numerical constants.
This loss in precision may be a price to pay for a non-asymptotic analysis in the general regime;
on the positive side, it will allow us to obtain
conditions that are easier to interpret.
For these reasons, we view the finite-sample and asymptotic perspectives as complementary.

\parag{Non-asymptotic guarantees}

We now discuss available non-asymptotic guarantees for the MLE in logistic regression from the literature, focusing on those that are most relevant to our setting.
First, it follows from the results of~\cite{chinot2020robust}
(specifically, combining Theorems~1 and~8 therein) that there is a constant $c > 0$ such that the following holds: if $n \geq e^{c \b} d$, then with probability $1 - 2 e^{-d/c}$ the MLE $\wh \theta_n$ exists and satisfies
\begin{equation}
  \label{eq:risk-bound-cll}
  L (\wh \theta_n) - L (\theta^*)
  \leq \frac{e^{c \b} d}{n}
  \, .
\end{equation}
On the positive side, this result is fully explicit and features an optimal dependence on the sample size $n$ and dimension $d$; the probability $1-e^{-d/c}$ under which the $O_{\b}(d/n)$ bound holds is also optimal, in light of the asymptotic results~\eqref{eq:mle-wilks} and~\eqref{eq:lowdim-asymptotic-mle}.
On the other hand, the dependence on the signal strength $\b$ is exponential, which turns out to be highly suboptimal for a Gaussian design.
In fact, the bound~\eqref{eq:risk-bound-cll} holds in a more general setting, where the model may be misspecified and where the design is only assumed to be sub-Gaussian.
As we will discuss below, some exponential dependence on the norm turns out to be unavoidable if one only assumes the design distribution to be sub-Gaussian.

Up until recently, the sharpest available non-asymptotic guarantees for the MLE in logistic regression with a Gaussian design were due to
Ostrovskii and Bach~\cite{ostrovskii2021finite}.
Specifically, combining Theorem~4.2 in~\cite{ostrovskii2021finite} with Proposition~D.1 therein shows that there is a constant $c > 0$ such that the following holds: for $\delta \leq 1/2$, if $n \geq c \log^4 (\b) \b^8 d \log (1/\delta)$, then with probability at least $1-\delta$ the MLE exists and satisfies
\begin{equation}
  \label{eq:risk-bound-ostr-bach}
  L (\wh \theta_n) - L (\theta^*)
  \leq \frac{\b^3 d \log (1/\delta)}{n}
  \, .
\end{equation}
Like the bound~\eqref{eq:risk-bound-cll}, this result features an optimal dependence on the dimension $d$ and sample size $n$; and while the bound involves a deviation term $d \log (1/\delta)$ proportional to the dimension (which is suboptimal for small $\delta$), it could be tightened to an additive deviation term $d + \log (1/\delta)$ with very minor changes to the proof of~\cite{ostrovskii2021finite}.
Importantly, this result significantly improves over the general bound~\eqref{eq:risk-bound-cll} in the case of a Gaussian design, by replacing the exponential dependence on the norm $\b$ by a polynomial one.
In addition, it is worth mentioning that the result of~\cite{ostrovskii2021finite} holds in the general misspecified case, and that in this case it is actually the best available guarantee in the literature.
This being said, as we will see below, the polynomial dependence on $\b$ in both the condition for existence of the MLE and in the risk bound can be improved.
For instance, in the well-specified case, the risk bound~\eqref{eq:risk-bound-ostr-bach} is larger than the asymptotic risk~\eqref{eq:lowdim-asymptotic-mle} by a factor of $\b^3$, which suggests possible improvements.

Recently and while this manuscript was under preparation, two additional works~\cite{kuchelmeister2023finite,hsu2023logistic} contributed significantly to the study of logistic regression with a Gaussian design, with an emphasis on the dependence on the signal strength $\b$.
Closest to our setting is the work of Kuchelmeister and van de Geer~\cite{kuchelmeister2023finite}, who study the MLE for logistic regression under a Gaussian design, but assuming that the conditional distribution of $Y$ given $X$ follows a probit rather than a logit model.
Despite real technical differences between the probit and logit models,
this is qualitatively related to the well-specified logit model.
With a natural notion of signal strength $\b$ in the probit model (the inverse of the parameter $\sigma$ in their work), Theorem~2.1.1 in~\cite{kuchelmeister2023finite} states that: for some absolute constant $c$, if $n \geq c \b (d \log n + \log (1/\delta))$, the MLE exists and satisfies
\begin{equation}
  \label{eq:result-kuch-vdg}
  \bigg\| \frac{\wh \theta_n}{\norm{\wh \theta_n}} - \frac{\theta^*}{\norm{\theta^*}} \bigg\|
  \leq c \sqrt{\frac{d \log n + \log (1/\delta)}{\b n}}
  \, , \quad
  \big| \norm{\wh \theta_n} - \norm{\theta^*} \big|
  \leq c \b^{3/2} \sqrt{\frac{d \log n + \log (1/\delta)}{n}}
  \, .
\end{equation}
While the bound~\eqref{eq:result-kuch-vdg} controls the estimation errors on the norm and direction of the parameter, we note that it can be equivalently restated in terms of excess logistic risk, as
\begin{equation}
  \label{eq:risk-bound-kuch-vdg}
  L (\wh \theta_n) - L (\theta^*)
  \leq c' \, \frac{d \log n + \log (1/\delta)}{n}
\end{equation}
for some constant $c' > 0$.
This guarantee matches the asymptotic risk~\eqref{eq:lowdim-asymptotic-mle} up to an additional $\log n$ factor, and as we will show below the condition for existence of the MLE from~\cite{kuchelmeister2023finite} is also almost sharp up logarithmic factors.
We also note that further results on linear separation in more general contexts have been obtained by Kuchelmeister~\cite{kuchelmeister2024probability} using tools from stochastic geometry.
In addition, as discussed in~\cite{cover1965geometrical,kuchelmeister2024probability}, by convex separation, the absence of linear separation (hence, existence of the MLE) amounts to the property that the origin $0 \in \R^d$ belongs to the interior of the convex hull $\conv (Y_1 X_1, \dots, Y_n X_n)$.
This allows one to connect to the rich literature concerned with the latter property, see \eg~\cite{wendel1962problem,cover1965geometrical,wagner2001continuous,kabluchko2020absorption,hayakawa2023estimating}.
We discuss the implications of these results in more detail in Section~\ref{sec:well-spec-gaussian}.

Hsu and Mazumdar~\cite{hsu2023logistic} consider the problem of estimating the parameter direction $\theta^*/\norm{\theta^*}$ (which suffices for the task of classification, namely of predicting the most likely value of $Y$ given $X$, as opposed to estimating conditional probabilities), again with an emphasis on the dependence on the signal strength $\b$.
Like~\cite{kuchelmeister2023finite} they consider the case of a Gaussian design, but assume that the data follows a logit model rather than a probit model.
Notably, they consider different estimators than the MLE for logistic regression, in particular the minimizer of a classification error.
They establish upper bounds on the estimation error of the same order as the first bound in~\eqref{eq:result-kuch-vdg}, again with logarithmic factors in $n$.
In addition, they establish minimax lower bounds on the estimation error of $\theta^*/\norm{\theta^*}$, which show that the previous upper bound is sharp up to logarithmic factors.
They also explicitly raise the question of whether or not the MLE achieves optimal upper bounds.

While these results constitute decisive advances, they leave some important questions.
First, the guarantees feature additional logarithmic factors in the sample size, which are presumably suboptimal but %
seem hard to avoid
in the analyses of~\cite{kuchelmeister2023finite} and~\cite{hsu2023logistic}, leaving a gap between upper and lower bounds.
Although logarithmic factors are admittedly a mild form of suboptimality,
logistic regression with a Gaussian design is arguably a basic enough problem 
to justify aiming for sharp results.
Second and perhaps more importantly, these results are specific to the case of a Gaussian design and a well-specified model, which raises the question of the behavior of the MLE for more general design distributions or under a misspecified model.

\subsection{Summary of contributions}
\label{sec:summ-contr}

We are now in position to %
provide a high-level overview of
our main results; we refer to subsequent sections for precise statements and additional comments.

\parag{Gaussian design, well-specified model}

First, in the case of a Gaussian design and a well-specified logit model, Theorem~\ref{thm:gaussian-well-specified} provides optimal (up to absolute constants) guarantees for the existence and accuracy of the MLE.
Specifically,
there exists a universal constant $c$ such that the following holds: for any $\delta \leq 1/2$, if $n \leq c^{-1} \b (d + \log (1/\delta))$, then
\begin{equation}
  \P ( \text{MLE exists} )
  \leq 1 - \delta
  \, .
\end{equation}
On the other hand, if $n \geq c \b (d + \log (1/\delta))$, then with probability \emph{at least} $1-\delta$ the MLE exists and satisfies
\begin{equation}
  \label{eq:risk-bound-gaussian-wellspec}
  L (\wh \theta_n) - L (\theta^*)
  \leq c \, \frac{d + \log (1/\delta)}{n}
  \, .
\end{equation}
This removes a $\log n$ factor from the bound~\eqref{eq:risk-bound-kuch-vdg} deduced from the work of~\cite{kuchelmeister2023finite} in the case of a probit model, and answers in the affirmative
(after translating this risk bound into a bound on the estimation error)
a question from~\cite{hsu2023logistic} on the optimality of the MLE.

In short, this result provides necessary and sufficient conditions on the sample size $n$ (up to numerical constant factors) for the MLE to exist with high probability, and shows that in the regime where the MLE exists, it achieves non-asymptotically the same risk as predicted by the asymptotic behavior~\eqref{eq:lowdim-asymptotic-mle} for fixed $\b,d, \delta$ and $n \to \infty$.

The previous result implies in particular that, if $n \gg \b d$, then
the MLE exists with probability at least $1 - \exp \big( - \frac{n}{c' \b} \big)$ for some constant $c'$, and that this estimate is optimal.
This %
provides
a quantitative version of the convergence to $1$ in the phase transition~\eqref{eq:phase-transition-mle} from~\cite{candes2020phase}.
On the other hand, in the regime where $n\ll \b d$, Theorem~\ref{thm:gaussian-well-specified} only shows that the probability of existence of the MLE is bounded by a constant (say, $1/2$), rather than converging to $0$ as in the phase transition~\eqref{eq:phase-transition-mle}.
We therefore complement Theorem~\ref{thm:gaussian-well-specified} by a result on non-existence of the MLE (Theorem~\ref{thm:strong-non-existence}), which states that if $n \ll \b d / \kappa$ for some $\kappa \geq 1$, then for some constant $c$,
\begin{equation}
  \label{eq:non-existence-mle-informal}
  \P ( \text{MLE exists} )
  \leq c \exp \big( - \max \big\{ \kappa \sqrt{d}, \kappa^2 d / \b^2 \big\} / c \big)
\end{equation}
This can be seen as a quantitative version of the convergence to $0$ in the phase transition~\eqref{eq:phase-transition-mle}.

\parag{Regular design, well-specified model}

The previous results are specific to the case of a Gaussian design, which can be seen as the most favorable case.
This raises the following natural question: which properties of the Gaussian distribution are responsible for the previously described behavior of the MLE?
Or equivalently, for which distributions of the design does the MLE behave (at least in the well-specified case) similarly as for a Gaussian design?

Perhaps a natural guess is that a
light-tailed design distribution
would lead to a similar behavior as a Gaussian design, and indeed this would be the case for linear regression.
However, this is far from being true for logistic regression: as previously alluded to, if the design distribution is only assumed to be sub-Gaussian (as in~\cite{chinot2020robust}), then an exponential dependence on the norm is unavoidable.

In Section~\ref{sec:assumptions}, we %
identify suitable
assumptions on the design distribution leading to a near-Gaussian behavior.
Aside from light tails (Assumption~\ref{ass:sub-exponential}), the assumptions include a condition on the behavior of one-dimensional linear projections of the design near $0$ (Assumption~\ref{ass:small-ball}), which is related to standard margin conditions in the classification literature~\cite{mammen1999margin,tsybakov2004aggregation}.
However, as shown in Proposition~\ref{prop:2-dim-margin-almost-necessary}, another assumption is necessary to obtain a near-Gaussian behavior (in a suitable sense); this non-standard %
condition (Assumption~\ref{ass:twodim-marginals}) bears on \emph{two-dimensional} linear projections of the design, rather than merely its one-dimensional marginals.
By analogy with the standard (one-dimensional) margin condition, we refer to this condition as ``two-dimensional margin assumption''.

Under these regularity assumptions on the design but still in the well-specified case, Theorem~\ref{thm:regular-well-specified} shows that the MLE behaves similarly as in the Gaussian case, up to poly-logarithmic factors in the norm $\b$.
Specifically, for some constant $c$ (depending on the constants of the regularity conditions), if $n \geq c \log^4 (\b) \b (d + \log (1/\delta))$, then with probability $1-\delta$ the MLE exists and satisfies
\begin{equation}
  \label{eq:risk-regular-wellspec-informal}
  L (\wh \theta_n) - L (\theta^*)
  \leq c \log^4 (\b) \, \frac{d + \log (1/\delta)}{n}
  \, .
\end{equation}

\parag{Regular design, misspecified model}

Finally, we turn to
the most general setting, where no assumption is made on the conditional distribution of $Y$ given $X$; in particular, it is no longer assumed that it belongs to the logit model.
This being said, as previously discussed it is still possible to define the minimizer $\theta^*$ of the logistic risk, and to consider the excess risk $L (\wh \theta_n) - L (\theta^*)$ of the MLE, which corresponds to the empirical risk minimizer (ERM) under the logistic loss.
This corresponds to the problem of Statistical Learning
under logistic loss.

As discussed in Section~\ref{sec:existing-results}, in many regimes of interest the best available guarantees for this problem in the literature are those from~\cite{ostrovskii2021finite}, namely the excess risk bound~\eqref{eq:risk-bound-ostr-bach} of order $\b^3 d \log(1/\delta)/n$, when the design is Gaussian but the model may be misspecified.
Theorem~\ref{thm:regular-misspecified} below improves these guarantees in the following way: it shows that if the design is regular (in the same sense as before) and $n \geq c \log^4 (\b) (\b d + \b^2 \log (1/\delta))$, with probability at least $1-\delta$ the MLE exists and satisfies
\begin{equation}
  \label{eq:risk-regular-misspecified-informal}
  L (\wh \theta_n) - L (\theta^*)
  \leq c \log^4 (\b) \frac{d + \b \log (1/\delta)}{n}
  \, ;
\end{equation}
here and as in~\eqref{eq:risk-regular-wellspec-informal}, $c$ is a constant that depends on the constants from the regularity conditions on $X$.
For instance, for constant $\delta$ and $\b \lesssim d$, this removes a factor of
almost $\b^3$
from the previous best guarantee~\eqref{eq:risk-bound-ostr-bach} for statistical learning with logistic loss.

Our guarantees in the misspecified case feature a stronger dependence on the norm $\b$ than in the well-specified case; specifically, the ``deviation terms'' (those that depend on the failure probability $\delta$) in both the condition for existence of the MLE and its excess risk bound are larger by a factor of $\b$.
This raises the question of whether this gap is essential or an artifact of the analysis.
As it turns out, even for a Gaussian design, the guarantee~\eqref{eq:risk-regular-misspecified-informal} is best possible (up to $\mathrm{polylog}(\b)$ factors) in the general misspecified setting,
both in the condition for existence of the MLE and for its excess risk bound,
as shown in
Theorem~\ref{thm:regular-misspecified}.

\parag{Distributions satisfying the regularity assumptions}

While identifying conditions on the design distribution under which the MLE behaves as in the Gaussian case may be interesting in its own right, 
for these general assumptions to constitute a genuine extension of the Gaussian case one must exhibit other meaningful examples of distributions satisfying them.

To illustrate these conditions, in Section~\ref{sec:regular-designs} we consider two families of distributions, namely log-concave distributions and product measures.
The log-concave case is overall similar to the Gaussian case (which it contains as a special case), in that the regularity conditions hold for any value of the parameter $\theta^* \in \R^d$---that is, for any parameter direction $u^* = \theta^*/\norm{\theta^*}$ and signal strength $\b= \max \{ e, \norm{\theta^*} \}$.
In contrast, the case of product measures is more subtle, since the regularity conditions are highly sensitive to the parameter direction $u^*$.
Indeed, we show that for designs with \iid coordinates (a prototypical example being the \emph{Bernoulli design}, with \iid coordinates uniform over $\set{-1, 1}$),
depending on the parameter direction the MLE may behave as in the Gaussian case either only for trivial (constant) signal strength $\b = O (1)$, or up to a large signal strength $\b = O (\sqrt{d})$.

\subsection{Additional related work}
\label{sec:addit-relat-work}

We now survey additional relevant prior work on logistic regression, beyond the results discussed in Section~\ref{sec:existing-results}.

\parag{Logistic regression as convex statistical learning}

Logistic regression is a special case of convex statistical learning (or convex stochastic optimization), allowing to leverage guarantees from this setting.
For instance, a standard uniform convergence argument using the Lipschitz property of logistic loss implies an excess risk bound of order $\b \sqrt{d/n}$ for ERM over a ball of radius $O(\b)$.
This upper bound exhibits a slow convergence rate of $n^{-1/2}$ as $n \to \infty$, as opposed to the actual asymptotic rate of $n^{-1}$.

In order to improve over the slow rate, a strengthening of mere convexity is needed, in the form of additional assumptions on the curvature of the loss or risk.
A common notion of curvature in optimization is strong convexity~\cite{boyd2004convex,bubeck2015convex,bach2024learning} of the loss, however the logistic loss
is not strongly convex with respect to the regression parameter $\theta$ as it only varies in one direction.
A more appropriate notion of curvature is ``exponential concavity'' (exp-concavity), which originates from online learning~\cite{vovk1998mixability,cesabianchi2006plg}.
Using this property of logistic loss, it is shown in~\cite{puchkin2023exploring} (see also~\cite{mehta2017expconcave,vilmarest2021stochastic}) that ERM for logistic regression constrained to a ball of radius $O (\b)$ achieves an excess risk of at most $O (d e^{c \b R}/ n)$, where $R> 0$ is such that $\norm{X} \leq R$ almost surely.
In the isotropic case where $\E [X X^\top] = I_d$, one has $R \geq \E [\norm{X}^2]^{1/2} = \sqrt{d}$, hence the previous guarantee scales at best as $d e^{c \b \sqrt{d}}/n$, with an exponential dependence on the norm $\b$ and (square root of the) dimension $d$.
This reflects the fact that logistic loss only possesses very weak deterministic curvature.

Another relevant property of logistic loss is (pseudo-)self-concordance (a bound on the third derivative of the loss in terms of the second derivative) which was put forward by~\cite{bach2010logistic}, and used to analyze logistic regression in a series of works~\cite{bach2010logistic,bach2014logistic,bach2013nonstrongly,ostrovskii2021finite}, the sharpest results in this direction being those of Ostrovskii and Bach~\cite{ostrovskii2021finite} discussed in Section~\ref{sec:existing-results}.

Finally, a classical condition to obtain fast rates for ERM in Statistical Learning Theory is a bound on the variance of loss differences in terms of the excess risk~\cite{massart2007concentration,koltchinskii2011oracle}, an assumption known as Bernstein condition~\cite{bartlett2006empirical}.
General guarantees for ERM in statistical learning under convex and Lipschitz loss are obtained in~\cite{alquier2019estimation} using the Bernstein condition.
These results are refined in the work~\cite{chinot2020robust} by using a local version of this condition; we discussed the instantiation of their results to logistic regression in Section~\ref{sec:existing-results}.

\parag{High-dimensional asymptotics}

As discussed in Section~\ref{sec:existing-results}, Candès and Sur~\cite{candes2020phase} characterized the phase transition for existence of the MLE in the well-specified case and with a Gaussian design in the high-dimensional asymptotic regime where $d/n \to \gamma \in (0, 1)$ and $b = \norm{\theta^*} \in \R^+$ is fixed.
This result on existence is complemented in~\cite{sur2019modern} by a result on the behavior of the MLE under the same assumptions and asymptotic regime; specifically, it is shown in this work that the joint distribution of the true and estimated coefficients converges to a certain distribution.
These results have been extended, among others,
to arbitrary covariance matrices of the design~\cite{zhao2022asymptotic},
to ridge-regularized logistic regression~\cite{salehi2019impact}, to more general binary models~\cite{taheri2020sharp}, to multinomial logistic regression~\cite{tan2024multinomial} and to missing data~\cite{verchand2024high}.

\parag{Worst-case design distributions, improper and robust estimators}

Our focus in this work is to characterize the performance of the MLE in ``regular'' situations, namely when the design distribution satisfies some suitable conditions ensuring a near-Gaussian behavior.
A rather different but complementary perspective consists in considering the performance of the MLE or other estimators for logistic regression under worst-case design distributions.

As one might expect,
the performance of the MLE is considerably degraded for worst-case design distributions.
In particular, a lower bound from~\cite{hazan2014logistic} for statistical learning with logistic loss implies that, when no assumption is made on the design asides from boundedness ($\norm{X} \leq R$ almost surely), then the MLE or any ``proper'' estimator (that returns a conditional density belonging to the logistic model) can achieve no better expected excess risk (with respect to the ball of radius $\b$) than $O (\b R/\sqrt{n})$, as long as $n \leq e^{c \b R}$.
This exponential dependence on the parameter norm can be bypassed by resorting to ``improper estimators'', that is, estimators that return conditional densities that do not belong to the logistic model; these include Bayesian model averaging~\cite{kakade2005online,foster2018logistic,qian2024refined} or adjusted estimators that account for uncertainty using ``virtual labels''~\cite{mourtada2022logistic,jezequel2020efficient}.
We also refer to~\cite{vijaykumar2021localization,vanderhoeven2023high} for alternative procedures achieving sharp high-probability guarantees, albeit at a high computational cost.

A related direction is that of robust logistic regression.
In~\cite{chinot2020robust},
high-probability risk bounds in logistic loss are established for an estimator based on medians-of-means, when the design $X$ may be heavy-tailed.
In addition, an estimator for logistic regression achieving near-optimal guarantees in Hellinger distance is proposed in~\cite{baraud2024robust}.
This direction is complementary to the one we pursue, for two reasons.
First, the improved robustness properties of the previous methods come at the cost of an increased computational complexity over that of the MLE.
Second, the study of robust estimators is complementary to the analysis of the behavior of the MLE, which is the simplest and most commonly used method for fitting logistic regression.

\subsection{Outline and notation}
\label{sec:outline-notation}

\parag{Paper outline}

This paper is organized as follows.
Section~\ref{sec:main-results} contains the precise statements of our main results, as well as the definition and discussion of the regularity assumptions we consider on the design distribution.
In Section~\ref{sec:regular-designs}, we illustrate these assumptions by investigating to which extent they hold for three standard classes of design distributions.
Section~\ref{sec:proof-scheme} describes the structure of the proofs of the main results from Section~\ref{sec:main-results}, including statements of the main lemmas.
In particular, a convex localization argument reduces the proof of existence and risk bounds for the MLE to two components: an upper bound on the gradient of the empirical risk, and a uniform lower bound on the Hessian of the empirical risk in a neighborhood of the true parameter.
Section~\ref{sec:gradients} is devoted to the proofs of upper bounds on the empirical gradients, while lower bounds on the empirical Hessian are established in Section~\ref{sec:hessians}.
Next, in Section~\ref{sec:linear-separation} we provide the proof of Theorem~\ref{thm:strong-non-existence} on linear separation (non-existence of the MLE) with high probability.
In Section~\ref{app:main-proofs}, we conclude the proofs of the main results of Section~\ref{sec:main-results} by putting together the results of Sections~\ref{sec:localization},~\ref{sec:gradients} and~\ref{sec:hessians} and providing additional lower bounds.
Finally, Section~\ref{sec:proof-examples-regular} contains the proofs of results from Section~\ref{sec:regular-designs}, namely regularity of log-concave and product measures.
Appendices~\ref{sec:tails-rv} and~\ref{sec:polar-coord-spher} gather technical facts on real random variables and polar coordinates, while Appendix~\ref{sec:proof-2dim-necessary} contains the proof of Proposition~\ref{prop:2-dim-margin-almost-necessary} on necessity of the two-dimensional margin condition.

\parag{Notation}

Throughout the paper, the sample size will be denoted by $n$ and the dimension by $d$.
We let $\innerp{\cdot}{\cdot}$ denote the standard inner product on $\R^d$, and $\norm{\cdot} = \norm{\cdot}_2$ the associated Euclidean norm.
We denote by $B_2^d$ the unit Euclidean ball and by $S^{d-1}$ the unit sphere in $\R^d$.
For a positive semi-definite matrix $A$, we let $\| \cdot \|_A$ be the semi-norm induced by $A$, defined by
$\|x\|_A^2 = \langle A x, x \rangle = \|A^{1/2} x \|^2$ for $x\in \R^{d}$.
The operator norm of a matrix $A$ is denoted by $\opnorm{A}$.
Given two $d\times d$ symmetric matrices $A,B$, we write $A \mleq B$ if $\innerp{A v}{v} \leq \innerp{B v}{v}$ for every $v \in \R^d$.

If $f : \R^d \to \R^d$ is a twice continuously differentiable function, we let $\nabla f (x) %
\in \R^d$ and $\nabla^2 f (x) %
\in \R^{d \times d}$ respectively denote its gradient and Hessian at $x \in \R^d$.
For $a,b \in \R$ we let $a \wedge b = \min (a, b)$ and $a \vee b = \max (a, b)$, as well as $a_+ = \max (a, 0)$ and $a_- = \max (-a, 0)$.

Recall the notation introduced in Section~\ref{sec:problem-setting}: we let $(X_{1}, Y_{1}), \dots, (X_{n}, Y_{n})$ be \iid pairs having the same distribution $P$ as a generic pair $(X,Y)$, and let $\wh \theta_n$ denote the MLE defined by~\eqref{eq:def-mle}.
Throughout, we assume without loss of generality that the design $X$ has an isotropic distribution, in the sense that $\E [X X^\top] = I_d$.
We say that the logit model is \emph{well-specified} if the conditional distribution of $Y$ given $X$ belongs to the model $\model$, that is if there exists $\theta^* \in \R^d$ such that $\P (Y=1|X) = \sigma (\innerp{\theta^*}{X})$; otherwise, the model is said to be misspecified.

\section{Main results}
\label{sec:main-results}

In this section, we provide the precise statements of our main results on logistic regression, which we presented informally in Section~\ref{sec:summ-contr}.

This section is organized as follows.
First, Section~\ref{sec:well-spec-gaussian} contains the results in the case of a well-specified model and a Gaussian design.
In Section~\ref{sec:assumptions}, we introduce and discuss the regularity assumptions we consider on the design to generalize the Gaussian case.
In Section~\ref{sec:well-spec-regular} we extend the results of Section~\ref{sec:well-spec-gaussian} to the case of a regular design, while still assuming that the model is well-specified.
Finally, in Section~\ref{sec:missp-regular} we
consider the most general case, where the design is regular but
no assumption is made on the conditional distribution of $Y$ given $X$.

\subsection{Well-specified model, Gaussian design}
\label{sec:well-spec-gaussian}

We start with the case of a well-specified model and a Gaussian design.
Theorem~\ref{thm:gaussian-well-specified} below provides a sharp condition (up to universal constant factors) on the sample size $n$ for the MLE to exist with high-probability, as well as an optimal upper bound in deviation on its excess risk. Its proof can be found in Section~\ref{sec:proof-well-gaussian}.
	
\begin{theorem}
\label{thm:gaussian-well-specified}
Assume that the design $X \sim \gaussdist (0,I_{d})$ is Gaussian and that
the model is well-specified with parameter $\theta^* \in \R^d$, and let
$\b=\max\{e,\|\theta^*\|\}$.
There exist universal constants $c_1,c_2, c_3 > 0$
such that the following holds.
For any $t \geq 0$, if 
\begin{equation}
	\label{eq:sample size}
	n \geq c_1 \b (d + t)\enspace,
\end{equation}
then with probability at least $1-e^{-t}$, the MLE $\hat{\theta}_n$ exists and satisfies
\begin{equation}
  L(\hat{\theta}_n) - L(\theta^*) \leq c_{2} \, \frac{d+t}{n}\enspace.
\end{equation}
Moreover, for any $d\geq 70$ and $t \geq 1$, if $n\leq c_3  \b (d + t )$ then
the MLE exists with probability at most $1 - e^{-t}$.
\end{theorem}

It follows from Theorem~\ref{thm:gaussian-well-specified} that, up to numerical constants, condition~\eqref{eq:sample size} is both necessary and sufficient for the MLE to exist with high probability,
and that whenever this condition holds, the MLE admits the same risk guarantee as the asymptotic one~\eqref{eq:lowdim-asymptotic-mle} in the regime where $\b,d$ are fixed while $n \to \infty$, which is optimal in light of the convergence in distribution~\eqref{eq:mle-wilks}.
In particular, the condition on $n$ that ensures that the MLE exists also ensures that it achieves its asymptotic excess risk.

We note also that, using Lemma~\ref{lem:localization}, the proof of Theorem~\ref{thm:gaussian-well-specified} also provides guarantees for the estimation error of the direction and norm of $\theta^*$: if $\norm{\theta^*} \geq e$, we have for some universal constants $c_3,c_4> 0$: if $n \geq c_3 \b (d+t)$, then with probability at least $1-e^{-t}$,
\begin{equation}
  \label{eq:direction-norm-estimation}
  \bigg\| \frac{\wh \theta_n}{\norm{\wh \theta_n}} - \frac{\theta^*}{\norm{\theta^*}} \bigg\|
  \leq c_4 \sqrt{\frac{d + t}{\b n}} \, , \qquad
  \big| \norm{\wh \theta_n} - \norm{\theta^*} \big|
  \leq c_4 \sqrt{\frac{\b^3 (d + t)}{n}}
  \, .
\end{equation}

The proof of Theorem~\ref{thm:gaussian-well-specified} can be found in Section~\ref{sec:proof-well-gaussian} (combining results from Sections~\ref{sec:gradient-wellspec-gaussian} and~\ref{sec:hessian-gaussian}), while the scheme of proof is described in Section~\ref{sec:proof-scheme}.
In particular, a key structural result in the analysis is Theorem~\ref{thm:hessian-gaussian}, which provides a sharp high-probability lower bound on the Hessian of the empirical risk $\wh H_n (\theta) = \nabla^2 \wh L_n (\theta)$, uniformly for $\theta$ belonging to a
neighborhood of $\theta^*$ that is ``as large as possible''.

Let us now come back to the question of existence of the MLE;
as noted in the introduction, non-existence of the MLE amounts to linear separation of the dataset.
Theorem~\ref{thm:gaussian-well-specified} implies in particular that if $n \geq 2 c_1 \b d$, then
\begin{equation}
  \label{eq:mle-exists-whp}
  1 - \exp \parens[\Big]{- \frac{n}{2 c_1 \b}}
  \leq \P (\text{MLE exists})
  \leq 1 - \exp \parens[\Big]{- \frac{n}{c_3 \b}}
  \, ,
\end{equation}
which provides an optimal quantitative estimate of the probability of existence of the MLE in the regime $n \gg \b d$.

This complements existing results on linear separation.
Specifically, it was recently shown by Kuchelmeister~\cite[Theorem~4]{kuchelmeister2024probability} that, for some absolute constant $c > 0$, if $n \geq c \b d$ then
\begin{equation}
  \label{eq:kuchelmeister-existence}
  \P (\text{MLE exists})
  \geq 1 - \exp \Big( - \frac{n}{c \b^2} \Big)
  \, .
\end{equation}
While this estimate is weaker than~\eqref{eq:mle-exists-whp} in our setting, the result~\cite[Theorem~4]{kuchelmeister2024probability} holds beyond the well-specified case (with the probability of existence of the MLE related to the misclassification error $\P (Y \innerp{\theta^*}{X} < 0)$), which makes it incomparable to~\eqref{eq:mle-exists-whp} in general.
In addition, using the equivalence discussed in Section~\ref{sec:existing-results} between existence of the MLE and the property that $0$ does not belong to the interior of $\conv (Y_1 X_1, \dots, Y_n X_n)$, a recent result from Hayakawa, Lyons, and Oberhauser~\cite[Theorem~14]{hayakawa2023estimating} implies the following: if $n \geq c \b d$ (for some suitable universal constant $c$), then
\begin{equation}
  \label{eq:hayakawa-existence}
  \Probab{\text{MLE exists}}
  \geq
  1 - 2^{-d}
  \, .
\end{equation}

Now, the lower bound in~\eqref{eq:mle-exists-whp} provides a quantitative version of the convergence to $1$ in the phase transition~\eqref{eq:phase-transition-mle} from~\cite{candes2020phase} for the existence of the MLE.
On the other hand, if $n \ll \b d$, then Theorem~\ref{thm:gaussian-well-specified} (with $t=0$) only implies that the probability of existence of the MLE is bounded away from $1$, rather than close to $0$ as in the phase transition~\eqref{eq:phase-transition-mle}.

Theorem~\ref{thm:strong-non-existence} below shows that if $n \ll \b d$, then the probability of existence of the MLE indeed approaches $0$, at a rate exponential in the dimension.
This can be seen as a quantitative version of the convergence to $0$ in the phase transition~\eqref{eq:phase-transition-mle}.
Theorem~\ref{thm:strong-non-existence} is proved in Section~\ref{sec:linear-separation}.

\begin{theorem}
\label{thm:strong-non-existence}
Let $d\geq 70$, and assume that $X \sim \gaussdist (0, I_d)$ and that the logistic model is well-specified.
For every $\kappa \geq 1$, if $n \leq \b d /(200 \kappa)$ then
\begin{equation}
  \label{eq:strong-non-existence}
  \P( \text{\emph{MLE exists}} ) 
  \leq \exp\big(-\max\big\{ \kappa\sqrt{d}, \kappa^{2}d/\b^{2}\big\}\big) 
  + 6e^{-d/24} \, .
\end{equation}
\end{theorem}

Let us now discuss the interpretation of this result.

First, the question of non-existence of the MLE (that is, linear separation) is mainly of interest when $n \geq d$, since for $n < d$ %
the dataset is linearly separated, as the points $X_1, \dots, X_n$ do not span $\R^d$.
We are therefore interested in the regime $d \ll n \ll \b d$, where linear separation no longer occurs deterministically because of the dimension, but instead with high probability due to the fact that the signal is strong ($\b \gg 1$).
Specifically, in this regime, a strong signal $\b$ entails that a large fraction of labels $Y_i$ will be of the same sign as the predictions $\innerp{\theta^*}{X_i}$, which effectively
constrains the directions of the vectors $Y_i X_i$, making it easier to find a $\theta \neq 0$ such that $\innerp{\theta}{Y_i X_i} \geq 0$ for every $i=1, \dots, n$.
Interestingly, if $n \gg \b$ (while $n \ll \b d$), then the true parameter $\theta^*$ will typically not satisfy this property; instead, linear separation will be achieved by another (random) parameter $\theta$, thanks to the %
flexibility due to the large dimension $d$.
As such, in the regime $\max \{ \b, d \} \ll n \ll \b d$, linear separation holds with high probability owing to the \emph{combination} of the ``signal strength'' and ``dimension'' effects, rather than one of the two taken individually.

We now comment on the quantitative bound~\eqref{eq:strong-non-existence}.
Theorem~\ref{thm:strong-non-existence} implies that if $n$ is small compared to the threshold of order $\b d$, then the probability of existence of the MLE is smaller than $\exp (- c \max\{ \sqrt{d}, d/\b^2 \})$ for some constant $c$.
In addition, the parameter $\kappa \geq 1$ quantifies how small the sample size $n$ is relative to the critical threshold of $\b d$, and the smaller the sample size (that is, the larger $\kappa$ is), the smaller the bound~\eqref{eq:strong-non-existence}.
In particular, in the regime where $n \asymp d$ and $\b \gg 1$ (so that $\kappa \asymp \b$), Theorem~\ref{thm:strong-non-existence} shows that the probability of existence of the MLE is smaller than $\exp (- c d)$ for some constant $c$.

It is worth relating the estimate from Theorem~\ref{thm:strong-non-existence} to existing results on linear separation.
Specifically, it was shown by Wagner and Welzl~\cite{wagner2001continuous} that, for any $n > d \geq 1$ and \iid random vectors $U_1, \dots, U_n$ in $\R^d$ with absolutely continuous distribution, one has
\begin{equation}
  \label{eq:wagner-separation}
  \Probab{0 \not\in \conv (U_1, \dots, U_n)}
  \geq \frac{1}{2^{n-1}} \sum_{k=0}^{d-1} \binom{n-1}{k}
  \, ;
\end{equation}
this inequality is in fact an equality when the distribution of $U_1$ is symmetric, as shown by a classical result of Wendel~\cite{wendel1962problem}.
The absolute continuity assumption was later relaxed in~\cite{hayakawa2023estimating} to the minimal assumption that $\P (U_1 \in H) = 0$ for any linear hyperplane $H \subset \R^d$.
Applying the inequality~\eqref{eq:wagner-separation} to $U_i = Y_i X_i$ shows that if $n > d$, then
\begin{equation}
  \label{eq:linear-sep-distribution-free}
  \Probab{\text{MLE exists}}
  \leq \sum_{k=0}^{n-d-1} \binom{n-1}{k}
  \, .
\end{equation}
In particular, for a suitable constant $c > 0$, if $d \geq c$ and $d < n < 1.9 d$, then the right-hand side of~\eqref{eq:linear-sep-distribution-free} is at most $\exp (- d / c)$, matching Theorem~\ref{thm:strong-non-existence} in this case.
On the other hand, if $n>2.1 d$, then the upper bound~\eqref{eq:linear-sep-distribution-free} is no longer small (it behaves as $1 - \exp (- d/c)$); whereas if $\b \gg 1$ and $d < n \ll \b d$, Theorem~\ref{thm:strong-non-existence} still provides an exponentially small bound.
This is %
because
the bound~\eqref{eq:linear-sep-distribution-free} is distribution-free, while Theorem~\ref{thm:strong-non-existence} takes the signal strength into account.

The proof of Theorem~\ref{thm:strong-non-existence}, which builds on the approach of Candès and Sur~\cite{candes2020phase}, can be found in Section~\ref{sec:linear-separation}.
Specifically, %
the starting point of the proof %
is to reformulate the property of linear separation into the property that a certain random cone $\Lambda$ in $\R^n$ has a non-trivial intersection with an independent uniform random subspace.
Now, it follows from the work~\cite{amelunxen2014living} that the probability of such an event depends on the dimension of the random subspace, and on a certain geometric parameter of the cone $\Lambda$ called ``statistical dimension''.
In order to control the probability of existence of the MLE, one must therefore combine two steps:
(i) conditionally on the cone $\Lambda$, apply a phase transition result showing that the probability that a random subspace does not intersect $\Lambda$ is small;
(ii) in order to apply the previous result to the random cone $\Lambda$, control of the statistical dimension of $\Lambda$ with high probability.

For the first point, Candès and Sur use a phase transition result from~\cite{amelunxen2014living}.
For the second point,
they establish that the statistical dimension of the random cone $\Lambda$ converges in probability to a deterministic value as $n, d \to \infty$ while $d/n \to \gamma$, for fixed $b = \norm{\theta^*}$.
To show this, they first relate the statistical dimension to (a family of) averages of \iid random variables, and then establish uniform convergence of the averages to the corresponding expectations.

While these arguments suffice to establish the $0$-$1$ law~\eqref{eq:phase-transition-mle} in this asymptotic regime,
several refinements are required in order to obtain the quantitative bound of Theorem~\ref{thm:strong-non-existence}.
First, a more precise phase transition result \cite[Theorem~6.1]{amelunxen2014living} must be used in order to finely capture the dependence
on the statistical dimension of $\Lambda$.
Second and more importantly, one must establish a refined high-probability control on the statistical dimension of the random cone.
This requires %
a high-probability bound on
the sum of \iid random variables that %
controls this dimension.
We achieve this by first obtaining a tight control on the moments of the individual summands, and then applying a sharp estimate of Lata{\l}a~\cite{latala1997estimation} for moments of sums of independent random variables.

\subsection{Regularity assumptions}
\label{sec:assumptions}

In this section, we introduce the regularity conditions on the design distribution, alluded to in the introduction, under which the MLE exhibits a similar behavior as in the Gaussian design case.
In the following sections, we will present two extensions of Theorem~\ref{thm:gaussian-well-specified}: first, to the case of a regular design under a well-specified model, and second, to the most general case of a regular design with a misspecified model.
Examples of distributions satisfying the regularity assumptions are discussed in Section~\ref{sec:regular-designs}.

The first assumption on the design is standard and states that the design $X$ is light-tailed; we refer to Definition~\ref{def:psi-alpha} in Appendix~\ref{sec:tails-rv} for the definition of the $\psi_1$-norm.
\begin{assumption}
  \label{ass:sub-exponential}
  The random vector $X$ is \emph{$K$-sub-exponential} for some $K \geq e$, in the sense that $\norm{\innerp{v}{X}}_{\psi_1} \leq K$ for every $v \in S^{d-1}$.
\end{assumption}

The second assumption is also standard in the binary classification literature.
It requires that the design distribution does not place too much mass near the separating hyperplane.
It is related to the \emph{margin assumption} underlying fast convergence rates in classification~\cite{mammen1999margin,tsybakov2004aggregation,audibert2007plugin}.
The assumptions below depend on a direction $u^* \in S^{d-1}$, which will be the parameter direction $\theta^*/\norm{\theta^*}$, and on a scale $\eta$, which will be the inverse parameter norm $\b^{-1}$.

\begin{assumption}
  \label{ass:small-ball}
  Let $u^* \in S^{d-1}$ and $\eta \in (0, 1]$.
  For some $c \geq 1$, one has for every $t \geq \eta$ that
  \begin{equation}
    \label{eq:small-ball}
    \P \big( \abs{\innerp{u^*}{X}} \leq t \big)
    \leq c t
    \, .
  \end{equation}
\end{assumption}

As it turns out, the standard Assumptions~\ref{ass:sub-exponential} and~\ref{ass:small-ball} above are not sufficient to ensure a Gaussian-like behavior in logistic regression.
Instead, we need a new, additional condition on the distribution, which concerns the behavior of its two-dimensional linear marginals.
We call this assumption the \emph{two-dimensional margin condition}.

\begin{assumption}
  \label{ass:twodim-marginals}
  Let $u^* \in S^{d-1}$, $\eta \in (0, 1/e]$ and $c \geq 1$.
   For every $v \in S^{d-1}$ such that $\innerp{u^*}{v} \geq 0$, one has
  \begin{equation}
    \label{eq:twodim-marginals-dist}
    \P \big( \abs{\innerp{u^*}{X}} \leq c \eta, \ \abs{\innerp{v}{X}} \geq c^{-1} \max \{ \eta, \norm{u^* - v} \} \big)
    \geq \eta / c
    \, .
  \end{equation} 
\end{assumption}

\begin{remark}
\label{rem:equivalent-two-dim}
Using that if $\innerp{u^*}{v} \geq 0$ then
\begin{equation*}
	\norm{u^*-v}/\sqrt{2} \leq \sqrt{1-\innerp{u^*}{v}^2} %
    = \norm{u^* - v} \sqrt{(1 + \innerp{u^*}{v})/2}
    \leq \norm{u^* - v}
    \, ,
\end{equation*}
another way of stating Assumption~\ref{ass:twodim-marginals} is that for every $v \in S^{d-1}$, one has
  \begin{equation}
    \label{eq:twodim-marginals-comp}
    \P \Big( \abs{\innerp{u^*}{X}} \leq c \eta, \ \abs{\innerp{v}{X}} \geq c^{-1} \max \Big\{ \eta
    ,
    \sqrt{1- \innerp{u^*}{v}^2} \Big\} \Big)
    \geq \eta / c
    \, .
  \end{equation}
  This only changes the value of the parameter $c$ from~\eqref{eq:twodim-marginals-dist} by a factor $\sqrt{2}$.
  This equivalent formulation is more convenient in some situations.
\end{remark}

Let us now discuss the meaning and necessity of this assumption.
We start with an intuitive discussion and proceed with a precise statement.
First, the condition $\ainnerp{u^*}{X} \lesssim \eta$ amounts to $\ainnerp{\theta^*}{X} \lesssim 1$.
The restriction to such values of $X$ stems from the fact that values of $X$ for which $\ainnerp{\theta^*}{X} \gg 1$ provide little information on the precise value of $\theta^*$.
Indeed, it follows from the form of the logistic model that for such $X$, one has $Y = \sign(\innerp{\theta^*}{X})$ with high probability; but for other parameters $\theta$ close to $\theta^*$, one also has $\ainnerp{\theta}{X} \gg 1$ and $Y = \sign (\innerp{\theta}{X})$.
Hence, the value of $Y$ does not allow one to distinguish between the values $\theta$ and $\theta^*$, both of which are highly consistent with this label.

We thus focus on values of $X$ such that $\ainnerp{\theta^*}{X} \lesssim 1$, which is the first condition.
The second condition is that for any other direction $v \in S^{d-1}$, we find sufficiently many values of $X$ for which additionally $\ainnerp{v}{X}$ is sufficiently large.
This intuitively comes from the fact that, in order to distinguish $\theta^*$ from another parameter $\theta = \theta^* + \eps v$ with $\eps > 0$ from a data point $(X, Y)$, we need the ``predictions'' of these two parameters at $X$, namely $\innerp{\theta^*}{X}$ and $\innerp{\theta}{X} = \innerp{\theta^*}{X} + \eps \innerp{v}{X}$, to be sufficiently different.
This precisely amounts to saying that $\ainnerp{v}{X}$ is not too small.
The threshold for $\ainnerp{v}{X}$ depends on the alignment between $v$ and $u^*$, and its value in~\eqref{eq:twodim-marginals-comp} comes from the fact that, in the baseline case where $X \sim \gaussdist (0, I_d)$, one has
\begin{equation*}
  \E \Big[ \innerp{v}{X}^2 \, \big| \, \ainnerp{u^*}{X} = \eta \Big]^{1/2}
  = \sqrt{\innerp{u^*}{v}^2 \eta^2 + (1 - \innerp{u^*}{v}^2)}
  \asymp \max \{ \eta, \sqrt{1 - \innerp{u^*}{v}^2} \}
  \, .  
\end{equation*}

We now argue more formally that Assumption~\ref{ass:twodim-marginals} is necessary for the MLE to behave similarly as in the Gaussian case.
To see why, let $H_X (\theta^*) = \nabla^2 L (\theta^*)$ denote the Hessian under the design $X$.
When the logistic model is well-specified, $H_X (\theta^*)$ coincides with the Fisher information matrix at $\theta^*$.
Hence, the MLE is asymptotically normal as $n \to \infty$, with
\begin{equation*}
  \sqrt{n} (\wh \theta_n - \theta^*) \cvd \gaussdist (0, H_X (\theta^*)^{-1})
  \, .
\end{equation*}
Therefore, for the error $\wh \theta_n - \theta^*$ of the MLE under design $X$ to be %
as small as
the error under a Gaussian design $G \sim \gaussdist (0, I_d)$, one must have $H_X (\theta^*) \mgeq c H_G (\theta^*)$ for some constant $c > 0$.
In addition, it follows from~\eqref{eq:HessianGaussian} and~\eqref{eq:components-hessian} that $H_G (\theta^*)$ is upper and lower bounded (up to absolute constant factors) by the matrix
\begin{equation}
	H = \frac{1}{\b^{3}} u^{*} {u^{*}}^{\top} + \frac{1}{\b} (I_{d} - u^{*} {u^{*}}^{\top}) \, ,
\end{equation}
where $\b = \max\{\|\theta^{*}\|, e\}$.
Hence, the previous condition is of the form $H_X (\theta^*) \mgeq c H$.

The following result shows that, whenever Assumptions~\ref{ass:sub-exponential} and~\ref{ass:small-ball} hold, Assumption~\ref{ass:twodim-marginals} is necessary (up to logarithmic factors) for the condition $H_X (\theta^*) \mgeq c H$ to hold.

\begin{proposition}
  \label{prop:2-dim-margin-almost-necessary}
  Let $\theta^* \in \R^d$, and set $u^* \in S^{d-1}$ such that $\theta^* = \norm{\theta^*} u^*$ as well as $\b = \max \set{\norm{\theta^*}, e}$.
  Let $X$ be an isotropic random vector satisfying Assumption~\ref{ass:sub-exponential} with parameter $K \geq e$ and Assumption~\ref{ass:small-ball} with parameters $u^*, \eta = \b^{-1}, c$.
  If $H_X (\theta^*) \mgeq c_0 H$ for some $c_0 \in (0, 1)$, then for some constants $c_1, c_2, c_3$ depending only on $K, c, c_0$, one has
  \begin{equation}
    \label{eq:two-dim-margin-necessary}
    \P \bigg( \ainnerp{u^*}{X} \leq \frac{c_1 \log \b}{\b} , \, \ainnerp{v}{X} \geq \frac{\max \{ \norm{u^* - v}, \b^{-1} \}}{c_2 \sqrt{\log \b}} \bigg)
    \geq \frac{c_3}{\b \log^2 (\b)}
    \, .
  \end{equation}
\end{proposition}
The proof of this result can be found in Appendix~\ref{sec:proof-2dim-necessary}.

We can now formulate the definition of ``regular distributions'' that we use throughout.
\begin{definition}
  \label{def:regular}
  Let $u^* \in S^{d-1}$, $\eta \in (0, e^{-1}]$ and $c \geq 1$.
  A random vector $X$ in $\R^d$ %
  is said to have an
  \emph{$(u^*, \eta, c)$-regular} distribution if it is isotropic (that is, $\E [X X^\top] = I_d$)  
  and
  satisfies Assumptions~\ref{ass:small-ball} and~\ref{ass:twodim-marginals} with
  parameters $u^*, \eta, c$.
\end{definition}
\subsection{Well-specified model, regular design}
\label{sec:well-spec-regular}

We can now state our main result on the performance of the MLE in the case of a regular design and a well-specified model, whose proof can be found in Section~\ref{sec:proof-well-regular}.

\begin{theorem}
  \label{thm:regular-well-specified}
  Assume that the model is well-specified, with parameter $\theta^* = \norm{\theta^*} u^*$ where $u^* \in S^{d-1}$ and let
  $\b=\max\{e,\|\theta^*\|\}$.
  Assume that $X$ satisfies Assumptions~\ref{ass:sub-exponential},~\ref{ass:small-ball} and~\ref{ass:twodim-marginals} with parameters $K\geqslant e$, $u^{*}$, $\eta=\b^{-1}$ and $c$.
  There exist constants $c_1, c_2$ that depend only on $c,K$ such that, if 
\begin{equation*}
  n
  \geq c_1 \b \log^4 (\b) (d + t)
  \, ,
\end{equation*}
then with probability at least $1-e^{-t}$, the MLE  $\wh \theta_{n}$ exists and satisfies
\begin{equation}
  L(\wh \theta_{n}) - L(\theta^{*})
  \leq %
  c_2
  \log^4 (\b) \frac{d+t}{n}
  \, .
\end{equation}
\end{theorem}

The guarantees of Theorem~\ref{thm:regular-well-specified} almost match (up to poly-logarithmic factors in $\b$) those of Theorem~\ref{thm:gaussian-well-specified} in the Gaussian case, which are optimal as discussed above.
In fact, one can almost recover (again up to $\log^4 (\b)$ factors) the guarantees of Theorem~\ref{thm:gaussian-well-specified} from this result, since one can show that the Gaussian design satisfies the regularity assumptions for all $u^*, \eta$ and with $c, K$ being universal constants.

\subsection{Misspecified model, regular design}
\label{sec:missp-regular}

We now turn to the general case where the logit model may be misspecified.
In this setting, the conditional distribution of $Y$ given $X$ is no longer determined by $\theta^{*}$, conversely $\theta^{*}$ is now a function of the joint distribution of $(X,Y)$, as is the case in Statistical Learning.
We define $\theta^{*}$ as the minimizer of the population risk $L$ (see~\eqref{eq:loss-risk}), namely
\begin{equation}
\label{eq:def-theta-star-learning}
	\theta^{*} \in \underset{\theta \in \R^{d}}{\argmin} \, L(\theta) \, , \quad 
	L(\theta) = \E\big[\log\big(1+e^{-Y \innerp{\theta}{X}}\big) \big] .
\end{equation}
By the discussion in the introduction, $\theta^*$ exists whenever the distribution of $(X,Y)$ is not linearly separated, meaning that there is no $\theta \neq 0$ such that $Y \innerp{\theta}{X} \geq 0$ almost surely---which we assume in this section.
In addition, $\theta^*$ is unique since we assume that $\E [X X^\top] = I_d$, which ensures strict convexity of $L$.
Theorem~\ref{thm:regular-misspecified} below is proved in Section~\ref{sec:proof-misspecified-regular}.

\begin{theorem}
\label{thm:regular-misspecified}
Suppose that $X$ satisfies Assumptions~\ref{ass:sub-exponential},~\ref{ass:small-ball} and~\ref{ass:twodim-marginals} with parameters $K\geqslant e$, $u^{*}$, $\eta=\b^{-1}$ and $c$, and that $\theta^{*} = \|\theta^{*}\| u^{*}$.
Let $\b = \max \{e, \norm{\theta^*} \}$.
There exist
constants $c_1, c_2$ that depend only on $c,K$ such that the following holds.
For any $t \geq 0$, if
\begin{equation*}
  n
  \geq c_1 \b \log^4 (\b) (d + \b t)
  \, ,
\end{equation*}
then with probability at least $1-e^{-t}$, the MLE  $\wh \theta_{n}$ exists and satisfies
\begin{equation}
  L(\wh \theta_{n}) - L(\theta^{*})
  \leq %
  c_2
  \log^4 (\b) \frac{d+\b t}{n}
  \, .
\end{equation}
Moreover, for any $\b \geq e$, there exists a distribution of $(X,Y)$ with $X \sim \gaussdist (0, I_d)$ and $\norm{\theta^*} = \b$
such that if $n \leq c_{3} \b (d+\b t)$ (for some universal constant $c_3$), then
\begin{equation}
\label{eq:non-existence-ms}
\P(\text{\emph{MLE exists}}) \leq 1- e^{-t} \, .
\end{equation}
In addition, for the same distribution,
\begin{equation}
  \label{eq:asymptotic-extra-B}
  \liminf_{n \to \infty} \P \Big( L (\wh \theta_n) - L (\theta^*) \geq c_{3} \frac{d + \b t}{n} \Big)
  \geq e^{-t} \, .
\end{equation}
\end{theorem}

Theorem~\ref{thm:regular-misspecified} improves the previous best guarantees for the MLE in logistic regression in the general misspecified case.
As discussed in Section~\ref{sec:existing-results}, these are from~\cite{chinot2020robust} for a sub-Gaussian design, and~\cite{ostrovskii2021finite} for a Gaussian design.
The guarantees in~\cite{chinot2020robust} in the sub-Gaussian case feature an exponential dependence on $\b$.
The guarantees in~\cite{ostrovskii2021finite} in the Gaussian case (a special case of regular design), which are actually the previous best guarantees for the MLE in a misspecified setting\footnote{Technically speaking, the guarantees in~\cite{ostrovskii2021finite} (obtained by combining Theorem~4.2 with Proposition~D.1) are stated in the well-specified case. However, they can be extended to the misspecified case through
  a very minor modification (renormalizing gradients by the Hesssian of the risk instead of their covariance matrix).
  Likewise, the deviation terms in $d \cdot t$ in these results can be tightened to $d+t$ with no changes to the analysis.
}, feature a polynomial dependence on $\b$ but a stronger one than in Theorem~\ref{thm:regular-misspecified}: the condition for existence of the MLE writes (ignoring $\mathrm{polylog}(\b)$ factors) $n \gtrsim \b^8 d t$, and the risk is bounded by $\b^3 d t/n$.

It should be noted that both the sample size needed for the MLE to exist and the bound on its excess risk in Theorem~\ref{thm:regular-misspecified} exhibit a stronger dependence on $\b$ compared to the well-specified case.
As shown by~\eqref{eq:non-existence-ms} and~\eqref{eq:asymptotic-extra-B}, this stronger dependence on $\b$ is in fact necessary in the misspecified case.
This shows that the non-asymptotic guarantees of Theorem~\ref{thm:regular-misspecified} for the existence and the excess risk of the MLE are sharp, up to polylogarithmic factors in $\b$.
It should be pointed out that the degradation only affects an additive term that does not multiply the dimension $d$, hence as long as $\b t = O (d)$ (a regime that covers many situations of interest), the guarantees in the misspecified case actually match those of the well-specified case.

\section{Examples of regular design distributions}
\label{sec:regular-designs}

In the previous section, we introduced certain regularity assumptions (Assumptions~\ref{ass:sub-exponential}, \ref{ass:small-ball} and~\ref{ass:twodim-marginals}) on the distribution of the design $X$, which we argued were essentially necessary and sufficient to obtain the same results as in the Gaussian case.
In this section, we provide examples of distributions that satisfy these assumptions.

The three examples we consider are: %
sub-exponential distributions when the signal strength is of constant order (Section~\ref{sec:regul-low-signal}), log-concave distributions (Section~\ref{sec:regularity-log-concave}), and product measures (Section~\ref{sec:regul-prod-meas}).

We recall that the regularity assumptions introduced in Section~\ref{sec:assumptions} depend on both a direction $u^* \in S^{d-1}$ and a scale parameter $\eta \in (0, e^{-1}]$.
When applied to logistic regression, these correspond respectively to the parameter direction $u^* = \theta^* / \norm{\theta^*}$ and inverse signal strength $\eta = 1/\b = 1/\max (\norm{\theta^*}, e)$.
In particular, the stronger the signal, the %
finer
the scale $\eta$ at which the regularity assumptions should hold for our guarantees of Section~\ref{sec:main-results} to apply.

\subsection{Regularity
  at constant scales%
}
\label{sec:regul-low-signal}

First, we note that the regularity assumptions at a lower-bounded scale $\eta$ (corresponding to a bounded signal strength) are automatically satisfied when the design is sub-exponential.

\begin{proposition}
  \label{prop:regularity-constant-scale}
  Let $X$ be an isotropic and $K$-sub-exponential random vector (Assumption~\ref{ass:sub-exponential}).
  Then $X$ is $(u^*, \eta, c_{K, \eta})$-regular for any $u^* \in S^{d-1}$ and $\eta \in (0, e^{-1}]$, where
  \begin{equation*}
    c_{K, \eta}
    = \max \Big\{ \frac{2 K \log (2K)}{\eta}, 2 K^4 \Big\}
    \, .
  \end{equation*}
\end{proposition}

The content of Proposition~\ref{prop:regularity-constant-scale} (proved in Section~\ref{sec:proof-regularity-constant}) is that the regularity assumptions are general enough to include all sub-exponential distributions, with the caveat that the involved constant $c$ depends on the scale $\eta$.
However, it should be noted that the bounds in Theorems~\ref{thm:regular-well-specified} and~\ref{thm:regular-misspecified} depend exponentially on $c$, leading to an exponential dependence on the signal strength $\b = \eta^{-1}$.
For this reason, Proposition~\ref{prop:regularity-constant-scale} is mainly relevant in the case of constant signal strength.

\subsection{Regularity of log-concave distributions}
\label{sec:regularity-log-concave}

The issue of the general reduction from sub-exponential to regular is that it ultimately leads to a poor (exponential) dependence on the signal strength in the guarantees of Section~\ref{sec:main-results}.
As we shall see in Section~\ref{sec:regul-prod-meas}, this exponential dependence is necessary in general, hence in order to obtain similar guarantees as for a Gaussian design, one must strengthen the assumptions on the design beyond merely sub-exponential tails.

A natural class of probability measures that contains Gaussian measures, and often exhibit similar properties, is the class of \emph{log-concave} distributions on $\R^d$.
Specifically, recall that the distribution $P_X$ on $\R^d$ is log-concave (see \eg \cite{saumard2014review}) if, for all Borel sets $S, T \subset \R^d$ and $\lambda \in (0, 1)$ such that $\lambda S + (1-\lambda )T = \{ \lambda s + (1-\lambda) t : s \in S, t \in T \}$ is measurable,
\begin{equation*}
  P_{X}(\lambda S + (1 - \lambda )T)
  \geq P_{X}(S)^{\lambda } P_{X}(T)^{1-\lambda } \, .
\end{equation*}
We are interested in the case where $X$ is centered and isotropic, in which case it is log-concave if and only if it admits a density on $\R^d$ of the form $\exp(-\phi)$, for some convex function $\phi : \R^d \to \R \cup \{ + \infty \}$.

The following result shows that centered isotropic and log-concave distributions are regular in all directions and at all scales.
	
\begin{proposition}
\label{prop:universal-reg-log-concave}
Assume that $X$ has a centered isotropic \textup(that is, $\E [X] = 0$ and $\E [X X^\top] = I_d$\textup) and log-concave distribution on $\R^d$.
Then $X$ is $c$-sub-exponential and $(u^*, \eta, c)$-regular with a universal
constant $c$, for every direction $u^*\in S^{d-1}$ and every scale $\eta \in (0, e^{-1}]$.
\end{proposition}

The proof of Proposition~\ref{prop:universal-reg-log-concave} is provided in Section~\ref{sec:proof-regularity-log-concave}.
The fact that log-concave distributions are regular (with universal constants) mainly comes from a key stability property: the distributions of their lower-dimensional linear projections are also log-concave~\cite{saumard2014review}, which is applied here to two-dimensional projections.
In addition, low-dimensional, centered and isotropic distributions admit a density that is upper and lower-bounded around the origin.
Hence, at small scales they are essentially equivalent to the Lebesgue (or Gaussian) measure, which admits a ``product'' or ``independence'' property for orthogonal linear projections that implies regularity.

\subsection{Regularity for \iid coordinates%
}
\label{sec:regul-prod-meas}

Besides log-concave measures, another class of distributions that tend to behave similarly to Gaussian distributions in many high-dimensional contexts is that of product measures, that is, distributions of random vectors with independent coordinates.
In this section, we therefore consider the question of regularity of product measures, which turns out to be much more subtle than in the log-concave case.

Specifically, in this section we consider the class of random vectors with \iid sub-exponential coordinates:

\begin{assumption}
  \label{ass:independent-coordinates}
  The random vector $X = (X_1, \dots, X_d)$ is such that: $X_1, \dots, X_d$ are \iid,
  with $\E [X_j] = 0$, $\E [X_j^2] = 1$ and $\norm{X_j}_{\psi_1} \leq K$ (for some $K \geq e$) for $j = 1, \dots, d$.
\end{assumption}

It is a simple fact (see Lemma~\ref{lem:sub-exponential-indep} in Section~\ref{sec:proof-regularity-iid} below) that such a random vector is $4K$-sub-exponential.
Hence, the main question is whether Assumptions~\ref{ass:small-ball} and~\ref{ass:twodim-marginals} are satisfied.

A concrete example which illustrates the main issues is the Bernoulli design $X = (X_1, \dots, X_d)$, whose coordinates are \iid random signs, namely $\P (X_j = 1) = \P (X_j = -1) = 1/2$ for $1 \leq j \leq d$ (that is, $X$ is uniform on the discrete hypercube $\set{-1, 1}^d$).
This design satisfies Assumption~\ref{ass:independent-coordinates}; in fact, its tails are even lighter than sub-exponential, since its coordinates are bounded and it is a sub-Gaussian random vector.
This design is similar to the Gaussian design in many ways; for instance, it
possesses strong
concentration properties.

Despite these facts, the behavior of the MLE under a Bernoulli design can be drastically different from the case of a Gaussian design.
Indeed, as noted below, an exponential dependence on the signal strength is necessary for the MLE to exist.
This contrasts with the linear dependence on $\b$ in the Gaussian case (Theorem~\ref{thm:gaussian-well-specified}).
As an aside, the example below shows that for a sub-Gaussian design, an exponential dependence on the norm is unavoidable in general, a fact we alluded to previously.
In what follows, we denote by $(e_1, \dots, e_d)$ the canonical basis of $\R^d$.

\begin{fact}
  \label{fac:exponential-dependence-bernoulli}
  Let $X = (X_1, \dots, X_d)$ be a Bernoulli design, and let $Y$ given $X$ follow the logit model with parameter $\theta^* = \b e_1$ for some $\b \geq e$.
  Given an \iid sample of size $n\geq 1$ from the same distribution as $(X,Y)$, if $n \leq 0.1 \exp (\b)$ then $\P (\text{\normalfont MLE exists}) \leq 0.1$.
\end{fact}

\begin{proof}
  First, since the model is well-specified, one readily verifies that
  \begin{equation*}
    \P (Y \innerp{\theta^*}{X} \leq 0)
    = \E \big[ \P (Y \innerp{\theta^*}{X} \leq 0|X) \big]
    = \E [ \sigma (- \ainnerp{\theta^*}{X}) ]
    \leq \E [ \exp (-\ainnerp{\theta^*}{X}) ]
    \, .
  \end{equation*}
  Now, note that $\ainnerp{\theta^*}{X} = \b |X_1| = \b$ since $X_1 = \pm 1$, and in particular $\ainnerp{\theta^*}{X} \geq \b$.
  Thus, the above formula shows that $\P (Y \innerp{\theta^*}{X} \leq 0) \leq \exp (-\b)$.
  Now, let $Z = Y X$ and define similarly $Z_1, \dots, Z_n$ from the \iid sample.
  If the MLE exists, then in particular $\theta^*$ does not linearly separate the dataset, hence there exists $1 \leq i \leq n$ such that $\innerp{\theta^*}{Z_i} \leq 0$.
  By a union bound, the probability of this event is lower than $n \exp (-\b) \leq 0.1$ by assumption on $n$.
\end{proof}

This exponential dependence on the norm $\b$ comes from the fact that $X$ is not regular at small scales in the direction $u^* = e_1$.
Indeed, the random variable $\innerp{e_1}{X} = X_1$ is a random sign, which puts no mass in the neighborhood $(-1, 1)$ of $0$, therefore violating Assumption~\ref{ass:twodim-marginals} for small $\eta$ and constant $c$.
This illustrates the fact that the existence of the MLE is sensitive to the behavior of linear marginals of $X$ around the origin, and not merely to the tails of $X$.
Hence, the ``discrete'' nature of the Bernoulli design $X$ (supported on a finite set) can lead to a very different behavior from the Gaussian case.

Although the previous example shows very different behaviors between the Gaussian and Bernoulli designs,
one should keep in mind that it concerns a very specific direction $u^* =(1, 0, \dots, 0)$, which is a coordinate vector.
This worst-case direction is highly ``sparse''; this contrasts with a typical vector on the sphere, which is ``dense'' or ``delocalized'' in the sense that most of its coordinates are small, namely of order $O (1/\sqrt{d})$.
One may expect that for such vectors, the behavior of the MLE is markedly different than for a sparse direction.

In order to capture this effect, we now consider the ``densest'' direction $u^* = (1/\sqrt{d}, \dots, 1/\sqrt{d})$, all of whose coefficients are small.
Our aim is to characterize the smallest scale $\eta = \eta^*_d$ for which a design $X$ with \iid coordinates satisfies the regularity assumptions (Definition~\ref{def:regular}) at scale $\eta$ in this direction $u^*$.
In particular, if one could show that $\eta^*_d \to 0$ as $d\to \infty$, then this would establish sensitivity of the behavior of the MLE to the structure of the
parameter direction $u^*$.

We start with Assumption~\ref{ass:small-ball} on the one-dimensional marginal $\innerp{u^*}{X} = \frac{1}{\sqrt{d}} \sum_{j=1}^d X_j$.
Under Assumption~\ref{ass:independent-coordinates}, this random variable is a normalized sum of \iid random variables.
It then follows from the Berry-Esseen inequality that
its distribution
approaches the standard Gaussian distribution, down to a scale of order $1/\sqrt{d}$.
This implies the following:

\begin{lemma}
  \label{lem:onedim-indep}
  Let $X$ satisfy Assumption~\ref{ass:independent-coordinates}.
  Then, for every $u \in S^{d-1}$ such that
  $\norm{u}_3 \leq K^{-1}$
  and any $t \in [K^3 \norm{u}_3^3, 1]$, one has
  \begin{equation}
    \label{eq:onedim-indep}
    \frac{t}{4}
    \leq
    \P \big( \abs{\innerp{u}{X}} \leq t \big)
    \leq t
    \, .
  \end{equation}
  In particular, if $d \geq K^6$ and $u^* = (1/\sqrt{d}, \dots, 1/\sqrt{d})$, then Assumption~\ref{ass:small-ball} holds with $\eta = K^3/\sqrt{d}$ and $c = 1$.
\end{lemma}

Lemma~\ref{lem:onedim-indep} (whose proof is provided in Section~\ref{sec:proof-regularity-iid}) shows that the one-dimensional marginal $\innerp{u^*}{X}$ exhibits the ``right'' behavior down to a scale $\eta \asymp 1/\sqrt{d}$.

However, as discussed in Section~\ref{sec:assumptions} (see Proposition~\ref{prop:2-dim-margin-almost-necessary}), Assumption~\ref{ass:small-ball} on the one-dimensional marginal $\innerp{u^*}{X}$ does not suffice to establish a near-Gaussian behavior of the MLE; indeed, for this task one must establish Assumption~\ref{ass:twodim-marginals} on two-dimensional marginals $(\innerp{u^*}{X}, \innerp{v}{X})$ for every $v \in S^{d-1}$.
In order to simplify the discussion, let us consider the special case where $v \in S^{d-1}$ is orthogonal to $u^*$.
In this case, Assumption~\ref{ass:twodim-marginals} is of the form
\begin{equation}
  \label{eq:twodim-orthogonal}
  \P \Big( \ainnerp{u^*}{X} \leq c \eta, \, \ainnerp{v}{X} \geq \frac{1}{c} \Big)
  \geq \frac{\eta}{c}
\end{equation}
for some constant $c$.
In the case where $X \sim \gaussdist (0, I_d)$ is Gaussian, condition~\eqref{eq:twodim-orthogonal} immediately follows from the fact that $\innerp{u^*}{X}$ and $\innerp{v}{X}$ are independent if $\innerp{u^*}{v} = 0$.
However, this property is highly specific to the Gaussian case, and does not extend to the more general case of product measures.

By analogy with the proof of Assumption~\ref{ass:small-ball}, a natural attempt to establish condition~\eqref{eq:twodim-orthogonal} is to resort to Gaussian approximation.
Specifically, by applying a two-dimensional Berry-Esseen inequality to the random vector $(\innerp{u^*}{X}, \innerp{v}{X}) = \sum_{j=1}^d X_j \omega_j$ with $\omega_j = (u^*_j, v_j) = (1/\sqrt{d}, v_j)$ (such that $\sum_{j=1}^d \omega_j \omega_j^\top = I_2$) and proceeding as in Lemma~\ref{lem:onedim-indep}, one can show that condition~\eqref{eq:twodim-orthogonal} holds down to $\eta \asymp \sum_{j=1}^d \norm{\omega_j}_2^3 \asymp \max \{ \norm{u^*}_3^3, \norm{v}_3^3 \} \asymp \max \{ 1/\sqrt{d}, \norm{v}_3^3 \}$.
This approach ensures that~\eqref{eq:twodim-orthogonal} holds for small $\eta$ whenever $v$ is sufficiently diffuse that $\norm{v}_3^3$ is small.
Unfortunately, %
condition~\eqref{eq:twodim-orthogonal} must hold for \emph{every $v \in S^{d-1}$} such that $\innerp{u^*}{v} = 0$, and in particular for non-diffuse vectors $v$ such that $\norm{v}_3^3 \asymp 1$ (for instance $v = (1/\sqrt{2}, -1/\sqrt{2}, 0, \dots, 0)$).
For such vectors $v \in S^{d-1}$, Gaussian approximation gives vacuous guarantees.

As it happens, an entirely different argument (based on ``approximate separation of supports'') can be used to handle the case of ``sparse'' vectors, which---when suitably combined with Gaussian approximation---allows one to establish regularity at a non-trivial scale $\eta_d \asymp d^{-1/4} \to 0$.
In order to convey the idea of this argument, and to illustrate how the $d^{-1/4}$ scaling naturally arises from this approach, we provide a high-level overview of the argument at the end of Section~\ref{sec:proof-regularity-iid}.
Since the estimate on $\eta_d$ obtained with this approach is sub-optimal and is improved in Lemma~\ref{lem:2d-margin-indep} below, we only provide a sketch of proof that omits significant technical details.

The argument we just alluded to leads to a scale of $d^{-1/4}$ for the two-dimensional margin assumption, 
which is larger than the scale of $d^{-1/2}$ obtained in Lemma~\ref{lem:onedim-indep} for one-dimensional marginals.
This naturally raises the question of whether the $d^{-1/4}$ scale can be improved to $d^{-1/2}$ by a refined analysis.
Lemma~\ref{lem:2d-margin-indep} below shows that this is indeed the case:

\begin{lemma}
  \label{lem:2d-margin-indep}
  Let $X = (X_1, \dots, X_d)$ have \iid coordinates, with $\E [X_1] = 0$, $\E [X_1^2] = 1$ and $\E[X_1^8] \leq \kappa^8$ for some $\kappa \geq 1$.
  Assume that $d \geq 2025 \kappa^6$, define $u^* = (1/\sqrt{d}, \dots, 1/\sqrt{d})$ and let $\eta \in [45 \kappa^3/\sqrt{d}, 1]$.
  Then, for every $v \in S^{d-1}$ such that $\innerp{u^*}{v} \geq 0$, one has
  \begin{equation}
    \label{eq:2d-margin-indep}
    \P \Big( \ainnerp{u^*}{X} \leq \eta, \, \ainnerp{v}{X} \geq 0.2 \max \{ \eta, \norm{u^* - v}_2 \} \Big)
    \geq \frac{\eta}{70\,000 \,\kappa^4}
    \, .
  \end{equation}
  In particular, if $X$ satisfies Assumption~\ref{ass:independent-coordinates}, then Assumption~\ref{ass:twodim-marginals} holds for any $\eta \in [18 K^3/\sqrt{d}, e^{-1}]$ with $c = 21\,000$.
\end{lemma}

Lemma~\ref{lem:2d-margin-indep} is a somewhat delicate result, so before discussing its implications we first explain the main idea behind its proof.
The detailed proof may be found in Section~\ref{sec:proof-regularity-iid}.

We need to show
that, conditionally on the fact that $\ainnerp{u^*}{X} \leq \eta$, the variable $\innerp{v}{X}$ fluctuates on a scale of order at least $\max \{ \eta, \norm{u^* - v} \}$.
Since $\innerp{v}{X} = \langle u^*,v\rangle \innerp{u^*}{X} + \sqrt{1-\innerp{u^*}{v}^2} \innerp{w}{X}$ with $\innerp{u^*}{w} = 0$, this means roughly speaking that the variables $\innerp{u^*}{X}$ and $\innerp{w}{X}$ %
behave as if they were independent.
Of course, the main difficulty is that these variables are not in fact independent, except in the very special case where the vectors $u^*$ and $w$ have disjoint supports.
In addition, Gaussian approximation on the vector $(\innerp{u^*}{X}, \innerp{w}{X})$ fails in general since $w \in S^{d-1}$ is arbitrary.

We therefore need to show that $\innerp{v}{X}$ exhibits some variability under the event that $\innerp{u^*}{X}$ is small, in the absence of independence properties.
The main idea to achieve this is to ``perturb'' the vector $X = (X_1, \dots, X_d)$ by randomly permuting its coordinates.
Specifically, given a permutation $\sigma \in \Sym_d$ of $\{1, \dots, d \}$, we let $X^\sigma = (X_{\sigma (1)}, \dots, X_{\sigma (d)})$.
We introduce an additional source of randomness (besides $X$) by taking $\sigma$ to be random, drawn uniformly over the symmetric group $\Sym_d$, and independent of $X$.
These transformations are useful thanks to the following properties:
\begin{enumerate}
\item The vector $X^\sigma$ has the same distribution as $X$ for a fixed $\sigma$, and thus also for random $\sigma$;
\item Permutations preserve $\innerp{u^*}{X}$, as
  $\innerp{u^*}{X^\sigma} = \frac{1}{\sqrt{d}} \sum_{j=1}^d X_{\sigma (j)} = \frac{1}{\sqrt{d}} \sum_{j=1}^d X_j = \innerp{u^*}{X}$;
\item Conditionally on $X$ (for most values of $X$), the quantity $\innerp{v}{X^\sigma} = \sum_{j=1}^d v_j X_{\sigma (j)}$ fluctuates on the desired scale of $\max \{ \eta, \norm{u^* - v} \}$, as the random permutation $\sigma$ varies.
\end{enumerate}
Since the first claim (exchangeability) follows immediately from Assumption~\ref{ass:independent-coordinates}, the main step is to justify the third claim.
We establish it by applying the Paley-Zygmund inequality, which reduces the task to lower-bounding one moment of $\innerp{v}{X^\sigma}$ (conditionally on $X$ and with respect to random $\sigma$), and to upper-bounding a higher-order moment, ideally to conclude that they are of the same order of magnitude.
In addition, one may explicitly evaluate the moments of even integer order, as this reduces to computations over symmetric polynomials in $X_1, \dots, X_d$.
After suitable simplifications (exploiting that $\sum_{j=1}^d w_j = \sqrt{d} \innerp{u^*}{w} = 0$), we can show that this is indeed the case, provided that $X = (X_1, \dots, X_d)$ satisfies some symmetric conditions that do hold with high probability.
We refer to Section~\ref{sec:proof-regularity-iid} for more details on this proof.

We can now gather the conclusions of Lemmas~\ref{lem:onedim-indep} and~\ref{lem:2d-margin-indep} into the following statement, which is the main result of the present section.

\begin{proposition}
  \label{prop:regularity-iid}
  Let $X = (X_1, \dots, X_d)$ satisfy Assumption~\ref{ass:independent-coordinates}, set $u^* = (1/\sqrt{d}, \dots, 1/\sqrt{d})$ and assume that $d \geq K^6$.
  Then $X$ is $4K$-sub-exponential and $(u^*, \eta, c)$-regular with $c = 21\,000$ for any $\eta \in [18 K^3/\sqrt{d}, e^{-1}]$.
\end{proposition}

It then follows from Theorem~\ref{thm:regular-well-specified} that, if $\theta^* = (\b/\sqrt{d}, \dots, \b/\sqrt{d})$, then the MLE behaves in a similar way as if the design was Gaussian as long as $\b = O (\sqrt{d})$.
Hence, in this direction, the ``discrete'' nature of the design has no impact, even for a moderately strong signal.

It is natural to ask if the sufficient condition $\b = O (\sqrt{d})$ is also necessary to exhibit a Gaussian-like behavior.
The following simple example shows that this is %
indeed the case.

\begin{fact}
  \label{fac:bernoulli-sqrt-d-optimal}
  Let $d$ be an odd integer, $X = (X_1, \dots, X_d)$ a Bernoulli design, and let $Y$ given $X$ follow the logit model with parameter $\theta^* = (\b/\sqrt{d}, \dots, \b/\sqrt{d})$ for some $\b \geq \sqrt{d}$.
  Given an \iid sample of size $n\geq 1$ from this distribution, if $n \leq 0.1 \exp (\b/\sqrt{d})$ then $\P (\text{\normalfont MLE exists}) \leq 0.1$.
\end{fact}

\begin{proof}
  The proof is the same as that of Fact~\ref{fac:exponential-dependence-bernoulli}, except that the condition $\ainnerp{\theta^*}{X} \geq \b$ therein is now replaced by $\ainnerp{\theta^*}{X} \geq \b/\sqrt{d}$.
  Indeed, one has $\ainnerp{\theta^*}{X} = \b | \sum_{j=1}^d X_j |/\sqrt{d} \geq \b/\sqrt{d}$ since $\sum_{j=1}^d X_j$ is an odd integer.
\end{proof}

In other words, if $\b \gg \sqrt{d}$ then some exponential dependence on $\b$ is again necessary for the MLE to exist.
In particular, the regularity scale of $\eta \asymp 1/\sqrt{d}$ is indeed optimal for the Bernoulli design in the direction $u_d^* = (1/\sqrt{d}, \dots, 1/\sqrt{d})$.

Now, since $u_d^* = (1/\sqrt{d}, \dots, 1/\sqrt{d})$ is the most ``well-spread'' vector in $S^{d-1}$, it is perhaps tempting to conjecture that it is the ``best'' direction from the perspective of logistic regression, that is, the one with the smallest regularity scale $\eta$.
If this were indeed the case, then for a ``typical'' direction $u^* \in S^{d-1}$ one would expect a regularity scale of $1/\sqrt{d}$ at best.

Interestingly, this is \emph{not} the case, at least for the one-dimensional Assumption~\ref{ass:small-ball}.
It turns out that, for a ``typical'' direction $u^* \in S^{d-1}$, Assumption~\ref{ass:small-ball} is satisfied down to a smaller scale, of order $1/d$ instead of $1/\sqrt{d}$.
This follows from a remarkable result of Klartag and Sodin~\cite{klartag2012variations}, which states that for a typical direction $u = (u_1, \dots, u_d) \in S^{d-1}$, the distribution of the linear combination $\innerp{u}{X} = \sum_{j=1}^d u_j X_j$ approaches the Gaussian distribution at a rate of $1/d$, which is faster than the $1/\sqrt{d}$ rate for the normalized sum $\frac{1}{\sqrt{d}} \sum_{j=1}^d X_j$.
We discuss the nature of this improvement and raise related open questions in Section~\ref{sec:impr-regul-scal}.

\section{Proof scheme and main lemmas}
\label{sec:proof-scheme}

In this section, we describe the general scheme of proof that we use to establish Theorems~\ref{thm:gaussian-well-specified},~\ref{thm:regular-well-specified} and~\ref{thm:regular-misspecified}, as well as the main lemmas in the analysis.

\subsection{Convex localization%
}
\label{sec:localization}

The following standard lemma, which is based on a convex localization argument, is used to both establish existence of the MLE and control its risk.
This reduction is general and deterministic: the only properties that it uses,
besides those explicitly stated in Lemma~\ref{lem:localization}, are that $\wh L_n,L$ are twice continuously differentiable, that $\wh L_n$ is convex and that $\theta^*$ is a global minimizer of $L$.

\begin{lemma}
  \label{lem:localization}
  Assume that there exists a positive-definite matrix $H \in \R^{d\times d}$ and real numbers $r_0, c_0, c_1, \nu > 0$ such that the following conditions hold: 
  \begin{itemize}
  \item $\norm{\nabla \wh L_n (\theta^*)}_{H^{-1}} \leq \nu$; %
  \item For every $\theta \in \R^d$ such that $\norm{\theta - \theta^*}_H \leq r_0$, one has $\nabla^2 \wh L_n (\theta) \mgeq c_0 H$;
  \item For every $\theta \in \R^d$ such that $\norm{\theta - \theta^*}_H \leq r_0$, one has $\nabla^2 L (\theta) \mleq c_1 H$.
  \end{itemize}
  If $\nu < c_0 r_0/2$, then the empirical risk $\wh L_n$ admits a unique global minimizer $\wh \theta_n$, which satisfies
  \begin{equation}
    \label{eq:norm-risk-bound-localization}
    \big\| \wh \theta_n - \theta^* \big\|_H
    \leq \frac{2 \nu}{c_0} 
    \qquad \mbox{and} \qquad
    L (\wh \theta_n) - L (\theta^*)
    \leq \frac{2 c_1 \nu^2}{c_0^2}
    \, .
  \end{equation}
  If in addition $\nu < c_0 r_0/4$, then for any $\wt \theta_n \in \R^d$ such that $%
  \wh L_n (\wt \theta_n) - \wh L_n (\wh \theta_n) < c_0 r_0^2/4$, one has
  \begin{equation}
    \label{eq:almost-min-localization}
    L (\wt \theta_n) - L (\theta^*)
    \leq \frac{c_1}{2} \big\| \wt \theta_n - \theta^* \big\|_H^2
    \leq \max \bigg\{ \frac{8 c_1 \nu^2}{c_0^2} , \frac{2 c_1}{c_0} %
    \Big[ \wh L_n (\wt \theta_n) - \wh L_n (\wh \theta_n) \Big]
    \bigg\}    
    \, .
  \end{equation}
\end{lemma}

Lemma~\ref{lem:localization} (proved in Section~\ref{sec:proof-localization}) reduces the proof of existence and risk bounds for the MLE (or approximate minimizers $\wt \theta_n$ of the empirical risk $\wh L_n$ thanks to~\eqref{eq:almost-min-localization}) to two main components:
\begin{itemize}
\item a high-probability upper bound on the $H^{-1}$-norm  $\norm{\nabla \wh L_n (\theta^*)}_{H^{-1}}$ of the empirical gradient at $\theta^*$;
\item a high-probability lower bound $\nabla^2 \wh L_n (\theta) \mgeq c_0 H$ on the Hessian of the empirical risk at $\theta$, uniformly over all $\theta \in \Theta = \{ \theta \in \R^d : \norm{\theta - \theta^*}_H \leq r_0 \}$.
\end{itemize}
The risk bound is then given by~\eqref{eq:norm-risk-bound-localization}, while the condition for existence of $\wh \theta_n$ is that $\nu < c_0 r_0/2$.
In particular, the smaller $\nu = \nu_n$ and the larger $r_0$, the weaker the condition for existence of $\wh \theta_n$.

Although the matrix $H$ (and the corresponding parameters $c_0,c_1, r_0, \nu$) from Lemma~\ref{lem:localization} can in principle be arbitrary, in order to obtain tight guarantees, a natural choice is to take $H$ to be equivalent up to constant factors to $\nabla^2 L (\theta^*)$, the Hessian of the risk at $\theta^*$, which coincides in the well-specified case with the Fisher information.

Indeed, in order to obtain sharp bounds we would like $c_0, c_1$ to be of constant order, and indeed in the Gaussian case these will be universal constants.
Now by assumption one has $\nabla^2 L (\theta) \mleq c_1 H$ for all $\theta \in \Theta = \{ \theta \in \R^d : \norm{\theta - \theta^*}_H \leq r_0 \}$, while $\nabla^2 \wh L_n (\theta) \mgeq c_0 H$ for any $\theta \in \Theta$.
Now for large $n$, by the law of numbers $\nabla^2 \wh L_n(\theta)$ should be close to its expectation $\E [ \nabla^2 \wh L_n (\theta) ] = \nabla^2 L (\theta)$, so for the latter condition to hold with high probability, one should also have $\nabla^2 L (\theta) \mgeq c_0 H$ for all $\theta \in \Theta$.
This implies that $H$ is equivalent to $\nabla^2 L (\theta^*)$, namely $c_0 H \mleq \nabla^2 L (\theta^*) \mleq c_1 H$.
In addition, this constrains the domain $\Theta$ (namely the parameter $r_0$), which must be contained in the set 
\begin{equation}
  \label{eq:domain-equiv-hessian}
  \Theta'
  = \Big\{ \theta \in \R^d : c_2^{-1} \nabla^2 L (\theta^*) \mleq \nabla^2 L (\theta) \mleq c_2 \nabla^2 L (\theta^*) \Big\}
\end{equation}
with $c_2 = c_1/c_0$.

It follows from these considerations that, in order to apply Lemma~\ref{lem:localization} effectively,
a first step is to understand the behavior of the Hessian $\nabla^2 L (\theta)$ for $\theta \in \R^d$---both to set the matrix $H \approx \nabla^2 L(\theta^*)$,
and to identify the largest possible region~\eqref{eq:domain-equiv-hessian} where the conditions of Lemma~\ref{lem:localization} could be expected to hold.

By rotation-invariance of the Gaussian distribution, when $X \sim \gaussdist (0, I_d)$ the Hessian $\nabla^2 L (\theta)$ commutes with any linear isometry of $\R^d$ that fixes $\theta$, and is therefore of the form
\begin{equation}\label{eq:HessianGaussian}
  \nabla^2L(\theta)
  = c_0(\|\theta\|) u u^\top + c_1(\|\theta\|)(I_d-u u^\top)\enspace,
\end{equation}
where $u = \theta/\norm{\theta}$ (for $\theta \neq 0$), and letting $G \sim \gaussdist(0,1)$ we have for $b \in \R^+$:
\[
  c_0(b)=\E[\sigma'(b G) G^2],\quad c_1(b)=\E[\sigma'(b G)]
  \, .
\]
In addition, one may verify (see Lemma~\ref{lem:MomGauss} in Section~\ref{sec:proof-localization}) that for some numerical constants $c_0',c_0'',c_1',c_1''$:
\begin{equation}
  \label{eq:components-hessian}
  \frac{c_0'}{(b + 1)^3}
  \leq c_0 (b)
  \leq \frac{c_0''}{(b + 1)^3}
  \, ,
  \qquad
  \frac{c_1'}{b + 1}
  \leq c_1 (b)
  \leq \frac{c_1''}{b + 1}
  \, .
\end{equation}
We will therefore set $H$ to be the matrix
\begin{equation}\label{eq:defH}
  H = \frac{1}{\b^{3}} u^{*} {u^{*}}^{\top} + \frac{1}{\b} (I_{d} - u^{*} {u^{*}}^{\top}),\qquad u^{*} = \frac{\theta^{*}}{\|\theta^{*}\|}\in S^{d-1},\qquad \b = \max(e,\|\theta^*\|)  \, ,  
\end{equation}
so that $c_0 H \mleq \nabla^2 L (\theta^*) \mleq c_1 H$ for some absolute constants $c_0,c_1$ for a Gaussian design.

In addition, it can be deduced from this characterization of $\nabla^2 L (\theta)$ that the region~\eqref{eq:domain-equiv-hessian} (for large $\b$ and constant $c_2$) where the Hessian is equivalent to $\nabla^2 L (\theta^*)$ coincides up to constants with an ellipsoid of the form $\{ \theta \in \R^d : \norm{\theta - \theta^*}_{H} \leq r_0 \}$, where $r_0 \asymp 1/\sqrt{\b}$.

\subsection{Upper bounds on the empirical gradient}
\label{sec:upper-bounds-gradient}

We now consider the first ingredient in the application of Lemma~\ref{lem:localization}, namely high-probability upper bounds on the $H^{-1}$-norm of the empirical gradient:
\begin{equation}
  \label{eq:empirical-gradient}
  \norm{\nabla \wh L_n (\theta^*)}_{H^{-1}}
  = \bigg\| \frac{1}{n} \sum_{i=1}^n \nabla \ell (\theta^*, (X_i,Y_i)) \bigg\|_{H^{-1}}
  = \bigg\| \frac{1}{n} \sum_{i=1}^n \sigma (- Y_i \innerp{\theta^*}{X_i}) H^{-1/2} X_i \bigg\|
  \, .
\end{equation}
We describe below our guarantees in the following three cases: (i) Gaussian design, well-specified model, (ii) regular design, well-specified model and (iii) regular design, misspecified model.
(We note in passing that
in order to control the gradient, we only require
Assumptions~\ref{ass:sub-exponential} and~\ref{ass:small-ball}.)

We start with the first case.
A natural approach (which is essentially that of~\cite{ostrovskii2021finite}) is to use that $\sigma_i =  \sigma ( -Y_i \innerp{\theta^*}{X_i}) \leq 1$ and that $X_i$ is sub-Gaussian for each $i$, to deduce that the individual summands in~\eqref{eq:empirical-gradient} are $H^{-1}$-sub-Gaussian.
By standard deviation bounds for sub-Gaussian vectors, this implies that for some constant $c > 0$, with probability at least $1-e^{-t}$,
\begin{equation*}
  \norm{\nabla \wh L_n (\theta^*)}_{H^{-1}}
  \leq c \sqrt{\frac{\tr (H^{-1}) + \opnorm{H^{-1}} t}{n}}
  \leq c \sqrt{\frac{\b d + \b^3 (t+1)}{n}}
  \, .
\end{equation*}
Unfortunately, this bound features a suboptimal dependence on the norm $\b$.
In order to improve it, the key observation is the following: if $
Y_i \innerp{\theta^*}{X_i} \geq 0
$, then the sigmoid $\sigma_i$ is bounded as
\begin{equation*}
  \sigma_i
  = \sigma (-Y_i \innerp{\theta^*}{X_i})
  = \sigma (- \ainnerp{\theta^*}{X_i})
  \leq \exp (- \ainnerp{\theta^*}{X_i})
  \, ,
\end{equation*}
which is very small if $\ainnerp{\theta^*}{X_i}$ is large.
On the other hand, if $
Y_i \innerp{\theta^*}{X_i} < 0
$, then the sigmoid is no longer small (specifically, $\frac{1}{2} \leq \sigma_i \leq 1$).
However, \emph{this configuration is highly unlikely if $\ainnerp{\theta^*}{X_i}$ is large}: indeed, using that the model is well-specified, one has
\begin{equation}
  \label{eq:misclassification-well-spec}
  \P (Y_i \innerp{\theta^*}{X_i} < 0 |X_i)
  = \sigma (- \ainnerp{\theta^*}{X_i})
  \leq \exp (- \ainnerp{\theta^*}{X_i})
  \, .
\end{equation}
Hence, the only remaining situation where $\sigma_i$ may not be small is when $\ainnerp{\theta^*}{X_i}$ is upper-bounded; but since $\innerp{\theta^*}{X_i} \sim \gaussdist (0, \norm{\theta^*}^2)$, the probability that $\ainnerp{\theta^*}{X_i} \lesssim 1$ is of order $1/\b$, which is small when $\b$ is large.

From a technical standpoint, the considerations above allow us to obtain improved upper bounds (compared to those obtained by bounding $|\sigma_i| \leq 1$) on the moments of the random variables $\innerp{v}{H^{-1/2} \nabla \ell (\theta^*, (X_i,Y_i))}$ for $v \in S^{d-1}$, whose supremum is precisely the norm~\eqref{eq:empirical-gradient}.
Specifically, %
these random variables can be shown %
to satisfy the \emph{sub-gamma} property~\cite[\S2.4]{boucheron2013concentration}, which we recall in Definition~\ref{def:sub-gamma}.
Using a deviation bound for sub-gamma random vectors (Lemma~\ref{lem:ConcSGVec}), we deduce the following result, proved in Section~\ref{sec:gradient-wellspec-gaussian}:

\begin{proposition}\label{prop:DevGradGauss}
  Assume that $X$ is Gaussian and the model is well-specified.
  Let $H$ be the matrix defined in \eqref{eq:defH}.
  For any $t\geq 0$, if $n\geqslant 16\b (d+t)$ then with probability at least $1 - 3e^{-t}$, one has
\begin{equation*}
  \big\| \nabla \wh L_{n} (\theta^{*}) \big\|_{H^{-1}}
  \leq %
  14 \sqrt{\frac{d + t}{n} } \enspace .
\end{equation*}
\end{proposition}

We now turn to the more general case of a regular design, but still assuming a well-specified model.
Here the guarantees are quite similar to the Gaussian case,
and the high-level argument sketched above remains valid.
However, two important properties of the Gaussian distribution that we used in the proof of Proposition~\ref{prop:DevGradGauss} no longer hold for general regular distributions:
(1) linear marginals $\innerp{u}{X}$ and $\innerp{v}{X}$ in orthogonal directions $u,v \in S^{d-1}$ are independent, and (2) the distribution of $\innerp{u^*}{X}$ admits a bounded (by $\frac{1}{\sqrt{2\pi}}$) density.
The lack of independence is handled by using that $X$ is sub-exponential (leading to an additional $\log \b$ factor); while to get around the lack of bounded density, we decompose the relevant expectations (that define the moments of the gradient) over a geometric grid of scales.
Using these arguments to again
show that gradients admit sub-gamma moments,
we obtain the following bound, proved in Section~\ref{sec:gradient-wellspec-regular}.

\begin{proposition}\label{prop:DevGradWSRD}
  Assume that $X$ satisfies Assumptions~\ref{ass:sub-exponential} and~\ref{ass:small-ball}
  with parameters $K$, $u^*$,
  $\eta = \b^{-1}$ and $c \geq 1$, and that the model is well-specified. 
For any $t\geq 0$, if $n\geqslant \b (d+t)$ then with probability at least $1 - 3e^{-t}$, one has
\begin{equation*}
  \big\| \nabla \wh L_{n} (\theta^{*}) \big\|_{H^{-1}} \leq c_0 \log (\b) \sqrt{\frac{d + t}{n} } \enspace ,
\end{equation*}
where $c_0>0$ %
is a constant that depends only on $K$ and $c$.
\end{proposition}

We now conclude with the most general case we consider, where the design is regular but the model is no longer assumed to be well-specified.
The fact that the model may be misspecified induces a significant change: the key bound~\eqref{eq:misclassification-well-spec} on the conditional probability of misclassification no longer holds.
Given that if $
Y_i \innerp{\theta^*}{X_i} < 0
$, then $\sigma_i = \sigma (- Y_i \innerp{\theta^*}{X_i}) \in [1/2, 1]$ is of constant order, and that no bound on $\P (
Y_i \innerp{\theta^*}{X_i} < 0
|X_i)$ is available, it might be tempting to simply bound $\abs{\sigma_i} \leq 1$, which as discussed above leads to a bound of order $\sqrt{(\b d + \b^3 t)/n}$.

As it happens, this bound is suboptimal and can be improved even in the misspecified case.
The reason for this is that, if the parameter $\theta^* = \argmin_{\theta \in \R^d} L (\theta)$ has a large norm $\b$, then
the (unconditional) probability $\P (Y \innerp{\theta^*}{X} < 0)$ of misclassification of $\theta^*$ must be small.
The key result that expresses this intuition is Lemma~\ref{lem:first-moments-misspecified}, which shows that the probability of misclassification $\P (Y \innerp{\theta^*}{X} < 0)$ and the first moment $\E [ \ainnerp{u^*}{X} \indic{Y \innerp{\theta^*}{X} < 0} ]$ are bounded in the general misspecified case in a similar way as in the well-specified case.
This allows one to refine the naive bound of $\sqrt{(\b d + \b^3 t)/n}$ into a near-optimal bound of $\log (\b) \sqrt{(d + \b t)/n}$.

\begin{proposition}
\label{prop:gradient-dev-MS}
Assume that $X$ satisfies
Assumptions~\ref{ass:sub-exponential} and~\ref{ass:small-ball} with parameters $K$ and $(u^{*}, \b^{-1}, c)$, but not that the model is well-specified.
For any $t \geq 0$, if $n\geqslant \b (d+\b t)
$, 
then with probability at least $1-3e^{-t}$,
\begin{equation*}
  \big\| \nabla \hat{L}_{n}(\theta^{*}) \big\|_{H^{-1}}
  \leq c_0 \log(\b)\sqrt{\frac{d+\b t}{n}} \, ,
\end{equation*}  
where $c_0>0$ is a constant that depends only on $K$ and $c$.
\end{proposition}

The proof of Proposition~\ref{prop:gradient-dev-MS} may be found in Section~\ref{sec:gradient-misspec-regular}.

\subsection{Lower bounds on empirical Hessian matrices}
\label{sec:lower-bounds-hessian}

We now turn to the second component of the proof scheme of Lemma~\ref{lem:localization}, namely a high-probability lower bound on the Hessian of the empirical risk:
\begin{equation}
  \label{eq:empirical-hessian}
  \wh H_n (\theta)
  = \nabla^2 \wh L_n (\theta)
  = \frac{1}{n} \sum_{i=1}^n \sigma' (\innerp{\theta}{X_i}) X_i X_i^\top
  \, ,
\end{equation}
where $\sigma' (s)
= \{(1+e^s)(1+e^{-s})\}^{-1}
$ for $s \in \R$, uniformly for $\theta$ in a neighborhood of $\theta^*$ that is as large as possible.
Specifically, it follows from the discussion of Section~\ref{sec:localization} that an ``ideal'' guarantee would be of the form: for $n$ large enough (depending on $\b,d,t$),
\begin{equation}
  \label{eq:uniform-lower-hessian}
  \P \Big( \forall \theta \in \Theta, \ \wh H_n (\theta) \mgeq c_0 H \Big)
  \geq 1 - e^{-t}
  \, , \quad
  \text{where} \quad
  \Theta
  = \Big\{ \theta \in \R^d :
  \norm{\theta - \theta^*}_H \leq \frac{c_1}{\sqrt{\b}} \Big\}
\end{equation}
for some constants $c_0,c_1$ that should not depend (or weakly depend) on $n,\b,d$, where the matrix $H$ is defined in~\eqref{eq:defH}.

As is clear from the expression~\eqref{eq:empirical-hessian}, the empirical Hessian matrix $\wh H_n (\theta)$ only depends on $X_1, \dots, X_n$ and not on the labels $Y_1, \dots, Y_n$.
Hence, the behavior of $\wh H_n (\theta)$ depends on the distribution of $X$ but not on the conditional distribution of $Y$ given $X$.
As such, there is no distinction between the well-specified and misspecified cases, and we only have to consider two cases: Gaussian design and regular design.
While the Gaussian design is a special case of regular design, we consider it separately because in this case we obtain sharper guarantees, involving universal constants rather than poly-logarithmic factors in $\b$.

We start with the general case of a regular design, because %
its analysis
is actually simpler than that of the Gaussian case.
Theorem~\ref{thm:hessians-regular-case} below provide an almost optimal uniform lower bound on the empirical Hessian of the form~\eqref{eq:uniform-lower-hessian}, up to logarithmic factors in $\b$.
We note in passing that this control on the Hessian only requires Assumptions~\ref{ass:sub-exponential} and~\ref{ass:twodim-marginals}, while Assumption~\ref{ass:small-ball} was used in the control of the gradient discussed in Section~\ref{sec:upper-bounds-gradient}.

\begin{theorem}
\label{thm:hessians-regular-case}
Let $X$ be a random vector satisfying Assumptions~\ref{ass:sub-exponential} and \ref{ass:twodim-marginals}
with parameter $K\geqslant e$, $u^* = \theta^*/\norm{\theta^*}$, $\eta=1/\b$ and $c\geq 1$.
There exist constants $c_1, c_2, c_3 > 0$ that depend only on $c$ and $K$ for which the following holds: for any $t \geq 0$, if
\begin{equation*}
  n \geq c_1 \b \big( \log (\b) d + t \big)
\end{equation*}
then
with probability at least $1 - e^{-t}$, one has
\begin{equation*}
  \wh H_n(\theta) \mgeq c_2 H
  \quad \text{for every }
  \theta \in \R^d
  \text{ such that}
  \quad
  \norm{\theta - \theta^*}_H \leq \frac{c_3}{\log (\b) \sqrt{\b}}
  \, .
\end{equation*}
\end{theorem}

Theorem~\ref{thm:hessians-regular-case} is proved in Section~\ref{sec:proof-hessian-regular},
and we discuss here the main ideas of the proof.
The key observation is that a certain property of the dataset implies the desired behavior~\eqref{eq:uniform-lower-hessian}.
Specifically, the first step is to notice that
if $X_1, \dots, X_n$ satisfy, for some constants $c_0,c_1$, 
\begin{equation}
  \label{eq:deterministic-condition-hessian}
  \inf_{\substack{u,v \in S^{d-1} \\ \norm{u - u^*} \leq c_0/\b, \, \innerp{u^*}{v} \geq 0}}
  \bigg[ \sum_{i=1}^n \bm 1 \Big\{\ainnerp{u}{X_i} \leq \frac{c_1}{\b} \, ; \, \ainnerp{v}{X_i} \geq \frac{\max\{ \b^{-1}, \norm{u^* - v} \}}{c_1} \Big\} \bigg]
  \geq \frac{n}{2 c_1 \b}
  \, ,
\end{equation}
then $\wh H_n (\theta) \mgeq c_2 H$ for every $\theta \in \R^d$ such that $\norm{\theta - \theta^*}_H \leq c_3/\sqrt{\b}$, for some constants $c_2, c_3$ that depend on $c_0,c_1$.
This follows from properties of the function $\sigma'$ and the structure of $H$.

It then remains to establish that condition~\eqref{eq:deterministic-condition-hessian} holds with high probability over the random draw of $X_1, \dots, X_n$.
To achieve this, observe first %
that condition~\eqref{eq:deterministic-condition-hessian} is essentially a variant of Assumption~\ref{ass:twodim-marginals}, with two differences: (i) it holds for any $u \in S^{d-1}$ such that $\norm{u - u^*} \leq c_0/\b$, rather than just for $u = u^*$, and (ii) it holds for the random sample $X_1, \dots, X_n$, rather than for the distribution $P_X$.

Condition~\eqref{eq:deterministic-condition-hessian} (or rather, a slightly weaker version with additional $\log \b$ factors) is thus established in two steps.
First, we show that Assumption~\ref{ass:twodim-marginals} on $P_X$ extends to all directions $u \in S^{d-1}$ such that $\norm{u-u^*} \leq c_3 / ({\b}\log \b)$.
Second, we show that this condition on $P_X$ is stable under random sampling with high probability, by using that the class of events in~\eqref{eq:deterministic-condition-hessian} is a Vapnik-Chervonenkis (VC) class with VC dimension at most $O (d)$,
and then applying a uniform lower bound on the empirical frequencies of a VC class of sets.

Theorem~\ref{thm:hessians-regular-case} applies to the general regular case, and in particular to the special case of a Gaussian design.
(That the Gaussian design satisfies Assumption~\ref{ass:twodim-marginals} may be verified using independence of orthogonal linear marginals, or alternatively from Proposition~\ref{prop:universal-reg-log-concave}.)
In fact, %
the identification of condition~\eqref{eq:deterministic-condition-hessian} as a structural property implying the ``right'' behavior of the empirical Hessian in the Gaussian case
is what
motivates the definition of Assumption~\ref{ass:twodim-marginals}.

At the same time, it should be noted that the guarantees of Theorem~\ref{thm:hessians-regular-case} feature additional $\log \b$ factors, compared to ``ideal'' guarantees that would lead to Theorem~\ref{thm:gaussian-well-specified} in the Gaussian case.
In particular, the (sufficient) condition on the sample size $n$ from Theorem~\ref{thm:hessians-regular-case} is stronger by a $\log \b$ factor than the necessary condition presented at the end of Theorem~\ref{thm:gaussian-well-specified}.

We address this suboptimality of Theorem~\ref{thm:hessians-regular-case} in the Gaussian case in Theorem~\ref{thm:hessian-gaussian} below, which provides an optimal uniform lower bound on empirical Hessian matrices.

\begin{theorem}
  \label{thm:hessian-gaussian}
  Assume that $X \sim \gaussdist (0, I_d)$ with $d \geq 2$ and that $\b = \norm{\theta^*} \geq e$.
  For any $t \geq 0$, if $n \geq 1\,200\,000 \, \b (d+t)$, then with probability at least $1-2e^{-t}$ one has
\begin{equation*}
  \wh H_n(\theta) 
  \succcurlyeq \frac1{1000} H
  \quad \text{for every }
  \theta \in \R^d
  \text{ such that}
  \quad
  \|\theta-\theta^*\|_H
  \leq \frac{1}{100\sqrt{\b}}
  \, .
\end{equation*} 
\end{theorem}

The proof of Theorem~\ref{thm:hessian-gaussian} is provided in Section~\ref{sec:hessian-gaussian}.
The proof of this sharp result in the Gaussian case happens to be significantly more delicate than that of the more general, but less precise, Theorem~\ref{thm:hessians-regular-case}.
This is because the techniques used to establish Theorem~\ref{thm:hessians-regular-case} (in particular, Vapnik-Chervonenkis arguments) inherently lead to additional logarithmic factors, so the proof of Theorem~\ref{thm:hessian-gaussian} requires a fundamentally different approach.

In order to obtain optimal results in the Gaussian case, we rely instead 
on the so-called PAC-Bayes method, which involves controlling a ``smoothed'' version of the process of interest.
The use of this technique in non-asymptotic statistics was pioneered by Audibert and Catoni~\cite{audibert2011robust}, and starting with Oliveira~\cite{oliveira2016covariance} has found several applications to the non-asymptotic study of random matrices~\cite{catoni2016pac,mourtada2022linear,zhivotovskiy2024dimension}.

In the logistic regression setting we consider, the presence of nonlinear terms (due to the sigmoid $\sigma'$) in the empirical Hessian~\eqref{eq:empirical-hessian} is an additional source of difficulty, which requires new technical ideas.
In particular, instead of applying the PAC-Bayes method to the process of interest, and later controlling the difference between the smoothed version of the process and the process itself, we apply it to an auxiliary process whose smoothed version is (a bound on) the process of interest.
In addition, the smoothing distributions we employ differ from the isotropic Gaussian distributions that have been used in previous works, as they exhibit an anisotropic structure, and as one of their component is far from being Gaussian.
Specifically, we rely on Lemma~\ref{lem:LBHess0} below, which shows that the $\sigma'(\cdot)$ sigmoid terms that appear in the empirical Hessian~\eqref{eq:empirical-hessian} can be lower-bounded by indicators which are smoothed according to a suitable distribution (Definition~\ref{def:PostThetaGauss}).
We then apply the PAC-Bayes inequality to an auxiliary process, where the sigmoid is replaced by an indicator, and control the remaining terms.

\section{Proofs of upper bounds on the empirical gradient%
}
\label{sec:gradients}

In this section, we prove the upper bounds in deviation on the empirical gradient~\eqref{eq:empirical-gradient} presented in Section~\ref{sec:upper-bounds-gradient}.
This section is organized as follows.
In Section~\ref{sec:concSGVec}, we describe the common scheme of proof used in all three settings described in Section~\ref{sec:upper-bounds-gradient}.
Then, Section~\ref{sec:gradient-wellspec-gaussian} contains the proof of Proposition~\ref{prop:DevGradGauss}, while we prove Proposition~\ref{prop:DevGradWSRD} in Section~\ref{sec:gradient-wellspec-regular} and Proposition~\ref{prop:gradient-dev-MS} in Section~\ref{sec:gradient-misspec-regular}.

\subsection{General approach to deviation bounds
}
\label{sec:concSGVec}

Although we will consider three situations of increasing generality, all of our deviation bounds on the norm of the empirical gradient will be established following a similar proof scheme.

First, in order to account for the anistropic structure of the Hessian $H$ defined in~\eqref{eq:defH}, it will be convenient to decompose the gradient along the parameter direction and orthogonal directions.
Specifically, letting $u^* \in S^{d-1}$ such that $\theta^* = \norm{\theta^*} u^*$, and denoting $\Pi :\R^d \to \R^d$ with $\Pi (x) = x - \innerp{u^*}{x} u^*$ the orthogonal projection on the orthogonal of $u^*$, we have for every $g \in \R^d$:
\begin{align*}
  \norm{g}_{H^{-1}}
  &\leq \ainnerp{u^*}{g} \cdot \norm{u^*}_{H^{-1}} + \norm{\Pi (g)}_{H^{-1}} 
    = \b^{3/2} \ainnerp{u^*}{g} + \b^{1/2} \norm{\Pi (g)}
    \, .
\end{align*}
Hence,
\begin{equation}
  \label{eq:decomp-norm-gradient}
  \norm{\nabla \wh L_n (\theta^*)}_{H^{-1}}
  \leq \b^{3/2} \ainnerp{u^*}{\nabla \wh L_n (\theta^*)} + \b^{1/2} \norm{\Pi (\nabla \wh L_n (\theta^*))}
  \, .
\end{equation}

Both terms in the decomposition~\eqref{eq:decomp-norm-gradient} will be controlled using the notion of sub-gamma random variables, which we recall in Definition~\ref{def:sub-gamma} below.
Specifically, we rely on the following standard concentration inequality for sub-gamma random vectors.

\begin{lemma}\label{lem:ConcSGVec}
  Let $W_1, \dots, W_n, W$ be \iid random vectors in $\R^d$.
  Assume that, for some $\nu, K > 0$, the random variable $\innerp{v}{W}$ is $(\nu^2, K)$-sub-gamma for every $v \in S^{d-1}$.
  Then, for any $t \geq 0$, 
  \begin{equation}
    \label{eq:conc-SG-vectors}
    \P \bigg( \norm[\bigg]{\frac{1}{n} \sum_{i=1}^n W_i}
    \geq 2\nu\sqrt{\frac{2(d\log5 + t)}n}+2K\frac{d\log 5 + t}n \bigg)
    \leq \exp (-t)
    \, .
  \end{equation}
\end{lemma}

The deviation bound~\eqref{eq:conc-SG-vectors} depends on two parameters of the random variable that control its variability: a weak parameter $\nu$ controlling moderate deviations, and a strong parameter $K$ controlling large deviations.
As discussed in Lemma~\ref{lem:sub-gamma-exponential}, $\nu$ and $K$ are closely related to the $L^2$ and sub-exponential norms, respectively; however, the sub-gamma condition allows one to obtain a logarithmic improvement in the ratio $K/\nu$.

\begin{proof}[Proof of Lemma~\ref{lem:ConcSGVec}]
By Point 3 in Lemma~\ref{lem:sub-gamma-exponential}, for every $v \in S^{d-1}$, the random variable
\[
  \frac1n\sum_{i=1}^n\innerp{v}{W_i}
\]
is $(\nu^2/n,K/n)$ sub-gamma.
Thus, by Bernstein's inequality, recalled in point 2 of Lemma~\ref{lem:sub-gamma-exponential}, for any $t\geq 0$, one has
\begin{equation}\label{eq:Bern}
  \P\bigg(\frac1n\sum_{i=1}^n\innerp{v}{W_i}
  \geq \nu\sqrt{\frac{2t}n}+K\frac tn\bigg)\leqslant \exp(-t)\enspace.  
\end{equation}
To make this bound uniform and conclude the proof, we use a standard $\epsilon$-net argument.
Let $\mathcal{N}$ denote a maximal set of $1/2$-separated points in $S^{d-1}$.
Then, for each $v \in S^{d-1}$, there exists $v' \in \mathcal{N}$ such that $\norm{v - v'} \leq 1/2$.
Therefore, for any $v\in S^{d-1}$, 
\begin{align*}
  \frac1n\sum_{i=1}^n\innerp{v}{W_i}
  &=  \frac1n\sum_{i=1}^n\innerp{v'}{W_i}+ \innerp[\Big]{v-v'}{\frac1n\sum_{i=1}^n W_i}
    \leqslant \max_{v'\in \mathcal{N}}\frac1n\sum_{i=1}^n\innerp{v'}{W_i}+\frac12 \norm[\bigg]{\frac1n\sum_{i=1}^n {W_i}}\enspace.
\end{align*}
Taking the supremum over $v \in S^{d-1}$ (so that the left-hand side equals $\norm{n^{-1} \sum_{i=1}^n W_i}$) and re-arranging, we obtain 
\[
  \norm[\bigg]{\frac1n\sum_{i=1}^n W_i}
\leqslant 2\max_{v'\in \mathcal{N}}\frac1n\sum_{i=1}^n\innerp{v'}{W_i}\enspace.
\]
Therefore, using~\eqref{eq:Bern} and a union bound, for any $t \geq 0$,
\[
  \P \bigg( \norm[\bigg]{\frac1n\sum_{i=1}^n W_i} \geq 2\nu\sqrt{\frac{2t}n}+2K\frac tn \bigg)
  \leqslant \P \bigg(\max_{v\in \mathcal{N}}\frac1n\sum_{i=1}^n\innerp{v}{W_i} \geq \nu\sqrt{\frac{2t}n}+K\frac tn \bigg)\leqslant |\mathcal{N}|\exp(-t).
\]
Now by \cite[Lemma 4.2.13]{vershynin2018high}, we have $|\mathcal{N}|\leqslant 5^{d}$, therefore the last inequality applied with $t'=d\log 5 + t$ shows the result.
\end{proof}

Our deviation bounds on the empirical gradient will be obtained by combining the decomposition~\eqref{eq:decomp-norm-gradient} with the sub-gamma deviation bound of Lemma~\ref{lem:ConcSGVec}, applied to both terms of the decomposition.
We will also establish the sub-gamma property through moment estimates.
This leads to the following template bound, which will be used in the three cases we consider:

\begin{lemma}
  \label{lem:template-grad-mom-dev}
  Assume that there exist $\nu_0, \nu_1, K_0, K_1 > 0$ such that, for every integer $p \geq 2$ and every $v \in S^{d-1}$ with $\innerp{u^*}{v} = 0$, one has
  \begin{align}
    \label{eq:grad-ssg-par}
    \Expect[\big]{\ainnerp{u^*}{\nabla \ell (\theta^*, Z)}^p}
    &\leq \frac{\nu_1^2 K_1^{p-2} p!}{2} \, , \\
    \Expect[\big]{\ainnerp{v}{\nabla \ell (\theta^*, Z)}^p}
    &\leq \frac{\nu_0^2 K_0^{p-2} p!}{2}
      \label{eq:grad-ssg-orth}
      \, .
  \end{align}
  Then, for every $t \geq 0$, with probability at least $1 - 3 e^{-t}$ one has
  \begin{equation}
    \label{eq:gradient-moment-dev}
    \norm[\big]{\nabla \wh L_n (\theta^*)}_{H^{-1}}
    \leq \b^{3/2} \nu_1 \sqrt{\frac{2 t}{n}} + \b^{3/2} K_1 \frac{t}{n} + 2 \b^{1/2} \nu_0 \sqrt{\frac{2 (d \log 5 + t)}{n}} + 2 \b^{1/2} K_0 \frac{d \log 5 + t}{n}
    \, .
  \end{equation}
\end{lemma}

The deviation bound above depends on four parameters: a weak and a strong sub-gamma parameters for the gradient $\innerp{v}{\nabla \ell (\theta^*, Z)}$, both in the parameter direction $v = u^*$ and in orthogonal directions $v \in S^{d-1}$.
Note that while the gradient term in the parameter direction in~\eqref{eq:decomp-norm-gradient} is multiplied by a larger factor, unlike the second term it does not depend in the dimension.
In addition, we will see that crucially, both the weak \emph{and} the strong sub-gamma parameters of the gradient are significantly smaller in the parameter direction than in orthogonal directions.

\begin{proof}
  By definition of $\theta^*$, the random vector $\nabla \ell (\theta^*, Z)$ is centered.
  In addition, by~\eqref{eq:grad-ssg-par} and Point~4 in Lemma~\ref{lem:sub-gamma-exponential}, the random variables $\pm \innerp{u^*}{\nabla \ell (\theta^*, Z)}$ are both $(\nu_1^2, K_1)$-sub-gamma.
  Thus, by Lemma~\ref{lem:sub-gamma-exponential} and a union bound, with probability at least $1 - 2 e^{-t}$ one has
  \begin{equation}
    \label{eq:proof-template-grad-1}
    \abs{\innerp{u^*}{\nabla \wh L_n (\theta^*)}}
    = \abs[\bigg]{\frac{1}{n} \sum_{i=1}^n \innerp{u^*}{\nabla \ell (\theta^*, Z_i)}}
    \leq \nu_1 \sqrt{\frac{2 t}{n}} + K_1 \frac{t}{n}
    \, .
  \end{equation}
  Next, define $W = \Pi (\nabla \ell (\theta^*, Z))$, and likewise for $(W_i)_{1 \leq i \leq n}$.
  In addition, for every $v \in S^{d-1}$, using that $\innerp{u^*}{\Pi (v)} = 0$ and $\norm{\Pi (v)} \leq 1$ together with~\eqref{eq:grad-ssg-orth}, we have
  \begin{equation*}
    \E [\ainnerp{v}{W}^p]
    = \E [\ainnerp{v}{\Pi (\nabla \ell (\theta^*, Z))}^p]
    = \E [\ainnerp{\Pi (v)}{\nabla \ell (\theta^*, Z)}^p]
    \leq \frac{\nu_0^2 K_0^{p-2} p!}{2}
    \, .
  \end{equation*}
  Since $\innerp{v}{W}$ is in addition centered, by Lemma~\ref{lem:sub-gamma-exponential} it is $(\nu_0^2, K_0)$-sub-gamma.
  Hence, Lemma~\ref{lem:ConcSGVec} implies that, with probability at least $1-e^{-t}$,
  \begin{equation}
    \label{eq:proof-template-grad-2}
    \norm{\Pi (\nabla \emp{L} (\theta^*))}
    = \norm[\bigg]{\frac{1}{n} \sum_{i=1}^n W_i}
    \leq 2 \nu_0 \sqrt{\frac{2(d\log5 + t)}n} + 2K_0 \frac{d\log 5 + t}n
    \, .
  \end{equation}
  Plugging the upper bounds~\eqref{eq:proof-template-grad-1} and~\eqref{eq:proof-template-grad-2} into the decomposition~\eqref{eq:decomp-norm-gradient} and using a union bound concludes the proof.
\end{proof}

In order to apply Lemma~\ref{lem:template-grad-mom-dev}, we need to control the moments of the random variables $\innerp{v}{\nabla \ell (\theta^*, Z)}$, both for $v = u^*$ and for orthogonal $v \in S^{d-1}$.
In all cases, the starting point towards such a control is the following lemma.

\begin{lemma}\label{lem:UBforGrad}
  Let $Z=(X,Y)$ denote a random variable taking value in $\R^d\times\{-1,1\}$.
  Let $u^*=\theta^*/\|\theta^*\|$ and let $p\geqslant 2$. 
  For any $v\in S^{d-1}$,
  \begin{equation}\label{eq:ToBeBounded}
    \Expect[\big]{\ainnerp{v}{\nabla\ell(\theta^*,Z)}^p}
    \leq \Expect*{\big(\exp(-|\innerp{\theta^{*}}{X}|)
      +\ind{Y\innerp{\theta^*}{X}<0}\big) |\innerp{v}{X}|^p}
    \, .
\end{equation}
Moreover, when the model is well-specified, %
we have 
\begin{equation}\label{eq:ToBeBoundedWS}
  \Expect[\big]{|\innerp{v}{\nabla\ell(\theta^*,Z)}|^p}
  \leq 2\, \Expect[\big]{\exp(-|\innerp{\theta^{*}}{X}|) |\innerp{v}{X}|^p}
  \enspace.
\end{equation}
\end{lemma}
\begin{proof}
    As $\nabla \ell(\theta^{*}, Z)
= -Y\sigma(-Y \innerp{\theta^{*}}{X})X$, for any $v\in S^{d-1}$, one has (using that $\sigma^p \leq \sigma$)
\begin{align*}
  \Expect[\big]{\ainnerp{v}{\nabla \ell(\theta^*,Z)}^p}
  = \Expect[\big]{\sigma(-Y \innerp{\theta^{*}}{X})^p |\innerp{v}{X}|^p}
  \leq \Expect[\big]{\sigma(-Y \innerp{\theta^{*}}{X}) |\innerp{v}{X}|^p}  
  \enspace .
\end{align*}
Moreover, as $\sigma(-|x|)\leqslant \exp(-|x|)$ and $\sigma(x)\leqslant 1$ for every $x \in \R$, we have
\begin{equation*}
  \sigma(-Y \innerp{\theta^{*}}{X})
  \leqslant \exp(-|\innerp{\theta^{*}}{X}|)+\ind{Y\innerp{\theta^*}{X}<0} \enspace,
\end{equation*}
which proves~\eqref{eq:ToBeBounded}.
Moreover, when the model is well-specified, we have
\begin{align}\label{eq:RedWS}
\E[\ind{Y\innerp{\theta^*}{X}<0}|X]&= \sigma(-|\innerp{\theta^{*}}{X}|)\leqslant \exp(-|\innerp{\theta^{*}}{X}|)\enspace.   
\end{align}
Hence, \eqref{eq:ToBeBoundedWS} follows by plugging this bound into \eqref{eq:ToBeBounded}.
\end{proof}

\subsection{Proof of Proposition~\ref{prop:DevGradGauss} (Gaussian design, well-specified model)}
\label{sec:gradient-wellspec-gaussian}

We follow the approach described in the previous section.
Since the model is well-specified, by the second part of Lemma~\ref{lem:UBforGrad}, we need to bound the random variables
$\E[\exp(-|\innerp{\theta^{*}}{X}|)|\innerp{v}{X}|^p]$
when the design $X$ is a standard Gaussian vector.

Consider first the case where $\norm{\theta^*} < e$, so that $\b =e$.
In this case,
for any $v\in S^{d-1}$ and integer $p \geq 2$, we have (using a simple induction on $p$ for the last bound)
\begin{gather}
  \label{eq:BGSN}
  \Expect[\big]{\exp(-|\innerp{\theta^{*}}{X}|)|\innerp{v}{X}|^p}
  \leqslant \E \big[|\innerp{v}{X}|^{p}\big]
		=\frac{2^{p/2}}{\sqrt{\pi}}\Gamma\bigg(\frac{p+1}2\bigg)
		\leqslant \frac{p!}{2}\enspace.
\end{gather}

Consider now the case where $\norm{\theta^*} \geq e$, so that $\b = \norm{\theta^*}$.
For any integer $p \geq 0$, using that
the density of $\innerp{u^*}{X} \sim \gaussdist (0, 1)$ is upper-bounded by $1/\sqrt{2\pi}$,
we obtain
\begin{align}\label{eq:MomGaussu*}
  \E \big[\exp(-|\innerp{\theta^{*}}{X}|)|\innerp{u^*}{X}|^p \big]
  &= \E \big[\exp(-\b|\innerp{u^{*}}{X}|)|\innerp{u^*}{X}|^p \big] \nonumber \\
  &\leq \frac{1}{\sqrt{2\pi}} \int_{\R} |x|^p \exp (-\b |x|) \di x
    = \sqrt{\frac2\pi}\frac{p!}{\b^{p+1}}
    \, .
\end{align}
In addition, for every $v \in S^{d-1}$ such that $\innerp{u^*}{v} = 0$, the random variables $\innerp{u^*}{X}$ and $\innerp{v}{X}$ are independent, thus
\[
\Expect[\big]{\exp(-|\innerp{\theta^{*}}{X}|)|\innerp{v}{X}|^p} = \Expect[\big]{\exp(-|\innerp{\theta^{*}}{X}|)} \, \Expect[\big]{|\innerp{v}{X}|^p}\enspace.
\]
We bound the first term in the right-hand side with \eqref{eq:MomGaussu*} with $p=0$ and the second one with \eqref{eq:BGSN} to obtain
\begin{equation}
  \label{eq:proof-grad-gauss-orth}
  \Expect[\big]{\exp(-|\innerp{\theta^{*}}{X}|)|\innerp{v}{X}|^p}
  \leq \sqrt{\frac2\pi}\frac{1}{\b} \times \frac{p!}{2}
  = \frac{p!}{\sqrt{2\pi} \b}\enspace.
\end{equation}

Combining~\eqref{eq:BGSN},~\eqref{eq:MomGaussu*} and~\eqref{eq:proof-grad-gauss-orth} with Lemma~\ref{lem:UBforGrad}, we get that regardless of the value of $\norm{\theta^*}$, for every integer $p \geq 2$ and every $v \in S^{d-1}$ with $\innerp{u^*}{v} = 0$ we have
\begin{align*}
  \Expect[\big]{\ainnerp{u^*}{\nabla \ell (\theta^*, Z)}^p}
  &\leq  \frac{1}{2} \frac{e^3}{\b^3} \Big( \frac{e}{\b} \Big)^{p-2} p!
    \, ; \\
  \Expect[\big]{\ainnerp{v}{\nabla \ell (\theta^*, Z)}^p}
  &\leq \frac{1}{2} \frac{e}{\b}\, p!
    \, .
\end{align*}
Hence, by Lemma~\ref{lem:template-grad-mom-dev}, for every $t \geq 0$, with probability at least $1 - 3 e^{-t}$ we have that
\begin{align*}
  \norm[\big]{\nabla \wh L_n (\theta^*)}_{H^{-1}}
  &\leq \b^{3/2} \frac{e^{3/2}}{\b^{3/2}} \sqrt{\frac{2 t}{n}} + \b^{3/2} \frac{e}{\b} \frac{t}{n} + 2 \b^{1/2} \frac{e^{1/2}}{\b^{1/2}} \sqrt{\frac{2 (d \log 5 + t)}{n}} + 2 \b^{1/2} \frac{d \log 5 + t}{n} \\
    &\leq \parens[\big]{2 \sqrt{2 e \log 5} + \sqrt{2} e^{3/2}} \sqrt{\frac{d + t}{n}} + \parens[\big]{e + 2 \log 5} \frac{\b^{1/2} (d+t)}{n}
    \, .
\end{align*}
In addition, since by assumption $n \geq 16 \b (d + t)$, we may further bound
\begin{equation*}
  \frac{\b^{1/2} (d + t)}{n}
  \leq \frac{1}{4} \sqrt{\frac{d + t}{n}}  
  \, .
\end{equation*}
Plugging this inequality into the previous bound and simplifying numerical constants concludes the proof of Proposition~\ref{prop:DevGradGauss}.

\subsection{Proof of Proposition~\ref{prop:DevGradWSRD} (regular design, well-specified model)}
\label{sec:gradient-wellspec-regular}

We now prove Proposition~\ref{prop:DevGradWSRD}, which is a deviation bound on the empirical gradient when the model is still well-specified, but the design is no longer Gaussian and instead satisfies Assumptions~\ref{ass:sub-exponential} and~\ref{ass:small-ball}
with parameters $u^{*}$, $\eta = \b^{-1}$ and $c \geq 1$. 

Since the model is well-specified, by \eqref{eq:ToBeBoundedWS}, we have to bound $\E[\exp(-|\innerp{\theta^{*}}{X}|)|\innerp{v}{X}|^p]$ when the design $X$ is regular to prove the result.
This is achieved in the following lemma.

\begin{lemma}\label{lem:MomRegDes}
   Suppose that $X$ satisfies Assumption~\ref{ass:sub-exponential} with parameter $K\geq e$ and Assumption~\ref{ass:small-ball} with parameters $\eta=1/\b$ and $c\geqslant 1$.
   Then, for any integer $p\geqslant 0$ and $v\in S^{d-1}$, one has
\begin{gather}
\label{eq:MomB1}     
  \E \big[\exp(-|\innerp{\theta^{*}}{X}|)|\innerp{u^*}{X}|^p \big]\leqslant \frac{9c}{\b}\bigg(\frac{e K\log(\b)}{2\b}\bigg)^pp!\enspace.\\
\label{eq:MomB2} \E \big[\exp(-|\innerp{\theta^{*}}{X}|)|\innerp{v}{X}|^p \big] \leqslant \frac{5ec}{\b}\bigg(\frac{K\log(\b)}{2}\bigg)^pp!\enspace. 
\end{gather}
\end{lemma}
\begin{proof}
  By Assumption~\ref{ass:sub-exponential}, $\|\innerp{v}{X}\|_{\psi_1}\leqslant K$, so by Definition~\ref{def:psi-alpha},
  \[
    \E \big[\exp(-|\innerp{\theta^{*}}{X}|)|\innerp{v}{X}|^p \big]
    \leq \E[|\innerp{v}{X}|^p]
    \leqslant \frac{K^p}{(2e)^p}p^p\leqslant \parens[\Big]{\frac{K}2}^pp!\enspace.
\]
This proves the result in the case where $\|\theta^*\|\leqslant e$ (since in this case $\b = e$).
From now on, we assume that $\|\theta^*\|>e$, so that $\b=\|\theta^*\|>e$.

Our first step is to control expectations of the form $\E [ \exp (- \alpha \ainnerp{\theta^*}{X}) ]$ for $\alpha > 0$.
This control only requires Assumption~\ref{ass:small-ball}, and relies on a dyadic decomposition.
Specifically, we have
\begin{align*}
  \exp (- \alpha \ainnerp{\theta^*}{X})
  &= \exp (- \alpha \b \ainnerp{u^*}{X}) \\
  &\leq \indic[\Big]{\ainnerp{u^*}{X} \leq \frac{1}{\b}} + \sum_{k\geq 0} \exp (- \alpha \b \ainnerp{u^*}{X}) \indic[\Big]{\frac{2^k}{\b} < \ainnerp{u^*}{X} \leq \frac{2^{k+1}}{\b}} \\
  &\leq \indic[\Big]{\ainnerp{u^*}{X} \leq \frac{1}{\b}} + \sum_{k\geq 0} \exp (- \alpha 2^k) \indic[\Big]{\ainnerp{u^*}{X} \leq \frac{2^{k+1}}{\b}}
    \, .
\end{align*}
Taking expectations above and using Assumption~\ref{ass:small-ball}, we obtain
\begin{align}
  \Expect{\exp (- \alpha \ainnerp{\theta^*}{X})}
  &\leq \Probab[\Big]{\ainnerp{u^*}{X} \leq \frac{1}{\b}} + \sum_{k\geq 0} \exp (- \alpha 2^k) \,\Probab[\Big]{\ainnerp{u^*}{X} \leq \frac{2^{k+1}}{\b}} \nonumber \\
  &\leq \frac{c}{\b} \braces[\bigg]{1 + \sum_{k\geq 0} 2^{k+1} \exp (- \alpha 2^k) }
    \leq \frac{c}{\b} \braces[\bigg]{1 + \sum_{k\geq 0} 4 \int_{2^{k-1}}^{2^{k}} \exp (-\alpha t) \di t} \nonumber \\
  &\leq \frac{c}{\b} \braces[\bigg]{1 + 4 \int_{0}^{\infty} \exp (-\alpha t) \di t}
    = \frac{c}{\b} \braces[\bigg]{1 + \frac{4}{\alpha}}
    \label{eq:bound-exp-regular}
    \, .
\end{align}

We now control the first expectation~\eqref{eq:MomB1}.
Using that $(x/2)^p/p! \leq e^{x/2}$ for every $x \in \R$ and integer $p \geq 0$, we obtain
\begin{align*}
  \Expect[\big]{\ainnerp{u^*}{X}^p \exp (- \ainnerp{\theta^*}{X})}
  &= \b^{-p} \cdot \Expect[\big]{\ainnerp{\theta^*}{X}^p \exp (- \ainnerp{\theta^*}{X})} \\
  &\leq \b^{-p} \cdot \Expect[\big]{2^p p! \exp(\ainnerp{\theta^*}{X}/2) \exp (- \ainnerp{\theta^*}{X})} \\
  &= \parens[\Big]{\frac{2}{\b}}^p p! \, \Expect{\exp (-\ainnerp{\theta^*}{X}/2)} \\
  &\leq \frac{9 c}{\b} \Big( \frac{2}{\b} \Big)^p p!
    \, ,
\end{align*}
where the last inequality comes from~\eqref{eq:bound-exp-regular}.

We now turn to the proof of inequality~\eqref{eq:MomB2}.
For any $\nu \in (0, 1)$, Hölder's inequality, the bound~\eqref{eq:bound-exp-regular}, and the inequality $(p/e)^p \leq p!$ imply that
  \begin{align*}
    \E \big[\exp(-|\innerp{\theta^{*}}{X}|)|\innerp{v}{X}|^p \big]
    &\leqslant \E[|\innerp{v}{X}|^{p/\nu}]^{\nu}\E \big[\exp\big(-|\innerp{\theta^{*}}{X}|/(1-\nu)\big)\big]^{1-\nu} \\
    &\leq \Big[ \Big( \frac{K p}{2 e \nu} \Big)^{p/\nu} \Big]^{\nu} \E \big[\exp\big(-|\innerp{\theta^{*}}{X}| \big)\big]^{1-\nu} \\
    &\leqslant \bigg(\frac{K}{2\nu}\bigg)^pp!\bigg(\frac{5c}{\b}\bigg)^{1-\nu}\enspace.
\end{align*}
Letting $\nu=1/\log (\b)$, this gives
\begin{equation*}
  \E \big[\exp(-\b|\innerp{u^{*}}{X}|)|\innerp{v}{X}|^p \big]\leqslant \frac{5ec}{\b}\bigg(\frac{K\log(\b)}{2}\bigg)^pp!\enspace,
\end{equation*}
which concludes the proof of Lemma~\ref{lem:MomRegDes}.
\end{proof}

\parag{Conclusion of the proof of Proposition~\ref{prop:DevGradWSRD}}
By combining Lemmas~\ref{lem:template-grad-mom-dev},~\ref{lem:UBforGrad} and~\ref{lem:MomRegDes}, we obtain that with probability at least $1-3 e^{-t}$, one has
\begin{align*}
  \norm[\big]{\nabla \wh L_n (\theta^*)}_{H^{-1}}
  &\leq \b^{3/2} \Big( \frac{c_1 \log^2 \b}{\b^3} \Big)^{1/2} \sqrt{\frac{2 t}{n}} + \b^{3/2} \frac{c_2 \log \b}{\b} \frac{t}{n} + \\
  & \quad + 2 \b^{1/2} \Big( \frac{c_3 \log^2 \b}{\b} \Big)^{1/2} \sqrt{\frac{2 (d \log 5 + t)}{n}} + 2 \b^{1/2} \times {c_4 \log \b} \frac{d \log 5 + t}{n} \\
  &\leq c_5 \log (\b) \bigg( \sqrt{\frac{d + t}{n}} + \frac{\b^{1/2}(d + t)}{n} \bigg)
    \, ,
\end{align*}
where $c_1, \dots, c_5$ are constants that only depend on $c$ and $K$.
Now since $n \geq \b (d+ t)$, one has $\b^{1/2} (d+t)/n \leq \sqrt{(d+t)/n}$; plugging this into the above inequality concludes the proof.

\subsection{Proof of Proposition~\ref{prop:gradient-dev-MS} (regular design, misspecified model)}
\label{sec:gradient-misspec-regular}

We now turn to the proof of Proposition~\ref{prop:gradient-dev-MS}, which provides a deviation bound on the empirical gradient in the misspecified case.
Specifically, $X$ satisfies Assumptions~\ref{ass:sub-exponential} and~\ref{ass:small-ball} and the model may not be well-specified.
The parameter $\theta^*$ is defined using the joint distribution of $Z=(X,Y)$ by
\begin{equation*}
	\theta^*= \underset{\theta \in \R^d}{\argmin} \, L (\theta), \quad \text{where} \quad
	L (\theta) = \E [ \ell (\theta,Z)] \, .
\end{equation*}

Recall that, from Lemma~\ref{lem:UBforGrad}, the main task is to bound, for all $p\geqslant 2$, the expectations
\begin{equation*}
    \E\big[ (\exp(-|\innerp{\theta^{*}}{X}|)
		+\ind{Y\innerp{\theta^*}{X}<0}) |\innerp{v}{X}|^p\big] \, .
\end{equation*}
The first term is bounded using Lemma~\ref{lem:MomRegDes}.
Therefore, we focus on the second term, namely
\begin{equation}
  \label{eq:misp-moment-todo}
  \E\big[\ind{Y\innerp{\theta^*}{X}<0}|\innerp{v}{X}|^p\big]\enspace.
\end{equation}
The main difficulty here is that unlike in the well-specified case, we no longer make any assumption on the conditional distribution of $Y$ given $X$.
In particular, the bound on the conditional probability of error
\begin{equation*}
  \P (Y \innerp{\theta^*}{X} < 0 | X)
  = \sigma (- \ainnerp{\theta^*}{X})
  \leq \exp (-\ainnerp{\theta^*}{X})
  \, ,
\end{equation*}
which implies the bound~\eqref{eq:ToBeBoundedWS}, no longer holds in the general misspecified case.

Given the lack of any information about the conditional distribution of $Y$ given $X$, it may seem reasonable to simply bound the indicator $\indic{Y \innerp{\theta^*}{X} < 0} \leq 1$ in the expectation~\eqref{eq:misp-moment-todo}.
One could show that this would lead to a bound of the form $\norm{\nabla \emp{L} (\theta^*)}_{H^{-1}} \lesssim \sqrt{(\b d + \b^3 t)/n}$ with probability at least $1-2e^{-t}$ for any $t \geq 1$.

However, as we show below, one can improve over this naive bound even in the misspecified case.
This comes from the fact that, if the minimizer $\theta^*$ of the risk has a large norm, then the (unconditional) probability of error $\P (Y \innerp{\theta^*}{X} < 0)$ (as well as another moment) must be small.
This implies the improved upper bound $\norm{\nabla \emp{L} (\theta^*)}_{H^{-1}} \lesssim \log (\b) \sqrt{(d + \b t)/n}$ of Proposition~\ref{prop:gradient-dev-MS}, which turns out to be best possible (except possibly for the $\log (\b)$ factor) in the general misspecified case.

\parag{Bounds on the first moments}

The key step in the analysis is the following lemma, which provides a precise control of the expectation~\eqref{eq:misp-moment-todo} in the case where $p=0$ and $p=1$.
In fact, the bounds of Lemma~\ref{lem:first-moments-misspecified} on these first two moments exhibit the same dependence on $\b$ in the general misspecified case as in the well-specified case.
These first estimates are an important step towards the bound for $p=2$ and then for general $p$.
\begin{lemma}
  \label{lem:first-moments-misspecified}
Suppose that $X$ satisfies Assumption~\ref{ass:small-ball} with parameters $(u^{*}, \b^{-1}, c)$.
Then,
  \begin{gather}
    \label{eq:bound-first-misspecified}
    \E \big[ \abs{\innerp{\theta^*}{X}} \indic{Y\innerp{\theta^*}{X} < 0} \big]
    \leq %
      \frac{6c}{\b^2} %
      \, ; \\
  \P(Y \innerp{\theta^*}{X} < 0)
   \leq \frac{3.21c}{\b}
  \, .
  \label{eq:bound-classif-misspecified}
\end{gather}    
\end{lemma}
\begin{remark}
 The second bound also holds when $\|\theta^*\|\leqslant e$ as it is trivial then and the first one also becomes trivial (and therefore holds) in this case as soon as $c\geqslant e^2/6\approx 1.23$.   
\end{remark}

\begin{proof}
Since $L$ is minimized in $\theta^*$, one has $\frac{\di}{\di t}\big|_{t=1} L (t \theta^*) = 0$.
Hence
\begin{align*}
  0
  &= \E \big[ Y \innerp{\theta^*}{X} \sigma (- Y \innerp{\theta^*}{X}) \big] \\
  &= \E \Big[ \abs{\innerp{\theta^*}{X}} \sigma \big( - \abs{\innerp{\theta^*}{X}} \big) \indic{Y \innerp{\theta^*}{X} \geq 0} - \abs{\innerp{\theta^*}{X}} \sigma \big( \abs{\innerp{\theta^*}{X}} \big) \indic{Y \innerp{\theta^*}{X} < 0} \Big]
    \, .
\end{align*}
Now, using that $\sigma (t) = 1-\sigma (-t)$, we obtain:
\begin{align*}
  0
  &= \E \Big[ \abs{\innerp{\theta^*}{X}} \Big\{ \sigma (- \abs{\innerp{\theta^*}{X}}) \big[ 1 - \indic{Y \innerp{\theta^*}{X} < 0} \big] - \big[ 1 - \sigma (- \abs{\innerp{\theta^*}{X}}) \big] \indic{Y \innerp{\theta^*}{X} < 0} \Big\} \Big] \\
  &= \E \Big[ \abs{\innerp{\theta^*}{X}} \big\{ \sigma (-\abs{\innerp{\theta^*}{X}}) - \indic{Y \innerp{\theta^*}{X} < 0} \big\} \Big]
    \, ,
\end{align*}
which writes
\begin{equation}
\label{eq:misclassif-identity}
\E \big[ \abs{\innerp{\theta^*}{X}} \indic{Y \innerp{\theta^*}{X} < 0} \big]
  = \E \big[ \abs{\innerp{\theta^*}{X}} \sigma (-\abs{\innerp{\theta^*}{X}}) \big] \, .  
\end{equation}
By \eqref{eq:MomB1} %
applied with $p=1$, this shows~\eqref{eq:bound-first-misspecified}.
Moreover, as $1 \leq \abs{\innerp{\theta^*}{X}} + \indic{\abs{\innerp{\theta^*}{X}} \leq 1}$, we have
\begin{align}
\P (Y \innerp{\theta^*}{X} < 0)
	&\leq \E \big[ \abs{\innerp{\theta^*}{X}} \indic{Y \innerp{\theta^*}{X} < 0} \big]
		+ \E\big[ \indic{\abs{\innerp{\theta^*}{X}} \leq 1} \indic{Y \innerp{\theta^*}{X} < 0} \big]
		\nonumber \\
\notag	&\leq \E \big[ \abs{\innerp{\theta^*}{X}} \sigma (-\abs{\innerp{\theta^*}{X}}) \big]
		+ \P(\abs{\innerp{\theta^*}{X}} \leq 1) \, .
\end{align}
Bounding the first term with~\eqref{eq:bound-first-misspecified} and the second using Assumption~\ref{ass:small-ball} concludes the proof.
\end{proof}

\parag{Bounds on the second moments}
In this paragraph, we bound the expectation of interest for $p=2$. 
We deduce an upper bound on the variance of the gradients.
\begin{lemma}\label{lem:SecMomGradGen}
Suppose that $X$ satisfies
Assumptions~\ref{ass:sub-exponential} and~\ref{ass:small-ball} with parameters $K\geqslant e$, $(u^{*}, \b^{-1}, c)$.
Then, for any $v\in S^{d-1}$,
\begin{align*}
  \E \big[\ind{Y\innerp{\theta^*}{X}<0}\innerp{u^*}{X}^2 \big]
  &\leqslant \frac{6 cK\log (K \b^2)}{\b^2}\, .\\
  \E \big[\ind{Y\innerp{\theta^*}{X}<0}\innerp{v}{X}^2 \big]
  &\leqslant \frac{\max\{4e^2,3.21c K^{2}\log^2 \b \}}{4e\b} \, .
\end{align*}
\end{lemma}
\begin{proof}
We start with the second inequality.
As $\E[\innerp{v}{X}^2]=1$, the left-hand side is smaller than $1$ while the right-hand side is at least $1$ if $\b=e$.
Therefore, we can assume that $\|\theta^*\|=\b>e$.

 By \eqref{eq:bound-classif-misspecified}, $\P(Y \innerp{\theta^*}{X} < 0) \leq \min\{1,  3.21c/\b\}$.
  Hence, by the first part of Lemma~\ref{lem:tuned-holder-order-2},
  \begin{equation*}
    \E \big[\ind{Y \innerp{\theta^*}{X} < 0} \innerp{v}{X}^2 \big]
    \leq \frac{3.21 cK^2 \log^2 ( \b )}{4e\b} \, ,
  \end{equation*}
 which shows the second inequality since $\b\geqslant e$.

   For the first inequality, when $\norm{\theta^*} < e$, the upper bound is larger than $1$ while 
   \[
   \E \big[\ind{Y\innerp{\theta^*}{X}<0}\innerp{u^*}{X}^2 \big] \leq \E [ \innerp{u^*}{X}^2 ] =1\enspace,
   \]
   so the inequality holds in this case.
  Hence, we may assume that $\norm{\theta^*} \geq e$.
  In this case, by \eqref{eq:bound-first-misspecified},
  \begin{equation*}
    \E \big[ \ainnerp{u^*}{X} \indic{Y \innerp{\theta^*}{X} < 0} \big]
    \leq \frac{6  c}{\b^2}
    \, .
  \end{equation*}
 Thus, applying the second part of Lemma~\ref{lem:tuned-holder-order-2} below with $U = \ainnerp{u^*}{X} \indic{Y \innerp{\theta^*}{X} < 0}$ and $V = \ainnerp{u^*}{X}$, we get
  \begin{equation*}
    \E \big[ \innerp{u^*}{X}^2 \indic{Y \innerp{\theta^*}{X} < 0} \big]
    = \E [U V]
    \leq \frac{6 c}{\b^2} K \log \Big( e \vee \frac{K\b^2}{6 c} \Big)
    \leq \frac{6 cK\log (K \b^2)}{\b^2} 
    \, . \qedhere
  \end{equation*}
\end{proof}

Now, combining Lemmas~\ref{lem:UBforGrad},~\ref{lem:MomRegDes} and~\ref{lem:SecMomGradGen}, we obtain that 
for some constants $c_1, c_2$ depending only on $c, K$ (and regardless of whether $\norm{\theta^*} \geq e$ or not), one has for any $v\in S^{d-1}$ that
\begin{gather}
  \label{eq:var1}  \E[\innerp{u^*}{\nabla\ell(\theta^*,Z)}^2]
  \leq \frac{c_1 \log^2 (\b)}{\b^2}
  \, ;\\
  \label{eq:Var2}\E[\innerp{v}{\nabla\ell(\theta^*,Z)}^2]\leqslant
  \frac{c_2 \log^2 (\b)}{\b}
  \, .
\end{gather}
In addition, for every $v \in S^{d-1}$, one has
$  \norm{\innerp{v}{\nabla \ell (\theta^*, Z)}}_{\psi_1}  \leq K$
since, by Assumption~\ref{ass:sub-exponential}, for every integer $p \geq 2$,
\begin{equation*}
  \norm{\innerp{v}{\nabla \ell (\theta^*, Z)}}_p
  = \Expect[\big]{\abs{Y \sigma (-Y \innerp{\theta^*}{X})}^p \ainnerp{v}{X}^p}^{1/p}
  \leq \norm{\innerp{v}{X}}_p
  \leq \frac{K p}{2 e}
  \, .
\end{equation*}
Therefore, by Lemma~\ref{lem:sub-gamma-exponential}, the random variable $\innerp{u^*}{\nabla \ell(\theta^*, Z)}$ is $(c_3 \log^2 (\b)/\b^2, c_4 \log (\b))$-sub-gamma, while $\innerp{v}{\nabla \ell(\theta^*, Z)}$ is $(c_5 \log^2 (\b)/\b, c_6 \log (\b))$-sub-gamma, where $c_3, c_4, \dots$ only depend on $c, K$.
By Lemma~\ref{lem:template-grad-mom-dev}, we deduce that for every $t \geq 0$, with probability at least $1-3e^{-t}$ we have that
\begin{align*}
  \norm[\big]{\nabla \emp{L} (\theta^*)}_{H^{-1}}
  &\leq c_7 \bigg[\b^{3/2} \Big( \frac{\log^2 \b}{\b^2} \Big)^{1/2} \sqrt{\frac{t}{n}} + \b^{3/2} \log (\b)\frac{t}{n} + \b^{1/2} \Big( \frac{\log^2 \b}{\b} \Big)^{1/2} \sqrt{\frac{d + t}{n}} + \\
  & \quad + \b^{1/2} \log (\b) \frac{d + t}{n}
    \bigg] \\
  &\leq c_8 \log (\b) \bracks[\bigg]{\sqrt{\frac{d + \b t}{n}} + \frac{\b^{1/2} (d + \b t)}{n}} \\
  &\leq 2 c_8 \log (\b) \sqrt{\frac{d + \b t}{n}}
    \, ,
\end{align*}
where the last inequality comes from the fact that $n \geq \b (d+ \b t)$.
This concludes the proof of Proposition~\ref{prop:gradient-dev-MS}.
We now conclude this section with the following lemma, which was used in the proof of Lemma~\ref{lem:SecMomGradGen} above.

\begin{lemma}
  \label{lem:tuned-holder-order-2}
  Let $U, V$ be nonnegative real random variables such that $\E [U] \leq \eps$ and $\norm{V}_{\psi_1} \leq K$ for some $\eps, K > 0$.
  \begin{enumerate}
  \item If $U \leq 1$ almost surely and $\eps \leq 1$, then $\E [U V^2] \leq \eps \cdot \frac{K^2 \log^2 ( e \vee \eps^{-1} )}{4 e}$.
  \item If $\norm{U}_{\psi_1} \leq K$, then $\E [U V]
      \leq \eps K \log (e \vee K/\eps)$.
  \end{enumerate}
\end{lemma}

\begin{proof}
  We start with the first point.
  Using Hölder's inequality, for any $p > 1$ we have (using that $u^{p/(p-1)} \leq u$ for $u \in [0,1]$)
  \begin{align*}
    \E \big[ U V^2 \big]  
    &\leq \E \big[ \abs{V}^{2p} \big]^{1/p} \, \E \big[ U^{p/(p-1)} \big]^{1 - 1/p} 
      \leq \| V \|_{2p}^2 \, \E [U]^{1 - 1/p} \\
    &\leq \Big( \frac{K p}{2 e} \Big)^2 \eps^{1-1/p}
      = \frac{K^2 \eps}{4 e^2} \eps^{-1/p} p^2 \, .
  \end{align*}
  Now, letting $p' \to p = \max (1, \log (1/\eps)) \geq 1$, we obtain
  \begin{equation*}
    \E [U V^2]
    \leq \frac{K^2 \eps}{4 e^2} \cdot e \cdot \max(1, \log^2 (1/\eps))
    = \frac{K^2}{4 e} \cdot \eps \log^2 ( e \vee \eps^{-1} ) \, .
  \end{equation*}

  We now prove the second inequality.
  For any $p > 1$, letting $q = p/(p-1)$ we have
  \begin{equation}
    \label{eq:second-holder-misspecified}
    \E [U V]
    \leq \E [ V^p ]^{1/p} \E [ U^q ]^{1/q}
    \leq \frac{K p}{2 e} \E [ U^q ]^{1/q} %
    \, ,
  \end{equation}
  where the second inequality comes from the fact that $\norm{V}_{\psi_1} \leq K$.
  Next, for any $r > 1$, write $q = 1 - \frac{1}{r} + \frac{q'}{r}$ with $q' = 1 + r (q - 1) > q$.
  We also have by Hölder's inequality
  \begin{align*}
    \E [ U^q ]
    &= \E [ U^{1-1/r} (U^{q'})^{1/r} ]
      \leq \E [ U ]^{1 - 1/r} \norm{U}_{q'}^{q'/r} 
      \leq \eps^{1 - 1/r} \Big( \frac{K q'}{2 e} \Big)^{q'/r} \\
    &= \eps^{1-1/r} \Big( \frac{K [1 + r (q-1)]}{2 e} \Big)^{q-1 + 1/r}
      = \eps^{q} \Big( \frac{K [1 + r (q-1)]}{2 e \eps} \Big)^{q-1 + 1/r}
      \, ,
  \end{align*}
  where we used that $\E [U] \leq \eps$ and $\norm{U}_{\psi_1} \leq K$.
  Hence, using that $q-1 = 1/(p-1)$, letting $r = p - 1$ (assuming $p > 2$) so that $q r= p$, we obtain 
  \begin{align*}
    \E [U^q]^{1/q}
    \leq \eps \Big( \frac{K [1 + r (q-1)]}{2 e \eps} \Big)^{1 - 1/q + 1/(qr)}
    = \eps \Big( \frac{K}{e \eps} \Big)^{2/p}
    \, .
  \end{align*}
  Plugging this inequality into~\eqref{eq:second-holder-misspecified} and letting  $p \to 2 \log (e \vee K/\eps) \geq 2$, so that $\lim (K/\eps)^{2/p} \leq e$, we get
  \begin{equation*}
    \E [U V]
    \leq \frac{K \cdot 2 \log (e \vee K/\eps)}{2 e} \cdot \eps e
    = \eps K \log (e \vee K/\eps)
    \, ,
  \end{equation*}
  which establishes the second point.
\end{proof}

\section{Proofs of lower bounds on empirical Hessian matrices}
\label{sec:hessians}

This section is devoted to the proofs of the uniform lower bound on empirical Hessian matrices stated in Section~\ref{sec:lower-bounds-hessian}.
Specifically, Sections~\ref{sec:proof-hessian-regular} and~\ref{sec:techn-lemm-hessian-regular} contain the proof of Theorem~\ref{thm:hessians-regular-case} (in the regular case), while Sections~\ref{sec:hessian-gaussian} and~\ref{sec:techn-lemm-hessian-gaussian} contain the proof of Theorem~\ref{thm:hessian-gaussian} (in the Gaussian case).

\subsection{Proof of Theorem~\ref{thm:hessians-regular-case} (regular design)}
\label{sec:proof-hessian-regular}

In this section, we prove Theorem~\ref{thm:hessians-regular-case}, namely the uniform lower bound on empirical Hessian matrices in the case of a regular design.
Specifically, we assume that $X$ satisfies Assumptions~\ref{ass:sub-exponential} and~\ref{ass:twodim-marginals} with parameters $K \geq e$, $u^* = \theta^*/\norm{\theta^*}$, $\eta = 1/\b$ and $c \geq 1$.
 
Fix $v\in S^{d-1}$ and $\theta\in \Theta$, we want to bound from below
 \begin{align*}
	\langle \wh H_n(\theta) v, v \rangle
		&=  \frac{1}{n} \sum_{i=1}^{n} \sigma'(\langle \theta, X_{i} \rangle)
			\langle v, X_{i} \rangle^{2} \, .
   \end{align*}
 The function $\sigma'(x)=e^x/(1+e^x)^2$ is even, non negative, non increasing on $[0, +\infty)$.
Therefore, for any $m, M>0$,
   \begin{align}
\label{VCB:Step1}\innerp{\wh H_n(\theta) v}{v}&     \geqslant \frac{\sigma'(m(1+r)\b)M^2}{n}\sum_{i=1}^{n}\ind{\ainnerp{u}{X_{i}} \leq m,\ \ainnerp{v}{X_{i}}\geqslant M}  \, ,
   \end{align}
   where we also used that, as $\|\theta-\theta^*\|_{H}\leqslant r/\sqrt{\b}$, $\|\theta\|\leqslant (1+r)\b$ by Lemma~\ref{lem:ellipsoids}.
It remains to bound from below the empirical process $n^{-1}\sum_{i=1}^{n}\ind{\big| \innerp{u}{X_{i}}\big| \leq m,\ |\innerp{v}{X_{i}}|\geqslant M}$ uniformly over $\theta\in \Theta$ and $v\in S^{d-1}$, for a proper choice of $m$ and $M$.
We want to apply Lemma~\ref{lem:vc-lower-bound}. 
For this, we have to estimate $\P(| \innerp{u}{X_{i}}| \leq m,\ |\innerp{v}{X_{i}}|\geqslant M)$ for each $\theta\in \Theta$ and $v\in S^{d-1}$.
If $\|\theta^{*}\| \leq e$, $\b=e$ so $\eta=1/e$, so Proposition~\ref{prop:regularity-constant-scale} shows that 
Assumption~\ref{ass:twodim-marginals} is satisfied with constant
$\max\{2eK\log(2K), 2K^{4}\} = 2K^{4}$. 
Therefore,
\begin{equation*}
  \P\left( \big| \langle u, X \rangle \big| \leq \frac{2K^4}\b \,; \, 
    \big| \langle v, X \rangle \big| \geq \frac{\max\left\{1/\b, \|u^{*} - v \| \right\}}{2K^4}\right)
  \geq \frac{1}{2K^4\b}
  \, .
\end{equation*}
When $\|\theta^{*}\| \geq e$, the third point of Lemma~\ref{lem:ellipsoids} implies that for every $\theta \in \Theta$,
\[
	\| u - u^{*} \|
		\leq \frac{\sqrt{2}}{[K\log(c(c+1)\b)-r]}\frac{r}\b
		\leq \frac{2r}{K\b \log(c(c+1)\b)} \, .
\]
By Lemma~\ref{lem:bidim-extension}, this implies that for all $\theta \in \Theta$ and $v \in S^{d-1}$,
one has for all $t \geq 1/\b$
\begin{equation*}
  \P\left( \big| \innerp{u}{X}\big| \leq \frac{c+1}\b \,; \, \big| \innerp{v}{X} \big| \geq \frac{\max\left\{1/\b, \|u^{*} - v \| \right\}}{c+1}\right)
  \geq \frac{1}{(c+1)\b}
  \, .
\end{equation*}
This suggests to choose $m=\gamma/\b$, $M=\max(1/\b,\|u^*-v\|)/\gamma$ in \eqref{VCB:Step1}, where $\gamma =c+1$ if $\|\theta^*\|\geqslant e$ and $\gamma=2K^{4}$ if $\|\theta^*\|<e$. 
With this choice, we have, for all $\theta \in \Theta$ and all $v\in S^{d-1}$,
\begin{align}
  \label{eq:prob-lb}
  \P\left( \big| \langle u, X \rangle \big| \leq \frac {\gamma}\b \,; 
  \,\big| \langle v, X \rangle \big| \geq \frac{\max\left\{1/\b, \|u^{*} - v \| \right\}}{\gamma}\right)
  \geq \frac1{\gamma \b}
    \enspace .
\end{align} 

The next step to apply Lemma~\ref{lem:vc-lower-bound} is to bound the VC dimension of the class of sets of the form $\big\{x: |\innerp{u}{x}| \leq m,\ |\innerp{v}{x}|\geqslant M\big\}$ for any $u,v\in S^{d-1}$ and any $m,M>0$.
For this, remark that each of these sets is the union of two intersections of $3$ half-spaces.
The class of all half-spaces in $\R^{d}$ has VC dimension $d$ \cite[Theorem 13.8]{devroye1996ptpr}. 
Therefore, by \cite[Theorem 1.1]{van2009note}, the class of all intersections of $3$ half-spaces is bounded from above by 
$6.9 \log(12) d$,
and therefore, by the same result, the VC dimension of the class of all unions of $2$ intersections of $3$ half spaces is bounded from above by 
\[
  4.6 \log (8) \times 6.9 \log(12) d
  \leq 165 d
  \, .
\]

Hence, Lemma~\ref{lem:vc-lower-bound} applies with $p = 1/(\gamma \b)$ and VC dimension $165 d$.
It shows that, for some absolute constant $c_1$, whenever
\begin{equation*}
  n
  \geq c_1 \big( \log (B) d + t \big)
  \, ,
\end{equation*}
with probability at least $1-e^{-t}$, one has simultaneously for all $\theta \in \Theta$ and $v \in S^{d-1}$,
\begin{equation*}
  \frac{1}{n}\sum_{i=1}^{n}\ind{\big| \innerp{u}{X_{i}}\big| \leq m,\ \innerp{v}{X_{i}}\geqslant M} \geq \frac{1}{2\gamma \b} \, .
\end{equation*}

Plugging this estimate into \eqref{VCB:Step1} shows that, on the same event,
\begin{align*}
	\langle \wh H_n(\theta) v, v \rangle^{2}
		&\geqslant \frac{\sigma'(\gamma(1+r))}{2\gamma^2\b} \max\bigg\{\frac1\b,\|u^*-v\|\bigg\}^2 \\
		&\geqslant \frac{\sigma'(\gamma(1+r))}{4\gamma^2} \left( \frac{\innerp{u^{*}}{v}^{2}}{\b^{3}}
			+ \frac{1 - \innerp{u^{*}}{v}^{2}}{\b} \right)
		= \frac{\sigma'(\gamma(1+r))}{4\gamma^2} \innerp{Hv}{v} \, .
\end{align*}
In addition, for all real $x$, $\sigma'(x)\geq \frac{e^{-|x|}}2 \indic{|x|\geq 1}$. One can also check that $x^{2}\exp(\alpha x) \leq \exp((\alpha+2/e)x)$ for every $x, \alpha \geq 1$. The result then follows by applying this with $x=\gamma$ and $\alpha = 1+r$. The condition on $r$ ensures that $1+r+2/e \leq 2$.

\subsection{Technical lemmas for the proof of Theorem~\ref{thm:hessians-regular-case}}
\label{sec:techn-lemm-hessian-regular}

This section gathers the technical tools that we used in the previous proof.
We used the following VC-type inequality (see \eg~\cite[Example~3.10]{boucheron2013concentration} for the definition of the VC dimension).

\begin{lemma}
  \label{lem:vc-lower-bound}
  Let $X_1, \dots, X_n$ be \iid random variables with distribution $P$ on $\Xs$,
  and let $\A$ be a VC class of subsets of $\Xs$ with VC dimension at most $d \geq 1$.
  Let $p \in (0,1)$, and assume that $P (A) %
  \geq p$ for any $A \in \A$.
  If
  \begin{equation}
    \label{eq:vc-condition-sample-size}
    n \geq \frac{180 \log (480/p)\, d}{p}
    \, ,
  \end{equation}
  then with probability at least $1- 2 e^{-np/64}$, one has
  \begin{equation}
    \label{eq:vc-lower-uniform}
    \inf_{A \in \A} \: %
    \frac{1}{n} \sum_{i=1}^n \indic{X_i \in A} %
    \geq \frac{p}{2}
    \, .
  \end{equation}
\end{lemma}

\begin{proof}
  For every $A \in \A$, denote the empirical frequency of $A$ by
  \begin{equation*}
    \emp{P} (A)
    = \frac{1}{n} \sum_{i=1}^n \indic{X_i \in A}
    \, .
  \end{equation*}
  By combining the inequality~\cite[Exercise~12.3 p.~356]{boucheron2013concentration} with Sauer's lemma (\eg,~\cite[Lemma~2.12 p.~211]{vanhandel2014probability}), we obtain that for every $\eps > 0$,
  \begin{equation}
    \label{eq:proof-vc-ratio-bound}
    \Probab[\bigg]{\sup_{A \in \A} \frac{P (A) - \wh P_n (A)}{\sqrt{P (A)}} \geq 2 \eps}
    \leq 2 \parens[\bigg]{\frac{2 e n}{d}}^d \exp \braces*{ - \frac{n \eps^2}{2}}
    \, .
  \end{equation}
  In addition, for every $A \in \A$, if $\emp{P} (A) \leq p/2$ then (since $P (A) \geq p$)
  \begin{equation*}
    \frac{P (A) - \emp{P} (A)}{\sqrt{P (A)}}
    \geq \frac{P (A) - P (A)/2}{\sqrt{P (A)}}
    = \frac{\sqrt{P (A)}}{2}
    \geq \frac{\sqrt{p}}{2}
    \, .
  \end{equation*}
  Combining this fact with~\eqref{eq:proof-vc-ratio-bound} gives
  \begin{align}
    \label{eq:proof-vc-tail-proba}
    \Probab*{\inf_{A \in \A} \emp{P} (A) \leq \frac{p}{2}}
    &\leq \Probab[\bigg]{\sup_{A \in \A} \frac{P (A) - \wh P_n (A)}{\sqrt{P (A)}} \geq \frac{\sqrt{p}}{2}} \nonumber \\
    &\leq 2 \exp \braces*{d \log \parens*{\frac{2 e n}{d}} - \frac{n p}{32}}
      \leq 2 \exp \braces*{- \frac{n p}{64}}
      \, ,
  \end{align}
  where the last inequality follows from the assumption~\eqref{eq:vc-condition-sample-size},
  together with the basic fact that if $u,v \geq 1$ satisfy $u \geq (1+e^{-1}) v \log \big( (e+1) v \big)$, then $\log (e u)/u \leq 1/v$ (applied to $u = 2n/d$ and $v = 128 /p$), and a bound on numerical constants.
\end{proof}

To apply Lemma~\ref{lem:vc-lower-bound} to the sets $\{x \in \R^d  :|\innerp{u}{x}|\leqslant m,|\innerp{v}{x}|\geqslant M\}$, we had to lower-bound the probability of these events.
This bound can be deduced from the fact that the lower bound on these event provided for $u=u^*$ by Assumption~\ref{ass:twodim-marginals} can be extended to all $u$ in a neighborhood of $u^*$ as shown in the following result.

\begin{lemma}
\label{lem:bidim-extension}
Suppose that Assumptions~\ref{ass:sub-exponential} and \ref{ass:twodim-marginals} hold with respective parameters $K\geq e$, $u^{*} \in S^{d-1}$, $c \geq 1$ and $\eta \in (0, 1)$.
Then for all $u, v \in S^{d-1}$ such that
\begin{equation}
  \label{eq:cap-neighborhood}
  \| u - u^{*} \| \leq \frac{ 2\eta}{K\log(c(c+1)/\eta)}
  \quad \text{and} \quad \innerp{u^*}{v} \geq 0\, ,
\end{equation}
one has
\begin{equation*}
  \P\Big( \abs{\innerp{u}{X}} \leq (c+1) \eta \, ; \,  \abs{\innerp{v}{X}}
  \geq \frac{ \max\big\{ \eta, \norm{u^{*} - v} \big\}}{c+1} \Big) \geq \frac{\eta}{c+1}\, .
\end{equation*}
\end{lemma}

\begin{proof}
  Let $u,v \in S^{d-1}$ satisfy~\eqref{eq:cap-neighborhood}.
  The triangle inequality
\begin{equation*}
  | \langle u, X \rangle |
  \leq | \langle u^{*}, X \rangle | + | \langle u - u^{*}, X \rangle |
\end{equation*}
implies that
\begin{align}
  \label{eq:proof-extend-2dim}
  &\P \Big( \abs{\innerp{u}{X}} \leq (c+1) \eta \, ; \,  \abs{\innerp{v}{X}} \geq \frac{ \max\big\{ \eta, \norm{u^{*} - v} \big\}}{c} \Big) \nonumber \\
  &\geq \P\Big( \abs{\innerp{u^*}{X}} \leq c \eta \,  ; \,  \abs{\innerp{v}{X}} \geq \frac{ \max\big\{ \eta, \norm{u^{*} - v} \big\}}{c}\Big)
    - \P\big(|\innerp{u-u^*}{X}|> \eta \big)\enspace.
\end{align}
Next, on the one hand, Assumption~\ref{ass:twodim-marginals} asserts that
\begin{equation*}
  \P\Big( \abs{\innerp{u^*}{X}} \leq c \eta \,  ; \,  \abs{\innerp{v}{X}} \geq \frac{ \max\big\{ \eta, \norm{u^{*} - v} \big\}}{c}\Big)
  \geq \frac{\eta}{c}
  \, ,
\end{equation*}
and on the other hand, Assumption~\ref{ass:sub-exponential} together with Point~1 in Lemma~\ref{lem:sub-gamma-exponential} implies that
\begin{equation*}
  \P\left( \big | \langle u - u^{*}, X \rangle \big | > \eta \right)
  \leq \exp\left(- \frac{2 \eta}{K\|u-u^{*}\|} \right)\leq \exp\left(- \log\big(c(c+1)/\eta \big) \right)
  \leq \frac{\eta}{c(c+1)} %
  \enspace .
\end{equation*}
Plugging the previous two inequalities into~\eqref{eq:proof-extend-2dim}
concludes the proof, since $\frac{\eta}{c} - \frac{\eta}{c(c+1)} = \frac{\eta}{c+1}$.
\end{proof}

\subsection{Proof of Theorem~\ref{thm:hessian-gaussian} (Gaussian design)}
\label{sec:hessian-gaussian}

In this section, we let
\begin{equation}
  \label{eq:def-ellipsoid-small}
  \Theta = \Big\{ \theta \in \R^{d} :  \|\theta - \theta^{*}\|_{H} \leq \frac{1}{100\sqrt{\b}} \Big\}
  \, .
\end{equation}
We would like to show that with high probability, one has $\emp{H} (\theta) \mgeq c H$ for every $\theta \in \Theta$, where $c$ is an absolute constant.
This property amounts to
\begin{equation*}
  \lamin (H^{-1/2} \emp{H} (\theta) H^{-1/2}) \geq c
\end{equation*}
for every $\theta \in \Theta$, and therefore to
\begin{equation}
  \label{eq:hessian-lower-inf-process}
  \inf_{\theta \in \Theta, \, v \in S^{d-1}} \innerp{H^{-1/2} \emp{H} (\theta) H^{-1/2} v}{v}
  = \inf_{\theta \in \Theta, \, v \in S^{d-1}} \braces[\bigg]{\frac{1}{n} \sum_{i=1}^n \sigma' (\innerp{\theta}{X_i}) \innerp{v}{H^{-1/2} X_i}^2}  
  \geq c
  \, .
\end{equation}
We are therefore led to control the infimum of an empirical process indexed by $(\theta, v) \in \Theta \times S^{d-1}$.

The main challenge towards such a control is the presence of the nonlinear terms $\sigma' (\innerp{\theta}{X_i})$, and the fact that we aim for a fully sharp estimate.
This rules out the use of VC arguments as in the regular case, since such methods would necessarily produce additional logarithmic factors in the low-noise regime we consider.

\subsubsection*{PAC-Bayes smoothing technique and proof outline}

We will achieve this by using the PAC-Bayes inequality; see \eg \cite[Proposition~2.1]{catoni2017dimension}.
The use of this inequality to control empirical processes was pioneered by Audibert and Catoni~\cite{audibert2011robust} in the context of robust linear regression, and has since found applications to matrix concentration~\cite{oliveira2016covariance,catoni2016pac,mourtada2022linear,zhivotovskiy2024dimension} and robust covariance matrix estimation~\cite{catoni2017dimension,giulini2018robust,abdalla2024covariance,oliveira2024improved,minasyan2025statistically}, among others.
In our logistic regression setting, it underpins the only approach we know of for establishing the optimal estimate of Theorem~\ref{thm:hessian-gaussian} on the empirical Hessian.

In order to state this inequality, recall that if $\mu$ and $\nu$ denote two probability measures on the same measurable space $(E, \mathcal{E})$ such that $\nu$ is absolutely continuous with respect to $\mu$, the Kullback-Leibler divergence (relative entropy) between $\nu$ and $\mu$ is defined by
\begin{equation}
  \label{not:kl}
  D(\nu \| \mu) %
  = \int_{E} \log \Big( \frac{\di \nu}{\di \mu} \Big) \di \nu\, .
\end{equation}

\begin{lemma}[PAC-Bayes inequality]
  \label{lem:pac-bayes-sum}
  Let $(E, \mathcal{E}, \pi)$ be a probability space and $Z=\left(Z(\omega)\right)_{\omega \in E}$ a measurable real process indexed by $\omega \in E$.
  Let $Z_1, \dots, Z_n$ be independent copies of the process $Z$. Let also $\lambda > 0$ be such that $\E \exp(\lambda Z(\omega)) < \infty$ for every $\omega \in \Omega$.
For any $t \geq 0$, we have with probability at least $1-e^{-t}$, simultaneously for every probability measure $\rho$ on $E$ dominated by $\pi$,
\begin{equation}
  \label{eq:pac_bayes_inequality}
  \frac{1}{n} \sum_{i=1}^{n} \int_{E} Z_i(\omega) \rho(\di \omega)
  \leq \frac1{\lambda }\int_{E} \log \Expect[\big]{\exp({\lambda Z(\omega)})} \rho(\di \omega) 
  + \frac{D(\rho \Vert \pi) + t}{\lambda n} \, .
\end{equation}
\end{lemma}

Since we aim to obtain a lower bound on the empirical process~\eqref{eq:hessian-lower-inf-process}, it will be more convenient to apply Lemma~\ref{lem:pac-bayes-sum} to the process $-Z$, which yields the following inequality:
\begin{equation}
  \label{eq:pac_bayes_lower}
  \frac{1}{n} \sum_{i=1}^{n} \int_{E} Z_i(\omega) \rho(\di \omega)
  \geq - \frac1{\lambda }\int_{E} \log  \E\big[ \exp({-\lambda Z(\omega)}) \big] \rho(\di \omega) 
  - \frac{D(\rho \Vert \pi) + t}{\lambda n} 
\end{equation}
for every $\lambda > 0$ such that $\E [\exp ({-\lambda Z (\omega)})] < \infty$, which always holds if $Z \geq 0$.

Inequality~\eqref{eq:pac_bayes_lower} depends on the choice of a fixed distribution $\pi$ on the index space $E$, which we call the ``prior''.
In addition, it yields a lower bound on a \emph{smoothed} version of the process $Z$,
obtained by averaging with respect to distributions $\rho$ on $E$ referred to as ``posteriors''.
Although the posterior distribution $\rho$ can in principle be arbitrary, the bound~\eqref{eq:pac_bayes_lower} depends on the relative entropy $D (\rho \Vert \pi)$ between $\rho$ and the prior $\pi$.
This prevents one from choosing posterior distributions that are (arbitrarily close to) Dirac masses at parameters, as such choices would render the bound vacuous.
Here, the relative entropy $D (\rho\Vert\pi)$ quantifies the ``complexity'' of the posterior $\rho$, and constitutes the price of requiring that the bound hold simultaneously for all posteriors.

In practice, we must select a prior distribution $\pi$ on a parameter set $E \subset \R^d \times S^{d-1}$; a process $Z$ to which the PAC-Bayes inequality~\eqref{eq:pac_bayes_lower} is applied; and for each $(\theta, v) \in \Theta \times S^{d-1}$, a posterior/smoothing distribution $\rho_{\theta, v}$ on $E$.
We must then obtain the following:
\begin{enumerate}
\item a lower bound on the \emph{log-Laplace transform}: $- \log \E [\exp ({-\lambda Z (\omega)})]$ for every $\omega \in E$;
\item an upper bound on the \emph{relative entropy} $D (\rho_{\theta, v} \Vert \pi)$ for every $(\theta, v) \in \Theta \times S^{d-1}$;
\item a control on the \emph{smoothing approximation error}, namely a lower bound on the quantity of interest $\innerp{H^{-1/2} \emp{H} (\theta) H^{-1/2} v}{v}$
  in terms of the smoothed process
  \begin{equation*}
    \frac{1}{n} \sum_{i=1}^n \int_E Z_i (\omega) \rho_{\theta, v} (\di \omega)
    \qquad
    \text{ for all }
    (\theta, v) \in \Theta \times S^{d-1}
    \, .
  \end{equation*}
\end{enumerate}
In particular, the above outline suggests a trade-off in the choice of the posterior $\rho_{\theta, v}$: on the one hand, it must be
concentrated enough near
its corresponding parameter $(\theta, v)$ that the smoothing error is small; on the other hand, it must be diffuse enough that the relative entropy $D (\rho_{\theta,v} \Vert \pi)$ is not too large, uniformly over $(\theta, v) \in \Theta \times S^{d-1}$.
The latter condition also requires a suitable choice of prior $\pi$.

Now, from~\eqref{eq:hessian-lower-inf-process}, the term we aim to control corresponds to an empirical process:
\begin{equation*}
  \innerp{H^{-1/2} \emp{H} (\theta) H^{-1/2} v}{v}
  = \frac{1}{n} \sum_{i=1}^n Z_{i}^{\mathrm{int}} (\theta, v)
  \, , \quad \text{where} \quad
  Z^{\mathrm{int}}_i (\theta, v)
  = \sigma' (\innerp{\theta}{X_i}) \innerp{v}{H^{-1/2} X_i}^2
  \, .
\end{equation*}
Hence, a natural approach is to apply the PAC-Bayes inequality to $Z = Z^{\mathrm{int}}$, the process of interest, and then to control the difference $Z (\theta, v) - \int_E Z (\omega) \rho_{\theta, v} (\di \omega)$ between this process and its smoothed version.

We will follow a different approach: we will apply the PAC-Bayes inequality to \emph{an auxiliary process $Z^{\mathrm{aux}}$ whose smoothed version yields} (a lower bound on) \emph{the process of interest} $Z^{\mathrm{int}}$.
Specifically, the sigmoid $\sigma'(\cdot)$ in the process of interest $Z^{\mathrm{int}} (\theta, v) = \sigma' (\innerp{\theta}{X}) \innerp{v}{H^{-1/2} X}^2$ will be replaced by a suitable indicator.
In addition, the posterior distribution $\rho_\theta$ on $\theta$ will critically involve a non-Gaussian component, in order to ensure the previous property.

Concretely, we will show that for a suitable choice of posteriors $\rho_{\theta, v} = \rho_{\theta} \otimes \rho_v$ (see Definition~\ref{def:PostThetaGauss} for the definition of $\rho_\theta$), we have for every $(\theta, v) \in \Theta \times S^{d-1}$,
\begin{align*}
  &\innerp{H^{-1/2}\wh H_n(\theta)H^{-1/2}v}{v} =
  \frac1n \sum_{i=1}^n \sigma'(\innerp{\theta}{X_i}) \innerp{v}{H^{-1/2} X_i}^2 \\
  &\geq \frac{c}{n} \sum_{i=1}^n \int_{\Theta \times S^{d-1}} \bm 1 \braces[\big]{\ainnerp{\theta'}{X_i} \leq 1, \, \|X_i\|\leqslant 2\sqrt{d}} \innerp{v'}{H^{-1/2} X_i}^2 \rho_{\theta} (\di \theta') \rho_v (\di v')
    - (\text{Remainder}) 
\end{align*}
for some constant $c > 0$.
A key step to achieve this is the smoothing result of Lemma~\ref{lem:LBHess0} below.
In addition, the remainder term above comes from the effect of smoothing over $v$, and depends on $X_1, \dots, X_n$.
It is itself bounded through a separate application of PAC-Bayes, this time over spherical caps in $S^{d-1}$ localized around $u^* = \theta^*/\norm{\theta^*}$ (Lemma~\ref{lem:PB2}).

We next proceed to the remainder of the proof, following the outline described above.
First, we define the auxiliary process $Z$ to which the PAC-Bayes inequality is applied, and control its log-Laplace transform.
Second, we define posterior distributions $\rho_\theta$ over parameters $\theta' \in \R^d$, and establish an approximation guarantee for smoothing under such posterior distributions.
Third, we define posterior distributions $\rho_v$ over directions $v' \in S^{d-1}$, and control the smoothing approximation error under these posterior distributions.
Fourth, we define the prior distribution $\pi$ over pairs $(\theta', v') \in \R^d \times S^{d-1}$, and bound the relative entropy $\kll{\rho_{\theta} \otimes \rho_{v}}{\pi}$ between the posterior and prior.
We then put things together to conclude the proof.
To keep the exposition streamlined, some technical lemmas used in this section are deferred to Section~\ref{sec:techn-lemm-hessian-gaussian}.

\subsubsection*{Auxiliary process and control of its log-Laplace transform}

We first define the index set $E$ and auxiliary process $Z$ to which the PAC-Bayes inequality~\eqref{eq:pac_bayes_lower} will be applied.
We let $E = \Theta' \times S^{d-1}$, where $\Theta'$ is the enlarged ellipsoid defined by
\begin{equation}
  \label{eq:def-ellipsoid-large}
  \Theta'=\Big\{ \theta' \in \R^{d} :  \|\theta' - \theta^{*}\|_{H} \leq \frac{1}{10\sqrt{\b}} \Big\}\supset \Theta
  \, .
\end{equation}
In addition, for $\omega = (\theta', v') \in \Theta' \times S^{d-1}$ we let
\begin{equation}
  \label{def:ProcessHessGauss}
  Z (\omega)
  = \ind{|\langle \theta', X \rangle | \leq 1 ; \|X \| \leq 2 \sqrt{d}} \big\langle v', H^{-1/2} X
  \big\rangle^2 \, .
\end{equation}
Likewise, for $i=1, \dots, n$, we define $Z_i (\omega)$ by replacing $X$ with $X_i$ in~\eqref{def:ProcessHessGauss}.

We now control the log-Laplace transform of the process $Z$, which will provide a lower bound on the first term of the right-hand side of~\eqref{eq:pac_bayes_lower}.

\begin{lemma}
  \label{lem:UBLTHG}
  For any $\omega = (\theta', v') \in E = \Theta' \times S^{d-1}$ and $\lambda \geq 0$,
  we have
\begin{equation*}
  - \log \Expect[\big]{\exp ({-\lambda Z (\omega)})}
  \geq 0.03 \lambda - %
  3 \lambda^{2} \b \, .
\end{equation*}
\end{lemma}

\begin{proof}%
  Since $\exp({-s}) \leq 1 - s + s^{2}/2$ for any $s \geq 0$, we have
  \begin{equation*}
    \Expect[\big]{\exp( - \lambda Z (\omega) )}
    \leq 1 - \lambda \Expect{Z (\omega)} + \frac{\lambda^{2}}{2} \Expect{Z (\omega)^{2}}
    \, .
  \end{equation*}
  Since in addition $- \log (1-s) \geq s$ for any $s \geq 0$, we deduce that
  \begin{equation}
    \label{eq:log-lap-lower-twomom}
    - \log \Expect[\big]{\exp (-\lambda Z (\omega))}
    \geq \lambda \Expect{Z (\omega)} - \frac{\lambda^2}{2} \Expect{Z (\omega)^2}
    \, .
  \end{equation}
  In light of~\eqref{eq:log-lap-lower-twomom}, in order to prove Lemma~\ref{lem:UBLTHG} it suffices to establish a lower bound on $\E [Z (\omega)]$ and an upper bound on $\E [Z (\omega)^2]$.

  Let $u' = \theta'/\norm{\theta'}$ ($u' \in S^{d-1}$ can be arbitrary if $\theta' = 0$), and define the matrices
  \begin{align}
    \label{eq:DefwtH}
    \wt H (\theta')
    &= \Expect*{\bm 1 \braces*{\innerp{\theta'}{X} \leq 1 ; \norm{X} \leq 2 \sqrt{d}} X X^\top} \, ; \\
    H_{\theta'}
    &= \frac{1}{\b^3} u' u'^{\top} + \frac{1}{\b} \parens*{I_d - u' u'^\top}
      \, .\label{eq:DefHtheta}
  \end{align}
  Lemmas~\ref{lem:hess-proxy-regularity} and~\ref{lem:truncated-hessian} below imply respectively that, since $\theta' \in \Theta'$,
  \begin{equation}
    \label{eq:ineqs-h-matrices-log-lap}
    H_{\theta'}
    \mgeq 0.75 \, H
    \quad \text{ and } \quad
    \wt H (\theta')
    \mgeq 0.05 \, H_{\theta'}
    \mgeq 0.03 \, H
    \, .
  \end{equation}

  With these prerequisites in place, we bound $\Expect{Z (\omega)}$ from below.
  Using~\eqref{eq:ineqs-h-matrices-log-lap}, we have
  \begin{align}
    \label{eq:lower-expect-Z}
    \Expect{Z (\omega)}
    &= \Expect*{\ind{\ainnerp{\theta'}{X} \leq 1; \norm{X} \leq 2 \sqrt{d}} \innerp{v'}{H^{-1/2} X}^2} \nonumber \\
    &= \innerp{H^{-1/2} \wt H (\theta') H^{-1/2} v'}{v'}
      \geq 0.03 \norm{v'}^2
      = 0.03
      \, .
  \end{align}

  Next, we bound $\E [ Z(\omega)^2 ]$ from above.
  Denoting $v = H_{\theta'}^{1/2} H^{-1/2} v'$, we have
  \begin{align*}
    \Expect{Z (\omega)^2}
    \leq \Expect[\big]{\bm 1 \braces{\ainnerp{\theta'}{X} \leq 1} \innerp{H^{-1/2} v'}{X}^4} 
    = \Expect[\big]{\bm 1\braces{\ainnerp{\theta'}{X} \leq 1} \innerp{H_{\theta'}^{-1/2} v}{X}^4}
      \, .
  \end{align*}
  In addition, since
  \begin{equation*}
    \innerp{H_{\theta'}^{-1/2} v}{X}
    = \b^{3/2} \innerp{u'}{v} \innerp{u'}{X} + \b^{1/2} \innerp{v-\innerp{u'}{v} u'}{X} \, ,
  \end{equation*}
  using that $\innerp{u'}{X}$ and $\innerp{v-\innerp{u'}{v} u'}{X}$ are independent centered Gaussian random variables, and bounding
  \begin{equation*}
    \ainnerp{u'}{v}
    \leq \norm{v}
    \quad \text{ and } \quad
    \norm{v - \innerp{u'}{v} u'}
    \leq \norm{v}
    \, ,
  \end{equation*}
  we obtain that
  \begin{equation*}
    \Expect[\big]{Z (\omega)^2}
    \leq \norm{v}^4 \, \E\big[ \ind{\ainnerp{\theta'}{X} \leq 1} \parens[\big]{\b^6 \innerp{u'}{X}^4 + 6\b^4 \innerp{u'}{X}^2 + \b^2} \big] \, .
  \end{equation*}
  Next, using that $\ainnerp{\theta'}{X} = \norm{\theta'} \cdot \ainnerp{u'}{X}$ and that the density of $\innerp{u'}{X} \sim \gaussdist (0, 1)$ is upper-bounded by $1/\sqrt{2\pi}$, we bound for $k \in \{ 0,1,2 \}$:
  \begin{align*}
    \E\big[ \ind{|\innerp{\theta'}{X} | \leq 1} \innerp{u'}{X}^{2k} \big]
    \leq \frac{1}{\sqrt{2\pi}} \int_\R \ind{\abs{x} \leq \norm{\theta'}^{-1}} x^{2k} \di x
    = \frac{\sqrt{2/\pi}}{(2k+1)\|\theta'\|^{2k+1}} 
    \enspace .
  \end{align*}
  In addition, the first inequality in~\eqref{eq:ineqs-h-matrices-log-lap} (and the fact that $\norm{v'}= 1$) implies that
  \begin{equation*}
    \norm{v}^4
    = \innerp{H^{-1/2} H_{\theta'} H^{-1/2} v'}{v'}^2
    \leq \parens[\Big]{\frac{4}{3} \norm{v'}^2}^2
    = \frac{16}{9}
    \, .
  \end{equation*}
  Finally, it follows from Point~1 in Lemma~\ref{lem:ellipsoids} that $\norm{\theta'} \geq 0.9 \b$ (as $\b = \norm{\theta^*} \geq e$).  
  Combining the previous inequalities, we obtain
  \begin{align}
    \label{eq:bound-z-squared-log-lap}
    \Expect[\big]{Z (\omega)^2}
    &\leq \frac{16}{9} \sqrt{\frac{2}{\pi}} \parens[\bigg]{\b^6 \times \frac{1}{5 \b^5} \parens[\Big]{\frac{10}{9}}^5 + 6 \b^4 \times \frac{1}{3 \b^3} \Big( \frac{10}{9} \Big)^3 + \b^2 \times \frac{10}{9 \b}}
    \leq 6 \b
      \, .
  \end{align}
  Plugging the bounds~\eqref{eq:lower-expect-Z} and~\eqref{eq:bound-z-squared-log-lap} into~\eqref{eq:log-lap-lower-twomom} concludes the proof.
\end{proof}

\subsubsection*{Posterior distributions and smoothing over parameters $\theta$}

As mentioned previously, the posteriors we will use on $E = \Theta' \times S^{d-1}$ will be of the form $\rho_{\theta, v} = \rho_\theta \otimes \rho_v$ for $\theta \in \Theta$ and $v \in S^{d-1}$, where $\rho_\theta$ is a distribution on $\Theta'$ and $\rho_v$ a distribution on $S^{d-1}$ (with a slight abuse of notation, we use similar notation for both posterior distributions).

In this section, for every $\theta \in \Theta$, we define the posterior distribution $\rho_\theta$ over parameters $\theta' \in \Theta'$, and establish an approximation result for smoothing under such distributions.

\begin{definition}\label{def:PostThetaGauss}
  For any $\theta \in \Theta$, we let $\rho_\theta$ denote the distribution of $\theta' = U \theta + Z$, where
  \begin{itemize}
  \item[(i)] $U, Z$ are independent;
  \item[(ii)] $U$
    is uniform over $[0.99,1.01]$;
  \item[(iii)] the distribution of $Z$
    is the conditional distribution of $Z' \sim \normal(0,(I_d-uu^\top)/(2\cdot 10^4\cdot d))$ given that $\norm{Z'} \leq 1/100$, where $u = \theta/\norm{\theta}$.
  \end{itemize}  
\end{definition}

The motivation behind the choice of the posterior (or smoothing distribution) $\rho_\theta$ in Definition~\ref{def:PostThetaGauss} is twofold.
On the one hand, it is sufficiently spread out that, for every $\theta \in \Theta$, the relative entropy between $\rho_\theta$ and a suitably chosen prior is controlled: Lemma~\ref{lem:UBKLGauss} below show that it is at most of order $d$, with no dependence on $\b$.
At the same time, it is sufficiently localized around $\theta$ (in particular, along the direction of $\theta$ itself) that
smoothing an indicator with respect to this distribution
provides a lower bound on the sigmoid, as shown in Lemma~\ref{lem:LBHess0} below.

\begin{lemma}\label{lem:LBHess0}
    For every $\theta \in \Theta$, the measure $\rho_\theta$ is supported on $\Theta'$.
    In addition, for every $x \in \R^d$ such that $\|x\|\leqslant 2\sqrt{d}$, one has
    \begin{equation}
      \label{eq:smoothing-indicator}
      \sigma'(\innerp{\theta}{x})
      \geqslant \frac1{15}\int_{\R^d} \bm 1 \big\{|\langle \theta' ,  x \rangle | \leq 1\big\} \rho_{\theta} (\di \theta')\, .
    \end{equation}
  \end{lemma}

We note in passing that such a lower bound would not hold if instead of being uniform on $[0.99, 1.01]$, the variable $U$ in Definition~\ref{def:PostThetaGauss} was Gaussian (say, if $U \sim \gaussdist (1, 0.01^2)$): in this case, the smoothed indicator would be of order $(1 + \ainnerp{\theta}{x})^{-1}$, which is much larger than $\sigma' (\ainnerp{\theta}{x})$.

  \begin{proof}%
    We start with the first claim.
    Let $\theta\in\Theta$ and $\theta'=U\theta+Z \sim \rho_\theta$, with $U,Z$ distributed as in Definition~\ref{def:PostThetaGauss}.
    By Lemma~\ref{lem:hess-proxy-regularity}, since $\|\theta-\theta^*\|_H\leqslant 1/100\sqrt{\b}$, if $\theta' \sim \rho_\theta$ then
\begin{align*}
    \|\theta'-\theta\|_H&\leqslant \frac{1}{0.97}\|\theta'-\theta\|_{H_\theta} =\frac{1}{0.97}\sqrt{\frac{(U-1)^2\|\theta\|^2}{\b^3}+\frac{\|Z\|^2}{\b}}\, .
\end{align*}
Now $|U-1|\leqslant 0.01$, $\|\theta\|/\b\leqslant 1.01$ and $\|Z\|\leqslant 1/100$ a.s., so by Lemma~\ref{lem:ellipsoids}, as $\theta\in\Theta$,
\begin{align}\label{eq:SuppRhotheta}
    \|\theta'-\theta\|_H\leqslant \frac{0.015}{\sqrt{\b}}\qquad \text{and}\qquad \|\theta'-\theta^*\|_H\leqslant \frac{0.025}{\sqrt{\b}}\, .
\end{align} 

We now prove inequality~\eqref{eq:smoothing-indicator}.
Let $x \in \R^d$ be such that $\|x\|\leqslant 2\sqrt{d}$.
We have
\begin{equation*}
  \int_{\R^d} \bm 1 \{|\langle \theta' ,  x \rangle | \leq 1 \} \rho_\theta (\di \theta')
		= \E \big[ \P\left( | U \langle \theta, x \rangle
			+ \langle Z, x \rangle | \leq 1 | U \right) \big] \, . \label{eq:cond}
\end{equation*}
If $Z' \sim \gaussdist (0, (I_d - u u^\top)/(2\cdot 100^2 \cdot d))$, we have, as $\P (\norm{Z'} \leq 1/100) \geq 1 - 100^2 \E \norm{Z'}^2 \geq 3/4$,
\[
\P\left( | U \langle \theta, x \rangle + \langle Z, x \rangle | \leq 1 | U \right)
	\leq \frac43\P\left( | U \langle \theta, x \rangle + \langle Z', x \rangle | \leq 1 | U \right) .
\]
Now, if $g$ is a standard Gaussian random variable, for any $a\in \R$, $b>0$ and $\sigma>0$
\[
\P(|\sigma g-a|\leqslant 1)\leqslant 2\P(g>(|a|-1)_+/\sigma)\leqslant \exp\bigg(-\frac{(|a|-1)_+^2}{2\sigma^2}\bigg)\leqslant C\exp(-b|a|)\, ,
\]
with $C=\exp\big(\frac{\sigma^2b^2}2+b\big)$.

We apply this result with 
\[
\sigma^2=\Var(\innerp{g'}{x})\leqslant \frac{\|x\|^2}{2\cdot(100)^2\cdot d}\leqslant\frac1{5000},\qquad b=\frac{1}{U}\leqslant \frac1{0.99},\qquad a=U\innerp{\theta}{x}\enspace.
\]
This shows that
\[
\P\left( | U \langle \theta, x \rangle
			+ \langle Z, x \rangle | \leq 1 | U \right)\leqslant 3.7\exp\big(-|\innerp{\theta}{x}|\big)\leqslant 15\sigma'(\innerp{\theta}{x})\, .
\]
This proves the lower bound~\eqref{eq:smoothing-indicator}.
\end{proof}

\subsubsection*{Posterior distributions and smoothing over directions $v$}

We now define, for every $v \in S^{d-1}$, the posterior $\rho_v$ over directions $v' \in S^{d-1}$.
We then control the approximation error that arises from smoothing over $\rho_v$.

\begin{definition}\label{def:DefRhov}
  Let $\epsilon\in(0,1)$. 
  For any $v\in S^{d-1}$, let $\rho_v$ denote the uniform distribution on the spherical cap $\mc (v, \eps)$ of radius $\eps$ around $v$, defined by
  \begin{equation}
    \label{eqdef:sphere-caps-polar}
    \mc(v, \eps) =
    \braces[\big]{v' \in S^{d-1} : \innerp{v}{v'} \geq \sqrt{1 - \eps^2}}
    \, .
  \end{equation}    
\end{definition}

We next show that the empirical Hessian can be lower-bounded in terms of the smoothed process and a remainder term $R_n$.

\begin{lemma}\label{lem:LBLinTGauss}
  Let $\rho_{\theta,v}=\rho_\theta\otimes\rho_v$ denote the posterior distribution defined as the product of the posterior $\rho_\theta$ of Definition~\ref{def:PostThetaGauss} and the posterior $\rho_v$ of Definition~\ref{def:DefRhov}.
  Then, for any $(\theta, v) \in \Theta \times S^{d-1}$,
  \begin{equation}\label{eq:LBLinTerm}
    \innerp{H^{-1/2} \wh H_n(\theta) H^{-1/2} v}{v}
    \geq \frac{1}{15 n} \sum_{i=1}^{n}\int_{\Theta'\times S^{d-1}} Z_i (\omega) %
    \rho_{\theta,v} (\di \omega) %
    - 22 \eps^2 \b R_n \enspace,    
  \end{equation}
  where 
  \begin{align*}
    R_n 
    &= \sup_{u \in \mc (u^*, 1/10\b)} \bigg\{ \frac{1}{n} \sum_{i=1}^{n} \exp\big( - 0.49 \, \b \ainnerp{u}{X_i} \big)
      \1\big(\|X_{i}\| \leq 2\sqrt{d}\big) \bigg\}
    \, .
  \end{align*}
\end{lemma}

\begin{proof}
  By Lemma~\ref{lem:LBHess0}, for every $v' \in S^{d-1}$ one has
  \begin{equation*}
    \frac{1}{n} \sum_{i=1}^{n} \sigma'(\innerp{\theta}{X_i}) \ind{\|X_i\|
      \leq 2\sqrt{d}} \innerp{X_i}{H^{-1/2} v'}^2
    \geq \frac{1}{15 n} \sum_{i=1}^{n}\int_{\Theta'} Z_i(\theta',v') \rho_\theta(\di \theta')\, .
  \end{equation*}
  Integrating over $v' \sim \rho_v$, we obtain
  \begin{align}
    \frac{1}{15 n} \sum_{i=1}^n \int_{\Theta' \times S^{d-1}} Z_i (\omega) \rho_{\theta, v} (\di \omega)
    &\leq \frac{1}{n} \sum_{i=1}^{n} \int_{S^{d-1}} \sigma'(\innerp{\theta}{X_i}) \ind{\|X_i\|\leqslant 2\sqrt{d}}\innerp{X_i}{H^{-1/2}v'}^2\rho_v(\di v') \nonumber \\
    &= \int_{\mc(v, \eps)}\innerp{H^{-1/2} \ol H_n (\theta) H^{-1/2}v'}{v'}\rho_v(\di v')
      \label{eq:first-proof-smooth-v}
      \, ,
  \end{align}
  where
  \begin{equation*}
    \overline{H}_n(\theta)=\frac{1}{n} \sum_{i=1}^{n}\sigma'(\innerp{\theta}{X_i})\bm 1 \big\{\|X_i\|\leqslant 2\sqrt{d} \big\} X_iX_i^\top \mleq \wh H_n(\theta)
    \, .
  \end{equation*}

Using the computations from \cite[Eqs.~(42) and (43)]{mourtada2022linear} and Fact~\ref{fact:angle-radius-correspondance}, we get
\begin{multline}
  \label{eq:main-approx}
  \int_{\mc(v,\eps)} \innerp{H^{-1/2}\overline{H}_n(\theta)H^{-1/2}v'}{v'} \rho_v(\di v') \\
  \leq  \big \langle H^{-1/2} \wh H_n(\theta) H^{-1/2} v , v \big \rangle
  + \frac{2\eps^2}{d-1} \tr\left( H^{-1/2} \overline{H}_{n}(\theta) H^{-1/2} \right) \, .
\end{multline} 

We now control the trace term in~\eqref{eq:main-approx}.
First, by Lemma~\ref{lem:hess-proxy-regularity}, one has $H^{-1} \mleq 1.03  H_\theta^{-1}$ for any $\theta \in \Theta$, thus
\begin{align*}
  &\tr \parens[\big]{H^{-1/2} \overline{H}_{n}(\theta) H^{-1/2}}
    = \tr \parens[\big]{\ol H_n (\theta)^{1/2} H^{-1} \ol H_n (\theta)^{1/2}} \\
  &\leq 1.03 \, \tr \parens{\ol H_n (\theta)^{1/2} H_\theta^{-1} \ol H_n (\theta)^{1/2}}
    = 1.03 \, \tr \parens[\big]{H_\theta^{-1/2} \overline{H}_{n}(\theta)H_\theta^{-1/2}}
    \, .
  \end{align*}
  Now, as $\|X_i-\innerp{u}{X_i} u\|^2\leqslant \|X_i\|^2$,
 \begin{align*}
   \tr\big(H^{-1/2} \overline{H}_{n}(\theta)& H^{-1/2} \big)	\\
&\leqslant \frac{1.03}{n} \sum_{i=1}^{n} \sigma'\big( \langle \theta , X_{i}
		\rangle\big) \ind{\|X_{i}\|\leq 2\sqrt{d}}\big(\b^3\innerp{u}{X_i}^2 + \b \|X_i-\innerp{u}{X_i}u\|^2\big)\\
  &\leqslant \frac{1.03\b}{n} \sum_{i=1}^{n} \sigma'\big( \langle \theta , X_{i}
		\rangle\big) \ind{\|X_{i}\|\leq 2\sqrt{d}}\big(\b^2\innerp{u}{X_i}^2 +4d\big)\, .
   \end{align*}
   Now, we use the inequalities $\sigma'(t)\leqslant e^{-|t|}$ and
   $t^2 e^{-t} = (t/2)^2 e^{t/2} \cdot 2^2 e^{-t/2} \leq 4 e^{-2} 2^2 e^{- t/2}$
   for all $t \geq 0$ and, by Lemma~\ref{lem:ellipsoids}, $\norm{\theta} \geq 0.99 \b$ (as $\theta \in \Theta$ and $\b = \norm{\theta^*} \geq e$)
   to get 
   \[
     \b^2\innerp{u}{X_i}^2\sigma'\big( \langle \theta , X_{i}
     \rangle\big)
     \leq 2.2 \exp\big(-0.49 \b | \innerp{u}{X_{i}}|\big)\, .
   \]
   Plugging this into the previous inequality gives (using that $d \geq 2$)
   \begin{align}
     \tr\big(H^{-1/2} \overline{H}_{n}(\theta) H^{-1/2} \big)
     &\leq {1.03 \b} \parens{2.2 + 4 d} \cdot \frac{1}{n} \sum_{i=1}^n \exp \parens[\big]{- 0.49 \b \ainnerp{u}{X_i}} \indic{\norm{X_i} \leq 2 \sqrt{d}} \nonumber \\
     &\leq {5.5 \b d} \cdot R_n
       \, ,
       \label{eq:bound-trace-remainder}
   \end{align}
   where the last inequality comes from the fact that $u \in \mc (u^*, 1/10 \b)$ as (by Lemma~\ref{lem:ellipsoids} and using that $\theta \in \Theta$) $\norm{u - u^*} \leq 1/10 B$.   

   Combining inequalities~\eqref{eq:first-proof-smooth-v},~\eqref{eq:main-approx} and~\eqref{eq:bound-trace-remainder} gives (bounding $d-1 \geq d/2$):
   \begin{equation*}
     \frac{1}{15 n} \sum_{i=1}^n \int_{\Theta' \times S^{d-1}} Z_i (\omega) \rho_{\theta, v} (\di \omega)
     \leq \big \langle H^{-1/2} \wh H_n(\theta) H^{-1/2} v , v \big \rangle
     + 22 \eps^2 \b R_n
     \, ,
   \end{equation*}
   which concludes the proof.
\end{proof}

Note that Lemma~\ref{lem:LBLinTGauss} features an empirical remainder term $R_n$, which must also be controlled.
Since this term is bounded using a second application of the PAC-Bayes inequality over a different space than $\Theta' \times S^{d-1}$, we defer the proof of this bound (Lemma~\ref{lem:PB2} below) to a paragraph following the conclusion of the proof.

\subsubsection*{Prior distribution and bound on the relative entropy
}

We now define the prior distribution $\pi$ and bound from above the relative entropy term $D(\rho_{\theta,v} \Vert \pi)$, where  $\rho_{\theta,v}$ was defined in Lemma~\ref{lem:LBLinTGauss}.

Let us start with the definition of the prior $\pi$.
For any $\mu\in \R^d$, $\Sigma\mgeq 0$ and any measurable subset $S\subset \R^d$, let $\gaussdist(\mu, \Sigma \mathop{|} S)$ denote the Gaussian distribution $\normal(\mu, \Sigma)$ conditioned on $S$, that is the distribution $\nu$ with density
\begin{equation}
\label{eq:def-truncated-gaussian}
	\di \nu = \frac{\1_{S}}{\gamma(S)} \, \di \gamma \, ,
\end{equation}
where $\gamma$ is the Gaussian distribution $\normal(\mu, \Sigma)$.
\begin{definition}\label{def:PriorGauss}
The prior distribution $\pi$ on $\Theta \times S^{d-1}$ is the product measure
$\pi = \pi_{\Theta} \otimes \pi_{S}$, where $\pi_S$ is the uniform distribution on $S^{d-1}$ and $\pi_{\Theta}=\gaussdist(\theta^*, \Gamma \mathop{|} \Theta')$, where
\[
\Gamma = \frac1{100^2}\bigg(\b^{2} u^{*}{u^{*}}^{\top}
+ \frac{1}{2d} (I_{d} -u^{*}{u^{*}}^{\top})\bigg) \, . 
\]    
\end{definition}

The relative entropy term is bounded in the following lemma:
\begin{lemma}\label{lem:UBKLGauss}
    Let $\theta \in \Theta$ and $v \in S^{d-1}$.
    Let $\rho_{\theta,v}$ denote the prior defined in Lemma~\ref{lem:LBLinTGauss} and $\pi$ denote the prior distribution of Definition~\ref{def:PriorGauss}.
Then,
\begin{equation}
  \label{eq:kl-total-bound}
  D(\rho_{\theta,v} \Vert \pi)
  \leq \Big(6.5 + \log\Big(1+\frac{2}{\eps}\Big) \Big) d \, .
\end{equation}
\end{lemma}
\begin{proof}%
Since the prior and all posterior distributions are product measures, the divergence writes
\begin{equation*}
D( \rho_{ \theta, v } \| \pi ) = D( \rho_v  \| \pi_S ) + D( \rho_\theta \| \pi_\Theta )\, .
\end{equation*}
On one hand, we have 
\begin{equation*}
	D( \rho_v  \| \pi_S )
		= \int_{S^{d-1}} \log	\Big( \frac{\di \rho_v}{\di \pi_{S}}\Big) \di \rho_{v}
		= \log\left( \frac{\vol_{d-1}(S^{d-1})}{\vol_{d-1}(\mc(v,\eps))} \right) \, .
\end{equation*}	
By \cite[\S4.4]{mourtada2022linear} and Fact~\ref{fact:angle-radius-correspondance}, this yields 
\begin{equation}\label{eq:kl-sphere}
	D( \rho_v  \| \pi_S ) \leq d \log\Big(1+ \frac 2 \eps \Big) \, .
\end{equation}

It remains to bound $D( \rho_\theta \| \pi_\Theta )$, which is more delicate. 
We first define an intermediate distribution $\tilde{\rho}_\theta$ and show, see~\eqref{eq:cyl-to-tronc-gau}, that
\[
D( \rho_\theta \| \pi_\Theta )\leqslant 1.5(\log(1.5)+D(\tilde{\rho}_\theta\Vert\pi_\Theta))\, .
\]
Then, we bound this last divergence.
The intermediate distribution $\tilde{\rho}_\theta=\gaussdist(\theta,\Gamma_\theta \mathop{|} \mathcal{E}_\theta)$ is the Gaussian distribution $\normal(\theta,\Gamma_\theta)$ conditioned on the ellipsoid $\mathcal{E}_\theta=\{\theta_0:\|\theta_0-\theta\|_H\leqslant \frac{0.02}{\sqrt{\b}}\}$, chosen such that, by Lemma~\ref{lem:LBHess0}, 
\[
\supp{\rho_\theta}\subset\supp{\wt \rho_\theta}\subset \Theta'=\supp{\pi_{\Theta}}\, .
\]
For any $\theta \in \Theta$, the covariance $\Gamma_{\theta}$ is defined as 
\begin{equation*}
	\Gamma_{\theta} = \frac1{100^2}\bigg(\b^{2} u u^{\top} + \frac{1}{2d} (I_{d} - u u^{\top})\bigg),
	\qquad u=\frac{\theta}{\|\theta\|} \, .
\end{equation*}

Before we bound the Kullback-Leibler divergences $D(\rho_\theta\Vert \pi_\Theta)$, we check the following facts.
\begin{enumerate}%
\item \label{it:density-domination} the density of  $\rho_{\theta}$ satisfies
  $\frac{\di \rho_{\theta}}{\di \trh_{\theta}} \leq 1.5$,
\item \label{it:normalization}
  one has
  $\gamma_{\theta}(\mathcal{E}_{\theta}) \geq 0.5$ and $\gamma(\mathcal{E}_{\theta})\geqslant 0.3$, where $\gamma_\theta=\normal(\theta,\Gamma_\theta)$ and $\gamma=\normal(\theta^*,\Gamma)$. 
\end{enumerate}

We start with point \ref{it:density-domination}. 
We have on one hand that the density $f_{\theta}$ of $\rho_{\theta}$ satisfies, for every $\theta_0=t u+z$,
\begin{align*}
 f_{\theta}(\theta_0)
	&\leq \frac{1}{p0.02\|\theta\|} \Big(\frac{(100)^2d}{\pi}\Big)^{\frac{d-1}{2}} \exp\Big(-\frac{d \|z\|^{2}}{100^2}\Big) \,
		\1\Big( t/\|\theta\| \in[0.99,1.01] ;\|z\|\leq 1/100\Big) 
\end{align*}
and on the other hand $\trh_{\theta}$ has density given for every $\theta_0 = t u+z$ by
\begin{align*}
	\tilde f_{\theta}(\theta_0)
		&= \frac{1}{\gamma_{\theta}(\mathcal{E}_{\theta})}
			\cdot \frac{e^{-\frac{(t-\|\theta\|)^{2}}{2(\b/100)^{2}}}}{\b/100} \cdot \frac{(2\cdot(100)^2d)^{\frac{d-1}{2}}}{(2\pi)^{d/2}}
			\exp\big(-(100)^2d \|z\|^{2}\big) \1(\theta_0 \in \mathcal{E}_{\theta}) \, .
\end{align*}
For any $t$ such that $|t/\|\theta\|-1|\leqslant 1/100$ and $\theta\in \Theta$ so, by Lemma~\ref{lem:ellipsoids}, $\|\theta\|/\b\in [0.99,1.01]$, we deduce that
\begin{equation}
\label{eq:density-domination}
	\frac{f_{\theta}(\theta_0)}{\tilde f_{\theta}(\theta_0)}
		\leq \frac1{2p}\frac{\b}{\|\theta\|}\sqrt{\frac{\pi}{2}} \gamma_{\theta}(\mathcal{E}_{\theta}) 
			\exp\Big(\frac{(t-\|\theta\|)^{2}}{2(\b/100)^{2}}\Big) \1\left( \frac{t}{\|\theta\|} \in[0.99,1.01] ;\theta_0 \in \mathcal{E}_{\theta} \right)
		\leq 1.5\cdot \1(\theta_0 \in \mathcal{E}_{\theta}) \, .
\end{equation}

Let us move to point~\ref{it:normalization}. 
Fix $\theta\in \Theta$ so by Lemma~\ref{lem:ellipsoids}, $\|u-u^*\|\leqslant 1/50\b$.
We have, if $N_\theta \sim \normal(\theta,\Gamma_\theta)$, by Chebychev's inequality,
\begin{equation*}
	1 - \gamma_\theta(\mathcal{E}_{\theta})
		= \P\left( \| N_\theta -\theta\|_{H} > \frac{0.02}{\sqrt{\b}} \right) 
		\leq \frac{\b \, \Expect{\| N_\theta -\theta\|_{H}^2}}{0.02^2}\, .
\end{equation*}
Besides,
\begin{align*}
  \Expect[\big]{\| N_\theta-\theta\|_{H}^{2}}
  &= \tr\big(H^{1/2} \Gamma_\theta H^{1/2} \big) =\frac1{100^2}\bigg(\bigg(\b^2-\frac1{2d}\bigg)\|H^{1/2}u\|^2+\frac1{2d}\tr(H)\bigg)\\
  &\leqslant \frac1{100^2} \bigg(\bigg(\b^2-\frac1{2d}\bigg) \bigg(\frac1{\b^3}+\frac{\|u-u^*\|^2}{2\b}\bigg) + \frac1{2d} \bigg(\frac{1}{\b^3} + \frac{d-1}{\b}\bigg) \bigg)
    \leqslant \frac{2}{100^2\b}\, .
\end{align*} 
This shows the first lower bound. 
For the second one, we proceed similarly: Let $N\sim \normal(\theta^*,\Gamma)$ so, by Chebychev's inequality,
\begin{equation*}
	1 - \gamma(\mathcal{E}_{\theta})
		= \P\left( \| N-\theta\|_{H} > \frac{0.02}{\sqrt{\b}} \right) 
		\leq \frac{\b \, \E [ \| N-\theta\|_{H}^2 ]}{0.02^2}\, .
\end{equation*}
Besides, as $\theta\in \Theta$,
\begin{align}
	\notag\E \| N-\theta\|_{H}^{2}
	&= \|\theta-\theta^*\|^2_H+\tr\big(H^{1/2} \Gamma H^{1/2} \big) \\
   \notag &\leqslant \frac{1}{100^2\b}+\frac1{100^2}\bigg(\bigg(\b^2-\frac1{2d}\bigg)\|H^{1/2}u^*\|^2+\frac1{2d}\tr(H)\bigg)\\
	&=\frac{1}{100^2\b}\bigg(2+\frac{d-1}{2d} \bigg)\leqslant \frac{2.5}{100^2\b} \label{eq:UBTr} \, .
\end{align} 
This concludes the proof of Point~\ref{it:normalization}.

We are now in position to bound the relative entropy $D(\rho_\theta\Vert \pi_\Theta)$.
By point~\ref{it:density-domination}, we have
\begin{equation}
 D(\rho_{\theta} \Vert \pi_{\Theta})
 	= \int_{\mathcal{E}_{\theta}} \log\left(\frac{\di \rho_{\theta}}{\di \pi_{\Theta}}\right)
		\di \rho_{\theta}  
	\leq \int_{\mathcal{E}_{\theta}} \log\left(\frac{1.5 \di \trh_{\theta}}{\di \pi_{\Theta}}\right)
		1.5 \di \trh_{\theta} 
	= 1.5 (\log 1.5 + D(\trh_{\theta} \Vert \pi_{\Theta}) )\, . \label{eq:cyl-to-tronc-gau}
\end{equation}
Now, denote $\gamma_\theta=\normal(\theta,\Gamma_\theta)$ and $\gamma=\normal(\theta^*,\Gamma)$ so $\trh_\theta$ and $\pi_\Theta$ are the restrictions of $\gamma_\theta$ and $\gamma$.
We have 
\[
D(\trh_{\theta} \Vert \pi_{\Theta})
=\int_{\mathcal{E}_\theta} \frac{\di \gamma_\theta}{\gamma_\theta(\mathcal{E}_\theta)}\log\bigg(\frac{\di\gamma_\theta/\gamma_\theta(\mathcal{E}_\theta)}{\di \gamma/\gamma(\mathcal{E}_\theta)}\bigg)+\log\bigg(\frac{\gamma(\Theta')}{\gamma(\mathcal{E}_\theta)}\bigg)\, .
\]
Using Lemma~\ref{lem:conditional-KL} and Point \ref{it:normalization} to bound the first term in the left-hand-side and Point \ref{it:normalization} for the second, we get,
\begin{equation}
\label{eq:KL-truncated-to-standard}
	D(\rho_{\theta} \Vert \pi_{\Theta})\leqslant 1.5\log (5)+3D(\gamma_{\theta} \Vert \gamma)\, .
\end{equation} 

Finally, we compute the divergence from $\gamma_{\theta}$ to $\gamma$.
Recall that, as $\det(\Gamma)=\det(\Gamma_\theta)$, it is equal to 
\[
  D(\gamma_{\theta} \Vert \gamma)
  =\frac12\big(\tr(\Gamma^{-1/2}\Gamma_\theta\Gamma^{-1/2})+\|\theta-\theta^*\|_{\Gamma^{-1}}^2-d\big)\, .
\]
As $\Gamma^{-1} \mleq 2(100)^2d \b H$, we have on one side, by \eqref{eq:UBTr},
\[
\tr(\Gamma^{-1/2}\Gamma_\theta\Gamma^{-1/2})\leqslant 2(100)^2d\b \tr(H^{1/2}\Gamma_\theta H^{1/2})\leqslant 3d\enspace,
\]
 and, on the other side,
\begin{align*}
 \|\theta-\theta^*\|_{\Gamma^{-1}}^2&\leqslant 2(100)^2d\b \|\theta-\theta^*\|^2_H\leqslant 2d\enspace.
\end{align*}
Thus, $D(\gamma_{\theta} \Vert \gamma) \leq 2d$ and, by \eqref{eq:KL-truncated-to-standard},
\begin{equation}
\label{eq:KL-bound}
	D( \rho_{\theta} \Vert \pi_{\Theta} ) \leq  6.5d \, .
\end{equation}
Combining this inequality with \eqref{eq:kl-sphere} concludes the proof.    
\end{proof}

\subsubsection*{Conclusion of the proof}

We apply the PAC-Bayes inequality~\eqref{eq:pac_bayes_lower} to the process $Z = (Z (\omega))_{\omega \in E}$ defined by~\eqref{def:ProcessHessGauss} on $E = \Theta' \times S^{d-1}$, with prior distribution $\pi$ on $E$ from Definition~\ref{def:PriorGauss} and parameter $\lambda > 0$ (to be chosen later).
This implies that with probability at least $1-e^{-t}$, for every $\theta \in \Theta$ and $v \in S^{d-1}$, 
\begin{equation}
  \label{eq:conc-pb-1}
  \frac{1}{n} \sum_{i=1}^{n} \int_{E} Z_i(\omega) \rho_{\theta, v} (\di \omega)
  \geq - \frac1{\lambda }\int_{E} \log  \E\big[ \exp({-\lambda Z(\omega)}) \big] \rho_{\theta, v} (\di \omega) 
  - \frac{D(\rho_{\theta, v} \Vert \pi) + t}{\lambda n}
  \, ,
\end{equation}
where $\rho_{\theta, v} = \rho_\theta \otimes \rho_v$ (see Definitions~\ref{def:PostThetaGauss} and~\ref{def:DefRhov}).

Next, we bound the left-hand side of~\eqref{eq:conc-pb-1} using Lemma~\ref{lem:LBLinTGauss}, and control the two terms in the right-hand side of~\eqref{eq:conc-pb-1} using Lemmas~\ref{lem:UBLTHG} and~\ref{lem:UBKLGauss} respectively.
This yields the following: with probability at least $1-e^{-t}$, for every $\theta \in \Theta$ and $v \in S^{d-1}$,
\begin{equation}
  \label{eq:conc-pb-2}
  \innerp{H^{-1/2} \emp{H} (\theta) H^{-1/2} v}{v}
  \geq \frac{1}{15} \parens[\big]{0.03 - 3 \lambda \b} - \frac{\parens{6.5 + \log(1+{2}{\eps}^{-1})} d + t}{15 \lambda n} - 22 \eps^2 \b R_n
  \, ,
\end{equation}
where $R_n$ is defined in Lemma~\ref{lem:LBLinTGauss}.
Now by Lemma~\ref{lem:PB2}, as $n \geq 1.1 (d + t)$, with probability at least $1-e^{-t}$ one has $R_n \leq 4/\b$.

Hence, with probability at least $1 - 2 e^{-t}$, one has for every $\theta \in \Theta$ and $v \in S^{d-1}$,
\begin{equation}
  \label{eq:hessian-almost-done}
  \innerp{H^{-1/2} \emp{H} (\theta) H^{-1/2} v}{v}
  \geq 0.002 - \frac{\lambda \b}{5} - \frac{\parens{6.5 + \log(1+{2}{\eps}^{-1})} d + t}{15 \lambda n} - 88 \eps^2
  \, .
\end{equation}
We choose $\lambda = 0.002/\b$ and $\eps = 0.001$, so that whenever $n \geq 500 \b (d + t)$ one has
\begin{equation*}
  \frac{\lambda B}{5}
  \leq 0.0004,
  \quad
  \frac{\parens{6.5 + \log(1+{2}{\eps}^{-1})} d + t}{15 \lambda n}
  \leq 0.0004,
  \quad
  88 \eps^2
  \leq 0.0002
  \, ,
\end{equation*}
and therefore the lower bound~\eqref{eq:hessian-almost-done} becomes
\begin{equation*}
  \innerp{H^{-1/2} \emp{H} (\theta) H^{-1/2} v}{v}
  \geq 0.001
  \, ,
\end{equation*}
which concludes the proof.

\subsection{Technical lemmas for the proof of Theorem~\ref{thm:hessian-gaussian}%
}
\label{sec:techn-lemm-hessian-gaussian}

This section gathers technical tools used in the proof of Theorem~\ref{thm:hessian-gaussian} in Section~\ref{sec:hessian-gaussian}.

The first lemma proves a lower bound on the expectation of $Z(\theta,v)$.
\begin{lemma}
\label{lem:truncated-hessian}
For any $\theta\in \R^d$, let
\begin{equation*}
	\wt H(\theta) = \E\left[ \ind{|\langle \theta, X \rangle | \leq 1 ; \,  \|X\| \leq 2\sqrt{d}} XX^{\top} \right].
\end{equation*}
For any $\theta \in \R^d$ such that $\|\theta-\theta^*\|_H\leqslant 1/10\sqrt{\b}$, we have 
\[
\wt H(\theta) \succcurlyeq 0.05 \cdot H_\theta
\]
\end{lemma}
\begin{proof}
Let $u = \theta / \|\theta\|$ and $v\in S^{d-1}$. We want to show that 
\[
\innerp{\wt H(\theta)v}{v}=\E\left[\ind{|\langle \theta, X \rangle | \leq 1 ; \,  \|X\| \leq 2 \sqrt{d}}\innerp{v}{X}^2\right]\geqslant 0.05\innerp{H_\theta v}{v}\, .
\]
write $v=\innerp{v}{u}{u}+(v-\innerp{v}{u}{u})$. As $\innerp{\theta}{X}$ is independent of $\innerp{v-\innerp{v}{u}u}{X}$ and $\innerp{v-\innerp{v}{u}u}{X}\sim \normal(0,1-\innerp{v}{u}^2)$, we have
\begin{multline}\label{eq:dvpAtheta}
\innerp{\wt H(\theta)v}{v}=\innerp{u}{v}^2 \E \big[ \1 \big\{ |\langle \theta, X \rangle | \leq 1 ; \,  \|X\| \leq 2\sqrt{d} \big\} \innerp{u}{X}^2 \big] \\
+(1-\innerp{u}{v}^2) \Expect[\big]{\bm 1 \braces[\big]{|\langle \theta, X \rangle | \leq 1 ; \,  \|X\| \leq 2\sqrt{d}}} \, .
\end{multline}
It remains to bound from below both expectations in the right-hand side term.
Let us start with the second one, we have $\|X\|^2=\innerp{u}{X}^2+\|X-\innerp{u}{X}u\|^2$, where $X-\innerp{u}{X}u$ is a Gaussian vector independent from $\innerp{u}{X}$.
Thus, if $\|X-\innerp{u}{X}u\|^2\leqslant 4d-1/\|\theta\|^2$ and $|\innerp{u}{X}|\leqslant 1/\|\theta\|$, then $|\langle \theta, X \rangle | \leq 1 ; \,  \|X\| \leq 2 \sqrt{d}$, so
\begin{align*}
  \E \Big[\ind{|\langle \theta, X \rangle | \leq 1 ; \,  \|X\| \leq 2\sqrt{d}}\Big]
  &\geqslant \P(|\innerp{u}{X}|\leqslant 1/\|\theta\|)\P(\|X-\innerp{u}{X} u\|^2\leqslant 4 d-1/\|\theta\|^2)\, .
\end{align*}
By Lemma~\ref{lem:ellipsoids}, since $\|\theta-\theta^*\|_H\leqslant 1/10\sqrt{\b}$ one has
\begin{equation}\label{eq:normtheta}
  \frac1{2}\leqslant 0.9\cdot \b \leqslant \|\theta\|\leqslant 1.1\cdot \b \enspace.
\end{equation}
By Markov's inequality, we therefore have
\[
\P(\|X-\innerp{u}{X} u\|^2\leqslant 4 d-1/\|\theta\|^2)\geqslant 1-\frac{d-1}{4d-1/\|\theta\|^2}\geqslant \frac{3}{4}\, .
\]
As $x\mapsto x\exp(-x^2)$ is non-decreasing on $[0,1/2]$,
\[
\P(|\innerp{u}{X}|\leqslant 1/\|\theta\|)\geqslant \frac2{\|\theta\|}\frac{\exp(-\|\theta\|^2/2)}{\sqrt{2\pi}}\geqslant \frac{0.07}{\b}\, .
\]
Thus
\begin{align*}
  \Expect[\Big]{\ind{|\langle \theta, X \rangle | \leq 1 ; \,  \|X\| \leqslant 2\sqrt{d}}}
  &\geqslant \frac{0.05}{\b}\, .
\end{align*}
We use the same arguments to bound the first expectation in the right hand side of \eqref{eq:dvpAtheta}, we get
\begin{align*}
	\E\big[\1\big(|\langle \theta, X \rangle | \leq 1 ; \,
		\|X\| \leq 2 \sqrt{d}\big) \innerp{u}{X}^2\big]
	\geqslant \frac34\E\big[\ind{|\innerp{u}{X}|
		\leqslant 1/\|\theta\|}\innerp{u}{X}^2\big]\, .
\end{align*}
We have
\begin{align*}
  \E[\ind{|\innerp{u}{X}|\leqslant 1/\|\theta\|}\innerp{u}{X}^2]&
  \geqslant \frac{\exp(-1/2\|\theta\|^2)}{\sqrt{2\pi}}\int_{-1/\|\theta\|}^{1/\|\theta\|}x^2\di x\\
  &=\sqrt{\frac{2}{\pi}}\frac{\exp(-1/2\|\theta\|^2)}{3\|\theta\|^3}\geqslant \frac{0.18}{\b^3}\, .
\end{align*}
Thus
\begin{align*}
    \E\big[\1\big(|\langle \theta, X \rangle | \leq 1 ; \,
		\|X\| \leq 2 \sqrt{d}\big) \innerp{u}{X}^2\big]
	\geqslant \frac{0.1}{\b^3}\, .
\end{align*}
Plugging these bounds into \eqref{eq:dvpAtheta} yields
\[
\innerp{\wt H(\theta)v}{v}\geqslant 0.05\bigg(\frac{2\innerp{u}{v}^2}{\b^3}+\frac{(1-\innerp{u}{v}^2)}{\b}\bigg)\geqslant 0.05\innerp{H_\theta v}{v}\, .
\qedhere
\]
\end{proof}

\begin{lemma}
\label{lem:hess-proxy-regularity}
Let $r\in[0,1/10]$. 
For %
every $\theta \in \R^d$ such that $\norm{\theta - \theta^*}_H \leq r/\sqrt{\b}$, we have 
\begin{equation*}
(1-2.35r)H\mleq  H_\theta\mleq (1+2.35r) H\, .
\end{equation*}
\end{lemma}

\begin{proof}
Let $v\in S^{d-1}$, $u=\theta/\|\theta\|$, $u^*=\theta^*/\|\theta^*\|$, we want to compare
\[
\innerp{Hv}{v}=\frac1{\b^3}\innerp{u^*}{v}^2+\frac1\b (1-\innerp{u^*}{v}^2),\qquad 
\innerp{H_\theta v}{v}=\frac1{\b^3}\innerp{u}{v}^2+\frac1{\b}(1-\innerp{u}{v}^2)\enspace.
\]
We have 
\begin{gather*}
|\innerp{v}{u}|\leqslant |\innerp{v}{u^*}|+\|u-\innerp{u}{u^*}u^*\|\|v-\innerp{v}{u^*}u^*\|\enspace,\\
v-\innerp{u}{v}u=(v-\innerp{u^*}{v}u^*)-\innerp{v-\innerp{u^*}{v}u^*}{u}u+\innerp{u^*}{v}(u^*-\innerp{u}{u^*}u)\enspace.
\end{gather*}
By Lemma~\ref{lem:ellipsoids}, $\|u-\innerp{u}{u^*}u^*\|\leqslant \frac{r}{(1-r)\b}$.
Using Cauchy-Schwarz inequality and $(a+b)^2\leqslant (1+r)a^2+(1+r^{-1})b^2$, we deduce
\begin{gather*}
\innerp{u}{v}^2
\leqslant (1+r)\bigg(\innerp{v}{u^*}^2+\frac{r}{(1-r)^2\b^2}(1-\innerp{v}{u^*}^2)\bigg)\enspace,\\   
1-\innerp{u}{v}^2\leqslant (1+r)\bigg((1-\innerp{v}{u^*}^2)+\frac{r}{(1-r)^2\b^2}\innerp{v}{u^*}^2\bigg)\enspace. 
\end{gather*}
Hence,
\begin{align*}
    \innerp{H_\theta v}{v}&\leqslant(1+r)\bigg(\frac{1+r(1-r)^{-2}}{\b^3}\innerp{u}{v}^2+\frac{1+r(1-r)^{-2}\b^{-4}}{\b}(1-\innerp{u}{v}^2)\bigg)\\
    &\leqslant (1+r)(1+r(1-r)^{-2})\innerp{Hv}{v}\leqslant (1+2.35r)\innerp{Hv}{v}\enspace,
\end{align*}
where the last inequality holds as $r\leqslant 0.1$.
The lower bound is obtained using similar arguments.
\end{proof}

\begin{lemma}
\label{lem:conditional-KL}
Let $\mathsf{P}, \mathsf{Q}$ be probability measures and $A$ an event
such that $\mathsf{P}(A) > 0$. One has
\begin{equation*}
	D(\mathsf{P}_{|A} \Vert \mathsf{Q}_{|A})
		\leq \frac{1}{\mathsf{P}(A)} D(\mathsf{P} \Vert \mathsf{Q}) \, .
\end{equation*}
\end{lemma}
\begin{proof}
Without loss of generality, let us assume that $\mathsf{P}$ and $\mathsf{Q}$
have densities $p$ and $q$ respectively, with respect to a common dominating
measure $\mu$ (\eg $\mathsf{P}+\mathsf{Q}$). Let also $p_{| A}$ and $q_{| A}$
denote their conditional densities. One has
\begin{equation}
\label{eq:kl-disjonction}
	 D(\mathsf{P} \Vert \mathsf{Q}) = \int_{A} p \log\Big(\frac{p}{q}\Big) \di \mu 
			+ \int_{A^{c}} p \log\Big(\frac{p}{q}\Big) \di \mu \, .
\end{equation}
By symmetry we do the computations on the event $A$.
\begin{align*}
	\int_{A} p \log\Big(\frac{p}{q}\Big) \di \mu
		& = \mathsf{P}(A) \int_{A} \frac{p}{\mathsf{P}(A)}
			\log\left( \frac{p/\mathsf{P}(A)}{q/\mathsf{Q}(A)}
				\cdot\frac{\mathsf{P}(A)}{\mathsf{Q}(A)} \right) \\
		&=\mathsf{P}(A) \int_{A} p_{| A} \log\left(\frac{p_{| A}}{q_{| A}}\right)
			+ \mathsf{P}(A) \log\left(\frac{\mathsf{P}(A)}{\mathsf{Q}(A)}\right)\\
		&=\mathsf{P}(A) D(\mathsf{P}_{| A} \Vert \mathsf{Q}_{| A})
			+ \mathsf{P}(A) \log\left(\frac{\mathsf{P}(A)}{\mathsf{Q}(A)}\right) \, .
\end{align*}
Hence, by symmetry,
\begin{equation*}
	D(\mathsf{P} \Vert \mathsf{Q}) = \mathsf{P}(A) D(\mathsf{P}_{| A} \Vert \mathsf{Q}_{| A})
		+ \mathsf{P}(A^{c}) D(\mathsf{P}_{| A^{c}} \Vert \mathsf{Q}_{| A^{c}})
		+ D(\mathsf{P}(A)\Vert \mathsf{Q}(A)) \, ,
\end{equation*}
where $D(\mathsf{P}(A)\Vert \mathsf{Q}(A))$ denotes the divergence between Bernoulli distributions with parameters $\mathsf{P}(A),\mathsf{Q}(A)$.
The last two terms being non-negative, the claim is proved. 
\end{proof}

Finally, the following lemma bounds the remainder term $R_n$ from Lemma~\ref{lem:LBLinTGauss}, which arises from smoothing over the direction $v \in S^{d-1}$.

\begin{lemma}\label{lem:PB2}
  If $n\geqslant 1.1 \b (d+t)$, then with probability at least $1-e^{-t}$, one has
  \begin{equation*}
    R_n
    = \sup_{u \in \mc(u^*,1/10\b)}
    \frac{1}{n}  \sum_{i=1}^{n}\exp\big( - 0.49\,\b \ainnerp{u}{X_i} \big) \ind{\|X_{i}\|\leq 2 \sqrt{d}}
    \leqslant \frac{4}{\b}\, .
  \end{equation*}
\end{lemma}
\begin{proof}
  We define the process $W = (W (u))_{u \in S^{d-1}}$ by (denoting $b = 0.49 \b$)
  \begin{equation*}
    W (u)
    = \exp \parens[\big]{- b \ainnerp{u}{X}} \ind{\norm{X} \leq 2 \sqrt{d}}
    \, ,
  \end{equation*}
  and likewise define $W_i (u)$ for $i=1, \dots, n$ by replacing $X$ by $X_i$ above.
  We need to bound
  \begin{equation*}
    \underset{u \in \mc(u^*,1/10\b)}{\sup}\frac{1}{n}  \sum_{i=1}^{n}W_i(u)    
    \, ,
  \end{equation*}
  which we achieve by applying the PAC-Bayes inequality of Lemma~\ref{lem:pac-bayes-sum}, with $\lambda = 1$.

  We apply this inequality with the collection of posteriors $(\rho_u)_{u\in \mc(u^*,1/10\b)}$ and the prior $\pi$ defined as follows:
For every $u\in \mc(u^*,1/10\b)$, $\rho_u$ is the uniform distribution over $\mc(u,1/10\b)$ and $\pi$ is the uniform distribution over $\mc(u^*,\sqrt{2}/5\b)$, chosen such that, for any $u\in \mc(u^*,1/10\b)$, the support of $\rho_u$ is included into the one of $\pi$.

\parag{Bound on the relative entropy}
We prove in this paragraph that 
\begin{equation}
  \label{eq:kl-spherical}
    D(\rho_{u} \Vert \pi)
		\leqslant 1.1 \cdot d\enspace.
\end{equation}
We have directly:
\begin{equation*}
	D(\rho_{u} \Vert \pi)
		=\log\Big(\frac{\vol_{d-1}(\mc(u^{*},\sqrt{2}/5\b))}{\vol_{d-1}(\mc(u,1/10\b))} \Big) \, ,
\end{equation*}
where $\vol_{d-1}$ denote the surface measure on $S^{d-1}$.
To compute these volumes, we let $U$ denote a random variable uniformly distributed on the sphere and $u\in S^{d-1}$.
It is %
a standard fact
that $\innerp{u}{U}$ has density given by
\[
f(s) = c_{d} \left(1-s^{2}\right)^{\frac{d-3}{2}}\indic{-1\leq s \leq 1} \, ,
\]
where $c_d$ is a normalizing constant. 
Therefore, for any $\eps\in (0,1)$,
\[
  \vol_{d-1}(\mc(u,\eps))
  = \P\big(\innerp{U}{u}>\sqrt{1-\eps^2}\big)
  = c_d\int_{\sqrt{1-\eps^2}}^1\left(1-s^{2}\right)^{\frac{d-3}{2}}\di s
  = \int_0^{\eps^2} \frac{t^{(d-3)/2}}{\sqrt{1-t}}\di t\, .
\]
Hence, 
\[
\frac{2c_d\eps^{d-1}}{d-1}\leqslant \vol_{d-1}(\mc(u,\eps))\leqslant \frac{1}{\sqrt{1-\eps^2}}\frac{2c_d\eps^{d-1}}{d-1}\, .
\]
Therefore,
\[
D(\rho_{u} \Vert \pi)
		\leqslant (d-1)\log\big(2\sqrt{2}\big)+\frac{1}2\log\bigg(\frac1{1-\sqrt{2}/5\b} \bigg)\, .
\]
As %
$\frac{1}2\log\big(\frac1{1-\sqrt{2}/5\b} \big)\leqslant \log\big(2\sqrt{2}\big)$, further bounding numerical constants gives the bound~\eqref{eq:kl-sphere}.

\parag{Bound on the Laplace transform}
In this paragraph, we prove that
\begin{equation}\label{eq:UBLTPB2}
  \log\E[\exp (W (u))]
  \leq \frac{3}{\b}
  \enspace.
\end{equation}
Indeed, since $0 \leq W(u) \leq \exp(-b |\innerp{u}{X}|)\leqslant 1$, one has (using that the density of $\innerp{u}{X}$ is bounded by $1/\sqrt{2\pi}$)
\begin{align*}
  \log \Expect{\exp (W (u))}
  &\leq \log \braces[\big]{1 + (e-1) \, \Expect{W (u)}}
    \leq (e - 1) \, \Expect{W (u)} \\
  &\leq (e-1) \Expect{\exp (-b \ainnerp{u}{X})}
    \leq \frac{e-1}{\sqrt{2\pi}} \int_\R e^{-b |s|} \di s
    = \frac{(e-1) \sqrt{2/\pi}}{b}
    \, ,
\end{align*}
which proves~\eqref{eq:UBLTPB2} after substituting $b = 0.49 \b$ and bounding the numerical constant.

\parag{Control of the smoothed process}
In this paragraph, we show that
\begin{equation}\label{eq:LBLinTPB2}
  \int_{\mc(u,1/10\b)} W_i (u') \rho_{u}(\di u')
  \geq 0.82 \cdot W_i (u)
  \, .
\end{equation}
We have by Jensen's inequality,
\begin{align*}
  \int_{\mc(u,1/10\b)} W_i&(u') \rho_{u}(\di u') = \ind{\|X_i\|\leqslant 2\sqrt{d}}\int_{\mc(u,1/10\b)}\exp(-b |\innerp{u'}{X_i}|)\rho_u(\di u')\\
&\geqslant \ind{\|X_i\|\leqslant 2\sqrt{d}}\exp\bigg(-b \int_{\mc(u,1/10\b)} \ainnerp{u'}{X_i} \rho_u(\di u')\bigg) \\
&\geqslant \ind{\|X_i\|\leqslant 2\sqrt{d}}\exp\bigg(-b \bigg(\int_{\mc(u,1/10\b)}\innerp{u'}{X_i}^2\rho_u(\di u')\bigg)^{1/2}\bigg)\enspace.
\end{align*}
On the other hand, using \cite[Eq~(42) and (43)]{mourtada2022linear} and Fact~\ref{fact:angle-radius-correspondance} for integration over spherical caps, one has under the event $\set{\norm{X_i} \leq 2 \sqrt{d}}$ that
\begin{align*}
  \int_{\mc(u,1/10\b)} \innerp{u'}{X_i}^2 \rho_u (\di u')
  \leq \innerp{u}{X_i}^{2} + \frac{2}{100(d-1)\b^2} \|X_i\|^2
  \leq \innerp{u}{X_i}^{2} + \frac{0.16}{\b^2}
  \, .
\end{align*}
Combining the previous two inequalities gives
\begin{align*}
  \int_{\mc(u,1/10\b)} W_i(u') \rho_{u}(\di u')
  \geq \ind{\norm{X_i} \leq 2 \sqrt{d}} \exp \bigg( -0.49 \, \b \bigg[ \ainnerp{u}{X_i} + \frac{0.4}{\b} \bigg] \bigg)
  \geq 0.82 \, W_i (u)
  \, .
\end{align*}

\parag{Conclusion of the proof}
By the PAC-Bayes inequality (Lemma~\ref{lem:pac-bayes-sum}), for any $t \geq 0$, we have with probability at least $1-e^{-t}$, simultaneously for all $u\in\mc(u^*,1/10\b)$,
\begin{align*}
  \frac{0.82}{n} \sum_{i=1}^{n} W_i (u)
  &\leq \frac{1}{n} \sum_{i=1}^n \int_{\mc (u, 1/10 \b)} W_i (u') \rho_u (\di u')
    \leq \frac{3}{\b} + \frac{1.1 d + t}{n}
    \leq \frac{4}{\b}
    \, ,
\end{align*}
where the last inequality comes from the condition on $n$.
\end{proof}

\section{Linear separation: Proof of Theorem~\ref{thm:strong-non-existence}%
}
\label{sec:linear-separation}

In this section, we provide the proof of Theorem~\ref{thm:strong-non-existence}  on linear separation for small sample sizes.

Throughout this section, we assume that the design is isotropic Gaussian and that the model is well-specified.
Specifically, given a dimension $d \geq 1$, a parameter $\theta^* \in \R^d$ with norm $b = \norm{\theta^*}$ and a sample size $n \geq d$, the dataset consists of $n$ \iid random pairs $(X_1, Y_1), \dots, (X_n, Y_n)$ with $X_i \sim \gaussdist (0, I_d)$ and $\P (Y_i = 1 |X_i) = \sigma (\innerp{\theta^*}{X_i})$.
Note that if $n \geq d$, then almost surely $X_1, \dots, X_n$ span $\R^d$, hence (by the discussion in the introduction) the MLE exists if and only if the dataset is not linearly separated.
In addition, by rotation invariance of the standard Gaussian distribution in $\R^d$, the probability of linear separation (non-existence of the MLE) only depends on $\theta^*$ through its norm $b$.

We start by defining the function $h : \R \to [0, 1/2]$ which appears in the asymptotic phase transition~\eqref{eq:phase-transition-mle}.
For $b \in \R^+$, let $(X',Y_b')$ be a random pair in $\R \times \{ -1, 1 \}$, with $X' \sim \gaussdist (0, 1)$ and $\P (Y_{b}' = 1 |X') = \sigma (b X')$, and let $V_b = Y_b' X'$.
In addition, let $Z \sim \gaussdist (0, 1)$ be independent of $V_b$.
Then, 
\begin{equation}
  \label{eq:boundary-function}
  h (b)
  = \min_{\lambda \in \R} \E \big[ (\lambda V_b - Z)_+^2 \big]
  \, .
\end{equation}

\subsection{Proof of Theorem~\ref{thm:strong-non-existence}}
\label{sec:proof-linear-separation}
The proof of Theorem~\ref{thm:strong-non-existence} relies on the approximate kinematic formula from conic geometry recalled hereafter.
First, define the \emph{statistical dimension}~\cite{amelunxen2014living} of a convex cone $\Lambda \subset \R^n$ as 
$\delta(\Lambda) = \E [\| \Pi_{\Lambda} \Z \|^{2}]$ where $\Pi_{\Lambda}$ is the Euclidean projection on $\Lambda$ and $\Z \sim \normal(0, I_{n})$.
\begin{lemma}
[Approximate kinematic formula, Theorem~7.1 in \cite{amelunxen2014living}]
\label{lem:bernstein-kinematic}
Let $\L_k$ be a random subspace of $\R^{n}$ drawn uniformly (that is, under the rotation-invariant probability measure) over the set of subspaces of dimension $k$ and let $\Lambda \subset \R^{n}$ be a convex cone.
For all $t \geq 0$,
if
\begin{equation}
\label{eq:codim-stat-dim}
	n-k \leq \delta(\Lambda) - t \, ,
\end{equation}
then
\begin{equation*}
	\P\left( \Lambda \cap \L_k \neq \{0\} \right)
		\geq 1 - 4\exp\left(-\frac{t^{2} /8}{\min \{\delta (\Lambda), n-\delta (\Lambda)\} + t} \right)  \, .
\end{equation*}
\end{lemma}

We can now proceed with the proof of Theorem~\ref{thm:strong-non-existence}.
First, if
$n \leq d$,
a simple induction shows that almost surely, the points $X_1, \dots, X_n$ are linearly independent in $\R^d$.
Hence, there exists $\theta \in \R^d$ such that for $i=1, \dots, n$, one has $\innerp{\theta}{X_i} = Y_i$ and thus $Y_i \innerp{\theta}{X_i} = Y_i^2 = 1 > 0$.
Thus $\inf_{\theta' \in \R^d} \wh L_n (\theta') = 0$, but $\wh L_n > 0$ on $\R^d$ and thus $\wh L_n$ admits no global minimizer in $\R^d$.
We thus assume from now on that $n > d$.
Since $n \leq B d /200$, this implies that $\norm{\theta^*} = B > e$.

First, following~\cite{candes2020phase}, we express the probability that the MLE does not exist as the probability that some random cone non-trivially intersects a random subspace of dimension $d-1$ in $\R^{n}$.
Let $\{(X_{i}, Y_{i})\}_{1\leq i\leq n}$ denote the dataset, where all the $X_{i}$'s are independently drawn from $\normal(0, I_{d})$.
Using the rotational invariance of the standard Gaussian distribution,  we can assume without loss of generality that for every $i\in \{1, \dots, n\}$, $\P(Y_{i} = 1 \cond X_{i}) = \sigma(\b X_{i}^{1})$ where $X_{i}^{j}$ denotes the $j$-th coordinate of $X_{i}$ for every $j\in \{1, \dots, d\}$. Below we let $U_{i} = X_{i}^{1}$ and $V_{i} = Y_{i}U_{i}$ for all $i\in \{1, \dots, n\}$. Let also $\V = (V_{1}, \dots, V_{n})\in \R^{n}$ and $\Lambda = \R \V + \R_{+}^{n}$, which is a random cone in $\R^{n}$.
The proof of Theorem~\ref{thm:strong-non-existence} relies on the following observation.
\begin{lemma}
\label{lem:cones-intersection}
Let $\L_{d-1}$ be a random subspace drawn uniformly over subspaces of dimension $d-1$ in $\R^{n}$.
Then 
\begin{equation*}
	\P(\text{\emph{MLE does not exist}}) \geq \P(\Lambda \cap \L_{d-1} \neq \{0\}) \, .
\end{equation*}
\end{lemma}
The proof of this result is postponed to the end of the section and is a straightforward adaptation of \cite[Proposition~2]{candes2020phase} to the case where the model does not include an intercept. 

In view of this characterization, we want to apply Lemma~\ref{lem:bernstein-kinematic} to the cone $\Lambda$ but we cannot do it in a straightforward way, as this cone is random. We therefore show that the sufficient condition \eqref{eq:codim-stat-dim} regarding the statistical dimension of $\Lambda$ is satisfied with high probability.
Hereafter we denote by $E = \{\Lambda \cap \L_{d-1} \neq \{0\}\}$,
and for every $t\geq 0$, we define the event %
\begin{equation}
  A_{t}
  = \{ n-d+1 \leq \delta(\Lambda) - t \}
  = \set{n - \delta (\Lambda) \leq d - 1 - t}
  \, .  
\end{equation}
Our main task in this proof is to show that
\begin{equation}
\label{eq:stat-dim-whp}
	\P(A_{\alpha d})
		\geq 1-\exp\big(-\max\big\{\kappa\sqrt{d}, \kappa^{2}d/\b^{2} \big\}\big)
			-2e^{-\tau d}
\end{equation}
for some $\tau\in (0,1/2)$ and $\alpha \in (1/2, 1)$.
We now establish~\eqref{eq:stat-dim-whp} with explicit constants. Conditionning on $\V$, one has $\delta(\Lambda)
= \E [ \norm{Z}^2 - \dist (\Z, \Lambda)^{2} \cond \V ]
= n - \E%
\big[\dist (\Z, \Lambda)^{2} \cond \V\big]$.
We define
\begin{equation*}
  F(\V)
  = n - \delta (\Lambda)
  = \E_{\Z} \big[ \dist (\Z, \Lambda)^{2} \cond \V \big]
  = \E_{\Z} \bigg[ \underset{\lambda \in \R}{\min} \, 
  \sum_{i=1}^{n}(\lambda V_{i} - Z_{i})_{+}^{2} \Big| \V \bigg] \, .
\end{equation*}
This way, we will prove \eqref{eq:stat-dim-whp} by showing that $F(\V) \ll d$ with high probability.
It is reasonable to believe that this is true in the regime of interest where
$n \ll \b d$.
Indeed, we note that
\begin{equation*}
  \E [F(\V)]
  \leq \underset{\lambda \in \R}{\min} \, \E\bigg[\sum_{i=1}^{n}
  (\lambda V_{i} - Z_{i})_{+}^{2} \bigg]
  =nh(\b) \, ,
\end{equation*}
where $h$ is the phase transition function~\eqref{eq:boundary-function}.
In addition, one can show that $h (\b) \lesssim 1/\b$ (we will not need this exact claim, hence we will not prove it, although it could be deduced from the analysis below).
Thus $\E [F (\bm V)] \lesssim n/\b \ll d$.

Hereafter, we let $\psi:s\in \R \mapsto \E\big[(s - Z)_{+}^{2}\big]$ with $Z\sim  \normal(0,1)$ and start by bounding
\begin{equation*}
	F(\V)
        = \E\bigg[ \underset{\lambda\in \R}{\min} \, 
			\sum_{i=1}^{n}(\lambda V_{i} - Z_{i})_{+}^{2} \Big| \V \bigg]
		\leq \underset{\lambda \in \R}{\min} \, \sum_{i=1}^{n}
			\E \big[(\lambda V_{i} - Z_{i})_{+}^{2} | V_{i} \big] 
		= \underset{\lambda \in \R}{\min} \, \sum_{i=1}^{n} \psi(\lambda V_{i})\, .
\end{equation*}
By Fact~\ref{fact:psi-estimates} below, for all $\lambda \geq 0$, recalling that $V_i = Y_i U_i$,
\begin{equation}
\label{eq:bound-function-psi-3terms}
	\psi(- \lambda V_{i})
		\leq e^{-\lambda^{2} U_{i}^{2}/2}
		+ \indic{Y_{i} U_{i} \leq 0}
		+ \lambda^2 U_{i}^{2} \indic{Y_{i} U_{i} \leq 0} \, .
\end{equation}
We thus define for all $i\in \{1, \dots, n\}$ and $\lambda \in \R$ the variables
\begin{equation}
\label{eq:3-aux-vars}
	\zeta_{i,\lambda} = e^{-\lambda^{2}U_{i}^{2}/2}\, , 
	\quad \eps_{i} = \indic{Y_{i}U_{i}\leq 0}\, ,
	\quad \psi_{i} = U_{i}^{2}\indic{Y_{i}U_{i}\leq 0} \, ,
\end{equation}
so that we can further bound
\begin{equation}
\label{eq:min-3-sums}
	F(\V)
	\leq \underset{\lambda > 0}{\min} \bigg\{ \sum_{i=1}^{n} \zeta_{i, \lambda} 
	+ \sum_{i=1}^{n} \eps_{i} + \lambda^{2}  \sum_{i=1}^{n} \psi_{i} \bigg\} \, .
\end{equation}
We now separately bound from above the three sums and then optimize the resulting bound over $\lambda$.
We use Bernstein's inequality
to bound the first two sums involving the $\zeta_{i,\lambda}$ and $\eps_{i}$, but bounding the sum of the $\psi_{i}$'s is a more subtle task, for which we resort to Lata{\l}a's bound on the moments of sums of independent variables \cite{latala1997estimation} to control the moments of $\sum_{i=1}^{n} \psi_{i}$. Let us start with the first sum. For every $i \in \{1, \dots, n\}$, every $\lambda > 0$ and $k \in \{1, 2\}$,
\begin{align*}
	\E [\zeta_{i,\lambda}^{k}]
	&=\E \Big[ \exp\Big(-\frac{k\lambda^{2}U_{i}^{2}}{2}\Big)\Big]
	= \int_{\R}\frac{e^{-(k\lambda^{2}+1)u^{2}/2}}{\sqrt{2\pi}} \di u
	= \frac{1}{\sqrt{k\lambda^{2}+1}} 
	\leq \frac{1}{\lambda} \, .
\end{align*}
Since in addition $\zeta_{i,\lambda} \leq 1$ almost surely, by Lemma~\ref{lem:bennett-mgf} and the second and third points of Lemma~\ref{lem:sub-gamma-exponential}, for all $t\geq 0$, with probability larger than $1-e^{-t}$,
\begin{equation}
\label{eq:bound-zeta}
	\sum_{i=1}^{n} \zeta_{i, \lambda} 
	\leq \frac n \lambda +  \sqrt{\frac{2nt}{\lambda}} + \frac t 3 \, .
\end{equation}
Regarding the second sum, inequality~\eqref{eq:MomGaussu*} shows that for every $i$, $\E [\eps_{i}] = \E[\exp(-\b|U_{i}|)] \leq \b^{-1}$. Since $\eps_{i}^{2} = \eps_{i}$ and $\eps_{i}\leq 1$, the same argument as before shows that for all $t\geq 0$, it holds with probability larger than $1-e^{-t}$ that
\begin{equation}
\label{eq:bound-eps}
	\sum_{i=1}^{n} \eps_{i} \leq \frac n \b + \sqrt{\frac{2nt}{\b}} + \frac t 3 \, .
\end{equation}
Finally, we turn to the control of the last sum, for which we use Lata{\l}a's bound, recalled hereafter.
\begin{lemma}[\cite{latala1997estimation}, Corollary~1]
\label{lem:latala}
  Let $\xi, \xi_{1}, \dots, \xi_{n}$ be \iid nonnegative random variables.
  Then for any $p \geq 1$,
\begin{equation*}
	\bigg\|\sum_{i=1}^{n}\xi_{i} \bigg\|_{p} 
	\leq 2e^{2} \sup\Big\{ \frac{p}{s}\Big(\frac{n}{p}\Big)^{1/s}  \|\xi\|_{s} :  
	1 \vee \frac{p}{n} \leq s \leq p \Big\} \, . 
\end{equation*}
\end{lemma}
From now on, we let $S_{n} = \psi_{1}+\cdots+\psi_{n} $ and $p\in [1,n]$. By Markov's inequality, $\P(S_{n} \leq e\|S_{n}\|_{p}) \geq 1-e^{-p}$, hence we want to bound $\|S_{n}\|_{p}$ from above by some factor of $d$, with $p$ as large as possible. 
We are thus led to bound the individual $L^{s}$ norms $\|\psi_{i}\|_{s}$ and then optimize over $s \in [1, p]$.
\parag{Bound on individual moments}
Regarding the bound on $\|\psi_{i}\|_{s}$, the result is
obtained by taking advantage of either the fact that $U_{i}^{2}$ is sub-exponential (by neglecting the indicator) or by conditioning on $U_{i}$,
which allows to use an exponential moment inequality. Let us formalize this. Let $U,Y,\psi$ denote random variables having the same distribution as $U_{i},Y_{i},\psi_{i}$. On the one hand, $\psi \leq U^{2}$, so for all $s\geq 1$,
\begin{equation*}
	\E [\psi^{s}] \leq \E [|U|^{2s}] = \frac{2^{s}}{\sqrt{2\pi}}\Gamma\Big(s+\frac 1 2 \Big) \, .
\end{equation*}
Hence, using \cite[Eq.~(5.6.1)]{olver2010handbook} and simplifying we obtain
\begin{equation}
\label{eq:bound-simple-sub-exp}
	\|\psi\|_{s} = \E[\psi^{s}]^{1/s} \leq (3/e) s \, .
\end{equation}
On the other hand, we use the fact that conditionally on $U$, $\{YU \leq 0\}$ happens with exponentially small probability. More precisely, we write
\begin{align*}
  \E[\psi^{s}]
  &= \E \big[ |U|^{2s} \E[\indic{YU \leq 0} \cond U]\big]
  = \E \big[ |U|^{2s} \sigma(-\b |U|) \big] \\
  &\leq \E \big[ |U|^{2s} \exp(-\b |U|) \big]
    \leq \sqrt{\frac 2 \pi} \frac{\Gamma(2s+1)}{\b^{2s+1}}
    \, .
\end{align*}
We then bound in a similar way $\Gamma(2s+1)^{1/s}$ and thus, combining
the previous two bounds
we deduce that
\begin{equation}
\label{eq:moments-psi}
	\|\psi\|_{s} \leq \frac{9}{e^{2}} \min\Big\{ \frac{s^{2}}{\b^{2}\b^{1/s}} , s \Big\} \, .
\end{equation}	

\parag{Upper bound on the supremum}
Using the control on the moments of the $\psi_{i}$'s~\eqref{eq:moments-psi}, it follows from Lata{\l}a's inequality (Lemma~\ref{lem:latala}) that
\begin{equation}
  \label{eq:latala-brut}
  \|S_{n}\|_{p} 
  \leq \frac{18}{\b^{2}} \sup_{1 \leq s \leq p} \min \braces[\Big]{p s \Big(\frac{n}{p\b}\Big)^{1/s}, \, 
  \b^{2}p\Big(\frac{n}{p}\Big)^{1/s}} \, .
\end{equation}
We now proceed with a bound on the supremum in the right-hand side
by a function of $p,n$ and $\b$. For this technical step, we define for every $s>0$
\begin{equation*}
	G(s) = \min\Big\{p s \Big(\frac{n}{p\b}\Big)^{1/s}, \, 
	\b^{2}p\Big(\frac{n}{p}\Big)^{1/s} \Big\}\, , \quad 
	G_{1}(s) = p s \Big(\frac{n}{p\b}\Big)^{1/s}, \quad
	G_{2}(s) = \b^{2}p \Big(\frac{n}{p}\Big)^{1/s},
\end{equation*}
and then bound $M(p) = \sup_{1\leq s\leq p} G(s)$ for every $p\geq 1$. We first note that $G_{2}$ decreases on $(0, +\infty)$ and that $G_{1}(s) \leq G_{2}(s)$ for every  $s\leq \b^{2}$. Let also
\begin{equation*}
	g(s) = \log\Big(s \Big(\frac{n}{p\b}\Big)^{1/s}\Big) 
	= \frac{1}{s} \log\Big(\frac{n}{p\b}\Big) + \log s \, .
\end{equation*}
Then
\begin{equation*}
	g'(s) = \frac{1}{s} \Big(1-\frac{1}{s} \log\Big(\frac{n}{p\b}\Big)\Big) \, .
\end{equation*}
Now let $s_{1} = \log(\frac{n}{p\b})$. We first deal with the case where $s_{1} \geq \min\{p, \b^{2} \}$. In this configuration, $G_{1}$ decreases on $[1,s_{1}]$, hence $G$ decreases on $[1,p]$ and therefore $M(p) = G (1) \leq G_{1}(1) = n/\b$. From now on we assume that $s_{1} < \min\{p, \b^{2}\}$. 

If $s_{1}\leq 1$, $g'(s) > 0$ for every $s>1$, hence $G_{1}$ increases on $[1, +\infty)$. Since $G_{2}$ decreases on $(0, +\infty)$,
we deduce that the supremum is attained at either the value $s_{\mathrm{c}}$ where $G_{1}$ and $G_{2}$ coincide or at $p$, depending on whether $p$ is smaller or larger than $s_{\mathrm{c}}$, hence $M(p) = \min\{G_{1}(p), G(s_{\mathrm{c}}) \}$. In addition $s_{\mathrm{c}}$ is solution to $s=\b^{2+1/s}$, therefore $s_c \geq \b^2$.
Consequently $G(s_{\mathrm{c}}) = G_{2}(s_{\mathrm{c}}) \leq G_{2}(\b^{2})$. Using the fact that $p^{1/p} \geq 1$, we deduce that in this configuration, 
\begin{equation*}
	\|S_{n}\|_{p} \leq \frac{18}{\b^{2}} \min \Big\{ p^{2}\Big(\frac{n}{\b}\Big)^{1/p}, \, 
		\b^{2}p \Big(\frac{n}{p}\Big)^{1/\b^{2}} \Big\} \, .
\end{equation*}

Now, if $s_{1} > 1$, $G$ decreases on $[1, s_{1}]$ and increases on $[s_{1}, \min\{s_{\mathrm{c}}, p\}]$, then decreases again on $[\min\{s_{\mathrm{c}}, p\}, p]$ as it coincides with $G_{2}$ on this last segment. Hence the only difference with the previous case is that the supremum might be attained at $s=1$.

Putting everything together, we conclude that, in all cases,
\begin{equation}
\label{eq:latala-bound-p}
	\|S_{n}\|_{p}
		\leq \frac{18}{\b^{2}} \max\Big\{ \frac n \b , \min \Big\{ p^{2}\Big(\frac{n}{\b}\Big)^{1/p}, \, 
		\b^{2}p \Big(\frac{n}{p}\Big)^{1/\b^{2}} \Big\} \Big\}\, .
\end{equation}

\parag{High-probability upper bound on $\sum_{i=1}^{n}\psi_{i}$} Using the bound on the moments of $S_{n}$ established above, we apply Markov's inequality to derive a high probability bound for $S_{n}$.
To that end, we let $p$ be as large as possible under the constraint that $\|S_{n}\|_{p}$ does not exceed
$O(\kappa^{2}d)$, namely $\|S_{n}\|_{p} \leq L_{0} \kappa^{2}d/\b^{2}$, where $L_{0}$ only depends on $C_{0}$. We prove that this is achieved by taking
\begin{equation}
\label{eq:p-opti}
	p = \max\{\kappa\sqrt{d}, \kappa^{2}d/\b^{2} \} \, .
\end{equation} 
To do so, we first show that if $p=\kappa\sqrt{d}$, then
\begin{equation}
\label{eq:p-bound}
	p^{2}\Big(\frac{n}{\b}\Big)^{1/p} \leq \kappa^{2} d \exp\Big(\frac{2}{e\sqrt{C_{0}}}\Big) \, .
\end{equation}
Using that $n/\b \leq d/(C_{0}\kappa)$,
\begin{equation*}
	p^{2}\Big(\frac{n}{\b}\Big)^{1/p}
		\leq \kappa^{2} d \Big(\frac{d}{C_{0}\kappa}\Big)^{1/(\kappa\sqrt{d})}
		= \kappa^{2} d \exp \bigg( \frac{\log\big(d/(C_{0}\kappa) \big)}{\kappa\sqrt{d}} \bigg) \, .
\end{equation*}
The function $h:t\mapsto\log(t)/\sqrt{t}$ reaches its maximum at $t=e^{2}$ and thus satisfies $h(t) \leq 2/e$ for all $t>0$. Hence 
\begin{equation}
	\frac{\log\big(d/(C_{0}\kappa) \big)}{\kappa\sqrt{d}}
		\leq \frac{2}{e\kappa^{3/2}\sqrt{C_{0}}} \, ,
\end{equation}
from which \eqref{eq:p-bound} follows, since $\kappa \geq 1$. 

Similarly, if $p=\kappa^{2} d/\b^{2}$, using that $n/d\leq \b/(C_{0}\kappa)$, one has
\begin{equation*}
	\b^{2}p \Big(\frac{n}{p}\Big)^{1/\b^{2}}
	\leq \kappa^{2} d \bigg( \frac 1 {C_{0}} \Big(\frac{\b}{\kappa}\Big)^{3} \bigg)^{1/\b^{2}} \, .
\end{equation*}
Using the same argument as before, we find that
\begin{equation*}
	\bigg( \frac 1 {C_{0}} \Big(\frac{\b}{\kappa}\Big)^{3} \bigg)^{1/\b^{2}} 
	\leq \exp\bigg(\frac{3}{2e C_{0}^{2/3}}\bigg) \, .
\end{equation*}
It is clear that $ \exp(3/\big(2e C_{0}^{2/3}\big)) \geq \exp(2/(e\sqrt{C_{0}}))$, hence, with $p = \max\big\{\kappa\sqrt{d}, \kappa^{2}d/\b^{2} \big\}$
\begin{equation*}
	\min \Big\{ p^{2}\Big(\frac{n}{\b}\Big)^{1/p}, \, \b^{2}p \Big(\frac{n}{p}\Big)^{1/\b^{2}} \Big\}
		\leq L_{0}\kappa^{2}d \, , 
		\quad L_{0} = \exp\bigg(\frac{3}{2e C_{0}^{2/3}}\bigg) \, .
\end{equation*}
Plugging this in \eqref{eq:latala-bound-p} and using again the assumption $n/\b\leq d/(C_{0}\kappa)$, we deduce that for $p = \max\{\kappa\sqrt{d}, \kappa^{2}d/\b^{2} \}$,
\begin{equation}
\label{eq:final-bound-moment}
	\|S_{n}\|_{p}
		\leq \frac{18}{\b^{2}} \max\Big\{ \frac{d}{C_{0}\kappa}, L_{0} \kappa^{2} d \Big\}
		= 18 L_{0} \frac{\kappa^{2} d}{\b^{2}} \, .
\end{equation}

\parag{High-probability bound on the statistical dimension}
We now apply Markov's inequality using the moment bound~\eqref{eq:final-bound-moment}. This yields
\begin{equation*}
	\P\Big( S_{n} \leq \frac{18eL_{0}\kappa^{2}d}{\b^{2}} \Big)
		\geq 1 - \exp\big( - \max\big\{\kappa\sqrt{d}, \kappa^{2}d/\b^{2} \big\} \big) \, .
\end{equation*}

Finally, we combine this with \eqref{eq:bound-zeta} and \eqref{eq:bound-eps}, showing that for every $\lambda > 0$ and every $t\geq0$, it holds with probability larger than $1-e^{-\max\{ \kappa\sqrt{d}, \kappa^{2}d/\b^{2}\}} - 2e^{-t}$ that
\begin{equation*}
	\sum_{i=1}^{n} \psi(-\lambda V_{i}) 
	\leq \frac n \b + \sqrt{\frac{2nt}{\b}} 
		+  \frac n \lambda + \sqrt{\frac{2nt}{\lambda}} + \frac{2t}{3}	
		+\frac{L\lambda^{2}\kappa^{2}}{\b^{2}} d \, , \quad \text{where } L=18eL_{0} \, .
\end{equation*}
	We then set $t = \tau d$ for some $\tau \in (0,1)$. As $n \leq \b d/(C_{0}\kappa)$, the above rewrites
\begin{equation*}
	\sum_{i=1}^{n} \psi(-\lambda V_{i}) 
	\leq \frac{d}{C_{0}\kappa} + \sqrt{\frac{2\tau}{C_{0}\kappa}} d  
		+  \frac{\b d}{C_{0} \kappa \lambda} + \sqrt{\frac{2 \b \tau}{C_{0} \kappa \lambda}} d
		+ \frac2 3 \tau d	
		+\frac{L\lambda^{2}\kappa^{2}}{\b^{2}} d \, .
\end{equation*}
Then, by the arithmetic mean--geometric mean inequality, 
\begin{equation}
\label{eq:psi-generic-lambda}
	\sum_{i=1}^{n} \psi(-\lambda V_{i})
		\leq \bigg[ \frac{3}{2C_{0}\kappa} + \frac 5 3 \tau 
			+ \frac{2 \b}{C_{0} \kappa \lambda} 
			+ \frac{L\kappa^{2} \lambda^{2}}{\b^{2}} \bigg] d \, .
\end{equation}
We now optimize the terms depending on $\lambda$ and write, using the arithmetic mean--geometric mean inequality
\begin{equation*}
	 \frac{2\b}{C_{0}\kappa \lambda} + \frac{L\kappa^{2} \lambda^{2}}{\b^{2}}
	 	= \frac{2}{3} \cdot \frac{3\b}{C_{0}\kappa \lambda} 
			+ \frac 1 3 \cdot \frac{3 L\kappa^{2} \lambda^{2}}{\b^{2}} 
		\geq \Big(\frac{3\b}{C_{0}\kappa \lambda} \Big)^{2/3} 
			\Big( \frac{3 L\kappa^{2} \lambda^{2}}{\b^{2}} \Big)^{1/3} \, ,
\end{equation*}
with equality if and only if $\lambda = \lambda^{*}$, the value such that $\frac{3\b}{C_{0}\kappa \lambda} = \frac{3 L\kappa^{2} \lambda^{2}}{\b^{2}}$, that is $\lambda^{*} = \frac{\b}{\kappa (C_{0}L)^{1/3}}$. Simplifying the constants yields
\begin{equation*}
	\underset{\lambda > 0}{\inf} \Big\{ \frac{2\b}{C_{0}\kappa \lambda} 
			+ \frac{L\kappa^{2} \lambda^{2}}{\b^{2}}\Big\} 
		= 3\cdot (18e)^{1/3} \frac{\exp\Big(\frac{1}{2eC_{0}^{2/3}}\Big)}{C_{0}^{2/3}} 
		=: \frac{C_{1}}{C_{0}^{2/3}} \, .
\end{equation*}
We finally plug this in \eqref{eq:psi-generic-lambda} and obtain that
\begin{equation*}
	\P\bigg(\sum_{i=1}^{n} \psi(-\lambda^{*} V_{i})
		\leq \bigg[\frac{5\tau}{3} + \frac{3}{2C_{0}} + \frac{C_1}{C_{0}^{2/3}} \bigg] d \bigg)
	\geq 1-e^{-\max\{ \kappa\sqrt{d}, \kappa^{2}d/\b^{2}\}} - 2e^{-\tau d} \, .
\end{equation*}
We then let $c_{0}(C_{0}) = 3/(2C_{0}) + C_{1}/C_{0}^{2/3}$ so that the last
inequality rewrites
\begin{equation*}
	\P\big( F(\V) \leq \big[5 \tau /3 +
		c_{0}(C_{0}) \big] d \big)
		\geq 1-e^{-\max\{ \kappa\sqrt{d}, \kappa^{2}d/\b^{2}\}} - 2e^{-\tau d} \, .
\end{equation*}
Given $\alpha \in (0,1)$, for any $\tau \in (0,1)$ and $C_{0} \geq 1$ such that
\begin{equation}
\label{eq:cst-tau-alpha-condition}
 	\bigg(\frac{5\tau}{3} + c_{0}(C_{0}) \bigg)	d \leq d-1 - \alpha d \, ,
\end{equation} it holds on the last event that $F(\V) \leq d-1 - \alpha d$. Equivalently, recalling that $\delta(\Lambda) = n - F(\V)$, this rewrites
\begin{equation}
\label{eq:bonne-dim-stat}
	\P(A_{\alpha d})
		= \P\big(n-d+1 \leq \delta(\Lambda) - \alpha d\big)
		\geq 1-e^{-\max\{ \kappa\sqrt{d}, \kappa^{2}d/\b^{2}\}} - 2e^{-\tau d} \, ,
\end{equation}
provided that $\alpha$, $\tau$ and $C_{0}$ satisfy~\eqref{eq:cst-tau-alpha-condition}. We have proved~\eqref{eq:stat-dim-whp}.

\parag{Conclusion of the proof}
The final step of the proof consists in applying the kinematic formula conditionally on the event where the statistical dimension of $\Lambda$ is well-behaved.
Let $\L_{d-1}$ be a random subspace drawn uniformly over $(d-1)$-dimensional subspaces in $\R^{n}$ and let $E = \{\Lambda \cap \L_{d-1} \neq \{0\} \}$ (the event that linear separation occurs). By Lemma~\ref{lem:bernstein-kinematic}, on the event $A_{\alpha d}$,\begin{equation}
\label{eq:conditional-kinematic}
	\P( E \cond \V)
		\geq 1-4\exp\Big(-\frac{(\alpha d)^{2}}
			{8(\min\{\delta (\Lambda), n-\delta (\Lambda)\}+\alpha d)}\Big)
		\geq 1-4e^{-\alpha^{2}d/8}
		\, .
\end{equation}
The last inequality stems from the fact that on $A_{\alpha d}$, it also holds that $\min\{F(\V), n - F(\V)\} \leq d-1 - \alpha d$
. We thus showed that, given $\alpha \in (0,1)$, for any $\tau$ and $C_{0}$ satisfying \eqref{eq:cst-tau-alpha-condition}, 
\begin{equation*}
	\P(A_{\alpha d})
		\geq \P\big( F(\V) \leq \big[5 \tau /3 + c_{0}(C_{0}) \big] d \big)
		\geq 1-e^{-\max\{ \kappa\sqrt{d}, \kappa^{2}d/\b^{2}\}} - 2e^{-\tau d} \, .
\end{equation*}

To conclude the proof, we bound from below the probability of $\L$ intersecting $\Lambda$ in a non trivial way by following the final steps of the proof of Theorem~1 in \cite{candes2020phase}. Using \eqref{eq:conditional-kinematic}, one has
\begin{align*}
	\1(A_{\alpha d})
		\leq \1\big(\P(E \cond \V) \geq 1-4e^{- \alpha^{2} d/8}\big)
		&= \1\big(\P(E \cond \V) + 4e^{- \alpha^{2} d/8} \geq 1\big) \\
		&\leq \P(E \cond \V) + 4e^{- \alpha^{2} d/8} \, .
\end{align*}
Taking expectation with respect to $\V$, this implies that
\begin{equation}
\label{eq:unconditional-separation}
	\P(E)
	\geq \P(A_{\alpha d}) - 4e^{-\alpha^{2} d/8}
	\geq 1-e^{-\max\{ \kappa\sqrt{d}, \kappa^{2}d/\b^{2}\}} 
		- 2e^{-\tau d} - 4e^{-\alpha^{2} d/8} \, .
\end{equation}
We will thus choose $\tau \geq \alpha^{2}/8$ so that $e^{-\tau d} \leq e^{-\alpha^{2}d/8}$, under the constraint~\eqref{eq:cst-tau-alpha-condition}.
This constraint rewrites
\begin{equation*}
	\tau \leq \frac 3 5 \bigg(1 - \alpha - \frac 1 d - c_{0}(C_{0})\bigg) \, .
\end{equation*}
We choose $\tau = \alpha^{2}/8$ and then saturate the constraint, that is, we choose $\alpha = \alpha(d,C_{0})$, the largest solution of the equation $\alpha^2 /8 = 3( 1- \alpha - 1/d - c_{0}(C_{0}))/5$. We further bound the numerical constants for the choice $C_{0}=200$ and under the assumption that $d\geq 70$, that is $\alpha=\alpha(d,C_{0})>0.5844...$ and in particular $\alpha^2/8 > 1/24$. The result then follows from~\eqref{eq:unconditional-separation}.

\subsection{Remaining proofs}
\label{sec:existence-remaining-proofs}
\begin{proof}[Proof of Lemma~\ref{lem:cones-intersection}]
By definition, there exists a separating hyperplane if there is some $\theta \in \R^{d} \setminus\{0\}$ such that for all $i \in \{1, \dots, n\}$, 
\begin{equation}
  \label{eq:zero-error}
  Y_{i} \langle \theta, X_{i} \rangle \geq 0 \, .
\end{equation}
From now on, for $1 \leq j \leq d$, we let $\X^{j}$ denote the $n$-dimensional vector $(X_{1}^{j},
\dots, X_{n}^{j})$ whose entries are all $j$-th coordinates of $X_{1}, \dots, X_{n}$.
For every $i$, the random vectors $(Y_i, X_i^1)$ and $(X_i^2, \dots, X_i^d)$ are independent (and the latter has a symmetric distribution), hence the vectors $(Y_i X_i^1, Y_i X_i^2, \dots, Y_i X_i^d)$ and $(Y_i X_i^1, X_i^2, \dots, X_i^d)$ have the same distribution.
Therefore, 
\begin{equation}
\label{eq:reduced-zero-error}
	\P\big(\exists \theta \in \R^{d} \setminus\{0\}, \, \forall i, Y_{i}
	\langle \theta, X_{i} \rangle \geq 0\big)
	=\P\Big(\exists \theta \in \R^{d} \setminus\{0\},\, 
	\theta^{1} \V + \sum_{j=2}^{d} \theta^{j} \X^{j} \in \R_+^n \Big) \, .
\end{equation}
Now, let $\L_{d-1} = \vspan\{\X^{2}, \dots, \X^{d}\}$.
Since $\X^2, \dots, \X^d$ are \iid random vectors with distribution $\gaussdist (0, I_n)$, the distribution of $\mathcal{L}$ is rotation-invariant and thus uniform over $(d-1)$-dimensional subspaces of $\R^n$.
Also, $\L_{d-1}$ is independent from $\Lambda = \R \V + \R_+^n$, and if $\Lambda \cap \L_{d-1} \neq \set{0}$, then there exists $\theta^1 \in \R$ and $(\theta^2, \dots, \theta^d) \in \R^{d-1}$, as well as $w \in \R_+^n$ such that $- \theta^1 \bm V + w = \sum_{j=2}^d \theta^j \bm X^j$, thus $\theta^1 \bm V + \sum_{j=2}^d \theta^j \bm X^j \in \R_+^n$.
Combining this fact with~\eqref{eq:reduced-zero-error} concludes the proof.
\end{proof}

\begin{fact}
\label{fact:signal-confidence}
Let $p\in (0,1/2)$ and $u^{*}\in S^{d-1}$ be such that $\P(Y\innerp{u^{*}}{X} < 0) \leq p$. For any $t>0$, if $n\leq t/(2p)$, then with probability at least $e^{-t}$ the dataset $(X_{1}, Y_{1}), \dots,(X_{n}, Y_{n})$ of \iid copies of $(X,Y)$ is linearly separated.
\end{fact}
\begin{proof}
We have
\begin{equation*}
	\P \big( \forall i \leq n ,\ Y_i \langle u^* , X_i \rangle \geq 0 \big)
		= \big( 1- \P(Y \langle u^*, X \rangle < 0) \big)^n \\
		\geq (1-p)^n
		= \exp\big(n\log(1 - p) \big) \enspace .
\end{equation*}
By concavity, for all $x \in [0, 1/2]$, $\log(1-x) \geq -2 \log(2) x$. Thus, since $n \leq t/(2p) \leq t/ (2\log(2)p)$, one has
 \begin{equation*}
	\P \big( \forall i \leq n, \ Y_i \langle \theta^* , X_i \rangle > 0 \big)
		\geq \exp\big(-2np\big)
		\geq \exp(-t) \, . 
	\qedhere
\end{equation*}
\end{proof}

\begin{fact}
\label{fact:psi-estimates}
Let $\psi(s) = \E\big[(s-Z)_{+}^{2}\big]$ for every $s\in \R$, with $Z\sim \normal(0,1)$. Then
\begin{equation*}
	\psi(s) \leq \frac{e^{-s^{2}/2} }{2}\indic{s<0} + (s^{2} + 1)\indic{s\geq 0} \, .
\end{equation*}
\end{fact}

\begin{proof}
  The case $s \geq 0$ follows from the inequality $(s - Z)_+^2 \leq (s-Z)^2$, as
  \begin{equation*}
    \psi (s)
    = \Expect{(s-Z)_+^2}
    \leq \Expect{(s-Z)^2}
    = s^2 + 1
    \, .
  \end{equation*}

  Consider now the case $s <0$.
  Denoting by $g$ the density of $\normal(0,1)$, one has
\begin{align*}
\psi(s)
	&= \E (s - Z)_{+}^{2} = \E \left[ (s - Z)^{2} \ind{s - Z > 0} \right]
		= \int_{-\infty}^{s} (s-z)^{2} g(z) \di z = \int_{0}^{+\infty} z^{2} g(s-z) \di z \\
	&= \int_{0}^{+\infty} z^{2} \frac{1}{\sqrt{2\pi}}
		\exp\Big(- \frac{s^{2} - 2 s z + z^{2} }{2} \Big) \di z
	= e^{-s^{2} / 2} \int_{0}^{+\infty} z^{2} e^{sz} g(z) \di z \\
	&= e^{-s^{2} / 2} \, \E \left[ Z^{2} e^{sZ} \ind{Z>0} \right] \, .
\end{align*}
Now since $s<0$, one has $e^{sZ} \ind{Z>0} \leq \ind{Z>0}$, thus
\begin{equation*}
  \E \left[ Z^{2} e^{-sZ} \ind{Z<0} \right] \leq \E \left[ Z^{2} \ind{Z<0} \right] = \frac{1}{2} \enspace ,
\end{equation*}
which concludes the proof.
\end{proof}

\section{Proofs of the main results%
}
\label{app:main-proofs}
This section gathers the results of Sections~\ref{sec:gradients} and~\ref{sec:hessians} to establish the upper bounds on the excess risk of the MLE thanks to Lemma~\ref{lem:localization}.

\subsection{Preliminaries: convex localization and Hessian}
\label{sec:proof-localization}

We start with the proof of the localization lemma (Lemma~\ref{lem:localization}).

\begin{proof}[Proof of Lemma~\ref{lem:localization}]
  Let $r$ be arbitrary such that $2 \nu /c_0 < r < r_0$, which exists since $r_0 > 2 \nu /c_0$ by assumption.
  For any $\theta \in \R^d$ such that $\norm{\theta - \theta^*}_H = r$, a Taylor expansion of order $2$ shows that
  \begin{align}
    \wh L_n (\theta) - \wh L_n (\theta^*)
    &= \innerp{\nabla \wh L_n (\theta^*)}{\theta - \theta^*} + \int_0^1 (1-t) \big\langle \nabla^2 \wh L_n \big( (1-t)  \theta^* + t \theta \big) (\theta - \theta^*), \theta - \theta^* \big\rangle \di t \nonumber \\
    &\geq - \norm{\nabla \wh L_n (\theta^*)}_{H^{-1}} \norm{\theta - \theta^*}_H + \frac{c_0}{2} \norm{\theta - \theta^*}_H^2 \label{eq:proof-loc-lower-hess} \\
    &\geq - \nu r + c_0 r^2/2
      \label{eq:proof-loc-cond-r}
      > 0
      \, ,
  \end{align}
  where inequality~\eqref{eq:proof-loc-lower-hess} comes from the fact that $\nabla^2 \wh L_n ((1-t) \theta^* + t \theta) \mgeq c_0 H$ by assumption, and~\eqref{eq:proof-loc-cond-r} from the condition $r > 2 \nu / c_0$.
  Now, for any $\theta' \in \R^d$ such that $r'= \norm{\theta' - \theta^*}_H \geq r$, the parameter $\theta = (1-t) \theta^* + t \theta'$ with $t = r / r' \in (0, 1]$ satisfies $\norm{\theta - \theta^*}_H = r$, hence by the preceding and by convexity of $\wh L_n$ one has
  \begin{equation*}
    (1-t) \wh L_n (\theta^*) + t \wh L_n (\theta')
    \geq \wh L_n \big( (1-t) \theta^* + t \theta' \big)
    = \wh L_n (\theta)
    > \wh L_n (\theta^*)
    \, ,
  \end{equation*}
  which simplifies to $\wh L_n (\theta') > \wh L_n (\theta^*)$.
  Hence $\inf_{\R^d} \wh L_n = \inf_{\theta \pp \norm{\theta - \theta^*}_H \leq r} \wh L_n (\theta)$, and the latter infimum is attained by compactness and continuity of $\wh L_n$.
  Since in addition $\wh L_n$ is strictly convex on the set $\{ \theta \in \R^d : \norm{\theta - \theta^*}_H \leq r_0 \}$ due to the second assumption, the function $\wh L_n$ admits a unique global minimizer $\wh \theta_n \in \R^d$, such that $\norm{\wh \theta_n - \theta^*}_H \leq r$.
  Since this holds for every $r \in (2 \nu/c_0, r_0)$, we deduce that $\norm{\wh \theta_n - \theta^*}_H \leq 2 \nu / c_0$.
  
  The excess risk bound~\eqref{eq:norm-risk-bound-localization} then follows from the fact that $L (\theta) - L (\theta^*) \leq \frac{c_1}{2} \norm{\theta - \theta^*}_H^2$ for any $\theta$ with $\norm{\theta - \theta^*}_H \leq r_0$, since $\nabla L (\theta^*) = 0$ and $\nabla^2 L \mleq c_1 H$ over this domain.

  To prove the second point, let $\eps = \wh L_n (\wt \theta_n) - \wh L_n (\wh \theta_n)$ and $r$ be such that $\max \{ 4 \nu / c_0, 2 \sqrt{\eps/c_0} \} < r < r_0$.
  For any $\theta$ such that $\norm{\theta - \theta^*}_H = r$, proceeding as before (and using that $\wh L_n (\wh \theta_n) \leq \wh L_n (\theta^*)$) we get
  \begin{equation*}
    \wh L_n (\theta) - \wh L_n (\wh \theta_n)
    \geq \wh L_n (\theta) - \wh L_n (\theta^*)
    \geq - \nu r + c_0 r^2/2
    \geq c_0 r^2/4
    > \eps
    \, ,
  \end{equation*}
  where the last two inequalities follow from the conditions on $r$.
  By the same convexity argument as before, this implies that $\wh L_n (\theta) - \wh L_n (\wh \theta_n) > \eps$ for any $\theta$ such that $\norm{\theta - \theta^*}_H \geq r$, hence $\norm{\wt \theta_n - \theta^*}_H < r$.
  Letting $r \to \max \{ 4 \nu/c_0, 2 \sqrt{\eps/c_0} \}$ and using that $L (\wt \theta_n) - L (\theta^*) \leq \frac{c_1}{2} \| \wt \theta_n - \theta^* \|_H^2$ concludes the proof.
\end{proof}

We now turn to the structure of the Hessian $\nabla^2 L (\theta^*)$ in the case of a Gaussian design, which is given by~\eqref{eq:HessianGaussian}.
It then remains to justify the estimates~\eqref{eq:components-hessian} on the components $c_0 (\cdot), c_1(\cdot)$ of the Hessian.
Lemma~\ref{lem:MomGauss} below shows that 
 \begin{align}
    \label{eq:2}
    \frac{2\sqrt{2}}{3e^4\sqrt{\pi}} \min \Big(1, \frac{1}{b^3} \Big) \leq\, &c_0 (b) \leq 2\sqrt{\frac2\pi}\min \Big( 1, \frac{1}{b^3} \Big) \, ; \\
    \frac1{2e^4}\sqrt{\frac2\pi} \min \Big( 1, \frac{1}{b} \Big) \leq \, &c_1 (b) \leq \sqrt{2}\min \Big( 1, \frac{1}{b} \Big)
              \, .
  \end{align}
  
\begin{lemma}\label{lem:MomGauss}
  Let $G \sim \gaussdist (0,1)$.
  For any $b>0$ and integer $k\geqslant 0$,
    \[
      \sqrt{\frac{2}{\pi}} \frac{2^{k+1}}{k+1} \min\bigg(\frac1{4e^{4}b^{k+1}},\frac{\sigma'(2)}{e^{2}}\bigg)
      \leqslant \E[\sigma'(b G) |G|^k]
      \leqslant
      \sqrt{\frac2\pi} \min \bigg(\Gamma\bigg(\frac{k+1}2\bigg), \frac{k!}{b^{k+1}}\bigg)\enspace.
    \]
\end{lemma}
\begin{proof}
We have
\[
\E[\sigma'(b G)|G|^k]=\sqrt{\frac{2}{\pi}}\int_0^{+\infty}x^k\sigma'(b x)\exp\bigg(-\frac{x^2}2\bigg)\di x\enspace.
\]
For the upper bound, we use that, for any $x$, $\sigma'(x)\leqslant \exp(-|x|)\leqslant 1$ and $\exp(-x^2/2)\leqslant 1$ to get
\[
\E[\sigma'(b G)|G|^k]\leqslant\sqrt{\frac{2}{\pi}}\min\bigg(\int_0^{+\infty}x^k\exp\bigg(-\frac{x^2}2\bigg)\di x,\int_0^{+\infty}x^k\exp\big(-b x\big)\di x\bigg)\enspace.
\]
Computing the integrals yields the upper bound.

For the lower bound, as the function we integrate is nonnegative and $\sigma'(x)\geqslant \exp(-x)/4$, we have 
\begin{align*}
\E[\sigma'(b G)|G|^k]&\geqslant \sqrt{\frac{2}{\pi}}\int_0^{2}x^k\sigma'(b x)\exp\bigg(-\frac{x^2}2\bigg)\di x\\
&\geqslant \sqrt{\frac{2}{\pi}}\max\bigg(\frac1{4e^{2}}\int_0^{2}x^k\exp(-b x) \di x,\sigma'(2b)\int_0^{2}x^k\exp\bigg(-\frac{x^2}2\bigg)\di x\bigg)\\
&=\sqrt{\frac{2}{\pi}}\max\bigg(\frac1{4e^{2}b^{k+1}}\int_0^{2}x^k\exp(-x) \di x,\sigma'(2b)\int_0^{2}x^k\exp\bigg(-\frac{x^2}2\bigg) \di x\bigg)\\
&\geqslant \sqrt{\frac{2}{\pi}}\frac{2^{k+1}}{k+1}\max\bigg(\frac1{4e^{4}b^{k+1}},\frac{\sigma'(2b)}{e^{2}}\bigg)\enspace.
\end{align*}
To get the lower bound, we use the first bound when $b>1$ and the second one when $b\leqslant 1$.
\end{proof}

\subsection{Proof of  Theorem~\ref{thm:gaussian-well-specified}}
\label{sec:proof-well-gaussian}

By Proposition~\ref{prop:DevGradGauss},
since $n \geq 16\b (d+t)$, with probability larger than $1-3e^{-t}$,
\begin{equation*}
	\big \| \nabla \wh L_{n} (\theta^{*}) \big \|_{H^{-1}}
		\leq 14 \sqrt{\frac{d+t}{n}} \, ,
\end{equation*}
Moreover, let
\begin{equation*}
  \Theta = \left\{ \theta \in \R^d : \, \|\theta - \theta^{*}\|_{H}
		\leq \frac{1}{100\sqrt{\b}} \right\} \, ,
\end{equation*}
Theorem~\ref{thm:hessian-gaussian} ensures that, if $n \geq 1200000\, \b (d+t)$, then with probability
at least $1 - 2 e^{-t}$, simultaneously for all $\theta \in \Theta$, 
$\wh H_n(\theta) \mgeq \frac1{1000} H$.  
Now,
as soon as $n\geqslant (2800 000)^2\b (d+t)$, 
\begin{equation*}
  14 \sqrt{\frac{d+t}{n}} < \frac{1}{2} \times \frac1{1000}\times\frac{1}{100\sqrt{\b}}\ ,
\end{equation*}
so by Lemma~\ref{lem:localization}, with probability at least $1-5e^{-t}$,
\begin{equation*}
	\| \wh \theta_{n} - \theta^{*} \|_{H} \leq 28000 \sqrt{\frac{d+t}{n}} \, .
\end{equation*}
By Lemmas~\ref{lem:BHGauss} and~\ref{lem:localization}, we also have on the same event 
\[
L(\wh\theta_{n}) - L(\theta^{*}) \leq 420\cdot(28)^2\cdot10^{6} \frac{d+t}{n}\enspace.
\]

Regarding the necessity of the sample size condition, we combine Theorem~\ref{thm:strong-non-existence} with Fact~\ref{fact:signal-confidence} which in the case of a well-specified model shows that the condition $n\gtrsim \b t $ is also necessary. Indeed, as the model is well-specified, $\P(Y\innerp{\theta^{*}}{X} < 0) = \Expect{\sigma(-\ainnerp{\theta^{*}}{X})}$. In addition, one has $\sigma(-|t|) \leq \min\{1/2, e^{-|t|}\}$ for every $t \in \R$, hence
\begin{equation*}
	\P(Y\innerp{\theta^{*}}{X} < 0) \leq \frac{1}{\max\{2, \|\theta^{*}\|\}} \leq \frac{e}{2\b} \, .
\end{equation*}
We use Fact~\ref{fact:signal-confidence} with $p=e/(2\b)$ to conclude that if $n\leq \b t/e$, then
\begin{equation}
\label{eq:confidence-influence-final}
	\P( \text{MLE does not exist} ) \geq \exp(-t) \enspace .
\end{equation}

With these results at hand, one has that whenever
\begin{equation}
\label{eq:disjonction-sample-size}
n \leq \frac \b 2 \Big(\frac d {200} + \frac t e \Big)
	\leq \max\Big\{\frac{\b d}{200}, \frac{\b t}{e}\Big\} \, ,
\end{equation}
either $n \leq \b t/e$ or $n \leq \b d/200$. The former is already dealt with by~\eqref{eq:confidence-influence-final}. The latter, by Theorem~\ref{thm:strong-non-existence} (with parameter $\kappa=1$), implies that
\begin{equation}
\label{eq:dimension-influence-final}
	\P( \text{MLE does not exist} ) 
		\geq 1 - \exp\Big(-\max\Big\{\sqrt{d}, \frac d{\b^{2}}\Big\}\Big) 
	 		- 6e^{-d/24} \, .
\end{equation}
As $d\geq 70$ and $t\geq 1$, one has $6e^{-d/24} < 1/2$ and $1/2-e^{-\sqrt{d}} \geq e^{-t}$. Hence, taking the minimum of the two lower bounds \eqref{eq:confidence-influence-final} and \eqref{eq:dimension-influence-final} shows that $\P( \text{MLE does not exist} ) \geq \exp(-t)$ and concludes the proof of Theorem~\ref{thm:gaussian-well-specified}.

\subsection{Proof of Theorem~\ref{thm:regular-well-specified}}
\label{sec:proof-well-regular}

By Proposition~\ref{prop:DevGradWSRD}, we have, for any $n\geqslant \b (d+t)$, for any $t>0$, with probability at least $1 - 3e^{-t}$,
\begin{equation*}
\big\| \nabla \wh L_{n} (\theta^{*}) \big\|_{H^{-1}} \leq c_0 \log \b \sqrt{\frac{d + t}{n} } \enspace ,
\end{equation*}
where $c_0$ only depends on $c$ and $K$.
Moreover, by Theorem~\ref{thm:hessians-regular-case},
if
\begin{equation*}
  n \geq c_1\b (\log(\b) d +  t)\ ,
\end{equation*}
then, with probability $1-\exp(-t)$,
\begin{equation*}
\wh H_n(\theta) \mgeq c_2 H \, , \quad \text{for every } \theta \text{ such that }
	\|\theta-\theta^*\|_H\leq\frac{c_3}{\log (\b)\sqrt{\b}}\, .
\end{equation*}
Now, for some constants $c_4, \dots$ depending only on $c,K$, if
\begin{equation*}
  n \geqslant c_4
  (\log \b)^4 \b (d+t) \, , 
\end{equation*}
then
\begin{equation*}
  c_0 \log (\b) \sqrt{\frac{d + t}{n}}
  < \frac{1}{2} \times c_2 \times \frac{c_3}{\log (\b) \sqrt{\b}}
  \, ,
\end{equation*}
and thus by Lemma~\ref{lem:localization},
with probability at least $1-4e^{-t}$,
\begin{equation}
	\| \wh \theta_{n} - \theta^{*} \|_{H} \leq c_5 \log (\b) \sqrt{\frac{d + t}{n} } \, .
\end{equation}
By Lemmas~\ref{lem:BHReg} and~\ref{lem:localization}, on the same event,
we have
\begin{equation*}
	L (\wh \theta_{n}) - L(\theta^{*}) \leq c_6 (\log \b)^4 \frac{d+t}{n} \, .
\end{equation*}

\subsection{Proof of Theorem~\ref{thm:regular-misspecified}}
\label{sec:proof-misspecified-regular}

By Proposition~\ref{prop:gradient-dev-MS}, if $n\geqslant \b (d+\b t)$, then with probability larger than $1-3e^{-t}$,
\begin{equation*}
  \big\| \nabla \hat{L}_{n}(\theta^{*}) \big\|_{H^{-1}}
  \leq c_0 \log(\b)\sqrt{\frac{d+\b t}{n}} \, .
\end{equation*}  
Moreover, by Theorem~\ref{thm:hessians-regular-case}, if
\begin{equation*}
  n \geq c_1\b (\log(\b) d +  t)\ ,
\end{equation*}
then, with probability $1-\exp(-t)$,
\begin{equation*}
\wh H_n(\theta) \succcurlyeq c_2 H \qquad \text{for every } \theta \text{ such that } \|\theta-\theta^*\|_H\leqslant \frac{c_3}{\log (\b)\sqrt{\b}}\, .
\end{equation*}
By Lemma~\ref{lem:localization}, it follows that (for constants $c_4, \dots$ depending only on $c,K$), if 
\begin{equation}
  \label{eq:sample-size-loc-reg-ws}
  n \geqslant
  c_4
  (\log \b)^4 \b (d+\b t) \, , 
\end{equation}
 with probability at least $1-4e^{-t}$,
\begin{equation}
  \| \wh \theta_{n} - \theta^{*} \|_{H}
  \leq %
  c_5 \log (\b)
  \sqrt{\frac{d + \b t}{n} } \, .
\end{equation}
By Lemmas~\ref{lem:BHReg} and~\ref{lem:localization}, on the same event,
\begin{equation*}
	L (\wh \theta_{n}) - L(\theta^{*}) \leq c_6(\log \b)^4 \frac{d+\b t}{n} \, .
\end{equation*}

Regarding the necessity of the sample size condition, the fact that the condition $n\gtrsim \b d$ is necessary comes from the well-specified case of Theorem~\ref{thm:gaussian-well-specified}, which is a special case of the current setting.
Regarding the necessity of the extra $\b$ factor in the sample size condition, consider the following distribution of $(X,Y)$: $X$ is a standard Gaussian vector and the conditional distribution of $Y$ given $X$ is
such that $\P(Y\langle u^{*}, X \rangle < 0 | X)$ is constant (see~\eqref{eq:worst-distrib}).
The first point of Lemma~\ref{lem:grad-ms-optimal} shows that for this distribution, $\P(Y \langle \theta^{*}, X \rangle < 0) \leq 1/\b^{2}$. It then follows from Fact~\ref{fact:signal-confidence} that if $n \leq \b^{2}t/2$,
\begin{equation*}
	\P(\text{MLE exists}) \leq e^{-t} \, .
\end{equation*}
The conclusion follows from the same argument as in the proof of Theorem~\ref{thm:regular-well-specified} in the previous section, see~\eqref{eq:disjonction-sample-size} and after.

We now turn to the optimality of our bound on the excess risk~\eqref{eq:asymptotic-extra-B}. It is known from asymptotic theory (see \eg~\cite[Example~5.25 p.~55]{vandervaart1998asymptotic}) that in the misspecified case,
\begin{equation*}
	\sqrt{n}H(\theta^{*})^{1/2} (\wh \theta_{n} - \theta^{*}) \underset{n\to\infty}{\cvd}  \normal(0, \Gamma)
	\, ,  \qquad  \Gamma = H(\theta^{*})^{-1/2}GH(\theta^{*})^{-1/2}\, ,
\end{equation*}
where $H(\theta^{*}) = \nabla^{2}L(\theta^{*})$ is the population Hessian and $G = \E[ \nabla \ell(\theta^{*}) \nabla \ell(\theta^{*})^{\top}]$ is the the covariance of the gradient at $\theta^{*}$.
Hence, the rescaled excess risk $2n(L(\wh \theta_{n}) - L(\theta^{*}))$ converges in distribution to $\|\xi\|^{2}$ where $\xi \sim	 \normal(0, \Gamma)$.
The argument showing the optimality of our result is twofold. First, in the case where the model is well-specified, $\Gamma=I_{d}$, so $\tr(\Gamma) = d$.
Second, the argument regarding the necessity of the deviation term builds upon the same conditional distribution that explains the necessity of $\b^{2}t$ in the sample size condition that we described above (\ie $X\sim  \normal(0, I_{d})$ and $Y|X$ is given by~\eqref{eq:worst-distrib}). Indeed, by Lemma~\ref{lem:BHGauss}, $\Gamma \succcurlyeq C_{1}^{-1}H^{-1/2}GH^{-1/2}$, with $C_{1} = 2\sqrt{2/\pi}$, so by the second point of Lemma~\ref{lem:grad-ms-optimal}, for this particular distribution, it holds that
\begin{equation*}
	\opnorm{\Gamma} \geq \frac \b {8C_{1}} \geq \frac \b {13} \, .
\end{equation*}
In addition, by standard concentration arguments, one can find an absolute constant $c_{1}$ such that on one hand, the median of the distribution $\chi^{2}(d)$ is at least $c_{1}d$; and on the other hand, if $v \in S^{d-1}$ denotes an eigenvector of $\Gamma$ associated to its largest eigenvalue,
\begin{equation*}
  \P (\norm{\xi}^2 \geq c_1 \opnorm{\Gamma} t)
  \geq \P(\langle v, \xi \rangle^{2} \geq c_{1} \opnorm{\Gamma} t)
  \geq e^{-t} \, .
\end{equation*}
This concludes the proof of Theorem~\ref{thm:regular-misspecified}.

\parag{Worst misspecified case}
In this paragraph we provide the example mentioned above of a conditional distribution of $Y$ given $X$ which accounts for the extra factors in the sample size and the risk bound on the MLE.

\begin{lemma}
\label{lem:grad-ms-optimal}
Let $X \sim  \normal(0, I_{d})$, $u^{*}\in S^{d-1}$ and $p\in (0, e^{-2}/2)$. Let $Y$ be such that
\begin{equation}
\label{eq:worst-distrib}
\P(Y\langle u^{*}, X \rangle < 0 | X) = p \, .
\end{equation}
Then
\begin{enumerate}
\item the signal strength $\b=\max\{e,\|\theta^{*}\|\} $ is related to the probability of misclassification by
\begin{equation*}
	\frac{1}{2\b^{2}} \leq p = \P(Y\langle u^{*}, X \rangle < 0) \leq \frac{1}{\b^{2}} \, .
\end{equation*}
\item The covariance of the gradient $G = \E[\nabla \ell(\theta^{*}, Z)\nabla \ell(\theta^{*}, Z)^{\top}]$ satisfies
\begin{equation*}
	\opnorm{H^{-1/2}GH^{-1/2}} \geq \frac \b 8 \, .
\end{equation*}
\end{enumerate}
\end{lemma}

\begin{proof}
Recall that since the model is misspecified, $\theta^{*}$ is defined as the unique minimizer of $L(\theta)$ (uniqueness follows from the strict convexity of $L$). We first note that for any isometry $Q$ such that $Qu^{*} = u^{*}$, it holds for all $\theta \in \R^{d}$
that
\begin{equation}
\label{eq:loss-invariance}
L(Q\theta) = L(\theta) \, .
\end{equation}
This stems from the fact that the distribution of $X$ is invariant under any isometry and the distribution of $Y$ given $X$ is invariant under any isometry that preserves $u^{*}$. This holds in particular at the point $\theta^{*}$. Hence $Q\theta^{*} = \theta^{*}$ and, letting $Q = 2u^{*}{u^{*}}^{\top} - I_{d}$, this shows that $\theta^{*} \in \R u^{*}$.  We show in addition that $\theta^* \in \R_+ u^*$, namely
\begin{equation}
\label{eq:theta-star-orientation}
	\theta^{*} = \|\theta^{*}\|u^{*} \, .
\end{equation} 
This amounts to showing that $L(-\|\theta^{*}\|u^{*}) > L(\|\theta^{*}\|u^{*})$ which we do next.
Let $\phi(t) = \log(1+e^{t})$ denote the logistic loss and write
\begin{align*}
  L(-\|\theta^{*}\|u^{*})
  &= \E[\phi(Y\|\theta^{*}\|\langle u^{*}, X \rangle)] \\
  &= (1-p) \E[\phi(\|\theta^{*}\||\langle u^{*}, X \rangle|)] + p \E[\phi(-\|\theta^{*}\||\langle u^{*}, X \rangle|)] %
    \\
  &> p \E[\phi(\|\theta^{*}\||\langle u^{*}, X \rangle|)] 
    + (1-p) \E[\phi(-\|\theta^{*}\||\langle u^{*}, X \rangle|)] \\
  &= \E[\phi(-Y\|\theta^{*}\|\langle u^{*}, X \rangle)] = L(\|\theta^{*}\| u^{*}) \, .
\end{align*}
This proves \eqref{eq:theta-star-orientation}.

It remains to show that $\b=\|\theta^{*}\| \geq e$ and that $\b \asymp p^{-1/2}$. In view of $\eqref{eq:worst-distrib}$, $\indic{Y\langle \theta^{*}, X \rangle < 0}$ is independent from $X$.
Hence \eqref{eq:misclassif-identity} rewrites
\begin{equation*}
	p \E\big[|\innerp{u^{*}}{X}|\big]
		=  \E\big[|\innerp{u^{*}}{X}|\sigma(-|\innerp{\theta^{*}}{X}|)\big] \, .
\end{equation*}
In addition, since $\E |\innerp{u^{*}}{X}| = \sqrt{2/\pi}$, one has
\begin{equation}
\label{eq:proba-is-order1-moment}
	p = \sqrt{\frac{\pi}{2}}\, \E\big[ |\innerp{u^{*}}{X}| \sigma(-|\innerp{\theta^{*}}{X}|)\big] \, .
\end{equation}
On one hand, $\sigma(-t)\geq e^{-t}/2$ for all $t\geq 0$. Using that
$\|\theta^{*}\| \leq \b$, it follows from Lemma~\ref{lem:order-1-exp-lower-bound} that \begin{equation}
\label{eq:order-1-lower-bound}
	\E\big[ |\innerp{u^{*}}{X}|\sigma(-|\innerp{\theta^{*}}{X}|)\big]
	\geq \frac{1}{2} \E\big[ |\innerp{u^{*}}{X}| \exp(-\b |\innerp{u^{*}}{X}|)\big]
	\geq  \frac{1}{\sqrt{2\pi}\b^2} \, .
\end{equation}
Hence, using~\eqref{eq:proba-is-order1-moment} and since $p\leq e^{-2}/2$ we deduce that
\begin{equation*}
	\b \geq \frac{1}{\sqrt{2p}} \geq e \, .
\end{equation*}
The lower bound of the first point of the lemma is therefore a straightforward consequence of~\eqref{eq:order-1-lower-bound} and~\eqref{eq:proba-is-order1-moment} and we have
\begin{equation}
\label{eq:signal-is-strong-lower-bound-prob}
	\b =\|\theta^{*}\|\geq e  \quad \text{and} \quad p\geq \frac{1}{2\b^{2}}.
\end{equation}
We now prove the upper bound of the first point, which is a consequence of the exponential moment bound~\eqref{eq:MomGaussu*}, since $\b=\|\theta^{*}\|\geq e$. Using that $\sigma(-t)\leq e^{-t}$ for all $t\geq 0$, we deduce
\begin{equation*}
	\E\big[ |\innerp{u^{*}}{X}| \sigma(-|\innerp{\theta^{*}}{X}|)\big] 
		\leq \sqrt{\frac{2}{\pi}} \cdot \frac{1}{\b^{2}} \, .
\end{equation*}
We plug this in \eqref{eq:proba-is-order1-moment} to get that $p \leq 1/\b^{2}$, which is the desired upper bound.

We now prove the second point. As $\sigma(t)\geq 1/2$ for every $t\geq 0$,
\begin{equation*}
	\innerp{Gu^*}{u^*}
		= \E\innerp{u^*}{\nabla\ell(\theta^*,Z)}^2
		=\E\big[\sigma(-Y\innerp{\theta^{*}}{X})^{2} \innerp{u^{*}}{X}^{2} \big]
		\geq \frac{1}{4} \E\big[ \indic{Y\innerp{\theta^{*}}{X} < 0}	
			\innerp{u^{*}}{X}^{2} \big] \, .
\end{equation*}
The distribution of $Y|X$ is designed so that $\E[\indic{Y\innerp{\theta^{*}}{X} < 0}\cond X]$ is actually not a function of $X$, but constant and equal to $p$. More precisely,
\begin{align*}
	\E\big[ \indic{Y\innerp{\theta^{*}}{X} < 0} \innerp{u^{*}}{X}^{2} \big]
			&= \E\big[ \E \big[ \indic{Y\innerp{\theta^{*}}{X} < 0}
				\innerp{u^{*}}{X}^{2} \cond X \big] \big] \\
			&= \E\big[ \innerp{u^{*}}{X}^{2}  \E \big[ \indic{Y\innerp{\theta^{*}}{X} < 0}
				\cond X \big] \big] \\
			&= p \, \E\innerp{u^{*}}{X}^{2} = p \, .
\end{align*}
Therefore
\begin{equation*}
	\innerp{Gu^*}{u^*} 
		\geq
		\frac p 4 
		\geq \frac{1}{8\b^{2}} \, .
\end{equation*}
Finally, since $H^{-1/2}u^{*} = \b^{3/2} u^{*}$, it follows that $\innerp{H^{-1/2}GH^{-1/2}u^{*}}{u^{*}} \geq \b/8$.
\end{proof}

\subsection{Technical tools}
In the previous proofs, we used the following lemmas linking the Hessians $\nabla^2L(\theta)=H(\theta)=\E[\sigma'(\innerp{\theta}{X})XX^\top]$ to $H$ to conclude the proof. Their proofs rely on the use of the family of proxies $H_{\theta}$ for the true Hessians $H(\theta)$ defined in the proof of Theorem~\ref{thm:hessian-gaussian}, Equation~\eqref{eq:DefHtheta}.

\begin{lemma}\label{lem:BHGauss}
    Let $\theta\in\R^d\setminus \{0\}$ denote a vector such that $\|\theta-\theta^*\|_H\leqslant 1/10\sqrt{\b}$, let $u=\theta/\|\theta\|$ and assume that $X \sim \gaussdist (0, I_d)$.
    Then,
    \[
   \frac{1}{500}H\preccurlyeq H(\theta)\preccurlyeq 420 H\enspace.
  \] 
\end{lemma}
\begin{proof}
We write that, for any $v\in S^{d-1}$
\begin{align*}
\innerp{H_\theta^{-1/2}H(\theta)H_\theta^{-1/2}v}{v} =\b^3\innerp{u}{v}^2\E[\sigma'(\innerp{\theta}{X})\innerp{u}{X}^2]+ \b (1-\innerp{u}{v}^2)\E[\sigma'(\innerp{\theta}{X})] \enspace. 
\end{align*}
We start with the upper bound.
By Lemma~\ref{lem:MomGauss},
\[
\E[\sigma'(\innerp{\theta}{X})|\innerp{u}{X}|^k]\leqslant \sqrt{\frac2\pi}\min\bigg(\Gamma\bigg(\frac{k+1}2\bigg),\frac{\Gamma(k+1)}{\|\theta\|^{k+1}}\bigg)\enspace.
\]
Then, if $\b=e$, we use the first bound to get
\begin{align*}
\innerp{H_\theta^{-1/2}H(\theta)H_\theta^{-1/2}v}{v} \leqslant\sqrt{2}e^3\enspace.  
\end{align*}
This proves that $H(\theta)\preccurlyeq \sqrt{2}e^3H_\theta$.
As $H_\theta\preccurlyeq e^{-1}I_d\preccurlyeq e^2H$, this proves the result when $\b=e$.

If $\b>e$, we have by Lemma~\ref{lem:ellipsoids}, $\|\theta\|\geqslant 0.9\cdot \b$, so the second bound on the moments gives
\begin{align*}
 \innerp{H_\theta^{-1/2}H(\theta)H_\theta^{-1/2}v}{v} \leqslant \sqrt{\frac2\pi}\frac{2}{0.9^3}\enspace.  
\end{align*}
This proves that $H(\theta)\preccurlyeq 2.2\cdot H_\theta$ and this proves the result in the case $\b>e$ since by Lemma~\ref{lem:hess-proxy-regularity} we also have $H_\theta\preccurlyeq 1.3\cdot H$.

We now turn to the lower bound.
By Lemma~\ref{lem:MomGauss},
\[
\E[\sigma'(\innerp{\theta}{X})|\innerp{u}{X}|^k]\geqslant \sqrt{\frac{2}{\pi}}\frac{2^{k+1}}{k+1}\min\bigg(\frac1{4e^{4}\|\theta\|^{k+1}},\frac{\sigma'(2)}{e^{2}}\bigg)\enspace.
\]
Then, if $\b=e$, we use the second bound to get
\begin{align*}
\innerp{H_\theta^{-1/2}H(\theta)H_\theta^{-1/2}v}{v} \geqslant 0.02\enspace.  
\end{align*}
This proves that $H(\theta)\succcurlyeq c_2H_\theta$ and as $H_\theta\succcurlyeq e^{-3}I_d\succcurlyeq e^{-2}H$, this proves the result when $\b=e$.

If $\b>e$, we have by Lemma~\ref{lem:ellipsoids}, $\|\theta\|\geqslant 0.9\cdot \b$, so the first bound on the moments gives
\begin{align*}
\innerp{H_\theta^{-1/2}H(\theta)H_\theta^{-1/2}v}{v} \geqslant 0.0027\enspace. 
\end{align*}
This proves the result in the case $\b>e$ since by Lemma~\ref{lem:hess-proxy-regularity} we also have $H_\theta\succcurlyeq 0.76 \cdot H$.
\end{proof}

\begin{lemma}\label{lem:BHReg}
    Let $\theta\in \R^d\setminus\{0\}$ be such that $\|\theta-\theta^*\|_H\leqslant 1/(10\sqrt{\b})$ and let $u=\theta/\|\theta\|$. 
   Suppose that $X$ satisfies Assumptions~\ref{ass:sub-exponential} with parameter $K>0$ and~\ref{ass:small-ball} with parameters $\eta=1/\b$ and $c\geqslant 1$.
    Then, there exists $c'$ depending on $c$ and $K$
    such that 
    \[
   \frac1{c'}H\preccurlyeq H(\theta)\preccurlyeq c' (\log \b)^{2} H \enspace.
  \]
\end{lemma}
\begin{proof}

We start with the proof of the upper bound.
    Let $v\in S^{d-1}$ and let $w\in S^{d-1}$ denote a vector such that $\innerp{u}{w}=0$ and $v-\innerp{u}{v}u=\sqrt{1-\innerp{u}{v}^2}w$. 
    As $\sigma'(x)\leqslant \exp(-|x|)$, we have
\begin{align*}
 \innerp{H_\theta^{-1/2}H(\theta)H_\theta^{-1/2}v}{v} &=\b^3\innerp{u}{v}^2\E[\exp(-|\innerp{\theta}{X}|)\innerp{u}{X}^2]\\
 &+ \b (1-\innerp{u}{v}^2)\E[\exp(-|\innerp{\theta}{X}|)\innerp{w}{X}^2] \enspace.  
\end{align*}
If $\b=e$, it follows from $\sigma'(x)\leqslant 1$ that $H(\theta)\preccurlyeq e^3 H_\theta$ and thus $H(\theta)\preccurlyeq e^5 H$ since $H_\theta\preccurlyeq e^2 H$ in this case.

If $\b>e$, we have by Lemma~\ref{lem:ellipsoids}, $\|\theta\|\geqslant (1-r)\b$, where we let $r=1/10$ for the rest of this proof.
Thus, by Lemma~\ref{lem:MomRegDes}, it follows that
\begin{equation*}
 \innerp{H_\theta^{-1/2}H(\theta)H_\theta^{-1/2}v}{v} \leqslant \frac{3c}{1-r}(K\log((1+r)\b))^2\enspace.  
\end{equation*}
This proves that $H(\theta)\preccurlyeq \frac{3c}{1-r}(K\log((1+r)\b))^2H_\theta$ and this proves the result in the case $\b>e$ since by Lemma~\ref{lem:hess-proxy-regularity} we also have $H_\theta\preccurlyeq (1+2.35 r)H$.

We now turn to the lower bound.
Let $v\in S^{d-1}$, we have 
\[
\innerp{H(\theta)v}{v}=\E[\sigma'(\langle \theta, X \rangle)
			\langle v, X \rangle^{2}]\, .
\]
The function $\sigma'(x)=\exp(x)/(1+\exp(x))^2$ is even, non negative, non increasing on $[0, +\infty)$.
Therefore, for any $m, M>0$,
   \begin{align}
\label{VCB:Step1*}\innerp{H(\theta) v}{v}&     \geqslant\sigma'(m(1+r)\b)M^2\P\big(\ainnerp{u}{X} \leq m,\ \ainnerp{v}{X}\geqslant M\big)  \, ,
   \end{align}
    where we also used that, as $\|\theta-\theta^*\|_{H} \leqslant r/\sqrt{\b}$, $\|\theta\|\leqslant (1+r)\b$ by Lemma~\ref{lem:ellipsoids}.
    
    If $\|\theta^{*}\| \leq e$, $\b=e$, so Proposition~\ref{prop:regularity-constant-scale} shows that Assumptions~\ref{ass:small-ball} holds with $c=e$ and Assumption~\ref{ass:twodim-marginals} is satisfied with constant $\max\{2eK\log(2K),2K^4\}=2K^4$.
Therefore,
\begin{equation*}
\P\left( \big| \langle u, X \rangle \big| \leq \frac{2K^4}\b \,; \, 
		\big| \langle v, X \rangle \big| \geq \frac{\max\left\{1/\b, \|u^{*} - v \| \right\}}{2K^4}\right)
		\geqslant \frac{1}{2K^4\b} \,  .
\end{equation*}
Hence, choosing $m=2K^4/\b$ and $M=\max\left\{1/\b, \|u^{*} - v \|\right\}/2K^4$ in \eqref{VCB:Step1*}, we get that 
\begin{align*}
\innerp{H(\theta) v}{v}&     \geqslant\frac{\sigma'((1+r)2K^4)}{8K^{12}}\frac1\b \max\left\{\frac1{\b^2}, \|u^{*} - v \|^2\right\}\geqslant \frac{\sigma'((1+r)2K^4)}{16K^{12}}\innerp{Hv}{v}  \, .
   \end{align*}

   When $\b> e$, the third point of Lemma~\ref{lem:ellipsoids} implies that for every $\theta \in \Theta$, 
\[
\| u - u^{*} \| \leq \frac{\sqrt{2}}{[K\log(c(c+1)\b)-1]}\frac{r}\b \leqslant \frac{2r}{K\b \log(c(c+1)\b)}.
\]
By Lemma~\ref{lem:bidim-extension}, this implies that for all $\theta \in \Theta$ and $v \in S^{d-1}$, one has for all $t \geq 1/\b$
\begin{equation*}
\P\left( \big| \innerp{u}{X}\big| \leq \frac{c+1}\b \,; \, 		\big| \innerp{v}{X} \big| \geq \frac{\max\left\{1/\b, \|u^{*} - v \| \right\}}{c+1}\right)\geqslant \frac{1}{(c+1)\b} \,  .
\end{equation*}
Hence, choosing $m=(c+1)/\b$, $M=\max(1/\b,\|u^*-v\|)/(c+1)$ in \eqref{VCB:Step1}, we get that 
\begin{align*}
\innerp{H(\theta) v}{v}&     \geqslant\frac{\sigma'((1+r)(1+c))}{(1+c)^3}\frac1\b \max\left\{\frac1{\b^2}, \|u^{*} - v \|^2\right\}\geqslant \frac{\sigma'((1+r)(1+c))}{2(1+c)^3}\innerp{Hv}{v}  \, .
\qedhere
\end{align*}
\end{proof}

\begin{lemma}
\label{lem:order-1-exp-lower-bound}
Let $N\sim \normal(0,1)$ and $\b \geq e$. Then
\begin{equation*}
	\E\big[|N| \exp(-\b|N|)]
	\geq \frac{1}{\sqrt{2\pi}\b^2} \, .
\end{equation*}
\end{lemma}

\begin{proof}
Let $g(t) = e^{-t^{2}/2}/\sqrt{2\pi}$ denote the standard real Gaussian density. First, by symmetry,
\begin{equation}
\label{eq:symmetric-moment}
	\E\big[|N| \exp(-\b|N|)\big] 
		= \sqrt{\frac{2}{\pi}} \int_{0}^{+\infty} te^{-\b t} e^{-t^{2}/2} \di t \, .
\end{equation}
Then we proceed with an expansion of the integral.
\begin{align*}
	\int_{0}^{+\infty} t e^{-\b t} e^{-t^{2}/2} \di t
		&= e^{\b^{2}/2} \int_{0}^{+\infty} t e^{-\frac{(t+\b)^{2}}2} \di t
		= e^{\b^{2}/2} \int_{\b}^{+\infty} (x-\b) e^{-x^{2}/2} \di x \\
		&= e^{\b^{2}/2} \Big(  \int_{\b}^{+\infty} x e^{-x^{2}/2} \di x 
			- \b \int_{\b}^{+\infty} e^{-x^{2}/2} \di x \Big) \\
		&= 1 - \b e^{\b^{2}/2} \int_{\b}^{+\infty} e^{-t^{2}/2} \di t \, .
\end{align*}
Now we use twice the formula
\begin{equation*}
	\int_{x}^{+\infty} \frac{1}{t^{k}} e^{-t^{2}/2}\di t = \frac{e^{-x^{2}/2}}{x^{k+1}}
		- (k+1)\int_{x}^{+\infty} \frac{1}{t^{k+2}} e^{-t^{2}/2} \di t \, ,
\end{equation*}
which holds for all $x>0$ and all $k\geq 0$ by a simple integration by parts. This yields
\begin{equation*}
	e^{\b^{2}/2} \int_{\b}^{+\infty} e^{-t^{2}/2} \di t
	= \frac{1}{\b} - \frac{1}{\b^{3}} + \frac{3}{\b^{5}} 
		-15 \int_{\b}^{+\infty} \frac{1}{t^{6}} e^{-t^{2}/2} \di t
	\leq \frac{1}{\b} - \frac{1}{\b^{3}} + \frac{3}{\b^{5}} \, .
\end{equation*}
Then
\begin{equation*}
	\int_{0}^{+\infty} t e^{-\b t} g(t) \di t
		= 1 - \b e^{\b^{2}/2} \int_{\b}^{+\infty} e^{-t^{2}/2} \di t
		\geq \frac{1}{\b^{2}} - \frac{3}{\b^{4}} = \frac{1}{\b^{2}}\Big(1 - \frac{3}{\b^{2}}\Big)  \, .
\end{equation*}
Finally, as $\b \geq e$, one has $3/\b^{2} \leq 3/e^{2} \leq 1/2$; and combining with \eqref{eq:symmetric-moment} proves the claim.
\end{proof}

\section{Proofs of results from Section~\ref{sec:regular-designs}}
\label{sec:proof-examples-regular}

\subsection{Proof of Proposition~\ref{prop:regularity-constant-scale} (regularity at constant scales)}
\label{sec:proof-regularity-constant}

  First, note that Assumption~\ref{ass:small-ball} holds with $c = \eta^{-1}< c(K, \eta)$, since
  $\P (\ainnerp{u^*}{X} \leq t) \leq 1 \leq \eta^{-1} \cdot t$ for any $t \geq \eta$.
  We now show that Assumption~\ref{ass:twodim-marginals} also holds for $c = c(K, \eta)$.
  We start by writing, for any $v \in S^{d-1}$ such that $\innerp{u^*}{v} \geq 0$ and $s, t > 0$,
  \begin{equation*}
    \P \big( \ainnerp{u^*}{X} \leq s, \, \ainnerp{v}{X} \geq t \big)
    \geq \P (\ainnerp{v}{X} \geq t) - \P (\ainnerp{u^*}{X} > s)
    \, .
  \end{equation*}
  In order to lower bound the first term above, we apply the Paley-Zygmund inequality~\eqref{eq:paley-zygmund} to $Z = \innerp{v}{X}^2$ (with $\E [Z] = 1$), which gives
  \begin{equation*}
    \P \Big( \ainnerp{v}{X} \geq \frac{1}{\sqrt{2}} \Big)
    = \P \Big( \innerp{v}{X}^2 \geq \frac{1}{2} \E [ \innerp{v}{X}^2 ] \Big) 
    \geq \frac{1}{4} \frac{\E [ \innerp{v}{X}^2 ]^2}{\E [\innerp{v}{X}^4]} 
    \geq \frac{1}{4 K^4} \, ,
  \end{equation*}
  where the last inequality follows from the fact that $\norm{\innerp{v}{X}}_{\psi_1} \leq K$, which by Definition~\ref{def:psi-alpha} implies that $\norm{\innerp{v}{X}}_4 \leq 4 K /(2e) \leq K$.
  In addition, since $\norm{\innerp{u^*}{X}}_{\psi_1} \leq K$, Lemma~\ref{lem:sub-gamma-exponential} implies that
  \begin{equation}
    \label{eq:constant-regularity}
    \P \big( \ainnerp{u^*}{X} > 2 K \log (2K) \big)
    \leq e^{-2 \times 2 \log (2K)}
    = \frac{1}{16 K^4}
    \, .
  \end{equation}
  Combining the previous inequalities and using that $\norm{u^* - v} \leq \sqrt{2}$ and $\eta \leq e^{-1}$, we obtain that
  \begin{align*}
    &\P \Big( \ainnerp{u^*}{X} \leq \frac{2 K \log (2K)}{\eta} \cdot \eta , \, \ainnerp{v}{X} \geq \frac{\max \{ \norm{u^*-v}, \eta \}}{2} \Big)
    \geq \frac{1}{4 K^4} - \frac{1}{16 K^4} \\
    &= \frac{3}{16 K^4}
      \geq \frac{3 e \eta}{16 K^4}
      \geq \frac{\eta}{2 K^4}
      \, ,
  \end{align*}
  which shows that Assumption~\ref{ass:twodim-marginals} holds with $c = c_{K, \eta}$ given by~\eqref{eq:constant-regularity}.

\subsection{Proof of Proposition~\ref{prop:universal-reg-log-concave} (regularity of log-concave distributions)}
\label{sec:proof-regularity-log-concave}

In this section, we show that centered isotropic log-concave distributions satisfy Assumptions~\ref{ass:sub-exponential},~\ref{ass:small-ball} and~\ref{ass:twodim-marginals}, in every direction $u^* \in S^{d-1}$ and at any scale $\eta \in (0,e^{-1})$.

First, it is a standard fact that log-concave measures are sub-exponential.

\begin{lemma}
\label{lem:log-concave-sub-exponential}
For every isotropic log-concave random vector $X$ in $\R^{d}$, one has $\| \innerp{v}{X}\|_{\psi_{1}}\leq \sqrt{2}e$ for every $v \in S^{d-1}$.
\end{lemma}

\begin{proof}
Corollary~5.7 in \cite{guedon2014concentration} with $q=2$ shows that for all $v \in S^{d-1}$
and $p \geq 1$, 
\begin{equation*}
  \|\langle v, X \rangle \|_{p}
  \leq \frac{(p!)^{1/p}}{2^{1/2}} \|\innerp{v}{X}\|_{2}
  \leq \frac{p}{\sqrt{2}}
  \, .
\end{equation*}
Hence $\langle v, X \rangle$ is sub-exponential with $\|\langle v, X \rangle\|_{\psi_{1}} \leq \sqrt{2} e$.
\end{proof}

\begin{lemma}
\label{lem:mass-at-0}
Let $X$ be an isotropic random vector in $\R^{d}$ with log-concave distribution. Then for all $u \in S^{d-1}$ and all 
$t > 0$,
\begin{equation}
\label{eq:mass-at-0}
\P\big( | \langle u, X \rangle | \leq t \big) \leq 2 t \enspace .
\end{equation}
\end{lemma}
In other words, $X$ satisfies Assumption~\ref{ass:small-ball} with constant $c_{1} = 2$,
for all $u\in S^{d-1}$ and $\eta > 0$.

\begin{proof}
  The random variable $\innerp{u}{X}$ is log-concave since the random vector $X$ is, and additionally $\E [\innerp{u}{X}] = 0$ and $\E [\innerp{u}{X}^2] = 1$.
  It then follows from \cite[Proposition~B.2]{bobkov2019onedimensional} that $\innerp{u}{X}$ admits a density $f_u$ that is upper-bounded by $1$ on $\R$, which proves~\eqref{eq:mass-at-0}.
\end{proof}

We now show that the two-dimensional margin condition is satisfied at all scales and in every direction.
Note that the case of constant scales follows from the sub-exponential tails, by Proposition~\ref{prop:regularity-constant-scale}.
For small scales, the proof uses the fact that centered and isotropic low-dimensional log-concave densities are lower-bounded on a neighborhood of the origin.

\begin{fact}[\cite{lovasz2007geometry}, Theorem~5.14 (a)]
  \label{fac:two-dim-mass-at-0}
  There exist constants $0 < c_0, c_{1} \leq 1$ such that any isotropic and centered log-concave distribution on $\R^2$ admits a density $f$ such that $f(z) \geq c_0$ for any $z \in [-c_1, c_1]^2$.
\end{fact}

\begin{lemma}
  There exist constants $c \geq 1$ and $\eps \in (0, 1)$ such that 
  for any $d \geq 1$, any centered isotropic log-concave random vector $X$ in $\R^d$, any $\eta \in (0, \eps)$ and $u, v \in S^{d-1}$ with $\innerp{u}{v} \geq 0$, one has
  \begin{equation}
    \label{eq:two-dim-log-conc}
    \P\Big( \ainnerp{u}{X} \leq \eta ,  \, \ainnerp{v}{X}
    \geq \frac{1}{c} \max\big\{\eta, \|u - v\|\big\}\Big)
    \geq \frac{\eta}{c} \, .
  \end{equation}
\end{lemma}

\begin{proof}
  Fix $u, v \in S^{d-1}$ with $\innerp{u}{v} \geq 0$.
  By Remark~\ref{rem:equivalent-two-dim}, it suffices to establish the condition~\eqref{eq:twodim-marginals-comp} in terms of $\alpha = \sqrt{1 - \innerp{u}{v}^2}$ instead of $\norm{u - v}$.
  First, let us write $v = \innerp{u}{v} u + \sqrt{1 - \innerp{u}{v}^2} \, w = \sqrt{1 - \alpha^2} \, u + \alpha w$ for some $w \in S^{d-1}$ with $\innerp{u}{w} =0$.
  We therefore have to prove that
  \begin{equation}
    \label{eq:proof-twodim-log-conc-u-w}
    \P \Big( \ainnerp{u}{X} \leq \eta , \,
    \abs[\big]{\sqrt{1-\alpha^2} \innerp{u}{X} + \alpha \innerp{w}{X}} \geq \frac{\max \set{\eta, \alpha}}{c} \Big)
    \geq \frac{\eta}{c}
  \end{equation}
  for all $\eta \in (0, \eps)$, for some absolute constants $\eps \in (0, 1)$ and $c \geq 1$ to be specified.

  Now, note that the random vector $(\innerp{u}{X}, \innerp{w}{X})$ on $\R^2$ is log-concave, as a linear image of the log-concave random vector $X$.
  In addition, it is centered and isotropic since $X$ is.
  Hence, by Fact~\ref{fac:two-dim-mass-at-0}, it admits a density $f$ on $\R^2$ which is lower-bounded by $c_0$ on $[-c_1, c_1]^2$.
  Thus
  \begin{align}
    &\P \Big( \ainnerp{u}{X} \leq \eta , \,
    \abs[\big]{\sqrt{1-\alpha^2} \innerp{u}{X} + \alpha \innerp{w}{X}} \geq \frac{\max \set{\eta, \alpha}}{c} \Big) \nonumber \\
    &\geq c_0 \int_{[-c_1, c_1]^2} \indic[\Big]{\abs{x_1} \leq \eta, \, \abs[\big]{\sqrt{1-\alpha^2}\, x_1 + \alpha x_2} \geq \frac{\max \set{\eta, \alpha}}{c}} \di x_1 \di x_2 \nonumber \\
    &= 4 c_0 c_1^2 \, \P \Big( \abs{X_1} \leq \eta , \,
      \abs[\big]{\sqrt{1-\alpha^2}\, X_1 + \alpha X_2} \geq \frac{\max \set{\eta, \alpha}}{c} \Big)
      \, ,
      \label{eq:proof-twodim-log-conc-reduc-unif}
  \end{align}
  where $X_1, X_2$ are independent uniform random variables on $[-c_1, c_1]$.
  We now let $\eps = c_1$ and $\eta \in (0, \eps)$, and proceed to lower-bounding the right-hand side of~\eqref{eq:proof-twodim-log-conc-reduc-unif}.

  First, assume that $\alpha \leq \eta/(4 c_1)$.
  Let $E = \set{\eta/2 \leq \abs{X_1} \leq \eta}$.
  Observe that since $\eta \leq c_1$, one has
  \begin{equation*}
    \P (E)
    = \frac{\eta}{2 c_1}
    \, .
  \end{equation*}
  In addition, under $E$, one has $\abs{X_1} \leq \eta$ and (using that $\sqrt{1-\alpha^2} \geq 3/4$ as $\alpha \leq 1/4$)
  \begin{align*}
    \abs[\big]{\sqrt{1-\alpha^2}\, X_1 + \alpha X_2}
    &\geq \sqrt{1-\alpha^2} \abs{X_1} - \alpha \abs{X_2}
      \geq \frac{3}{4} \times \frac{\eta}{2} - \alpha c_1
      = \frac{3 \eta}{8} - \frac{\eta}{4}
      = \frac{\eta}{8}
      \, .
  \end{align*}
  Hence, we have in this case that 
  \begin{equation}
    \label{eq:proof-log-conc-eta}
    \P \Big( \abs{X_1} \leq \eta, \, \abs[\big]{\sqrt{1-\alpha^2}\, X_1 + \alpha X_2} \geq \frac{c_1 \max \set{\eta, \alpha}}{8} \Big)
    \geq \frac{\eta}{2 c_1}
    \, .
  \end{equation}

  Now, assume that $\alpha > \eta / (4 c_1)$.
  We have, for any $c > 0$,
  \begin{align*}
    &\Probab[\Big]{\abs{X_1} \leq \eta, \, \abs[\big]{\sqrt{1-\alpha^2}\, X_1 + \alpha X_2} \geq \frac{\alpha}{c}}
      = \Expect[\Big]{\indic{\abs{X_1} \leq \eta} \cdot \Probab[\Big]{\abs[\big]{\sqrt{1-\alpha^2}\, X_1 + \alpha X_2} \geq \frac{\alpha}{c} \Big| X_1} }
      .
  \end{align*}
  In addition, since $X_1, X_2$ are independent and $X_2$ has density bounded by $1/(2 c_1)$, we have 
  \begin{equation*}
    \Probab[\Big]{\abs[\big]{\sqrt{1-\alpha^2}\, X_1 + \alpha X_2} < \frac{c_1 \alpha}{2} \Big| X_1}
    \leq \frac{1}{2 c_1} \times c_1
    = \frac{1}{2}
    \, .
  \end{equation*}
  Therefore
  \begin{equation*}
    \Probab[\Big]{\abs{X_1} \leq \eta, \, \abs[\big]{\sqrt{1-\alpha^2}\, X_1 + \alpha X_2} \geq \frac{c_1 \alpha}{2}}
    \geq \frac{1}{2} \Probab{\abs{X_1} \leq \eta}
    = \frac{\eta}{2 c_1}
    \, .
  \end{equation*}
  Note also that in this case, $\alpha = \max\set{\alpha, \eta/(4 c_1)} \geq \max\set{\alpha, \eta}/4$, and therefore~\eqref{eq:proof-log-conc-eta} also holds.

  Hence, inequality~\eqref{eq:proof-log-conc-eta} holds for all values of $\alpha$, and plugging it into~\eqref{eq:proof-twodim-log-conc-reduc-unif} concludes the proof.
\end{proof}

\subsection{Proof of Proposition~\ref{prop:regularity-iid} (regularity for i.i.d.~coordinates)%
   }
\label{sec:proof-regularity-iid}

This section contains the proofs of the results from Section~\ref{sec:regul-prod-meas}.
Specifically, we show that random vectors $X$ with \iid sub-exponential coordinates (Assumption~\ref{ass:independent-coordinates}) satisfy Assumptions~\ref{ass:sub-exponential},~\ref{ass:small-ball} and~\ref{ass:twodim-marginals} down to a scale $\eta \asymp 1/\sqrt{d}$ in the ``diffuse'' direction $u^* = (1/\sqrt{d}, \dots, 1/\sqrt{d})$.

\subsubsection*{Assumption~\ref{ass:sub-exponential}}

We first recall the standard fact that a random vector with independent sub-exponential coordinates is itself sub-exponential.

\begin{lemma}
  \label{lem:sub-exponential-indep}
  If $X_1, \dots, X_d$ are independent centered real random variables with $\norm{X_j}_{\psi_1} \leq K$ for $1 \leq j \leq d$, for every $v = (v_j)_{1 \leq j \leq d} \in S^{d-1}$, letting $X = (X_j)_{1 \leq j \leq d}$ one has $\norm{\langle v, X\rangle}_{\psi_1} \leq 4 K$.
\end{lemma}

\begin{proof}
  By the sixth point of Lemma~\ref{lem:sub-gamma-exponential}, $X_j$ is $(K^2/2, K/2)$-sub-gamma for every $j$.
  By independence and the third point of the same lemma, $\langle v, X\rangle = \sum_{j=1}^d v_j X_j$ is sub-gamma, with parameters $K^2/2 \cdot \sum_{j=1}^d v_j^2 = K^2/2$ and $K/2 \cdot \max_{1\leq j \leq d} |v_j| \leq K/2$.
  Since the same also holds for $-\langle v, X\rangle = \langle -v, X\rangle$, the fifth point of Lemma~\ref{lem:sub-gamma-exponential} implies that $\norm{\innerp{v}{X}}_{\psi_1} \leq 2 \sqrt[3]{2e} \max (K/\sqrt{2}, 2\cdot K/2) = 2 \sqrt[3]{2 e} K \leq 4 K$.
\end{proof}

\subsubsection*{Assumption~\ref{ass:small-ball}: proof of Lemma~\ref{lem:onedim-indep}}

The second condition %
on one-dimensional marginals holds because, if $u \in S^{d-1}$ is sufficiently ``diffuse'', then the distribution of $\langle u, X\rangle$ is close to that of a standard Gaussian variable.
This fact follows
from the Berry-Esseen theorem (see, \eg, \cite{feller1968berry}); we will use the version with small numerical constants from~\cite{shevtsova2010improvement,tyurin2012refinement}.

\begin{lemma}[\cite{tyurin2012refinement}, Theorem~1]
  \label{lem:berry-esseen-1d}
  Let $Z_1, \dots, Z_d$ be independent centered random variables with $\sum_{j=1}^d \E [Z_j^2] = 1$.
  Let $Z = \sum_{j=1}^d Z_j$ and $G \sim \gaussdist (0, 1)$.
  Then, for every $t \in \R$, one has
  \begin{equation}
    \label{eq:berry-esseen-1d}
    \big| \P (Z \leq t) - \P (G \leq t) \big|
    \leq 0.56 \cdot \sum_{j=1}^d \E \big[ |Z_j|^3 \big]
    \, .
  \end{equation}
\end{lemma}

We now proceed with the proof of Lemma~\ref{lem:onedim-indep}, which states that Assumption~\ref{ass:small-ball} holds.

\begin{proof}[Proof of Lemma~\ref{lem:onedim-indep}]
  First, applying Lemma~\ref{lem:berry-esseen-1d} to $-Z_1, \dots, -Z_d$ %
  and $-s$ gives a similar bound as~\eqref{eq:berry-esseen-1d} for $\P (Z < s)$.
  After taking differences we deduce that, under the assumptions of Lemma~\ref{lem:berry-esseen-1d}, for every $s, t \in \R$ with $s \leq t$, one has
  \begin{equation}
    \label{eq:berry-esseen-1d-twosided}
    \big| \P (s \leq Z \leq t) - \P (s \leq G \leq t) \big|
    \leq 1.12 \cdot \sum_{j=1}^d \E \big[ |Z_j|^3 \big]
    \, .
  \end{equation}
  We apply this inequality to $t \in [K^3 \norm{u}_3^3, 1]$, $s = -t$ and $Z_j = u_j X_j$, so that $\E [Z_j] = 0$, $\sum_{j=1}^d \E [Z_j^2] = \sum_{j=1}^d u_j^2 = 1$, and $\E [ |Z_j|^3 ] = \abs{u_j}^3 \norm{X_j}_3^3 \leq \abs{u_j}^3 (3 K / 2e)^3$.
  As $Z %
  = \langle u, X\rangle$, 
  \begin{equation}
    \label{eq:onedim-indep-berry}
    \big| \P (\abs{\innerp{u}{X}} \leq t) - \P (|G| \leq t) \big| 
    \leq 1.12 \cdot \sum_{j=1}^d \Big( \frac{3 K}{2 e} \Big)^3 \abs{u_j}^3 
    \leq \frac{K^3}{5} \norm{u}_3^3
      \, ,
  \end{equation}
  where we used that $1.12 \times (\frac{3}{2e})^3 \leq 1/5$.
  Now, since the density of $G$ is between $e^{-1/2}/\sqrt{2\pi}$ and $1/\sqrt{2\pi}$ on $[-1, 1]$, one has
  \begin{equation*}
    \frac{2t}{\sqrt{2\pi e}} \leq \P (|G| \leq t) \leq \frac{2 t}{\sqrt{2\pi}}
    \, .
  \end{equation*}
  Plugging these inequalities into~\eqref{eq:onedim-indep-berry} and using that $K^3 \norm{u}_3^3 \leq t$ gives
  \begin{equation*}
    \Big( \frac{2}{\sqrt{2\pi e}} - \frac{1}{5} \Big) t
    \leq \P (|\innerp{u}{X}| \leq t)
    \leq \Big( \sqrt{\frac{2}{\pi}} + \frac{1}{5} \Big) t
    \, ,
  \end{equation*}
  which implies~\eqref{eq:onedim-indep} by further bounding numerical constants.
\end{proof}

\subsubsection*{Assumption~\ref{ass:twodim-marginals}: proof of Lemma~\ref{lem:2d-margin-indep}}
We now establish the two-dimensional margin condition.

\begin{proof}[Proof of Lemma~\ref{lem:2d-margin-indep}]
  As discussed in Section~\ref{sec:regul-prod-meas}, the idea of the proof is to perturb the vector $X$ by a random permutation of its coordinates, and use the fact that such transformations do not affect the distribution of $X$ nor the value of $\innerp{u^*}{X}$, but induce some variability in the quantity $\innerp{v}{X}$.
  
  \parag{Perturbation by random permutations}
Let $\sigma$ be a permutation of $\set{1, \dots, d}$.
For $x = (x_1, \dots, x_d) \in \R^d$, we let $x^\sigma = (x_{\sigma (1)}, \dots, x_{\sigma (d)})$ denote the vector obtained by permuting the coordinates of $x$ by $\sigma$.
First, since $X_1, \dots, X_d$ are \iid, the vector $X^\sigma$ has the same distribution as $X$.
In addition, one has
\begin{equation*}
  \innerp{u^*}{X^\sigma}
  = \frac{1}{\sqrt{d}} \sum_{i=1}^d X_{\sigma (i)}
  = \frac{1}{\sqrt{d}} \sum_{i=1}^d X_{i}
  = \innerp{u^*}{X}
  \, .
\end{equation*}
It follows that, for any $v \in S^{d-1}$ and $s > 0$,
\begin{equation*}
  \P \big( \ainnerp{u^*}{X} \leq \eta, \, \ainnerp{v}{X} \geq s \big)
  = \P \big( \ainnerp{u^*}{X} \leq \eta, \, \ainnerp{v}{X^\sigma} \geq s \big)
  \, .
\end{equation*}
From now on, we let $\sigma$ denote a random permutation, drawn uniformly from the set $\Sym_d$ of all permutations of $\set{1, \dots, d}$ and independent of $X$.
We let $\P_\sigma$ and $\E_\sigma$ respectively denote the probability and expectation with respect to $\sigma$, conditionally on $X$.
From the equality above applied to any $\sigma' \in \Sym_d$, one has
\begin{align}
  \P \big( \ainnerp{u^*}{X} \leq \eta, \, \ainnerp{v}{X} \geq s \big)
  &= \frac{1}{d!} \sum_{\sigma' \in \Sym_d} \P \big( \ainnerp{u^*}{X} \leq \eta, \, \ainnerp{v}{X^{\sigma'}} \geq s \big) \nonumber \\
  &= \E \big[ \P \big( \ainnerp{u^*}{X} \leq \eta, \, \ainnerp{v}{X^{\sigma}} \geq s \, | \, \sigma \big) \big] \nonumber \\
  &= \E \big[ \P_\sigma \big( \ainnerp{u^*}{X} \leq \eta, \, \ainnerp{v}{X^{\sigma}} \geq s \big)  \big] \nonumber \\
  &= \E \big[ \bm 1 \{ \ainnerp{u^*}{X} \leq \eta \} \P_\sigma \big( \ainnerp{v}{X^{\sigma}} \geq s \big) \big]
  \label{eq:conditioning-exchangeable}
    \, .
\end{align}
Hence,
in order to lower bound the left-hand side of~\eqref{eq:conditioning-exchangeable}, it suffices to lower bound $\P_\sigma \big( \ainnerp{v}{X^{\sigma}} \geq s \big)$
when $X$ satisfies $\ainnerp{u^*}{X} \leq \eta$ (we will actually require additional symmetric conditions on $X$, but we omit them here for simplicity).
In other words, we need to show that for such values of $X$, the fraction of permutations $\sigma \in \Sym_d$ such that $\ainnerp{v}{X^{\sigma}} \geq s$ is lower-bounded.

We will achieve this by resorting to the Paley-Zygmund inequality (\eg, \cite[eq.~(6.15) p.~181]{talagrand2021upper}), which asserts that for any non-negative random variable $Z$ with $0 < \E [Z^2] < +\infty$, one has
\begin{equation}
  \label{eq:paley-zygmund}
  \P \Big( Z \geq \frac{1}{2} \E [Z] \Big)
  \geq \frac{1}{4} \frac{\E [Z]^2}{\E [Z^2]}
  \, .
\end{equation}
Applying this inequality to the random variable $Z = \innerp{v}{X^\sigma}^2$ conditionally on $X$ gives
\begin{equation}
  \label{eq:paley-zygmund-permutation}
  \P_\sigma \Big( \ainnerp{v}{X^\sigma} \geq \frac{1}{\sqrt{2}} \E_\sigma [\innerp{v}{X^\sigma}^2]^{1/2} \Big)
  \geq \frac{1}{4} \frac{\E_\sigma [\innerp{v}{X^\sigma}^2]^2}{\E_\sigma [\innerp{v}{X^\sigma}^4]}
  \, .
\end{equation}
We are therefore led to bound $\E_\sigma [\innerp{v}{X^\sigma}^2]^{1/2}$ from below and $\E_\sigma [\innerp{v}{X^\sigma}^4]^{1/4}$ from above, ideally to conclude that these two quantities are
both of the order of the value from Lemma~\ref{lem:2d-margin-indep}.
The advantage of this approach is that it reduces to evaluating expectations of polynomials of the variables $X_{\sigma (i)}, 1 \leq i \leq d$ under the uniform distribution on $\Sym_d$, which can be computed exactly.

\parag{Lower bound on the second moment}

Denote for $p \in \N$,
\begin{align*}
  \phi
  = \phi (X)
  = \innerp{u^*}{X}
  = \frac{1}{\sqrt{d}} \sum_{i=1}^d X_i \, ,
  \qquad
  \mu_p
  = \mu_p (X)
  = \frac{1}{d} \sum_{i=1}^d X_i^p
    \, .
\end{align*}
In particular, one has $|\mu_p| \leq \mu_4^{p/4}$ for $1 \leq p \leq 4$.
In the following, we assume that $X$ satisfies $\mu_2 (X) \geq 1/2$ and $|\phi (X)| \leq \eta \leq 1$.

Now, for any $v \in S^{d-1}$, 
\begin{equation*}
  \E_\sigma \big[ \innerp{v}{X^\sigma}^2 \big]
  = \E_\sigma \bigg[ \bigg( \sum_{i=1}^d v_i X_{\sigma (i)} \bigg)^2 \bigg]
  = \sum_{1 \leq i,j \leq d} v_i v_j \E_\sigma [X_{\sigma (i)} X_{\sigma (j)}]
  \, .
\end{equation*}
For $i=j$, since $\sigma (i)$ is uniformly distributed on $\set{1, \dots, d}$ one has
\begin{equation*}
  \E_{\sigma} [X_{\sigma (i)} X_{\sigma (j)}]
  = \E_{\sigma} [X_{\sigma (i)}^2 ]
  = \frac{1}{d} \sum_{k=1}^d X_{k}^2
  = \mu_2
  \, .
\end{equation*}
On the other hand, if $i \neq j$, then $(\sigma (i), \sigma (j))$ is distributed uniformly on pairs $(k,l)$ such that $k \neq l$, thus
\begin{equation*}
  \E_\sigma [X_{\sigma(i)} X_{\sigma (j)}]
  = \frac{1}{d(d-1)} \sum_{k \neq l} X_k X_l
  = \frac{1}{d (d-1)} \bigg\{ \bigg( \sum_{k=1}^d X_k \bigg)^2 - \sum_{k=1}^d X_k^2 \bigg\}
  = \frac{\phi^2 - \mu_2}{d-1}
  .
\end{equation*}
Combining the previous two equations, we get for any $i,j$ that
\begin{equation*}
  \E_\sigma [X_{\sigma(i)} X_{\sigma (j)}]
  = \frac{\phi^2 - \mu_2}{d-1} + \Big( \mu_2 + \frac{\mu_2 - \phi^2}{d-1} \Big) \indic{i= j}
  \, .
\end{equation*}
Hence,
\begin{align*}
  \E_\sigma \big[ \innerp{v}{X^\sigma}^2 \big]
  &= \frac{\phi^2 - \mu_2}{d-1} \bigg( \sum_{i=1}^d v_i \bigg)^2 + \Big( \mu_2 + \frac{\mu_2 - \phi^2}{d-1} \Big) \sum_{i=1}^d v_i^2 \\
  &= \big( \phi^2 - \mu_2 \big) \frac{d}{d-1} \innerp{u^*}{v}^2 + \mu_2 + \frac{\mu_2 - \phi^2}{d-1} \\
  &= \frac{d}{d-1} \Big[ \mu_2 (1 - \innerp{u^*}{v}^2) + \phi^2 \innerp{u^*}{v}^2 - \frac{\phi^2}{d} \Big] \\
  &\geq \mu_2 (1 - \innerp{u^*}{v}^2) + \phi^2 \innerp{u^*}{v}^2 - \frac{\phi^2}{d}
    \, .
\end{align*}
Recalling that $\mu_2 \geq 1/2$, that $|\phi| \leq \eta \leq 1$ and $d \geq 2025$, then
either $\innerp{u^*}{v}^2 \geq 1/4$ and
\begin{equation*}
  \mu_2 (1 - \innerp{u^*}{v}^2) + \phi^2 \innerp{u^*}{v}^2 - \frac{\phi^2}{d}
  \geq \frac{1 - \innerp{u^*}{v}^2}{2} + 0.97 \innerp{u^*}{v}^2 \phi^2
  \, ,
\end{equation*}
or $\innerp{u^*}{v}^2 < 1/4$ and then
\begin{equation*}
  \mu_2 (1 - \innerp{u^*}{v}^2) + \phi^2 \innerp{u^*}{v}^2 - \frac{\phi^2}{d}
  \geq \frac{3}{8} - \frac{1}{2025}
  \geq 0.37 \big[ 1 - \innerp{u^*}{v}^2 + \innerp{u^*}{v}^2 \phi^2 \big]
  \, .
\end{equation*}
Combining the previous inequalities, we get that in all cases 
\begin{equation}
  \label{eq:perm-lower-second-moment}
  \E_\sigma \big[ \innerp{v}{X^\sigma}^2 \big]
  \geq 0.37 \big[ 1 - \innerp{u^*}{v}^2 + \innerp{u^*}{v}^2 \phi^2 \big]
  \, .
\end{equation}

\parag{Upper bound on the fourth moment}

We now turn to the control the conditional fourth moment.
Let $v \in S^{d-1}$ such that $\innerp{u^*}{v} \geq 0$; we may write $v = \sqrt{1-\alpha^2} u^* + \alpha w$ where $\alpha = \sqrt{1-\innerp{u^*}{v}^2}$ and $w \in S^{d-1}$ is such that $\innerp{u^*}{w} = 0$.
We then have
\begin{align}
  \E_\sigma \big[ \innerp{v}{X^\sigma}^4 \big]
  &=\E_\sigma \big[ \big( \alpha \innerp{w}{X^\sigma} + \sqrt{1-\alpha^2} \innerp{u}{X^\sigma} \big)^4 \big] \nonumber \\
  &\leq 8 \, \E_\sigma \big[ \alpha^4 \innerp{w}{X^\sigma}^4 + (1-\alpha^2)^2 \innerp{u}{X^\sigma}^4 \big] \nonumber \\
  &= 8 \big\{ (1 - \innerp{u^*}{v}^2)^2 \, \E_\sigma \big[ \innerp{w}{X^\sigma}^4 \big] + \innerp{u^*}{v}^4 \phi^4 \big\}
    \label{eq:perm-decomp-fourth}
    \, .
\end{align}
In light of~\eqref{eq:perm-decomp-fourth}, it remains to show that $\E_\sigma [\innerp{w}{X^\sigma}^4] \lesssim_\kappa 1$.

We start by writing:
\begin{equation}
  \label{eq:fourth-moment-expansion}
  \E_\sigma \big[ \innerp{w}{X^\sigma}^4 \big]
  = \sum_{1 \leq i,j,k,l \leq d} w_i w_j w_k w_l \E [ X_{\sigma (i)} X_{\sigma (j)} X_{\sigma (k)} X_{\sigma (l)} ]
    \, .
\end{equation}
We abbreviate ``pairwise distinct'' (indices in $\set{1, \dots, d}$) by p.d., and denote for $i, j, k, l$ p.d.,
\begin{align*}
  \alpha_4
  &= \E_\sigma [X_{\sigma (i)}^4] \\
  \alpha_{31}
  &= \E_\sigma [ X_{\sigma (i)}^3 X_{\sigma (j)} ] \\
  \alpha_{22}
  &= \E_\sigma [ X_{\sigma (i)}^2 X_{\sigma (j)}^2 ] \\
  \alpha_{211}
  &= \E_\sigma [ X_{\sigma (i)}^2 X_{\sigma (j)} X_{\sigma (k)} ] \\
  \alpha_{1111}
  &= \E_\sigma [ X_{\sigma (i)} X_{\sigma (j)} X_{\sigma (k)} X_{\sigma (l)} ]
    \, ;
\end{align*}
these quantities are independent of $i,j,k,l$ p.d.~since $\sigma$ is distributed uniformly on the symmetric group, hence $(\sigma (i), \sigma (j), \sigma (k), \sigma (l))$ is distributed uniformly on the set of p.d.~indices.
Hence, collecting the terms in the right-hand side of~\eqref{eq:fourth-moment-expansion} depending on the distinct indices, we obtain
\begin{align}
  \E_\sigma \big[ \innerp{w}{X^\sigma}^4 \big]
  &= \bigg( \sum_{i} w_i^4 \bigg) \alpha_4 + 4 \bigg( \sum_{i,j \pd} w_i^3 w_j \bigg) \alpha_{31} + 3 \bigg( \sum_{i,j \pd} w_i^2 w_j^2 \bigg) \alpha_{22} + \nonumber \\
  & \qquad + 6 \bigg( \sum_{i,j,k \pd} w_i^2 w_j w_k \bigg) \alpha_{211} + \bigg( \sum_{i,j,k,l \pd} w_i w_j w_k w_l \bigg) \alpha_{1111}
    \, .
    \label{eq:fourth-moment-perm}
\end{align}

We control the sum in~\eqref{eq:fourth-moment-perm} by separately controlling the $\alpha_{\cdot}$ terms (that depend on $X$) and their coefficients depending on $w$.
The control of the former terms is simple, as we simply bound all these terms by the empirical fourth moment $\mu_4$:
for every $1 \leq r \leq 4$ and $\iota_1 \geq \dots \geq \iota_r \geq 1$ such that $\iota_1 + \dots + \iota_r = 4$, we have
\begin{equation}
  \label{eq:alpha-bound-fourth}
  |\alpha_{\iota_1, \dots, \iota_r} |
  \leq \mu_{4}
  \, .
\end{equation}
To show~\eqref{eq:alpha-bound-fourth}, first note that since $\sigma (1)$ is uniformly distributed in $\set{1, \dots, d}$, we have
\begin{equation*}
  \alpha_4
  = \E_{\sigma} [ X_{\sigma (1)}^4 ]
  = \frac{1}{d} \sum_{i=1}^d X_i^4
  = \mu_4
  \, .
\end{equation*}
Now for $\iota_1, \dots, \iota_r$ as above, Hölder's inequality (with $\iota_1/4 + \dots + \iota_r/4 = 1$) implies that
\begin{align*}
  |\alpha_{\iota_1, \dots, \iota_r}|
  &\leq \E_{\sigma} [ |X_{\sigma (1)}|^{\iota_1} \cdots |X_{\sigma (r)}|^{\iota_r} ] 
  = \E_{\sigma} [ (X_{\sigma (1)}^4)^{\iota_1/4} \cdots (X_{\sigma (r)}^4)^{\iota_r/4} ] \\
  &\leq \E_\sigma [ X_{\sigma (1)}^{4} ]^{\iota_1/4} \cdots \E_\sigma [ X_{\sigma (r)}^{4} ]^{\iota_r/4}
    = \mu_4
    \, .
\end{align*}

We now turn to the control of the coefficients in~\eqref{eq:fourth-moment-perm} that depend on $w$.
Although one could in principle use the same method as above, namely Hölder's inequality combined with the fact that $\norm{w}_4^4 \leq \norm{w}_2^4 = 1$, this would result in a highly suboptimal bound in $O(d^2)$.
In order to improve this bound, we exploit the additional information that $w$ is orthogonal to $u^*$, namely
\begin{equation*}
  0
  = \innerp{u^*}{w}
  = \frac{1}{\sqrt{d}} \sum_{i=1}^d w_i
  \, ,
\end{equation*}
so that $\sum_{i} w_i = 0$.
We will therefore decompose the sums in~\eqref{eq:fourth-moment-perm} by making the quantities $\sum_{i} w_i = 0$ and $\sum_{i} w_i^2 = 1$ appear.

For the first term, we have
\begin{equation*}
  0
  \leq \sum_{i} w_i^4
  \leq \sum_{i} w_i^2
  = 1
  \, .
\end{equation*}
For the second term, we write
\begin{align*}
  \sum_{i,j \, \text{p.d.}} w_i^3 w_j
  &= \Big( \sum_{i} w_i^3 \Big) \Big( \sum_{j} w_j \Big) - \sum_{i} w_i^4
    = - \sum_{i} w_i^4 \in [-1, 0]
    \, .
\end{align*}
For the third term, 
\begin{align*}
  \sum_{i, j \, \text{p.d.}} w_i^2 w_j^2
  &= \Big( \sum_{i} w_i^2 \Big)^2 - \sum_{i} w_i^4 
    = 1 - \sum_{i} w_i^4 \in [0, 1]
    \, .
\end{align*}
For the fourth term, by distinguishing the different possible configurations of $i,j,k \in \set{1, \dots, d}$,
\begin{equation}
  \label{eq:decomp-211}
  \sum_{i,j,k \, \text{p.d.}} w_i^2 w_j w_k
  = \Big( \sum_{i} w_i^2 \Big) \Big( \sum_{j} w_j \Big) \Big( \sum_{k} w_k \Big) - \sum_{i} w_i^4 - \sum_{i,j \, \text{p.d.}} w_i^2 w_j^2 - 2 \sum_{i,j \, \text{p.d.}} w_i^3 w_j
    \, .
\end{equation}
Plugging the previous identities in~\eqref{eq:decomp-211}, we obtain
\begin{align*}
  \sum_{i,j,k \, \text{p.d.}} w_i^2 w_j w_k
  &= - \sum_{i} w_i^4 - \Big( 1 - \sum_{i} w_i^4 \Big) - 2 \Big( - \sum_{i} w_i^4 \Big) \\
  &= 2 \sum_{i} w_i^4 - 1 \in [-1, 1]
    \, .
\end{align*}
Finally, for the fifth term, we write (collecting the terms similarly to~\eqref{eq:fourth-moment-perm})
\begin{align}
  \label{eq:decomp-1111}
  \sum_{i,j,k,l \, \text{p.d.}} w_i w_j w_k w_l
  &= \Big( \sum_{i} w_i \Big) \Big( \sum_{j} w_j \Big) \Big( \sum_{k} w_k \Big) \Big( \sum_{l} w_l \Big)
    - \sum_{i} w_i^4 - \nonumber \\
  &\ - 4 \sum_{i, j \, \text{p.d.}} w_i^3 w_j - 3 \sum_{i,j \, \text{p.d.}} w_i^2 w_j^2 - 6 \sum_{i,j,k \, \text{p.d.}} w_i^2 w_j w_k
    \, .
\end{align}
Using the identities for the previous four terms, equation~\eqref{eq:decomp-1111} becomes
\begin{align*}
  \sum_{i,j,k,l \, \text{p.d.}} w_i w_j w_k w_l
  &= - \sum_{i} w_i^4 - 4 \Big( - \sum_{i} w_i^4 \Big) - 3 \Big( 1 - \sum_{i} w_i^4 \Big) - 6 \Big( 2 \sum_{i} w_i^4 - 1 \Big) \\
  &= - 6 \sum_{i} w_i^4 + 3 \in [-3, 3]
    \, .
\end{align*}
Finally, injecting the previous bounds into the decomposition~\eqref{eq:fourth-moment-perm}, we obtain
\begin{align*}
  \E_\sigma \big[ \innerp{w}{X^\sigma}^4 \big]
  &\leq \Big| \sum_{i} w_i^4 \Big| \cdot |\alpha_4| + 4 \bigg| \sum_{i,j \pd} w_i^3 w_j \bigg| \cdot |\alpha_{31}| + 3 \bigg| \sum_{i,j \pd} w_i^2 w_j^2 \bigg| \cdot \abs{\alpha_{22}} + \\
  & \qquad + 6 \bigg| \sum_{i,j,k \pd} w_i^2 w_j w_k \bigg| \cdot \abs{\alpha_{211}} + \bigg| \sum_{i,j,k,l \pd} w_i w_j w_k w_l \bigg| \cdot \abs{\alpha_{1111}} \\
  &\leq \big( 1 + 4 \times 1 + 3 \times 1 + 6 \times 1 + 3 \big) \mu_4
    = 17 \, \mu_4
    \, ,
\end{align*}
which combined with~\eqref{eq:perm-decomp-fourth} gives
\begin{equation}
  \label{eq:final-bound-fourth-moment-perm}
  \E_\sigma \big[ \innerp{v}{X^\sigma}^4 \big]
  \leq 8 \big\{ 17 (1 - \innerp{u^*}{v}^2)^2 \mu_4 + \innerp{u^*}{v}^4 \phi^4 \big\}
  \, .
\end{equation}

\parag{Symmetric condition}

So far, we have established a lower bound on the second moment $\E_\sigma [ \innerp{v}{X^\sigma}^2 ]$ and an upper bound on the fourth moment $\E_\sigma [ \innerp{v}{X^\sigma}^4 ]$, both over the random permutation $\sigma$ and conditionally on $X$.
These upper and lower bounds are of the desired order whenever $X$ satisfies the following three conditions: $\eta/2 \leq \ainnerp{u^*}{X} \leq \eta$, $\mu_2 (X) \geq 1/2$, and $\mu_4 (X) = O_\kappa (1)$.
We are therefore reduced to lower-bounding the probability that $X$ simultaneously satisfies those three conditions, which are symmetric in the coordinates of $X$.

We thus establish a lower bound on
$\P \big(\eta/2 \leq \ainnerp{u^*}{X} \leq \eta, \, \mu_2 (X) \geq 1/2 , \, \mu_4 (X) \leq 2 \kappa^4 \big)$.
We start by writing
\begin{align}
  & \P \big( \eta/2 \leq \ainnerp{u^*}{X} \leq \eta, \, \mu_2 (X) \geq 1/2 , \, \mu_4 (X) \leq 2 \kappa^4 \big) \nonumber \\
  &= \P \big( \eta/2 \leq \ainnerp{u^*}{X} \leq \eta \big) - \P \big( \set{\ainnerp{u^*}{X} \leq \eta} \cap \set{\mu_2 (X) < 1/2 \text{ or } \mu_4 (X) > 2 \kappa^4} \big) \nonumber \\
  &\geq \P \big( \eta/2 \leq \ainnerp{u^*}{X} \leq \eta \big) - \P \big( \mu_2 (X) < 1/2 \big) - \P \big( \mu_4 (X) > 2 \kappa^4 \big)
    \label{eq:symmetric-diff}
    \, .
\end{align}
Now, applying the Berry-Esseen inequality (Lemma~\ref{lem:berry-esseen-1d}) and proceeding as in the proof of Lemma~\ref{lem:onedim-indep}, using that $\E [(|X_i|/\sqrt{d})^3] \leq \E [|X_i|^8]^{3/8}/d^{3/2} \leq \kappa^3/d^{3/2}$, we get
\begin{equation}
  \label{eq:symmetric-m1}
  \P \big( \eta/2 \leq \ainnerp{u^*}{X} \leq \eta \big)
  \geq \frac{\eta}{\sqrt{2\pi e}} - \frac{2.24 \kappa^3}{\sqrt{d}}
  \geq 0.24 \, \eta - \frac{2.25 \kappa^3}{\sqrt{d}}
  \, .
\end{equation}

We now upper bound $\P (\mu_4 (X) > 2 \kappa^4)$.
Applying Chebyshev's inequality to $\sum_{i=1}^d X_i^4$ gives, for any $t > 0$,
\begin{equation*}
  \P \Big( \Big| \frac{1}{d} \sum_{i=1}^d X_i^4 - \E [X_1^4] \Big| > t \Big)
  \leq \frac{\E [X_1^8]}{d \cdot t^2}
  \leq \frac{\kappa^8}{d \cdot t^2}
  \, .
\end{equation*}
In particular, taking $t = \kappa^4$, applying the triangle inequality and using that $\E [X_1^4] \leq \E [X_1^8]^{1/2} \leq \kappa^4$ by assumption, we get
\begin{equation}
  \label{eq:symmetric-m4}
  \P \big( \mu_4 (X)
  > 2 \kappa^4 %
  \big)
  \leq 1/d
  \, .
\end{equation}
Likewise, Chebyshev's inequality implies that
\begin{align}
  \label{eq:symmetric-m2}
  \P \big( \mu_2 (X) < 1/2 \big)
  \leq 
  \P \Big( \Big| \frac{1}{d} \sum_{i=1}^d X_i^2 - 1 \Big| > \frac{1}{2} \Big)
  \leq \frac{4 \,\E [X_1^4]}{d}
  \leq \frac{4 \kappa^4}{d}
  \, .
\end{align}

Plugging inequalities~\eqref{eq:symmetric-m1},~\eqref{eq:symmetric-m4} and~\eqref{eq:symmetric-m2} into~\eqref{eq:symmetric-diff} gives
\begin{equation*}
  \P \big( \eta/2 \leq \ainnerp{u^*}{X} \leq \eta, \, \mu_2 (X) \geq 1/2 , \, \mu_4 (X) \leq 2 \kappa^4 \big)
  \geq 0.24 \, \eta - \frac{2.25 \kappa^3}{\sqrt{d}} - \frac{1}{d} - \frac{4 \kappa^4}{d}
  \, ,
\end{equation*}
which is larger than $%
0.12\, \eta$ whenever $\eta \geq \max (45 \kappa^3/\sqrt{d}, 80 \kappa^4/d)$.
Now since $d \geq 2025 \kappa^6$ by assumption, one has $\sqrt{d} \geq 45 \kappa^3 \geq 45 \kappa$ and thus $80 \kappa^4/d \leq 80 \kappa^3 / (45 \sqrt{d}) < 45 \kappa^3 / \sqrt{d}$; therefore, the previous condition reduces to $\eta \geq 45 \kappa^3/\sqrt{d}$, which is satisfied by assumption.

\parag{Putting things together}
We now conclude the proof.
Define the event $E$ by
\begin{equation*}
  E
  = \big\{ \eta/2 \leq \ainnerp{u^*}{X} \leq \eta, \, \mu_2 (X) \geq 1/2 , \, \mu_4 (X) \leq 2 \kappa^4 \big\}
  \, ,
\end{equation*}
so that $\P (E) \geq 0.12 \eta$ by the above.
In addition, it follows respectively from~\eqref{eq:perm-lower-second-moment} and~\eqref{eq:final-bound-fourth-moment-perm} that, under the event $E$,
\begin{align*}
  \E_\sigma \big[ \innerp{v}{X^\sigma}^2 \big]
  &\geq 0.37 \big[ 1 - \innerp{u^*}{v}^2 + \innerp{u^*}{v}^2 (\eta/2)^2 \big] \, ; \\
  \E_\sigma \big[ \innerp{v}{X^\sigma}^4 \big]
  &\leq 8 \big\{ 34 \kappa^4 (1 - \innerp{u^*}{v}^2)^2 + \innerp{u^*}{v}^4 \eta^4 \big\}
  \, .
\end{align*}
Plugging these upper and lower bounds into~\eqref{eq:paley-zygmund-permutation} gives:
\begin{align*}
  &\P_\sigma \Big( \ainnerp{v}{X^\sigma} \geq \frac{0.6}{\sqrt{2}} \big[ 1 - \innerp{u^*}{v}^2 + \innerp{u^*}{v}^2 \eta^2/4 \big]^{1/2} \Big) \\
  &\geq \frac{1}{4} \frac{0.37^2 \big[ 1 - \innerp{u^*}{v}^2 + \innerp{u^*}{v}^2 \eta^2/4 \big]^2}{8 \big[ 34 \kappa^4 (1 - \innerp{u^*}{v}^2)^2 + \innerp{u^*}{v}^4 \eta^4 \big]} \\
  &\geq \frac{0.37^2}{32} \frac{(1 - \innerp{u^*}{v}^2)^2 + \innerp{u^*}{v}^4 \phi^4/16}{34 \kappa^4 (1 - \innerp{u^*}{v}^2)^2 + \innerp{u^*}{v}^4 \eta^4} 
    \geq \frac{1}{8000 \kappa^4}
    \, .
\end{align*}

Now, let $s = \frac{0.6}{\sqrt{2}} [ 1 - \innerp{u^*}{v}^2 + \innerp{u^*}{v}^2 \eta^2/4 ]^{1/2}$.
From~\eqref{eq:conditioning-exchangeable} and the above, we obtain
\begin{align*}
  \P \big( \ainnerp{u^*}{X} \leq \eta, \, \ainnerp{v}{X} \geq s \big)
  &= \E \big[ \bm 1 \{ \ainnerp{u^*}{X} \leq \eta \} \P_\sigma \big( \ainnerp{v}{X^{\sigma}} \geq s \big) \big] \\
  &\geq \E \big[ \bm 1_E \cdot \P_\sigma \big( \ainnerp{v}{X^{\sigma}} \geq s \big) \big] \\
  &\geq \frac{\P (E)}{8000 \kappa^4} 
    \geq \frac{0.12 \eta}{8000 \kappa^4}
    \geq \frac{\eta}{70\,000 \kappa^4}
    \, .
\end{align*}
To conclude, note that
\begin{align*}
  s
  &= \frac{0.6}{\sqrt{2}} \big[ 1 - \innerp{u^*}{v}^2 + \innerp{u^*}{v}^2 \eta^2/4 \big]^{1/2} 
    \geq \frac{0.6}{\sqrt{2}} \max \big\{ 1 - \innerp{u^*}{v}^2,  \eta^2/4 \big\}^{1/2}
    \, ,
\end{align*}
and that by Lemma~\ref{lem:polar-cartesian}, if $\innerp{u^*}{v} \geq 0$ then $\sqrt{1 - \innerp{u^*}{v}^2} \geq \norm{u^* - v}/\sqrt{2}$; the numerical constant in Lemma~\ref{lem:2d-margin-indep} is obtained by lower-bounding $0.6/(2 \sqrt{2}) > 0.2$.

Finally, the last part of Lemma~\ref{lem:2d-margin-indep} follows from~\eqref{eq:2d-margin-indep}, since under Assumption~\ref{ass:independent-coordinates} one has
$\E [X_1^4]^{1/4} \leq \kappa = \frac{4}{2 e} \norm{X_1}_{\psi_1} \leq \frac{2}{e} K$, which gives the desired claims by substituting for $\kappa$ and bounding the numerical constants.
\end{proof}

\subsubsection*{Sketch of the argument to obtain the $d^{-1/4}$ scaling}

We now provide an (incomplete) high-level sketch of the argument alluded to in Section~\ref{sec:regul-prod-meas}, that leads to a nontrivial guarantee by combining Gaussian approximation with approximate separation of supports.

The main idea is that an arbitrary vector $v \in S^{d-1}$ either admits a ``dense'' sub-vector $v_I = (v_i)_{i \in I}$ (for some $I \subset \{1, \dots, d \}$) with lower-bounded $\ell^2$ norm, or a ``sparse'' sub-vector $v_I$ with lower-bounded $\ell^2$ norm.
In the first case one may resort to Gaussian approximation, and in the second case one may argue that the supports of the vectors $u^*$ and $v$ are ``almost separated''.
In addition, in both cases we use the fact that the random vectors $(\innerp{u_I^*}{X_I}, \innerp{v_I}{X_I})$ and $(\innerp{u_{I^c}^*}{X_{I^c}}, \innerp{v_{I^c}}{X_{I^c}})$ are independent for any subset $I \subset \{ 1, \dots, d \}$ (since they depend on disjoint subsets of the independent variables $(X_j)_{1 \leq j \leq d}$).

Specifically, let $v \in S^{d-1}$ be arbitrary.
Without loss of generality one may assume that $|v_1| \geq \dots \geq |v_d|$.
Define $k = \min \{ 1 \leq k \leq d : \sum_{j=1}^k v_j^2 \geq 0.01 \}$ and let $I = \{ 1, \dots, k \}$.
In particular, one has $\sum_{j=1}^k v_j^2 \geq 0.01$ and $k \leq 0.01 d$, and either $k = 1$ or $\sum_{j > k} v_j^2 > 0.98$.

On the one hand, if $k > 1$, we have $\sum_{j > k} |v_j|^3 \leq |v_k| \sum_{j>k} v_j^2 \leq |v_k| \leq 1/\sqrt{k}$, since $k v_k^2 \leq \sum_{j=1}^k v_j^2 \leq 1$.
Combining this with the fact that $\sum_{j > k} v_j^2 > 0.98$, that $\sum_{j >k} (u_j^*)^2 = (d-k)/d > 0.99$ and $|\sum_{j > k} u_j^* v_j| = | \innerp{u^*}{v} - \sum_{j=1}^k u_j^* v_j| = |\sum_{j=1}^k u_j^* v_j| \leq \sqrt{(\sum_{j=1}^k (u_j^*)^2) (\sum_{j=1}^k v_j^2)} \leq \sqrt{k / d} \leq %
0.1$, applying the Berry-Esseen Gaussian approximation bound on $(\innerp{u^*_{I^c}}{X_{I^c}}, \innerp{v_{I^c}}{X_{I^c}})$ and using independence with the remaining variables, one may show that condition~\eqref{eq:twodim-orthogonal} holds with $\eta \asymp 1/\sqrt{k}$.

On the other hand, regardless of the value of $k \leq 0.01 d$, one has $\sqrt{\sum_{j=1}^k (u_j^*)^2} = \sqrt{k / d}$ while $\sum_{j=1}^k v_j^2 \geq 0.01$.
In other words, a constant fraction of the ``energy'' of the vector $v$ is supported in $I$, while if $k \ll d$ only a small fraction of the energy of $u^*$ is supported on $I$.
This ``approximate separation'' of the supports of $u^*, v$ implies that $\sum_{j=1}^k u_j^* X_j$ is very small, while $\sum_{j=1}^k v_j X_j$ fluctuates on a constant scale.
By using (one-dimensional) Gaussian approximation on $\sum_{j>k} u_j^* X_j$, conditioning and independence with $\sum_{j=1}^k u_j^* X_j, \sum_{j=1}^k v_j X_j$, and the fact that $|\sum_{j=1}^k u_j^* X_j | \lesssim \sqrt{k/d}$ with high probability, one may show that condition~\eqref{eq:twodim-orthogonal} holds with $\eta \asymp \sqrt{k/d}$.

Taking the best of the two guarantees above (depending on the value of $k = k(v)$), condition~\eqref{eq:twodim-orthogonal} holds down to $\eta \asymp \min (\sqrt{k/d}, 1/\sqrt{k}) \leq d^{-1/4}$ for any $v \in S^{d-1}$.

\subsection{Improved regularity scales in generic directions?}
\label{sec:impr-regul-scal}

We now discuss the phenomenon alluded to in Section~\ref{sec:regul-prod-meas}, namely that Assumption~\ref{ass:small-ball} holds down to a scale of $1/d$ in ``typical'' directions $u^* \in S^{d-1}$.
This is a consequence of the following result of Klartag and Sodin~\cite{klartag2012variations},
which states that for a ``typical'' vector $u = (u_1, \dots, u_d) \in S^{d-1}$, if $X = (X_1, \dots, X_d)$ has \iid coordinates then the distribution of the linear combination $\innerp{u}{X} = \sum_{j=1}^d u_j X_j$ approaches the Gaussian distribution at a rate of order $1/d$.
This rate is faster than the usual $1/\sqrt{d}$ rate from the Berry-Esseen theorem for the usual normalized sum $\innerp{u^*_d}{X} = \frac{1}{\sqrt{d}} \sum_{j=1}^d X_j$.

\begin{theorem*}[Theorem~1.1 in~\cite{klartag2012variations}]
  There exists a constant $c \geq 1$ such that the following holds.
  Let $\eps \in (0, 1/2)$ and $d \geq 1$.
  Assume that $X = (X_1, \dots, X_d)$ has independent coordinates, with $\E [X_j] = 0$ and $\E [X_j^2] = 1$ for $j = 1, \dots, d$ and with finite fourth moment.
  Let
  \begin{equation*}
    \kappa
    = \bigg( \frac{1}{d} \sum_{j=1}^d \E [X_j^4] \bigg)^{1/4}
    \, .
  \end{equation*}
  Then, there is a subset $A_\eps \subset S^{d-1}$ with $\mu_{d-1} (A_\eps) \geq 1-\eps$ (where $\mu_{d-1}$ stands for the uniform probability measure on $S^{d-1}$) such that, for any $u \in A_\eps$, one has
  \begin{equation}
    \label{eq:berry-esseen-fast}
    \sup_{a,b \in \R, \, a \leq b} \Big| \P (a \leq \innerp{u}{X} \leq b) - \frac{1}{\sqrt{2\pi}} \int_{a}^{b} e^{-s^2/2} \di s \Big|
    \leq \frac{c \log^2 (1/\eps) \kappa^4}{d}
    \, .
  \end{equation}
\end{theorem*}

This immediately implies that there exists a subset $A$ of $S^{d-1}$ with $\mu_{d-1} (A) \geq 1-1/d \to_{d\to \infty} 1$ such that, for any $u \in A$, the margin probability $\P (\ainnerp{u}{X} \leq t)$ is of order $t$ as long as $t \gtrsim \log^2 (d)/d$ (hence, Assumption~\ref{ass:small-ball} holds at least down to $\eta \asymp \log^2 (d)/d$).

The reason why a ``generic'' direction $u \in S^{d-1}$ leads to a faster rate of Gaussian approximation (and therefore a smaller scale $\eta$ for Assumption~\ref{ass:small-ball}) than $u^*_d = (1/\sqrt{d}, \dots, 1/\sqrt{d})$ is the following.
For the parameter $u^*_d$, the rate of Gaussian approximation of order $1/\sqrt{d}$ cannot be improved due to an arithmetic obstruction: if $X_1, \dots, X_d$ are \iid Bernoulli, the quantity $\innerp{u_d^*}{X} = \frac{1}{\sqrt{d}} \sum_{j=1}^d X_j$ takes values in the lattice $\Zn/\sqrt{d}$.
This is due to the strong additive structure of $u_d^*$, all of whose coefficients are equal:
hence, there are many cancellations in the sum $\innerp{u_d^*}{X} = \frac{1}{\sqrt{d}} \sum_{j=1}^d X_j$, as any two opposite signs $X_j,X_k$ cancel out.
This means that many different values of the vector $X$ lead to the same value of $\innerp{u_d^*}{X}$.
However, this arithmetic obstruction vanishes for a ``generic'' direction $u = (u_1, \dots, u_d) \in S^{d-1}$, which
is much less structured
(for instance, all ratios $u_j/u_k$ with $j \neq k$ are irrational numbers with probability $1$).

These results suggest that, for a generic parameter direction $u^* \in S^{d-1}$, the regularity conditions (Definition~\ref{def:regular}) may hold at a scale $\eta_d \ll 1/\sqrt{d}$.
However, we do not know how to prove this for the \emph{two-dimensional} margin Assumption~\ref{ass:twodim-marginals}.
Indeed, as previously discussed, a key difficulty is that the property~\eqref{eq:twodim-marginals-comp} must be established for \emph{every} direction $v \in S^{d-1}$, including those for which Gaussian approximation fails.
In addition, it is not clear how to extend our arguments in Lemma~\ref{lem:2d-margin-indep} from the case of $u^* = u_d^*$ to a generic $u^* \in S^{d-1}$ lacking additive structure,
while incorporating the $1/d$ improvement of~\cite{klartag2012variations} in this case.
We therefore leave this question as an open problem:

\begin{question}
  \label{prob:opt-regularity-iid}
  Does there exist a sequence $(\eta_d)_{d\geq 1}$ with $\sqrt{d} \cdot \eta_d \to 0$ as $d \to \infty$ such that the following holds?
  Let $X = (X_1, \dots, X_d)$ be a random vector with \iid sub-exponential coordinates (Assumption~\ref{ass:independent-coordinates} with $K \lesssim 1$), for instance a Bernoulli design.
  There exists a subset $A_{d} \subset S^{d-1}$ with $\mu_{d-1} (A_d) \to  1$ as $d \to \infty$, such that for every $u^* \in S^{d-1}$, the distribution $X$ satisfies Assumption~\ref{ass:twodim-marginals} with parameter $u^*, \eta_d$ and $c \lesssim 1$.

  In addition, does $\eta_d = 1/d$ satisfy this property?
  And what is the smallest order of magnitude of $\eta_d$ such that this property holds?
\end{question}

In short, Problem~\ref{prob:opt-regularity-iid} asks about the regularity scale of product measures (such as the Bernoulli design) in ``typical'' directions.
By Theorem~\ref{thm:regular-well-specified} and Proposition~\ref{prop:2-dim-margin-almost-necessary}, this amounts to investigating the values of the parameter norm (for typical parameter directions) for which the MLE for logistic regression behaves as in the case of a Gaussian design.

\appendix

\section{Tail conditions on real random variables%
}
\label{sec:tails-rv}

In this section, we gather some definitions and basic properties regarding tails of real valued random variables.
These are well-known that are simply recalled here to fix the constants.
We start with the definition of the sub-exponential and sub-Gaussian norms:

\begin{definition}[$\psi_\alpha$-norm]
  \label{def:psi-alpha}
  Let $\alpha > 0$.
  If $X$ is a real random variable, its \emph{$\psi_\alpha$-norm} is defined as
  \begin{equation}
    \label{eq:def-psi-alpha}
    \norm{X}_{\psi_\alpha}
    = \sup_{p \geq 2} \bigg[ \frac{2^{1/\alpha} e \norm{X}_p}{p^{1/\alpha}} \bigg]
    \in [0, + \infty]
    \, ,
  \end{equation}
  where the supremum is taken over all real values of $p \geq 2$.
  We say that $X$ is \emph{sub-exponential} if $\norm{X}_{\psi_1} < + \infty$, and \emph{sub-Gaussian} if $\norm{X}_{\psi_2} < + \infty$.
\end{definition}

We mostly consider the cases $\alpha = 1$ and $\alpha = 2$.
We refer to~\cite[\S2.5 and \S2.7]{vershynin2018high} for equivalent definitions of the $\psi_1$ and $\psi_2$-norms.

Note that the normalization in the definition~\eqref{eq:def-psi-alpha} ensures that (i) if $\E [X^2] = 1$, then $\norm{X}_{\psi_\alpha} \geq e$ and (ii) if $\alpha \leq \alpha'$, then $\norm{X}_{\psi_\alpha} \leq \norm{X}_{\psi_{\alpha'}}$.
In addition, one has $\norm{X + X'}_{\psi_\alpha} \leq \norm{X}_{\psi_\alpha} + \norm{X'}_{\psi_\alpha}$ for every real valued random variables $X, X'$ and every parameter $\alpha > 0$.

In order to obtain sharp guarantees, we need the additional notion of \emph{sub-gamma} random variables~\cite[\S2.4]{boucheron2013concentration}.

\begin{definition}[Sub-gamma random variables]
  \label{def:sub-gamma}
  Let $X$ be a %
  real valued random variable and $\sigma, K > 0$.
  We say that $X$ is \emph{$(\sigma^2,K)$-sub-gamma} if for every $\lambda \in [0, 1/K)$ one has
\begin{equation}
  \label{eq:def-sub-gamma}
  \E \exp (\lambda X)
  \leq \exp \bigg( \frac{\sigma^2 \lambda^2}{2 (1 - \lambda K)} \bigg)
  \, .
\end{equation}
\end{definition}

Recall that $X$ is said to be {centered} if $\E [X] = 0$.
The basic properties of sub-gamma and sub-exponential variables are gathered in the following lemma:

\begin{lemma}
  \label{lem:sub-gamma-exponential}
  Let $X$ be a real random variable and $\sigma, K > 0$.
  \begin{enumerate}
  \item If $\norm{X}_{\psi_\alpha} \leq K$, then for every $t \geq 1$ one has
    \begin{equation}
      \label{eq:tail-psi-alpha}
      \P \big( |X| \geq K t^{1/\alpha} \big)
      \leq e^{-2t}
      \, .
    \end{equation}
  \item If $X$ is $(\sigma^2, K)$-sub-gamma, then for every $t \geq 0$, one has
    \begin{equation}
      \label{eq:tail-subgamma}
      \P (X \geq \sigma \sqrt{2 t} + K t)
      \leq e^{-t}
      \, .
    \end{equation}
  \item If $X_1, \dots, X_n$ are independent random variables such that $X_i$ is
  $(\sigma_i^2, K_i)$-sub-gamma \textup(with $\sigma_i, K_i > 0$\textup) for every $i = 1, \dots, n$,
  then $X_1 + \dots + X_n$ is $(\sigma_1^2 + \dots + \sigma_n^2, \max (K_1, \dots, K_n))$-sub-gamma.
    Also, if $X$ is $(\sigma^2, K)$-sub-gamma and $\alpha \geq 0$, then $\alpha X$ is $(\alpha^2 \sigma^2, \alpha K)$-sub-gamma.
  \item If $X$ is centered and satisfies for every integer $p \geq 2$ that
    \begin{equation}
      \label{eq:moment-sub-gamma}
      \E [\abs{X}^p]
      \leq \sigma^2 K^{p-2} p! / 2
      \, ,
    \end{equation}
    then $X$ is $(\sigma^2, K)$-sub-gamma.
  \item If $X$ and $-X$ are $(\sigma^2, K)$-sub-gamma, then $\var (X) \leq \sigma^2$ and $\norm{X}_{\psi_1} \leq 2 \sqrt[3]{2 e} \max (\sigma, 2 K)$.
  \item If $X$ is centered, $\var (X) \leq \sigma^2$ and $\norm{X}_{\psi_1} \leq K$ \textup(where $K \geq e\sigma$\textup), then $X$ is $(\sigma^2, K \log (K/\sigma))$-sub-gamma.
    In addition, $X$ is $(K^2/2, K/2)$-sub-gamma.
  \end{enumerate}
\end{lemma}

  In particular, it follows from the last two points of Lemma~\ref{lem:sub-gamma-exponential} that if $X$ is centered and $K \geq e \sigma$, the property that $X$ (and $-X$) is $(\sigma^2, K)$-sub-gamma is closely related to the conditions $\var (X) \leq \sigma^2$ and $\norm{X}_{\psi_1} \leq K$.
  The sub-gamma condition is however slightly stronger, and allows one to gain a factor of order $\log (K/\sigma)$.
  We actually use this improvement in order to avoid additional $\log \b$ factors in the setting of Theorem~\ref{thm:gaussian-well-specified}.

\begin{proof}
  For the first point, for any $p \geq 2$, Markov's inequality implies that
  \begin{equation*}
    \P (|X| \geq e \norm{X}_p)
    = \P \big( |X|^p \geq e^p \norm{X}_p^p \big)
    \leq \frac{\E [ |X|^p ]}{e^p \norm{X}_p^p}
    = e^{-p}
    \, .
  \end{equation*}
  Letting $p = 2 t$ and bounding $e\norm{X}_p \leq \norm{X}_{\psi_{\alpha}} (p/2)^{1/\alpha} \leq K t^{1/\alpha}$ concludes.
  The second point is established in~\cite[p.~29]{boucheron2013concentration}.
  The third point follows from 
  the fact that, by independence,
  \begin{equation*}
    \E \big[ e^{\lambda (X_1 + \dots + X_n )} \big]
    = \E [ e^{\lambda X_1} ] \cdots \E [ e^{\lambda X_n} ]
    \, .
  \end{equation*}

  We now turn to the fourth point.
  For every $\lambda \in [0, 1/K)$, one has
  \begin{align*}
    \E e^{\lambda X}
    &\leq 1 + \lambda \E [X] + \sum_{p \geq 2} \frac{\lambda^p\, \E [ |X|^p ]}{p!} 
      \leq 1 + \frac{\sigma^2 \lambda^2}{2} \sum_{p \geq 2} \frac{\lambda^{p-2} K^{p-2} p!}{p!} \\
    &= 1 + \frac{\sigma^2 \lambda^2}{2 (1- \lambda K)} 
      \leq \exp \bigg( \frac{\sigma^2 \lambda^2}{2 (1-\lambda K)} \bigg)
      \, .
  \end{align*}

  For the fifth point, we first note that, as $\E [ e^{|X|/(2K)} ] \leq \E [ e^{X/(2K)} ] + \E [ e^{-X/(2K)} ] < \infty$, by dominated convergence the function $\phi : \lambda \mapsto \log \E [ e^{\lambda X} ]$ is well-defined and twice continuously differentiable over $(-1/(2K), 1/(2K))$, with $\phi (0) = 0$, $\phi' (0) = \E [X]$ and $\phi''(0) = \var (X)$.
  Hence, $\phi (\lambda) = \E [X] \lambda + \var (X) \lambda^2/2 + o(\lambda^2)$ as $\lambda \to 0$, and by assumption one has $\phi(\lambda) \leq \frac{\sigma^2\lambda^2}{2 (1-\lambda K)} = \sigma^2 \lambda^2/2 + o(\lambda^2)$, hence $\E [X] = 0$ and $\var (X) \leq \sigma^2$.
  Next, in order to bound $\norm{X}_{\psi_1}$, we
  apply the sub-gamma condition~\eqref{eq:def-sub-gamma} to $\lambda = 1/(\sigma \vee 2K)$, which gives:
  \begin{align*}
    \E \bigg[ \exp \bigg( \frac{|X|}{\sigma \vee 2 K} \bigg) \indic{X \geq 0} \bigg]
    \leq \E \bigg[ \exp \bigg( \frac{X}{\sigma \vee 2 K} \bigg) \bigg]
    \leq \E \bigg[ \exp \bigg( \frac{\sigma^2 / \sigma^2}{2 (1 - K / (2K))} \bigg) \bigg] = e \, .
  \end{align*}
  Applying the same inequality to $-X$ and summing gives:
  \begin{equation*}
    \E \bigg[ \exp \bigg( \frac{|X|}{\sigma \vee 2 K} \bigg) \bigg]
    \leq 2 e
    \, .
  \end{equation*}
  Now, a simple analysis of function shows that $e^{u} \geq e u$ for any $u \geq 0$, hence (applying this to $u/p$) $\big( \frac{e u}{p} \big)^p \leq e^u$.
  Hence, for any $p \geq 3$, one has
  \begin{align*}
    \E \bigg[ \bigg( \frac{e |X|}{p (\sigma \vee 2K)} \bigg)^p \bigg]
    \leq \E \bigg[ \exp \bigg( \frac{|X|}{\sigma \vee 2 K} \bigg) \bigg]
    \leq 2 e \, ,
  \end{align*}
  so that $2 e \norm{X}_p / p \leq 2 (2 e)^{1/p} (\sigma \vee 2K) \leq 2 \sqrt[3]{2 e} (\sigma \vee 2 K)$, which proves the desired bound since we also have $2e \norm{X}_2/2 \leq e \sigma \leq 2 \sqrt[3]{2e} \sigma$.

  Let us now establish the sixth point.
  For every $p>2$, one has for $r > 1$, using Hölder's inequality:
  \begin{align*}
    \E [ |X|^p ]
    & = \E [ |X|^{2(1-1/r)} |X|^{p-2+2/r} ]  
     \leq \E [X^2]^{1-1/r} \E [ X^{(p-2)r+2} ]^{1/r} \\
    & \leq \sigma^{2-2/r} \norm{X}_{(p-2)r + 2}^{[(p-2)r + 2]/r}  \\
    & \leq \sigma^{2-2/r} \Big[ \frac{[(p-2)r + 2] K}{2 e} \Big]^{p-2+2/r} \\
    & = \sigma^2 \Big( \frac{K r}{2} \Big)^{p-2} \Big( \frac{r}{2} \Big)^{2/r} \Big( \frac{K}{\sigma} \Big)^{2/r} \Big( \frac{p-2 + 2/r}{e} \Big)^{p-2 + 2/r}
      \, .
  \end{align*}
  Now let $r /2 = \log (K/\sigma) \geq 1$,
  so that $(K/\sigma)^{2/r} = e$.
  A direct analysis shows that the function $u \mapsto (u/e)^u$ increases on $[1, + \infty)$, and since $r/2 \geq 1$ one has $1 \leq p-2 \leq p-2+2/r \leq p-1$.
  Hence, for any integer $p\geqslant 3$,
  \begin{equation*}
    \Big( \frac{p-2 + 2/r}{e} \Big)^{p-2 + 2/r}
    \leq \Big( \frac{p-1}{e} \Big)^{p-1}
    \leq (2 \pi (p-1))^{-1/2} (p-1)!
    \leq p!/(6 \sqrt{\pi})
    \, ,
  \end{equation*}
  where we used the standard Stirling-type inequalities
  \begin{equation}
    \label{eq:stirling-inequalities}
    \sqrt{2\pi p} \Big( \frac{p}{e} \Big)^p
    \leq p!
    \leq p^p
    \, .
  \end{equation}
  In addition $t^{1/t} \leq e^{1/e}$ for $t>0$, so $(r/2)^{2/r} \leq e^{1/e}$.
  Combining the previous inequalities, we obtain
  \begin{align}
    \label{eq:moment-sub-exp}
    \E [ |X|^p ]
    & \leq \sigma^2 \big( K \log (K/\sigma) \big)^{p-2} e^{1+1/e} p!/(6 \sqrt{\pi}) \nonumber \\
    & \leq \sigma^2 \Big( K \log \Big( \frac{K}{\sigma} \Big) \Big)^{p-2} p!/2
      \, ,
  \end{align}
  where we used that $e^{1+1/e}/(3 \sqrt{\pi}) = 0.738\dots \leq 1$.
  By the fourth point, this implies that
  $X$ is $(\sigma^2, K \log (K/\sigma))$-sub-gamma.
  For the last statement, using the inequality $(\frac{p}{e})^p \leq p!$ for $p \geq 2$, we obtain
  \begin{equation}
    \label{eq:sub-gamma-subexp}
    \E [|X|^p]
    \leq \Big( \frac{K p}{2 e} \Big)^p
    \leq \Big( \frac{K}{2} \Big)^p p!
    = \frac{1}{2} \frac{K^2}{2} \Big( \frac{K}{2} \Big)^{p-2} p!
    \, ,
  \end{equation}
  so by the fourth point $X$ is $(K^2/2, K/2)$-sub-gamma.
\end{proof}

Finally, we will also use the following consequence of Bennett's inequality.

\begin{lemma}
\label{lem:bennett-mgf}
Let $X$ be a random variable such that $\E [X^{2}] \leq \sigma^{2}$ and $X \leq b$ almost surely, for some $\sigma^{2}>0$ and $b>0$. Then 
\begin{enumerate}
\item $X-\E [X]$ is $(\sigma^{2}, b/3)$-sub-gamma.
\item For all $\lambda \in [0, b^{-1}]$, $\log\E e^{\lambda X} \leq \lambda \E [X] + \sigma^{2}/b^{2}$.
\end{enumerate} 
\end{lemma}

\begin{proof}
By homogeneity we assume that $b=1$. Let $X'=X - \E [X]$. Using Bennett's inequality \cite[Theorem~2.9]{boucheron2013concentration}, one has, for all $\lambda \geq 0$,
\begin{equation}
\label{eq:log-mgf-bennett}
	\log\E e^{\lambda X'} \leq \sigma^{2} \phi(\lambda)\, , 
	\quad \phi(\lambda) = e^{\lambda} - \lambda - 1 \, .
\end{equation}
Moreover, for every $\lambda \in [0, 1/3)$
\begin{equation*}
  \phi(\lambda)
  = \sum_{k\geq 2} \frac{\lambda^{k}}{k!}
  =\frac{\lambda^{2}}{2} \sum_{k\geq 0} \frac{\lambda^{k}}{(k+2)!/2}
  \leq \frac{\lambda^{2}}{2} \sum_{k\geq 0} \frac{\lambda^{k}}{3^{k}}
  =\frac{\lambda^{2}}{2(1-\lambda/3)}
  \, ,
\end{equation*}
where we used that
$(k+2)!/2 = \prod_{j=3}^{k+2} j \geq 3^k$ for $k \geq 1$.
The first point is proved. For the second point, we start from \eqref{eq:log-mgf-bennett} and use that $\phi(\lambda)\leq \phi (1) = e-2\leq 1$ for all $\lambda \in [0,1]$.
\end{proof}

\section{Polar coordinates and spherical caps%
}
\label{sec:polar-coord-spher}

	\subsection{Polar %
  		coordinates}
	\label{sec:polar-cart-param}

Depending on the situation, it may be more convenient to express the position of $\theta$ relative to $\theta^*$ (with direction $u^* = \theta^* / \norm{\theta^*}$) in either of the following two equivalent ways: (1) in terms of the component $\innerp{u^*}{\theta}$ parallel to $u^*$ and of the orthogonal component $\theta - \innerp{u^*}{\theta} u^*$; or (2) in terms of the norm $\norm{\theta}$ and of the direction $u = \theta / \norm{\theta}$.
The following lemma gathers inequalities relating the two representations.
\begin{lemma}
  \label{lem:polar-cartesian}
  Let $\theta, \theta^* \in \R^d$, and set $u = \theta / \norm{\theta}$, $u^* = \theta^*/\norm{\theta^*} \in S^{d-1}$ and $\theta_\perp = \theta - \innerp{u^*}{\theta} u^*$.
  \begin{enumerate}
  \item If $\innerp{u^*}{u} \geq 0$, then 
    \begin{equation}
      \label{eq:orthogonal-equiv-polar}
      \frac{\norm{u - u^*}}{\sqrt{2}}
      \leq \frac{\norm{\theta_\perp}}{\norm{\theta}}
      =  \sqrt{1 - \innerp{u}{u^*}^2}
      \leq \norm{u - u^*}
      \, .
    \end{equation}
  \item If $\norm{u - u^*} \leq 1$, then $\norm{\theta}/2 \leq \innerp{u^*}{\theta} \leq \norm{\theta}$.
  \item One has
    \begin{equation}
      \label{eq:parallel-from-polar}
      \ainnerp{u^*}{\theta - \theta^*}
      \leq \big| \norm{\theta} - \norm{\theta^*} \big| + \norm{\theta^*}
      \cdot \frac{\norm{u - u^*}^2}{2}
      \, .
    \end{equation}
  \item One has
    \begin{equation}
      \label{eq:norm-from-cartesian}
      \big| \norm{\theta} - \norm{\theta^*} \big|
      \leq | \innerp{u^*}{\theta - \theta^*} | + \frac{\norm{\theta_\perp}^2}{\norm{\theta} + \norm{\theta^*}}
      \, .
    \end{equation}
  \end{enumerate}
\end{lemma}

\begin{proof}
  We start with the first point.
  By orthogonality,
  \begin{align}
    \norm{\theta_\perp}^2 
    &= \norm{\theta}^2 - \innerp{u^*}{\theta}^2 
    = \norm{\theta}^2 \big[ 1 - \innerp{u^*}{u}^2 \big] \label{eq:polar-orth} \\
    &= \norm{\theta}^2 \big[ 1 - \innerp{u^*}{u} \big] \big[ 1 + \innerp{u^*}{u} \big] 
    = \frac{1}{2} \norm{\theta}^2 \norm{u - u^*}^2 \big[ 1 + \innerp{u^*}{u} \big] \nonumber
      \, .
  \end{align}
  Hence, if $\innerp{u^*}{u} \geq 0$, then
  \begin{equation*}
    \frac{1}{2} \norm{\theta}^2 \norm{u - u^*}^2
    \leq \norm{\theta_\perp}^2 %
    \leq \norm{\theta}^2 \norm{u - u^*}^2
    \, ,
  \end{equation*}
  which together with the identity~\eqref{eq:polar-orth} proves the first claim.
  The second point follows from the fact that
  \begin{equation*}
    \frac{\innerp{u^*}{\theta}}{\norm{\theta}}
    = \innerp{u}{u^*}
    = 1 - \frac{1}{2} \norm{u - u^*}^2
    \in \Big[ \frac{1}{2}, 1 \Big]
    \, .
  \end{equation*}

  We now turn to the third point.
  Since $\innerp{u^*}{\theta^*} = \norm{\theta^*}$, we have
  \begin{align*}
    \ainnerp{u^*}{\theta - \theta^*}
    = \big| \norm{\theta} \innerp{u^*}{u} - \norm{\theta^*} \big| 
    &\leq \big| ( \norm{\theta} - \norm{\theta^*} ) \innerp{u^*}{u} \big| + \big| \norm{\theta^*} (\innerp{u^*}{u} - 1) \big| \\
    &\leq \big| \norm{\theta} - \norm{\theta^*} \big| + \norm{\theta^*} \cdot \frac{\norm{u - u^*}^2}{2}
      \, ,
  \end{align*}
  where we used that $\norm{u - u^*}^2 = 2 (1 - \innerp{u}{u^*})$.
  For the fourth point, note that
  \begin{align*}
    \big( \norm{\theta} + \norm{\theta^*} \big) \big| \norm{\theta} - \norm{\theta^*} \big|
    &= \big| \norm{\theta}^2 - \norm{\theta^*}^2 \big| \\
    &= \big| \norm{\theta_\perp}^2 + \innerp{u^*}{\theta}^2 - \innerp{u^*}{\theta^*}^2 \big| \\
    &\leq \norm{\theta_\perp}^2 + \ainnerp{u^*}{\theta + \theta^*} \cdot \ainnerp{u^*}{\theta - \theta^*} \\
    &\leq \norm{\theta_\perp}^2 + (\norm{\theta} + \norm{\theta^*}) \cdot \ainnerp{u^*}{\theta - \theta^*}
      \, ;
  \end{align*}
  dividing by $\norm{\theta} + \norm{\theta^*}$ gives the claimed inequality.
\end{proof}

We also use repeatedly the following result, which controls the behavior of the norm and direction within $H$-ellipsoids.

\begin{lemma}
\label{lem:ellipsoids}
Let $\theta^*\in\R^d$ be such that $\b=\|\theta^{*}\| \geq e$, set $H= \b^{-3} u^{*}{u^{*}}^{\top}
+ \b^{-1} (I_{d} - u^{*}{u^{*}}^{\top} )$ where $u^* = \theta^*/\norm{\theta^*}$ and let $r\in (0,1)$. 
Then, for any $\theta \in \R^{d}$ such that
$\|\theta-\theta^{*}\|_{H}\leq r/\sqrt{\b}$, letting $u=\theta/\|\theta\|$, one has:
\begin{enumerate}
\item $(1-r)\b \leq \|\theta\| \leq (1+r) \b$;
\item $\frac{\|\theta-\innerp{u^{*}}{\theta}u^{*}\|}{\|\theta\|}=\|u-\innerp{u^*}{u}u^*\| \leq \frac{r}{(1-r)\b}$;
\item $\| u - u^{*}\| \leq \frac{\sqrt{2} r}{(1-r)\b}$.
\end{enumerate}
\end{lemma}

\begin{proof}
  For the first point, using that $H \mgeq \b^{-3} I_d$ we obtain
  \begin{equation*}
    \abs[\big]{\norm{\theta} - \b}
    = \abs[\big]{\norm{\theta} - \norm{\theta^*}}
    \leq \norm{\theta - \theta^*}
    \leq \b^{3/2} \norm{\theta - \theta^*}_H
    \leq r \b
  \end{equation*}
  as desired.
  For the second point, note that the constraint $\|\theta-\theta^{*}\|_{H}\leq r/\sqrt{\b}$ writes
\begin{equation}\label{eq:constDev}
  \frac{(\innerp{u^*}{\theta}-\b)^2}{\b^3} + \frac{\|\theta - \innerp{u^*}{\theta} u^*\|^2}{\b}
  = \frac{\innerp{u^*}{\theta - \theta^*}^2}{\b^3} + \frac{\norm{\theta}^2 - \innerp{u^*}{\theta}^2}{\b}
  \leqslant \frac{r^2}{\b}\enspace.
\end{equation}
Hence $\|\theta-\innerp{u^*}{\theta}u^*\|\leqslant r$, and from the first point $\|\theta\|\geqslant (1-r)\b$, proving the second point.

For the last point, first note that, as $|\innerp{u^*}{\theta}-\b|\leq r\b$, we have $\innerp{u^*}{u}\geq 0$.
Thus $\innerp{u}{u^*} \in [0, 1]$ and hence $\innerp{u}{u^*}^2 \leq \innerp{u}{u^*}$.
This gives
\begin{align*}
  \|u-u^*\|^2 = 2(1-\innerp{u}{u^*})
  \leqslant 2(1-\innerp{u}{u^*}^2)
  = \frac{2}{\norm{\theta}^2} \parens*{\norm{\theta}^2 - \innerp{u^*}{\theta}^2}
  \leq \frac{2 r^2}{\norm{\theta}^2}
  \, ,
\end{align*}
where the last inequality comes from~\eqref{eq:constDev}.
The proof is concluded by Point 1.
\end{proof}

\subsection{Spherical caps}	
\label{sec:spherical-caps}
In Section~\ref{sec:hessian-gaussian}, we defined spherical caps as follows: for any $u \in S^{d-1}$ and $\eps \in [0,1]$,
\begin{equation}
  \mc(u, \eps) = \big\{v\in S^{d-1} : \innerp{u}{v} \geq \sqrt{1 - \eps^2} \big\} \, ,
\end{equation}
Spherical caps can be equivalently defined using the Euclidean distance by
\begin{equation}
\label{eq:sphere-caps-euclid}
	\wt \mc (u, r) = \big\{ v \in S^{d-1}, \, \|u-v\| \leq r \big\} \, ,
\end{equation}
as shown in the following fact.

\begin{fact}
\label{fact:angle-radius-correspondance}
For every $\eps \in [0,1]$, $\mc (u, \eps) = \wt \mc(u, r_{\eps})$, where
$r_{\eps} = \sqrt{2\big(1-\sqrt{1-\eps^{2}}\big)}$. Moreover, it holds that
$\eps \leq r_{\eps} \leq \sqrt{2} \eps$ and thus
\begin{equation}
\label{eq:sphere-caps-equivalence}
	\mc (u, \eps/\sqrt{2}) \subset \wt \mc(u,\eps) \subset \mc (u,\eps)\, .
\end{equation}
\end{fact}

\begin{proof}
  The equivalence comes from the fact that, for any $u,v \in S^{d-1}$, 
  \begin{equation*}
    \|u-v\|^{2}
    = 2\big(1 - \langle u, v \rangle \big)
    \, .
  \end{equation*}
In addition, by concavity, for all $t\in [0,1]$, $1-t \leq \sqrt{1-t} \leq 1- t/2$, which implies the claimed inequalities and inclusions.
\end{proof}	

\section{%
Proof of Proposition~\ref{prop:2-dim-margin-almost-necessary}%
}
\label{sec:proof-2dim-necessary}
In this section we provide the proof of Proposition~\ref{prop:2-dim-margin-almost-necessary} from Section~\ref{sec:assumptions}, regarding the necessity of the two-dimensional margin assumption.

\begin{proof}[Proof of Proposition~\ref{prop:2-dim-margin-almost-necessary}]
  Recall that $H_X (\theta^*) = \E [ \sigma' (\innerp{\theta^*}{X}) X X^\top ]$, and fix $v \in S^{d-1}$.
  For every $c_1 \geq 1$, we have
  \begin{align*}
    \innerp{H_X (\theta^*) v}{v}
    &= \E \big[ \sigma' (\b \innerp{u^*}{X}) \innerp{v}{X}^2 \big] \\
    &= \E \Big[ \sigma' (\b \innerp{u^*}{X}) \innerp{v}{X}^2 \bm 1 \Big( \ainnerp{u^*}{X} > \frac{c_1 \log \b}{\b} \Big) \Big] +  \\
    & \quad + \E \Big[ \sigma' (\b \innerp{u^*}{X}) \innerp{v}{X}^2 \bm 1 \Big( \ainnerp{u^*}{X} > \frac{c_1 \log \b}{\b} \Big) \Big] \\
    &\leq \sigma' (c_1 \log \b ) \cdot \E [ \innerp{v}{X}^2 ] + \underbrace{\E \Big[ \innerp{v}{X}^2 \bm 1 \Big( \ainnerp{u^*}{X} \leq \frac{c_1 \log \b}{\b} \Big) \Big]}_{:= \mathrm{I}} \\
    &\leq \b^{-c_1} + \mathrm{I}
      \, ,
  \end{align*}
  where we used that $\sigma'(t) = (1+e^t)^{-1}(1+e^{-t})^{-1} \leq e^{-|t|}$ for any $t \in \R$.
  Now, set $m = \max \set{\b^{-1}, \norm{u^*-v}}$.
  We proceed to bounding term $\mathrm{I}$ above: for any $C_2 > 0$,
  \begin{align*}
    \mathrm{I}
    &= \E \Big[ \innerp{v}{X}^2 \bm 1 \Big( \ainnerp{u^*}{X} \leq \frac{c_1 \log \b}{\b} \Big) \Big] \\
    &\leq \E \Big[ \Big( \frac{m}{C_2} \Big)^2 \bm 1 \Big( \ainnerp{u^*}{X} \leq \frac{c_1 \log \b}{\b} \Big) \Big] + \E \Big[ \innerp{v}{X}^2 \bm 1 \Big( \ainnerp{u^*}{X} \leq \frac{c_1 \log \b}{\b} ; \ainnerp{v}{X} \geq \frac{m}{C_2} \Big) \Big] \\
    &= \mathrm{II} + \mathrm{III}
      \, .
  \end{align*}
  To control term $\mathrm{II}$, we use Assumption~\ref{ass:small-ball} (and the fact that $c_1 \log \b \geq 1$) to obtain
  \begin{align*}
    \mathrm{II}
    &= \Big( \frac{m}{C_2} \Big)^2 \, \P \Big( \ainnerp{u^*}{X} \leq \frac{c_1 \log \b}{\b} \Big)
      \leq \Big( \frac{m}{C_2} \Big)^2 \cdot \frac{c \, c_1 \log \b}{\b}
      \, .
  \end{align*}
  Next, we control term $\mathrm{III}$ by further decomposing it as follows: for $\lambda \geq m/C_2$,
  \begin{align*}
    \mathrm{III}
    &= \E \Big[ \innerp{v}{X}^2 \bm 1 \Big( \ainnerp{u^*}{X} \leq \frac{c_1 \log \b}{\b} ; \frac{m}{C_2} \leq \ainnerp{v}{X} \leq \lambda \Big) \Big] + \\
    &\quad + \E \Big[ \innerp{v}{X}^2 \bm 1 \Big( \ainnerp{u^*}{X} \leq \frac{c_1 \log \b}{\b} ; \ainnerp{v}{X} > \lambda \Big) \Big] %
    \\
    &\leq \lambda^2 \P \Big( \ainnerp{u^*}{X} \leq \frac{c_1 \log \b}{\b} ; \ainnerp{v}{X} \geq \frac{m}{C_2} \Big) + \mathrm{IV}
      \, .
  \end{align*}
  To bound term $\mathrm{IV}$, let $\eta = c_1 \log (\b)/\b$ and write $v = u^* + \norm{u^* - v} w$ with $w \in S^{d-1}$.
  Then, note that under the event $\set{\ainnerp{u^*}{X} \leq \eta}$, one has
  \begin{equation*}
    \ainnerp{v}{X}
    \leq \ainnerp{u^*}{X} + \norm{u^* - v} \cdot \ainnerp{w}{X}
    \leq \eta + m \ainnerp{w}{X}
    \, .
  \end{equation*}
  In particular, $\set{\ainnerp{u^*}{X} \leq \eta, \ainnerp{v}{X} \geq \lambda} \subset \set{\ainnerp{w}{X} \geq (\lambda - \eta)/m}$.
  Plugging these inequalities into the definition of $\mathrm{IV}$ gives
  \begin{align*}
    \mathrm{IV}
    &\leq \E \Big[ \innerp{v}{X}^2 \bm 1 \Big( \ainnerp{w}{X} \geq \frac{\lambda - \eta}{m} \Big) \Big]
      \, .
  \end{align*}
  Now, set $\lambda = 2 K c_1 \log (\b) m$.
  One has $\lambda \geq 2 c_1 \log (\b)/\b = 2\eta$, thus $(\lambda - \eta)/m \geq \lambda/(2m) = K c_1 \log (\b)$.
  Using this together with the Cauchy-Schwarz inequality, the fact that (by Assumption~\ref{ass:sub-exponential}) $\norm{\innerp{v}{X}}_{\psi_1}, \norm{\innerp{w}{X}}_{\psi_1} \leq K$, the first point of Lemma~\ref{lem:sub-gamma-exponential} and Definition~\ref{def:psi-alpha}, we obtain
  \begin{align*}
    \mathrm{IV}
    &\leq \E \Big[ \innerp{v}{X}^2 \bm 1 \Big( \ainnerp{w}{X} \geq K c_1 \log (\b) \Big) \Big] 
      \leq \sqrt{\E \big[\innerp{v}{X}^4 \big] \, \P \Big( \ainnerp{w}{X} \geq K c_1 \log (\b) \Big)} \\
    &\leq \sqrt{\Big( \frac{4 K}{2 e} \Big)^4 e^{-2 c_1 \log (\b)}}
      \leq %
      K^2 \b^{-c_1}
      \, .
  \end{align*}
  Putting together the previous inequalities and substituting for the definition of $\lambda$, we get
  \begin{align}
    \innerp{H_X (\theta^*) v}{v} 
    &\leq \b^{-c_1} + \Big( \frac{m}{C_2} \Big)^2 \cdot \frac{c \, c_1 \log \b}{\b} + K^2 \b^{-c_1} + \nonumber \\
    &\quad + \big( 2 K c_1 \log (\b) m \big)^2 \, \P \Big( \ainnerp{u^*}{X} \leq \frac{c_1 \log \b}{\b} ; \ainnerp{v}{X} \geq \frac{m}{C_2} \Big)
      \label{eq:proof-necessity-bound-hx}
      \, .
  \end{align}
  We now choose $c_1 \geq 1$ and then $C_2 > 0$ in order to further control the right-hand side of~\eqref{eq:proof-necessity-bound-hx}.
  The sum of the first and third term can be bounded as follows, using that $\b \geq e$:
  \begin{align*}
    (K^2 + 1) \b^{-c_1}
    \leq \b^{-3} (K^2 + 1) e^{- (c_1 - 3)}
    = \frac{c_0}{6 \b^3}
    \leq \frac{c_0 m^2}{6 \b}
    \, ,
  \end{align*}
  where the equality is obtained by choosing $c_1 = 3 + \log (6 (K^2 +1) / c_0) \geq 1$.
  Next, the second term in the right-hand side of~\eqref{eq:proof-necessity-bound-hx} equals
  \begin{align*}
    \Big( \frac{m}{C_2} \Big)^2 \cdot \frac{c \, c_1 \log \b}{\b}
    = \frac{c_0 m^2}{6 \b}
  \end{align*}
  when choosing $C_2 = c_2 \sqrt{\log \b}$ with $c_2 = \sqrt{6 c c_1/c_0}$.
  
  On the other hand, we have from Remark~\ref{rem:equivalent-two-dim} that
  \begin{equation*}
    m^2
    = \max \big\{ \b^{-2}, \norm{u^* - v}^2 \big\}
    \leq \max \big\{ \b^{-2}, 2 (1 - \innerp{u^*}{v}^2) \big\}
    \leq 2 \Big[ \b^{-2} \innerp{u^*}{v}^2 + (1 - \innerp{u^*}{v}^2) \Big]
  \end{equation*}
  where the last inequality uses that $\b^{-2} \leq 1$, hence either $\innerp{u^*}{v}^2 \leq 1/2$, and then the maximum equals $2 (1 - \innerp{u^*}{v}^2)$, or $\innerp{u^*}{v}^2 \geq 1/2$, and then $\b^{-2} \leq 2 \b^{-2} \innerp{u^*}{v}^2$.
  Recalling the expression~\eqref{eq:defH} of $H$, this implies that
  \begin{equation}
    \label{eq:proof-necessity-m-h}
    m^2 \leq 2 \b \innerp{H v}{v}
    \, .
  \end{equation}
  Hence, the assumption that $H_X (\theta^*) \mgeq c_0 H$ implies that $\innerp{H_X (\theta^*) v}{v} \geq c_0 m^2 / (2 \b)$.
  Plugging the previous inequalities into~\eqref{eq:proof-necessity-bound-hx}
  gives
  \begin{align*}
    \frac{c_0 m^2}{2 \b}
    \leq \frac{c_0 m^2}{3 \b} + 4 K^2 c_1^2 \log^2 (\b) m^2 \, \P \Big( \ainnerp{u^*}{X} \leq \frac{c_1 \log \b}{\b} ; \ainnerp{v}{X} \geq \frac{m}{c_2 \sqrt{\log \b}} \Big)
    \, .
  \end{align*}
  Re-arranging, substituting for the value of $m$ and letting $c_3 = c_0 / (24 K^2 c_1^2)$, this gives
  \begin{equation*}
    \P \Big( \ainnerp{u^*}{X} \leq \frac{c_1 \log \b}{\b} ; \ainnerp{v}{X} \geq \frac{\max \{ \norm{u^* - v}, \b^{-1} \}}{c_2 \sqrt{\log \b}} \Big)
    \geq \frac{c_3}{\b \log^2 (\b)}
    \, ,
  \end{equation*}
  which concludes the proof.
\end{proof}

\parag{Acknowledgements}

This research is supported by a grant of the French National Research Agency (ANR), ``Investissements d’Avenir'' (LabEx Ecodec/ANR-11-LABX-0047).

% {%
%   \bibliography{biblio-logistic-abbrv}
%   \bibliographystyle{alpha}
% }

\newcommand{\etalchar}[1]{$^{#1}$}

\end{document}